%% file: main.tex
\documentclass[11pt]{amsart} 
\input{preamble}

\title[Invariants of Legendrian knots in thickened convex surfaces]{Invariants of Legendrian knots in thickened convex surfaces}

\author[Nancy Mae Eagles]{Nancy Mae Eagles}
\address{Department of Mathematics\\University of California at Berkeley\\Berkeley, CA\\94703\\USA}
\email{nm.eagles@berkeley.edu}

\author[Zijian Rong]{Zijian Rong}
\address{Department of Mathematics\\University of Southern California\\Los Angeles, CA\\90007\\USA}
\email{zijianro@usc.edu}
\date{\today}

\begin{document}

\begin{abstract}
    We define a differential graded algebra associated to Legendrian knots in thickened convex surfaces $\Sigma\times \R$. The algebra is defined in the same spirit as the Chekanov-Eliashberg DGA for Legendrians in $\R^3$, but makes use of the data of the dividing set $\Gamma$ of $\Sigma$. The algebra is generated by countably many Reeb chords of the Legendrian $\Lambda$, and its differential counts certain immersed polygons in the projection $\pi:\Sigma\times \R\to \Sigma\times \{0\}$ with boundary on $\pi(\Lambda)\cup \Gamma$. We show that the differential squares to zero and that the stable tame isomorphism type of the DGA is invariant under Legendrian isotopy. Finally, we compute several examples and use the invariant to distinguish Legendrian knots in thickened convex surfaces that cannot be distinguished by the classical invariants.
\end{abstract}

\maketitle

\tableofcontents

\section{Introduction}

Let $(M,\xi)$ be a contact $3$-manifold with contact form $\lambda$. An embedding $\Lambda:S^1\hookrightarrow M$ is a \textit{Legendrian knot} if $T_p\Lambda\subset \xi_p$ for all $p\in \Lambda(S^1)$. That is, $\Lambda$ is everywhere tangent to the contact structure. A smooth map $\varphi:S^1\times \R\to M$ is a \textit{Legendrian isotopy} if for each $t\in I$, $\varphi(\cdot, t):S^1\to M$ is a Legendrian knot. A powerful tool to study Legendrian knots in certain contact manifolds is Legendrian contact homology, developed by Chekanov in \cite{chekanov2002dga} and independently by Eliashberg, Givental, and Hofer in \cite{egh2000SFT}. The homology theory falls under the ever-expanding framework of symplectic field theory, which comprises invariants built out of pseudoholomorphic curves in certain non-compact symplectic manifolds like the symplectization $\R\times M$. As initially proposed in \cite{egh2000SFT}, under suitable circumstances one can associate to a knot $\Lambda\subset M$ a differential graded algebra $A(\Lambda)$ generated by \textit{Reeb chords} of the Legendrian knot. Its differential counts certain J-holomorphic boundary punctured disks in $\R\times M$ with asymptotic ends at Reeb chords. Some analytical care must be taken to assure that one can actually \textit{count} such disks. Following the standard setup of Floer theory, one hopes to describe the collection of $J$-holomorphic disks as a $0$-dimensional compact manifold, seen as the solution space of some appropriate Cauchy-Riemann equation with Lagrangian boundary conditions. In the event that the homology theory is well-defined, though the DGA itself is not a Legendrian isotopy invariant, its \textit{stable tame isomorphism class} is. Consequently, the \textit{Legendrian contact homology} $H_*(A(\Lambda),\partial)$ is a Legendrian isotopy invariant up to isomorphism. 

Well-definedness and computability of Legendrian contact homology have been achieved in several settings, the first and simplest of which is in the standard contact $\R^3$. Legendrian contact homology has also been defined in circle bundles over Riemann surfaces with negative curvature \cite{sabloff2003circle}, in $J^1(S^1)$ \cite{NgTraynor2004}, in $\R^{2n+1}$ \cite{ekholm2005R2n+1}, in $P\times \R$ where $P$ is an exact symplectic manifold \cite{ekholm2007PxR} \cite{Bjorklund2016}, in lens spaces \cite{licata2011lensspaces}, in Seifert fibered spaces with a transverse contact structure \cite{licata2013seifert}, in connected sums of $S^1 \times S^2$ viewed as boundaries of subcritical Weinstein $4$-manifolds \cite{ekholmNg2015weinstein}, and in open book decompositions with once-punctured torus pages \cite{Brand2023}. As by its original definition in \cite{chekanov2002dga}, one can avoid certain analytic technicalities in $\R^{3}$ entirely and describe the algebra combinatorially using the knot’s image under the \textit{Lagrangian projection} $\pi:\R^3\to \R^2; (x,y,z)\mapsto (x,y)$. The key observation is that a combinatorial model arises whenever the Reeb flow admits a projection to a fibration over a suitable two–dimensional complex base. Under such a projection, one may identify Reeb chords with double points of the knot, and holomorphic disks with certain rigid polygons. This is proved in \cite{ENS2002}.

\subsection{Main results}
Let $(\Sigma \times \R, \xi_\Gamma)$ be a thickened closed convex surface, i.e. $\Sigma \times \R_t$ is a contact manifold such that $\frac{\partial}{\partial t}$ is a contact vector field and $\Gamma = \{p \in \Sigma : (\frac{\partial}{\partial t})_p \in (\xi_\Gamma)_p\}$ is a union of circles dividing $\Sigma$ into two surfaces with boundary $\Sigma_\pm$. See \Cref{sec:background} for a review of convex surface theory. Assume $(\Sigma\times \R,\xi_\Gamma)$ is hypertight and $c_1(\xi_\Gamma) = \chi(\Sigma_+) - \chi(\Sigma_-) = 0$. Let $\Lambda \subset \Sigma \times \R$ be null-homologous Legendrian knot.  
In this paper, we prove the following theorem:

\begin{theorem}\label{thm: main 1} \it
    To $\Lambda \subset (\Sigma \times \R, \xi_\Gamma)$ one can associate a combinatorially defined differential graded algebra $(\widehat{\mathcal{A}},|\cdot|,\widehat{\partial})$, whose stable tame isomorphism type is invariant under Legendrian isotopy.
\end{theorem}

Our invariant does not comprise the full algebra of \cite{egh2000SFT} but can be thought of as some variant of a quotient of the full SFT algebra. When the Legendrian knot has a representative whose projection onto the convex surface is disjoint from $\Gamma$ (i.e. it is entirely contained in $\Sigma_+$ or $\Sigma_-$), our definition agrees with that of \cite{ekholm2007PxR}. Our definition is motivated in part by \cite{ekholm2016nonloose}, particularly Lemma 2.3. In \cite{avdek2023algebraicgiroux} the sutured contact homology (as defined in \cite{CGHH2011suturedcontacthomology}) of thickened convex surfaces $\Sigma \times \R$ is computed as the bilinearized contact homology of its dividing set $\Gamma$ viewed as a contact manifold with Liouville fillings $\Sigma_+$ and $\Sigma_-$. Avdek remarks that it is the closed-string analogue of \cite{ekholm2016nonloose} which (in the $1$-dimensional case) can essentially be viewed as computing the Legendrian contact homology of a Legendrian knot $\Lambda$ in $\Sigma \times \R$ built from a $0$-dimensional Legendrian link $\pi(\Lambda) \cap \Gamma$ in $\Gamma$ and two embedded exact Lagrangian fillings $\pi(\Lambda) \cap \Sigma_+$ and $\pi(\Lambda) \cap \Sigma_-$ as a bilinearized Legendrian contact homology. This paper aims to extend this computation to all Legendrian knots in $\Sigma \times \R$, which means that the exact Lagrangian fillings $\pi(\Lambda) \cap \Sigma_\pm$ are allowed to be immersed, i.e. $\pi(\Lambda) \cap \Sigma_\pm$ can have double-points. Although \cite{ekholm2016nonloose} actually proves that the bilinearized LCH indeed computes the LCH of $\Lambda$ using Morse flow trees, we will only use his computation as a definition. Our definition also appears to be similar to that of \cite{sivek2011bordered} \cite{pan2021functorial}. 

Let $\pi$ be the \textit{convex surface projection} $\Sigma \times \R \rightarrow \Sigma$ sending $(p, t) \mapsto p$. Up to Legendrian perturbation, any Legendrian knot $\Lambda$ in a thickened convex surface has a \textit{good projection} onto $\Sigma$, i.e. it has only transverse double-points and transverse intersections with the dividing set $\Gamma$. The Legendrian DGA of $\Lambda$ is generated by the double-points and (possibly self-overlapping) arcs along $\Gamma$ joining two points in $\pi(\Lambda) \cap \Gamma$. In particular, there are infinitely many generators along $\Gamma$. If near $\Gamma$ we have that $\Lambda$ intersects $\Gamma$ orthogonally and is contained in $\Sigma \times \{0\}$, then the generators correspond exactly to the Reeb chords. An appropriate Maslov grading can be assigned to each generator in essentially the same way as \cite{licata2013seifert}. The differential is defined by counting certain immersed polygons with boundary along $\pi(\Lambda) \cup \Gamma$. A more detailed definition of the construction is given in \Cref{sec:algebra}.

The well-definedness (\Cref{sec:d2 is 0}) and invariance (\Cref{sec:invariance}) of our invariant will be proved completely combinatorially, in a similar style as \cite{chekanov2002dga}. Namely, we use a broken-heart type argument for $\partial^2 = 0$ and an invariance argument by discretizing a Legendrian isotopy into a sequence of Reidemeister moves for the convex surface projection. The Reidemeister theorem we use is covered in \Cref{sec:Leg knots}. As an example, we prove in \Cref{sec:linearization}:

\begin{theorem}\label{thm: main 2} \it
    There exist two nullhomologous but not nullhomotopic Legendrian knots $\Lambda_1$ and $\Lambda_2$ in $(\Sigma_2\times \R, \xi_\Gamma)$ with identical classical invariants that are not Legendrian isotopic, where $(\Sigma_2, \Gamma)$ is a convex genus-$2$ surface with dividing curve $\Gamma$ as in \Cref{fig:combined chekanov}.
\end{theorem}

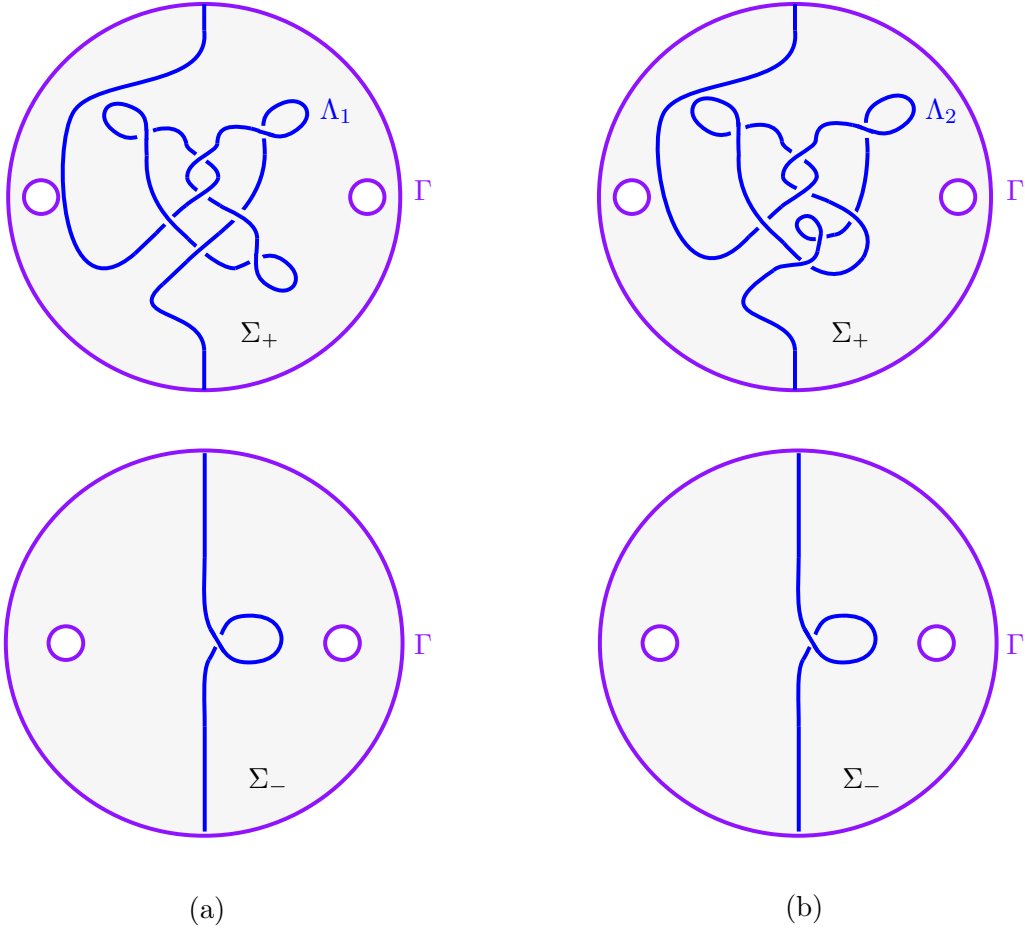
\begin{figure}
    \centering
    \input{images/Chekanov_example_combined}
    \caption{Two Legendrian knots in $\Sigma_2$ that are not Legendrian isotopic, but share the same classical invariants. The surface $\Sigma_2$ is constructed by identifying the pair of pants $\Sigma_+$ and $\Sigma_-$ along the dividing set $\Gamma$ in the obvious way. }
    \label{fig:combined chekanov}
\end{figure}


It may be possible to extend the results of this paper to define contact homology type invariants for arbitrary Legendrian knots in open book decompositions.


\subsection{Relation to the Eliashberg-Givental-Hofer invariants}
In this subsection we motivate how our combinatorial definition can be thought of as some variant of a simplification to the geometric model from \cite{egh2000SFT}.

First we claim that we may quotient out the Reeb orbits. In \cite{egh2000SFT} the algebra $\mathcal{A}(Y, \Lambda)$ is given by $S(C(Y, \alpha)) \otimes T(C(\Lambda))$, where $S(C(Y, \alpha))$ is the symmetric algebra on the vector space generated by the Reeb orbits $\gamma_i$ and $T(C(\Lambda))$ is the tensor algebra on the vector space generated by the Reeb chords $c_i$. The differential $\widehat{\partial}$ on $\mathcal{A}(Y, \Lambda)$ is defined by \[\widehat{\partial}(\gamma) = \sum \#(\mathcal{M}(\gamma; \{\gamma_1, \dots , \gamma_k\})/\mathbb{R}) \gamma_1\dots \gamma_k\] and \[\widehat{\partial}(c) = \sum \#(\mathcal{M}(c; \{\gamma_1, \dots , \gamma_k\}, \{c_1, \dots , c_l\})/\mathbb{R}) \gamma_1\dots \gamma_k c_1\dots c_l,\] where $\mathcal{M}(\gamma; \{\gamma_1, \dots , \gamma_k\})$ is the $1$-dimensional moduli space of $J$-holomorphic curves with a positive end asymptotic to $\gamma$ and $k$ negative ends asymptotic to $\gamma_1, \dots , \gamma_k$, and \\ $\mathcal{M}(c; \{\gamma_1, \dots , \gamma_k\}, \{c_1, \dots , c_l\})$ is the $1$-dimensional moduli space of $J$-holomorphic curves with a positive end asymptotic to $c$, $l$ negative ends asymptotic to $c_1, \dots , c_l$, and $k$ interior punctures asymptotic to $\gamma_1, \dots , \gamma_k$.  Let $\mathcal{A}(Y, \Lambda)^+$ be the ideal in $\mathcal{A}(Y, \Lambda)$ generated by words of the form $\gamma_1\dots \gamma_k c_1\dots c_l$ such that $k \geq 1$. Our hypertightness assumption implies that $\mathcal{A}(Y, \Lambda)^+$ is a differential graded ideal, i.e. $\widehat{\partial}(\mathcal{A}(Y, \Lambda)^+) \subset \mathcal{A}(Y, \Lambda)^+$. Indeed, the condition fails only if $\widehat{\partial}$ counts a $J$-holomorphic disk bounding $\gamma_i$, which would imply that $\gamma_i$ is a contractible Reeb orbit. See \Cref{fig:J curve}. Therefore, our definition agrees with the quotient $\mathcal{A}_{\text{red}}(Y, \Lambda)$ of $\mathcal{A}(Y, \Lambda)$ by $\mathcal{A}(Y, \Lambda)^+$ on the level of generators.

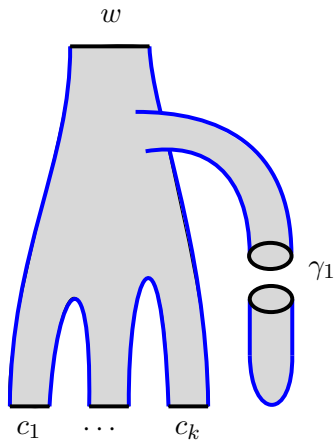
\begin{figure}[h]
    \centering
    \input{images/curve_justification}
    \caption{Contractible Reeb orbits are forbidden.}
    \label{fig:J curve}
\end{figure}

Second, we claim that the differential on $c_\pm$ should count rigid polygons contained wholly in $\Sigma_\pm$, or correspondingly $J$-holomorphic punctured disks in the symplectization of $\Sigma_\pm \times \R\subset \Sigma\times \R$.  This can be arranged by scaling the contact form near $\Gamma$ to be very large. After such a scaling, the height property obstructs polygons with a positive end in $\Sigma_\pm$ and negative ends on $\Gamma$ or $\Sigma_\mp$ from lifting to $J$-holomorphic curves. Hence, our differential agrees with the differential in \cite{chekanov2002dga} when acting on generators in $\Sigma_\pm$. 

Third, for the differential on $c_\Gamma$ we use the formula in \cite[Lemma 2.3]{ekholm2016nonloose}, which counts certain $J$-holomorphic curves with a positive end asymptotic to a Reeb chord along $\Gamma$ and several negative ends asymptotic to Reeb chords along $\Gamma$ or double-points in $\Sigma_\pm$. The formula does not involve any $J$-holomorphic curves with more than one negative ends along $\Gamma$ and we currently do not have any good conceptual reasons to exclude these curves so our differential on $c_\Gamma$ may be different from the actual SFT differential. However, we have the following conjecture: \begin{conjecture}
    Let $\Lambda \subset (\Sigma \times \R, \xi_\Gamma)$ be a Legendrian knot. Let $Y(\Sigma, \Gamma) \subset \Sigma \times \R$ be the convex sutured contact manifold associated to $(\Sigma, \Gamma)$ such that $\Lambda \subset Y(\Sigma, \Gamma)$. Then the sutured Legendrian contact homology dg-algebra $\mathcal{A}_{\text{red}}(Y, \Lambda)$ is quasi-isomorphic to the combinatorially-defined DGA $\widehat{\mathcal{A}}$ in this paper.
\end{conjecture} 

\subsection{Organization}
This paper is organized as follows. \Cref{sec:background} recalls some necessary background from contact geometry and introduces the contact manifold in which we will work. It also includes a review of some algebra background. \Cref{sec:Leg knots} develops a combinatorial framework for studying Legendrian knots in thickened convex surfaces. Specifically, we make use of a projection map to define Reidemeister-type moves for Legendrians and give combinatorial formulas for the classical invariants. \Cref{sec:algebra} defines the Chekanov-Eliashberg DGA for Legendrian knots, and \Cref{sec:linearization} computes the linearized homology for some first examples. Finally, \Cref{sec:d2 is 0} and \Cref{sec:invariance} establish the well-definedness of the DGA, proving that $\partial^2=0$ and that the DGA is stable-tame invariant under Legendrian isotopy respectively.

\subsection*{Acknowledgments}
    This paper originated as the thesis project of the second author proposed by his advisor Julian Chaidez. We thank him for formulating the problem, outlining the proof strategy, and other helpful discussions. We also thank Michael Hutchings, the first author's advisor, along with John Etnyre, Wenyuan Li, and Joshua Sabloff for many helpful discussions too. 
\section{Background}\label{sec:background}

\subsection{Convex surface theory}

We assume familiarity with the standard definitions of contact geometry. We recall some of them here for convenience, but refer the reader to \cite{etnyre2003convex} for a gentler introduction. In this subsection we give a brief review of convex surface theory (see also \cite[Sec. 4.6, 4.8]{geiges2006contact} and \cite{etnyre2025convex}). In later sections we will also assume basic facts about Legendrian knots (see \cite[Sec. 3.1, 3.2, 3.5]{geiges2006contact}) and the basic construction of Legendrian contact homology for the standard contact $\R^3$ (see \cite{Ghiggini2012} and \cite{EtnyreNg2022}).

Let $\Sigma$ be a closed, orientable surface. Let $(M,\xi)$ be a contact $3$-manifold, and suppose there exists an embedding $\Sigma\hookrightarrow M$. The surface $\Sigma$ is said to be \textit{convex} in $M$ if there exists a \textit{contact vector field} $v$ that is transverse to $\Sigma$. That is, $v$ is a vector field whose flow preserves the contact structure.  The \textit{dividing set}  of a convex surface $\Sigma \subset M$ is the locus where the contact form evaluates to zero along $v$:
    \[
    \Gamma = \{ p \in \Sigma \ | \ \alpha_p(v_p) = 0 \}.
    \]
    Equivalently, it consists of points where the vector field $v$ lies in the contact plane $\xi$. 

    Foundational work of Giroux \cite{giroux1991convex} says that a surface $\Sigma$ is convex if and only if its \textit{characteristic foliation} $\Sigma_\xi := \xi\cap T\Sigma$ is divided by a collection of embedded circles into $\Sigma_+$ and $\Sigma_-$, in the sense that there exists a volume form $\omega$ on $\Sigma$ and a vector field $X$ on $\Sigma$ directing $\Sigma_\xi$ pointing from $\Sigma_+$ to $\Sigma_-$ such that $\pm \mathcal{L}_X \omega > 0$ on $\Sigma_\pm$. This collection of embedded circles is the dividing set $\Gamma$ of $\Sigma$ with respect to some contact vector field $v$. 
    
    The dividing set $\Gamma$ determines a germ of the contact structure on a neighborhood of $\Sigma$ in $M$. On a neighborhood $\Sigma\times \R_t$ we can write $\alpha = \beta+u\de t$, where $\beta$ is some one form on $\Sigma$ and $u:\Sigma\to \R$ is some function, and $\Gamma = u^{-1}(0)$. The contact condition $\alpha \wedge \de \alpha > 0$ translates to \[\beta \wedge \de u + u \de \beta > 0.\] We can find a contact form $\alpha_0$ such that $u\equiv \pm 1$ away from a neighborhood of $\Gamma$. This is (a version of) the \textit{standard $\R$-invariant contact structure} with respect to $\Gamma$ on $\Sigma\times \R$, which we denote by $\xi_\Gamma$. Notice that $\Sigma\times \{0\}$ (and more generally $\Sigma \times \{t\}$) is convex in $\Sigma\times \R$. 
    We restrict from now on to the contact $3$-manifold $\Sigma\times \R$ equipped with the standard $\R$-invariant contact structure, along with some chosen dividing set $\Gamma$ so that $(\Sigma\times\{0\},\Gamma)$ is convex in $(\Sigma\times \R,\xi_\Gamma)$. We will often refer to $\Sigma\times \{0\}$ as simply $\Sigma$ and similarly $\Gamma\times \{0\}$ as $\Gamma$.

Define the \textit{Reeb vector field} $R$ to be the unique vector field on $\Sigma\times \R$ satisfying the conditions
\[\iota_R\de \alpha \equiv 0, \qquad \iota_R\alpha\equiv1.\]
For the standard contact form $\alpha_0$, it is easy to check that on $\Sigma_\pm$ we can identify $R$ with $\pm\partial_t$, and on $\Gamma$ we have that $R\in T\Sigma$. By definition, this means that on $\Gamma$, $R$ is tangent to $\Gamma$, see \Cref{fig:orientationsSigmapm}. We orient $\Sigma_\pm$ by a right hand rule relative to the Reeb vector field. Namely, loops that run counterclockwise in $\Sigma_+$ are considered positively oriented, and loops that run counterclockwise in $\Sigma_-$ are negatively oriented. 

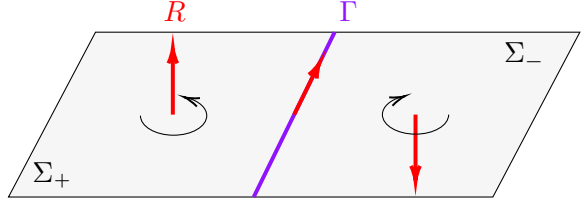
\begin{figure}[h]
    \centering
    \input{images/orientations_on_Sigmapm}
    \caption{Relative orientations on $\Sigma_\pm$.}
    \label{fig:orientationsSigmapm}
\end{figure}

\subsection{Stable tame automorphisms for countably generated DGAs}

In this subsection we review the notions of infinite composition of elementary automorphisms and infinite stabilization introduced by \cite[Sec. 2.6]{ekholmNg2015weinstein} extending their finite analogs in \cite[Sec. 2]{chekanov2002dga}.

\begin{definition}
    Let $\mathcal{A}$ be a semi-free DGA over $\Z_2$ generated by $c_1, c_2 \dots$. An ordering $c_{\sigma(1)} < c_{\sigma(2)} < \dots$ of the generators of $\mathcal{A}$ (where $\sigma$ is a permutation) induces a filtration \[\Z_2 = \mathcal{F}^0 \mathcal{A} \subset \mathcal{F}^1 \mathcal{A} \subset \mathcal{F}^2 \mathcal{A} \subset \dots \] by $\mathcal{F}_i = \Z_2\langle c_{\sigma(1)}, \dots, c_{\sigma(i)} \rangle$. An \textit{elementary} automorphism of $\mathcal{A}$ is a DGA automorphism from $\mathcal{A} \rightarrow \mathcal{A}$ such that there exists a generator $c_{\sigma(i)}$ with \begin{align*}
        c_{\sigma(i)} &\mapsto c_{\sigma(i)} + v \\
        c &\mapsto c \text{ if } c \neq c_{\sigma(i)},
    \end{align*} where $v \in \mathcal{F}^{i-1} \mathcal{A}$. An ``infinite composition of elementary automorphisms'' is a DGA automorphism such that there exist generators $c_{\sigma(i_1)}, \dots, c_{\sigma(i_n)}$ such that for all $1 \leq k \leq n$, \begin{align*}
        c_{\sigma(i_k)} &\mapsto c_{\sigma(i_k)} + v_k \\
        c &\mapsto c \text{ if } c \neq c_{\sigma(i_k)},
    \end{align*} where $v_{i_k} \in \mathcal{F}^{i_k-1} \mathcal{A}$. (Note that the elementary automorphism only depends on the existence of such a filtration and the particular choice of filtration is irrelevant.) A \textit{tame} automorphism of $\mathcal{A}$ is a finite composition of ``infinite compositions of elementary automorphisms''. Note that the orderings for the infinite compositions of elementary automorphisms need not be the same.
\end{definition}

\begin{remark}
    An elementary automorphism in the sense of \cite{ekholmNg2015weinstein} is a possibly infinite composition of elementary automorphisms in the sense of \cite{chekanov2002dga}, where a finite composition of elementary automorphisms is called a tame automorphism. A tame automorphism in the sense of \cite{ekholmNg2015weinstein} is a finite composition of elementary automorphisms in the sense of \cite{ekholmNg2015weinstein}. In this paper we call an elementary automorphism in the sense of \cite{chekanov2002dga} elementary and an elementary or tame automorphism in the sense of \cite{ekholmNg2015weinstein} tame.
\end{remark}

Finally the notion of countably many stabilization operations extends naturally from the finite case.

\begin{definition}
    Let $E = T(\{a_i, b_i\}_{i \in I})$, where $I = \{1, 2, \dots, n\}$ or $\N$, $|b_i| = |a_i| - 1$, and $\partial_E a_i = b_i$, $\partial_E b_i = 0$. Then a \textit{stabilization} of $\mathcal{A}$ is the free product $\mathcal{A} \sqcup (E, \partial_E)$.
\end{definition}

\section{Legendrian knots in thickened convex surfaces}\label{sec:Leg knots}

In this section we study Legendrian knots in $\Sigma \times \R$ and Legendrian Reidemeister moves for the convex surface projection $\pi: \Sigma \times \R \rightarrow \Sigma\times \{0\}$.

Let $(\Sigma, \Gamma)$ be a convex surface with dividing curve $\Gamma$. We equip $\Sigma\times \R$ with the contact structure $ \xi_\Gamma:=\ker(\alpha_0)$ as in the previous section. Let $\Lambda$ be a Legendrian knot in $(\Sigma \times \R, \xi_\Gamma)$. Later in this paper we assume the following conditions: (1)  $c_1(\xi_\Gamma) = 0$, (2) $\Lambda$ is null-homologous, and (3) $(\Sigma\times \R,\ker(\alpha_0))$ is hypertight, although none of these assumptions are needed for this section.

First we investigate the question of lifting an immersed loop $K$ in $\Sigma$ to a Legendrian arc in $\Sigma \times \R$. A necessary condition is that $TK|_\Gamma \subset \ker \beta$. The following lemma shows that it is also sufficient.

\begin{lemma}\label{lem:legendrian lift}
    Let $K:S^1\hookrightarrow \Sigma$ be an immersion. Then $K$ lifts to a Legendrian arc $\Lambda \subset \Sigma \times \R$ if and only if $TK|_\Gamma \subset \ker \beta$, i.e. $K$ is tangent to the characteristic foliation along $\Gamma$.
\end{lemma}

\begin{remark}
    We emphasize that the lift of $K$ in the above lemma is a Legendrian \textit{arc} which may not be closed. There should be a condition for closedness in terms of some area constraint similar to the $\R^3$ case (and also a combinatorial criterion similar to that of \cite[Sec. 11]{chekanov2002dga}). Instead of formulating these we will give a procedure in the next subsection that checks closedness in some special cases which will suffice for all but one examples in this paper.
\end{remark}

\begin{proof}
    Let $f:[0,1]\to \Sigma$ be a parametrization of $K$. We divide $K$ into sub-arcs $\{K_j\}_{j\in J}$ by cutting along $\Gamma$, and denote the induced parametrization of $f$ on each $K_j$ by $f_j:[0,1]\to \Sigma$. Note that $u(f_j(i))=0$ for $i=0,1$. First, suppose that $K$ lifts to a Legendrian. Then
    \[TK\subset \ker(\alpha)\implies (\beta+u\de t)_{f_j(i)}(f_j'(i))=0 \implies \beta_{f_j(i)}(f_j'(i))=0.\]
    Now suppose that we only know $TK|_\Gamma\subset \ker(\beta)$, i.e. that $\beta_{f_j(i)}(f_j'(i))=0$. We'd like to build a lift of each arc $f_j:[0,1]\to \R$ in $\Sigma$ to some $\tilde{f_j}:[0,1]\to \Sigma\times \R$ such that 
    \begin{enumerate}
        \item each $\tilde{f_j}$ is a parametrization of a Legendrian arc, and 
        \item $\tilde{f}_j(1)=\tilde{f}_{j+1}(0)$ for all except possibly one $j\in J$, so that their concatenation $\tilde{f} := \cup_{j\in J} \tilde{f}_j$ is a connected Legendrian lift of $f$.
    \end{enumerate}

    To accomplish (1), fix some $j\in J$. Since $\alpha = \beta+u\de t$ is contact, 
\[0 \neq \iota_{f_j'(i)}(\beta \wedge \de u + u\de \beta) = \iota_{f_j'(i)} (\beta \wedge \de u) = -\de u(f_j'(i))\beta,\] 
    which implies $(u \circ f_j)'(i) \neq 0$. By L'Hôpital's rule, the limit \[\lim_{s \rightarrow i} \frac{-\beta_{f_j(s)}(f_j'(s))}{u(f_j(s))} = \lim_{s \rightarrow i} \frac{-\frac{\de}{\de s} \beta_{f_j(s)}(f_j'(s))}{(u \circ f_j)'(s)} =: L\] exists. 
    (Since $\beta_{f_j(i)}(f_j'(i))=0$, by continuity $\lim_{s\to i}\beta_{f_j(s)}(f_j'(s))=0$ too).
    Define $g: [0, 1] \rightarrow \R$ by \[g(s) = \begin{cases}
        \frac{-\beta_{f_j(s)}(f_j'(s))}{u(f_j(s))} \text{ if } s \in (0, 1) \\
        L \text{ if } s = 0, 1
    \end{cases}.\]
    Fix any $s_0 \in (0, 1)$ and $C \in \R$. Set $t_j(s_0) := C$ and \[t_j(s) := t_j(s_0) + \int_{s_0}^s g(v) \de v.\] Then $(f_j(s), t_j(s))$ is a Legendrian lift of $f_j(s)$. By picking suitable $t_j(s_0)$ one can arrange that the Legendrian lifts of $f_j$ match up at all $f_j(i)$, $i = 0, 1$ except possibly at one $f_j(1) = f_{j+1}(0)$.
\end{proof}

\begin{remark}
    It is clear from the proof above that the Legendrian lift of a loop that lifts to a closed Legendrian arc (i.e. a Legendrian knot) is unique up to vertical translation.
\end{remark}

\subsection{Examples}\label{subsec:examplesOfLegendrians} \text{ }

In this subsection we describe a procedure for producing examples of Legendrian knots in $\Sigma \times \R$ when $\Sigma_+$ and $ \Sigma_-$ are homeomorphic.

\begin{construction}\vphantom{}\label{LegConstr}
    \begin{enumerate}[leftmargin=3.5em]
        \item[\textit{Step 1:}] Start with a Legendrian knot in $\R^3$. \cite[Corollary 11.3]{chekanov2002dga} gives a combinatorial criterion for a diagram in $\R^2$ to admit a Legendrian lift. Drawing this diagram in the interior of $\Sigma_+$ already gives a Legendrian knot in $\Sigma \times \R$. 
        \item[\textit{Step 2:}] Isotope the Legendrian knot so that it has an even number of cusps along $\Gamma$. See \Cref{addCusp}. Note that the isotopy is Legendrian since the area constraint is always satisfied.
        \item[\textit{Step 3:}] Using the symmetry of the contact form $\alpha_0$ with respect to $\Gamma$ we can ``flip'' half of the diagram from $\Sigma_+$ to $\Sigma_-$. This gives a Legendrian knot in $\Sigma \times \R$ whose projection intersects $\Gamma$.
    \end{enumerate}
\end{construction}

\begin{figure}[h]
    \begin{center}
    	\input{images/addCusp}
        \caption{An isotopy of a Legendrian arc which adds a cusp point along $\Gamma$.}
    	\label{addCusp}
    \end{center}
\end{figure}

See  \Cref{stdUnknotConstruction} and \Cref{unknotConstruction} for some examples.

\begin{figure}[h]
    \begin{center}
    	\input{images/stdUnknotConstruction}
        \caption{A Legendrian unknot in $\T^2 \times \R$.}
    	\label{stdUnknotConstruction}
    \end{center}
\end{figure}
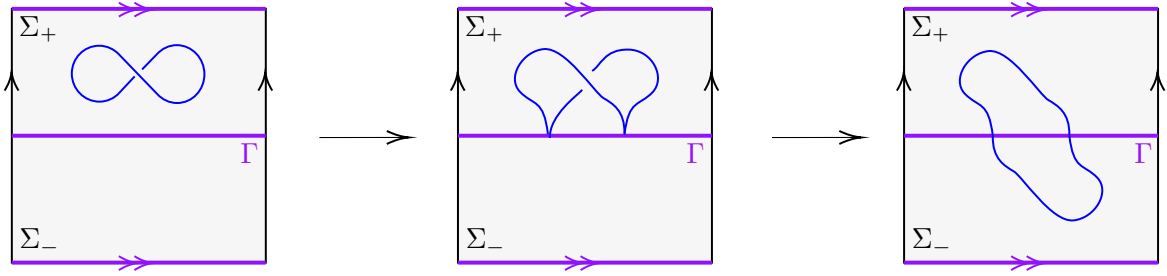

\begin{figure}[h]
    \begin{center}
    	\input{images/unknotConstruction}
        \caption{Another Legendrian knot in $\T^2 \times \R$.}
    	\label{unknotConstruction}
    \end{center}
\end{figure}
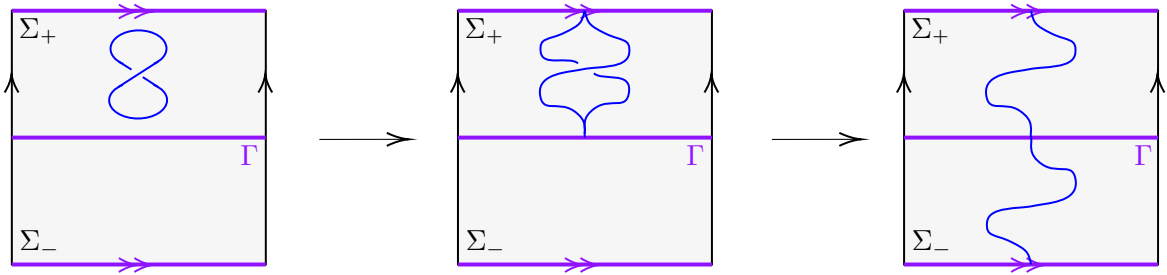


\subsection{Legendrian Reidemeister theorem for the convex surface projection}

In this subsection we state the Legendrian Reidemeister theorem which will be used later to prove the invariance of our invariant under Legedrian isotopy. The proof will be postponed to a future paper \cite{Rong2026}.

First we define the notions of good projections for Legedrian knots and isotopies.

\begin{definition}\label{dfn: reidemeister}
	Let $(\Sigma \times \R, \xi_\Gamma)$ be a thickened convex surface. Let $\Lambda \subset \Sigma \times \R$ be a Legendrian knot. Then $\Lambda$ has a \textit{good} projection if $\pi(\Lambda)$ is an immersed submanifold of $\Sigma$ whose only self-intersections are transverse double-points off $\Gamma$ and $\pi(\Lambda)$ intersects $\Gamma$ transversely.

    Let $\Lambda_t$ be a Legendrian isotopy in $(\Sigma \times \R, \xi_\Gamma)$ between two Legendrian knots $\Lambda_0$ and $\Lambda_1$ having good projections. Then $\Lambda_t$ has a \textit{good} projection if it can be discretized into a sequence of Reidemeister-like moves:
	\begin{enumerate}[label=(\alph*)]
	    \item  Reidemeister I move with bifurcation point on $\Gamma$,
        \item  Reidemeister II move with bifurcation point off $\Gamma$,
        \item  Reidemeister IIIa and IIIb move with bifurcation point off $\Gamma$,
        \item  Reidemeister III-like move with two strands of $\Lambda$ and $\Gamma$, which we call the Reidemeister IV move.
	\end{enumerate}
	These moves are represented locally in \Cref{Reidemeister}.
\end{definition}

\begin{figure}[h]
    \begin{center}
    	\input{images/Reidemeister_moves}
    	\caption{Reidemeister moves for the convex surface projection.}
    	\label{Reidemeister}
    \end{center}
\end{figure}
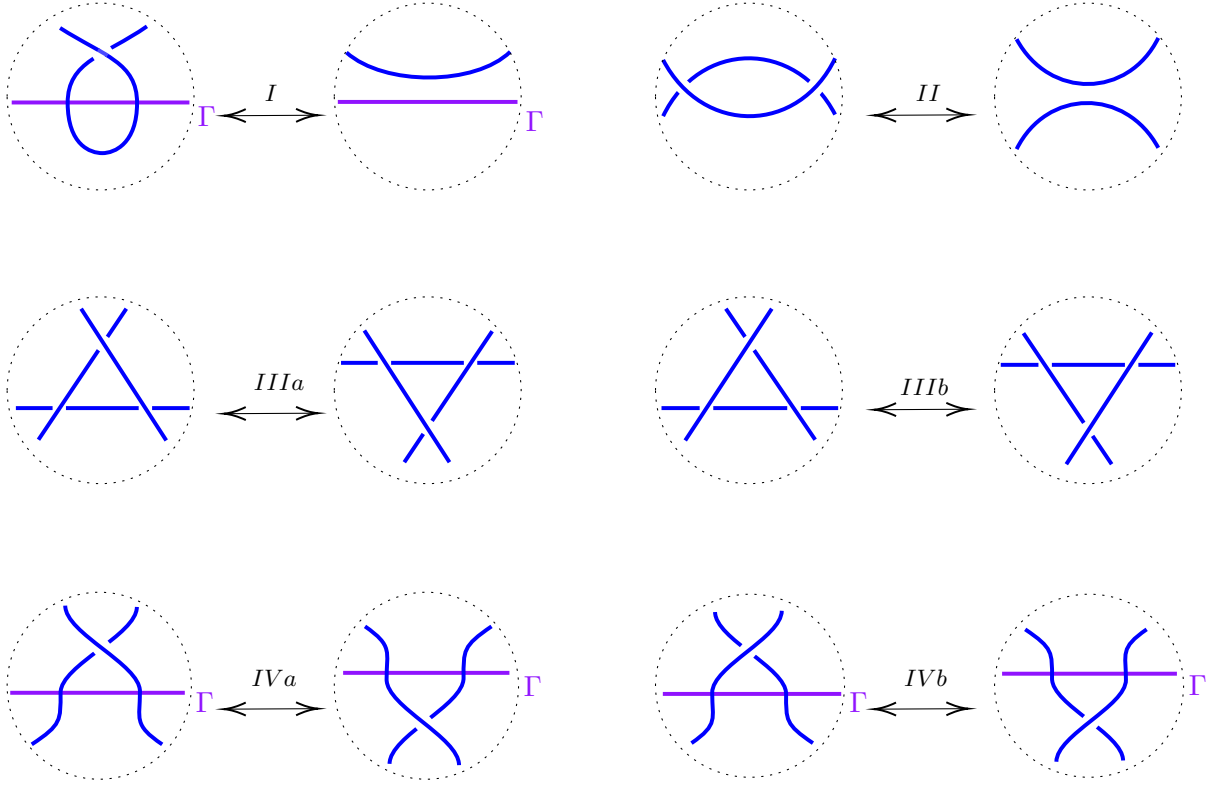

\begin{theorem}[Legendrian Reidemeister theorem]\vphantom{} \it
Let $(\Sigma \times \R, \xi_\Gamma)$ be a thickened convex surface. Let $\Lambda$ be a Legendrian knot and $\Lambda_t$ be a Legendrian isotopy in $\Sigma \times \R$. Then we have:
\begin{enumerate}[label=(\alph*)]
    \item There exists an arbitrarily small Legendrian perturbation $\Lambda'$ of $\Lambda$ such that $\Lambda'$ has a good projection.
    \item There exists an arbitrarily small Legendrian perturbation $\Lambda'_t$ of $\Lambda_t$ such that $\Lambda'_0 = \Lambda_0$, $\Lambda'_1 = \Lambda_1$, and $\Lambda'_t$ has a good projection.
\end{enumerate}
\end{theorem}




\subsection{Classical invariants}

In this subsection we compute the classical invariants of $\Lambda$ in terms of its convex surface projection assuming that $\Lambda$ is nullhomologous.

\begin{lemma} \it
	The Thurston-Bennequin number of $\Lambda$ is given by \[\tb(\Lambda) = -\frac{1}{2}|\pi(\Lambda) \cap \Gamma| + \writhe(\pi(\Lambda)).\]
\end{lemma}

\begin{proof}
    The lemma follows by combining the computation of $\tb(\Lambda)$ in the Lagrangian projection (see, e.g., \cite[Prop. 3.5.11]{geiges2006contact}) and \cite[Thm. 3.5.1]{etnyre2025convex}. Let $\Lambda'$ be the push-off of $\Lambda$ in the direction of a section $v$ of $\xi$ such that $v$ is transverse to $T\Lambda$ in $\xi$. By an argument analogous to \cite[Prop. 3.4.14]{geiges2006contact}, $\tb(\Lambda)$ equals the signed number of times $\Lambda'$ crosses under $\Lambda$, where the sign is determined as in \Cref{tb-sign}. Note that the local push-off diagrams near $\pi(\Lambda) \cap \Gamma$ in \Cref{tb-invariant}(a) must come in consecutive pairs as we trace along $\Lambda$, since for each segment $\Lambda_i$ contained in $\Sigma_\pm \times \R$ and its push-off $\Lambda_i'$ either $\Lambda_i$ is always to the left of $\Lambda_i'$ or $\Lambda_i$ is always to the right of $\Lambda_i'$ (the projection of $\xi|_{\Sigma_\pm}$ onto $T\Sigma_\pm$ is an isomorphism). Hence crossings of $\Lambda'$ under $\Lambda$ along $\Gamma$ contribute $-\frac{1}{2}|\pi(\Lambda) \cap \Gamma|$ to $\tb(\Lambda)$. For each self-intersection of $\pi(\Lambda)$ in $\Sigma_\pm$, there is a crossing of $\Lambda'$ under $\Lambda$, whose sign agrees with that of the self-intersection, as shown in \Cref{tb-invariant}(b). So the contribution of the double-points sum up to the writhe of $\pi(\Lambda)$.
\end{proof}

\begin{figure}[h]
    \begin{center}
    	\input{images/sign_of_a_crossing}
    	\caption{The sign of a crossing of $\Lambda$ over a Legendrian push-off $\Lambda'$.}
    	\label{tb-sign}
    \end{center}
\end{figure}
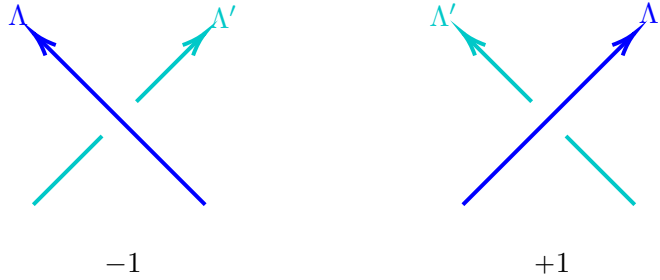

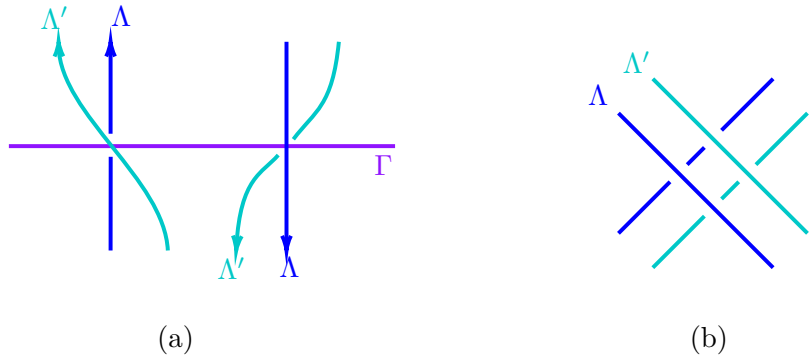
\begin{figure}[h]
    \begin{center}
    	\input{images/tb_invariant}
    	\caption{The contact framing of $\Lambda$, given by taking a Legendrian push-off $\Lambda'$.}
    	\label{tb-invariant}
    \end{center}
\end{figure}

\begin{remark}
    One can easily check that the formula for $\tb(\Lambda)$ is invariant under Legendrian Reidemeister moves for the convex surface projection.
\end{remark}

The rotation number $\rot(\Lambda)$ can be computed in a similar way as the grading of generators. See \Cref{sec: grading} for details.

\begin{example}
    Let $\Lambda$ be the standard Legendrian unknot in $\mathbb{T}^2 \times \R$ whose projection is shown in \Cref{stdUnknotConstruction}. Then $\tb(\Lambda) = -1$ and $\rot(\Lambda) = 0$.
\end{example}

\section{The algebra}\label{sec:algebra}

In this section we define the Chekanov-Eliashberg DGA in the convex surface projection. As already mentioned in the introduction, our definition is motivated by \cite[Lemma 2.3]{ekholm2016nonloose}. More precisely, our definition agrees with Ekholm's formula when $\pi(\Lambda)$ has no double-points,so that $\Gamma$ divides $\pi(\Lambda)$ into $1$-dimensional Lagrangian fillings of $0$-dimensional Legendrian links in $\Gamma$.

\subsection{Moduli space of rigid polygons} \text{ }

As a first step to defining the DGA, in this section we identify the compact $0$-dimensional moduli spaces of certain immersed polygons in $\Sigma_\pm$ with boundary in $\pi(\Lambda) \cup \Gamma$. These polygons will be used later to define the differential.

Assuming the projection $\pi$ is good, $\pi(\Lambda)\cup \Gamma$ divides $\Sigma\times \{0\}$ into \textit{immersed polygons}, whose corners can locally be perturbed to be multiples of $\pi/2$. For the remainder of this section, we consider any boundary arc of a polygon along $\Gamma$ as an acute corner of this type. Under this convention, all corners of immersed polygons are identified with Reeb chords of $\Lambda$. As in classical knot projections, we may remember overstrands and understrands of $\Lambda$ after projection to $\Sigma\times \{0\}$ by indicating a break in the understrand. We can define orientations for these disks by assigning decorations to each double point/crossing according to the rule in \Cref{fig:corner orientations}. A polygon has a \textit{positive end} at a Reeb chord $a$ if it has an acute corner at $a$ with a positive decoration or if it has a side along $\Gamma$ at $a$; it has a \textit{negative end} at $c$ if that corner carries a negative decoration.

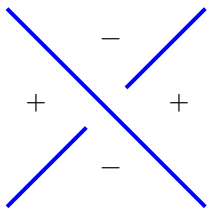
\begin{figure}[ht]
    \centering
    \input{images/corner_orientations}
    \caption{Orientations of corners in $\Sigma_\pm$.}
    \label{fig:corner orientations}
\end{figure}

\begin{definition}[Rigid polygon]
     A polygon is \textit{rigid} if it has exactly one positive end, and arbitrarily many negative ends. For a Reeb chord $c$ along $\Gamma$, denote by $\mathcal{M}_\pm(c,d_1,\dots,d_n)$ the moduli space of rigid polygons in $\Sigma_\pm$ with a positive end $c$ in $\Sigma_\pm$ or on $\Gamma$ and negative ends at Reeb chords $d_1,\dots,d_n$ in $\Sigma_\pm$.
\end{definition}

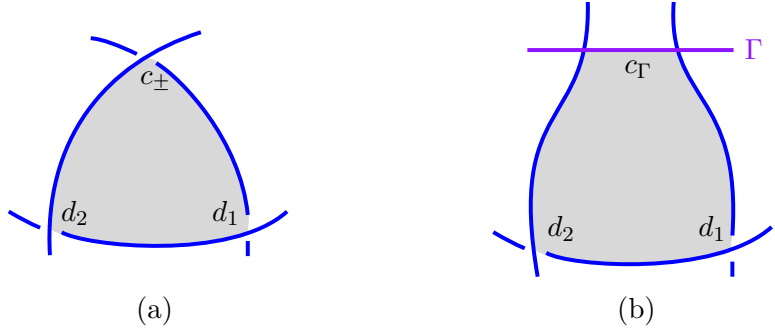
\begin{figure}
    \centering
    \input{images/immersed_polygons}
    \caption{(a) A rigid polygon with positive end $c_\pm$ in $\Sigma_\pm$ and (b) a rigid polygon with positive end $c_\Gamma$ on $\Gamma$.}
    \label{fig:placeholder}
\end{figure}

As in the standard construction of \cite[Sec. 6]{chekanov2002dga} for Legendrians in $(\R^3,\xi_{\text{std}})$, the presence of a rigid disk with positive end and negative end imposes constraints on the relative lengths of these chords: whenever two Reeb chords are connected by such a disk, they satisfy a corresponding \textit{height inequality}.

\begin{definition}[Height of a Reeb chord]
	Fix $\Sigma_+' \subset \Sigma_+$ such that all double-points of $\pi(\Lambda)$ in $\Sigma_+$ are contained in $\Sigma_+'$. Then fix a contact form $\alpha_0 = \beta + u\de t$ on $\Sigma\times \R$ for $\xi_\Gamma$ such that $u \equiv 1$ on $\Sigma_+'$. Let $c$ be either a double-point of $\pi(\Lambda)$ or a (possibly self-overlapping) arc along $\Gamma$ connecting two points in $\pi(\Lambda) \cap \Gamma$. In the former case, let $\gamma_c$ be the arc in $\Lambda$ from the lower strand to the upper strand of $c$; in the latter case, let $\gamma_c = \gamma_c^\parallel \cup \gamma_c^\perp$, where $\gamma_{c}^\parallel$ is a positively oriented (as determined by $\alpha_0|_{\Gamma}$) arc in $\Sigma_+'$ parallel to that along $\Gamma$ between the two points in $\pi(\Lambda) \cap \Gamma$, and $\gamma_c^\perp$ is the union of the vertical line segments joining the two ends of $\gamma_c^\parallel$ and $\Lambda$. See \Cref{height}. Define the \textit{height of $c$} as follows: if $c$ is a self-intersection of $\pi(\Lambda)$, then \[H(c) := t(\gamma_c(1)) - t(\gamma_c(0))\] where $t$ is the coordinate of $\R$. If $c$ is an arc along $\Gamma$, then \[H(c) := t(\gamma_c^\perp(1)) - t(\gamma_c^\perp(0)) + \int_{\gamma_c^\parallel} \beta.\] Denote by $R(\gamma_c^\parallel)$ the region in $\Sigma_+$ bounded by $\pi(\Lambda)$, $\Gamma$ and $\gamma_c^\parallel$. Define the \text{area} of an immersed polygon $\mathcal{P}$ in $\Sigma_+$ to be \[A(\mathcal{P}) := \int_{\mathcal{P} - R(\gamma_c^\parallel)} d\beta.\]
\end{definition}

\begin{figure}
    \begin{center}
    	\includegraphics[width=0.7\textwidth]{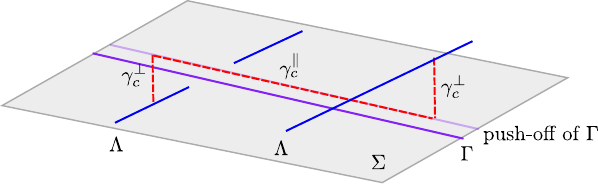}
    	\caption{Height of a Reeb chord along $\Gamma$.}
    	\label{height}
    \end{center}
\end{figure}

\begin{remark}
    Both the height of $c$ along $\Gamma$ and the area of a polygon with sides along $\Gamma$ depend on the choice of $\gamma_c^\parallel$. Specifically, the height and area increase as  $\gamma_c^\parallel$ approaches $\Gamma$.
\end{remark}

\begin{lemma}\label{heightProperty}
    Let $\mathcal{P}$ be an immersed polygon in $\Sigma_+$ with boundary in $\Gamma \cup \pi(\Lambda)$. Denote its set of positive/negative ends by $Q_\pm$. Then \[\sum_{c \in Q_+} H(c) - \sum_{c \in Q_-} H(c) > 0.\]
\end{lemma}

\begin{proof}
    Throughout the proof we use $\mathcal{P}$ to denote $\mathcal{P} - R(\gamma_c^\parallel)$. Lift $\mathcal{P}$ to a polygon $P$ in $\Sigma_+' \times \R$ and write $\iota(P) = \mathcal{P}$. The sides of $P$ consist of arcs $\gamma_1, \dots , \gamma_n$ of $\Lambda$ and $\gamma_{c_1}, \dots , \gamma_{c_n}$ for some double-points $c_1, \dots, c_n$ corresponding to the ends of $\mathcal{P}$. Then \begin{align*}
		\int_{\{\gamma_i\}_{i = 1}^n \cup \{\gamma_{c_j}\}_{j = 1}^n} \alpha &= \sum_{j = 1}^n \int_{\gamma_{c_j}} \alpha \text{ since } \alpha|_{\gamma_i} = 0 \\
		&= \sum_{c \in Q_+} H(c) - \sum_{c \in Q_-} H(c).
	\end{align*} At the same time, we have \begin{align*}
		\int_{\{\gamma_i\}_{i = 1}^n \cup \{\gamma_{c_j}\}_{j = 1}^n} \alpha &= \int_{\partial P} \alpha \\
		&= \int_P \de \alpha \text{ by Stokes' Theorem}\\
		&= \int_{\mathcal{P}} \de \beta + \int_{P} \de u \wedge \de t \\
		&= \int_{\mathcal{P}} \de \beta \text{ since } u \equiv 1 \text{ on } \mathcal{P}.
	\end{align*} Thus we have \[\sum_{c \in Q_+} H(c) - \sum_{c \in Q_-} H(c) = \int_{\mathcal{P}} \de \beta = \int_{\mathcal{P}} \beta \wedge \de u + u\de \beta > 0.\]
\end{proof}

\begin{lemma}\label{finitePolygon}
    Assume $c$ is a Reeb chord in $\Sigma_+$ or along $\Gamma$ and $d_1, \dots , d_n$ are Reeb chords in $\Sigma_+$. Then \[|\mathcal{M_+}(c, d_1, \dots , d_n)| < \infty.\] 
\end{lemma}

\begin{proof}
     This is the analogue of \cite[Lem. 3.9]{sabloff2003circle}. By the proof of \Cref{heightProperty}, $A(\mathcal{P}) \leq H(c)$. Thus, $\mathcal{P}$ cannot cover arbitrarily many subregions of $\Sigma_+-\pi(\Lambda)$ counting multiplicity.
\end{proof}

\subsection{The algebras $\mathcal{A}_\pm$ and $\mathcal{A}_\Gamma$} \text{ }

We begin to construct the DGA of \Cref{thm: main 1} by first defining three intermediate DGAs associated to $\Lambda$: $(\mathcal{A}_+,\partial_+)$, $(\mathcal{A}_-,\partial_-)$ and $(\mathcal{A}_\Gamma,\partial_\Gamma)$, along with DGA maps $\Phi_\pm:\mathcal{A}_\Gamma\to \mathcal{A}_\pm$.

\subsubsection{Generators}

Let $\mathcal{A}_\pm$ be the free algebra generated over $\Z_2$ by Reeb chords of $\Lambda$ in $\Sigma_\pm\times \R$. If $\Lambda$ is in good position, then we may identify the generators of $\mathcal{A}_\pm$ with double points in $\pi(\Lambda)\cap \Sigma_\pm$. In such a case, $\mathcal{A}_\pm$ is finitely generated, and the elements of this algebra are finite formal sums of words in Reeb chords. Next, let $\mathcal{A}_\Gamma$ be the algebra freely generated over $\Z_2$ by ``Reeb chords'' of $\Lambda$ in $\Gamma\times \R$, i.e. any arc in $\Gamma$ with ends on $\pi(\Lambda\times \R)\cap \Gamma\times \{0\}$, including those with overlaps. If a chord can be written as a concatenation of two nontrivial chords then it is said to be \textit{decomposable}, and otherwise \textit{indecomposable}. Any Legendrian knot $\Lambda\subset \Sigma\times \R$ admits finitely many indecomposable Reeb chords in $\Gamma$, but infinitely many decomposable Reeb chords, and so $\mathcal{A}_\Gamma$ is infinitely generated.

\subsubsection{Grading}

 We will postpone the grading of these algebras in \Cref{subsec:Ahat} but remark that a Reeb chord $c\in \mathcal{A}_\pm$ admits a relative grading, defined as in \cite[Sec.~7.2]{chekanov2002dga} or \cite{ekholm2007PxR}, provided that $c$ admits a \textit{capping path} entirely contained in $\Sigma_\pm$. 

\subsubsection{Differential}
For a Reeb chord $c_+\in \mathcal{A}_+$ and any word $w\in \mathcal{A}_+$, we set $\<\partial_+ c_+,w\>$ to be the $\Z_2$-count of all immersed polygons in $\Sigma_+$ with sides being arcs of $\pi(\Lambda)$ in $\Sigma_+$ such that $c$ is a positive corner and all other corners are negative at $w$. Define $(\mathcal{A}_-, \partial_-)$ similarly for $\Sigma_-$. For a Reeb chord $c_\Gamma\in \mathcal{A}_\Gamma$, we set
 \[\partial_\Gamma c_\Gamma = \sum_{c_\Gamma^1 *c_\Gamma^2 = c_\Gamma} c_\Gamma^1 c_\Gamma^2,\] where $c * d$ denotes concatenation of Reeb chords and $cd$ denotes the product operation of the underlying ring structure. Both $\partial_+$ and $\partial_\Gamma$ extend to each whole algebra by demanding linearity and a Leibniz rule. 
 
 \subsubsection{Transfer morphisms} The algebras $\mathcal{A}_\Gamma$ and $\mathcal{A}_\pm$ are related by maps $\Phi_\pm$ which we now define. Our definition is similar to \cite[Sec. 2.2]{sivek2011bordered} and \cite[Dfn. 6.7]{pan2021functorial}.
 
\begin{definition}\label{dfn:phipm}
    On a generator $c_\Gamma\in \mathcal{A}_\Gamma$, define $\<\Phi_\pm(c_\Gamma), w\>$ to be the $\Z_2$-count of all polygons with positive end at $c_\Gamma$ and negative ends at the ordered word in Reeb chords $w$ in $\Sigma_\pm$. We extend the map to the entire algebra by demanding that \[\Phi_\pm(c_\Gamma+c_\Gamma') = \Phi_\pm(c_\Gamma)+\Phi_\pm(c_\Gamma')\qquad \text{and}\qquad \Phi_\pm (c_\Gamma c_\Gamma') = \Phi_\pm(c_\Gamma)\Phi_\pm(c_\Gamma').\] 
    We will see later in \Cref{cor:phipmgrading} and \Cref{prop:phipmDGA} that the maps $\Phi_\pm$ are not just algebra morphisms, but DGA morphisms. 
\end{definition}

\begin{lemma}\label{lem:dphiwelldefined}
    The maps $\partial_\pm c_\pm$ and $\Phi_\pm(c_\Gamma)$ are well-defined for any generators $c_\pm$ of $\mathcal{A}_\pm$ and $c_\Gamma$ of $\mathcal{A}_\Gamma$.
\end{lemma}

\begin{proof}
    Note that any monomial $d_1d_2\dots d_n$ appearing in $\partial_\pm(c_\pm)$ or $\Phi_\pm(c_\Gamma)$ must satisfy \[\sum_{i = 1}^n H(d_i) < H(c_{+/-/\Gamma}).\] 
    So, finitely many monomials may appear in either $\partial_\pm(c_\pm)$ or $\Phi_\pm(c_\Gamma)$. It then follows from  \Cref{finitePolygon} that the coefficients of all candidates are finite.
\end{proof}

\subsection{The algebra $\widehat{\mathcal{A}}$}\label{subsec:Ahat} \text{ }

Finally, we combine our previous constructions to define the DGA of \Cref{thm: main 1}.
\subsubsection{Generators} Let $\widehat{\mathcal{A}}$ be the free algebra generated over $\Z_2$ by all Reeb chords in $\Sigma_\pm$ and along $\Gamma$. The algebra splits as the free product of the three previously-defined algebras: \[\widehat{\mathcal{A}} = \mathcal{A}_+* \mathcal{A}_\Gamma* \mathcal{A}_-.\] 

\subsubsection{Grading}\label{sec: grading}
We define a $\Z_{2\rot(\Lambda)}$-grading on the generators of $\widehat{\mathcal{A}}$, where $2\rot(\Lambda)$ is the Maslov number of $\Lambda$. This is possible since we assume that $c_1(\xi_\Gamma) = 0$. Our definition is adapted from that of \cite{sabloff2003circle} and \cite{licata2013seifert} .

Fix once and for all a collection of generators $\textbf{X}$ of $H_1(\Sigma; \Z)$. Given a generator $c\in \widehat{\mathcal{A}}$, choose an oriented loop $\gamma_c$ in $\pi(\Lambda)$ based at $c$ which runs from the overstrand to the understrand. For generators $c\in \mathcal{A}_\Gamma$ where this is ambiguous, we follow the direction of the Reeb vector field so that the loop runs negatively in time between the endpoints of $c$. We call such a path a \textit{capping path} of $c$. Then we apply a variant of Seifert's algorithm to construct a \textit{formal capping surface} of $\gamma_c$ via \Cref{con:Seifert's algorithm}. Note that the corners associated to the crossing $c$ are preserved, although this is not required for defining $|c|$. Indeed, for any capping surface of $\gamma_c$, the contributions coming from these corners cancel. Thus, one could equivalently simplify the computation by applying the standard Seifert resolution at $c$, replacing the crossing with two disjoint, smoothed arcs. We retain the corners in order to keep our definition closer in spirit to that of \cite{licata2013seifert}.

\begin{figure}[h]
    \centering
    \input{images/Seifert_resolution_1}
    \caption{The rule for resolution of double points along a capping path $\gamma_c$.}
    \label{fig:seifert alg 1}
\end{figure}
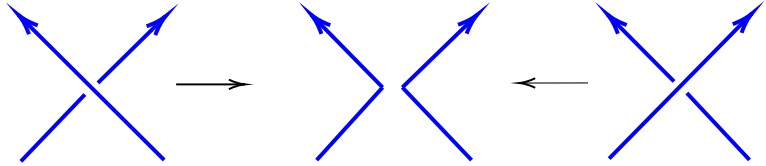
\begin{construction}\label{con:Seifert's algorithm}
    First, suppose that $\Sigma$ is simply connected. Given a capping path $\gamma_c$ of $c$, resolve any crossing according to the rule in \Cref{fig:seifert alg 1}. The result is a disjoint union of simple loops with corners $\sqcup_{l=1}^k\gamma_l$. Since $\Sigma$ is simply connected, each of these $\gamma_l$ are separating, bounding an oriented surface in $\Sigma$ which possibly intersects $\Gamma$ such that the orientation of $\gamma_l$ agrees with the boundary orientation induced by $\Sigma$. The choice of surface is not unique, but will ultimately be shown to be only auxiliary in defining the grading. If an arc of a loop intersects $\Gamma$ transversely, then there is a natural cutting along $\Gamma$ of the region it bounds into finitely many subregions, see \Cref{fig:cut along Gamma}. Let $\sqcup_{i=1}^j \tilde{\gamma_i}$ denote the final set of loops obtained after all such cuttings. We define the formal sum 
    \[S_c:=\sum_{i=1}^j c_i u_i,\]
    where $u_i$ are the chosen regions bound by each $\tilde{\gamma_i}$, and $c_i$ are coefficients assigned via the following rule: if the induced orientation of $\tilde{\gamma}_i$ agrees with the orientation given in \Cref{fig:orientationsSigmapm}, then set $c_i=+1$. Similarly, if it disagrees, set $c_i=-1$. 
\end{construction}

\begin{figure}[h]
    \centering
    \input{images/dividing_curves_along_Gamma}
    \caption{For a component of a resolved capping path which crosses $\Gamma$, we may split the region it bounds along the dividing curve into subregions. }
    \label{fig:cut along Gamma}
\end{figure}
If $\Sigma$ is not simply connected, then applying the resolution of \cref{fig:seifert alg 1} may result in loops $\gamma_i$ that are non-separating. In this case, each $\gamma_i$ is homologous to a linear combination of generators in $\mathbf{X}_i$. We amend each $\gamma_i$ by adding this linear combination of generators with reversed orientation. Such a curve is now separating, and we may proceed with \Cref{con:Seifert's algorithm}.

\begin{definition}
	Let $u: S \rightarrow \Sigma_\pm \setminus\pi(\Lambda)$ be a subregion with boundary such that $u$ is orientation-preserving if $u(S) \subset \Sigma_+$ and orientation-reversing if $u(S) \subset \Sigma_-$. If $u$ has $k$ corners, define the \textit{rotation number} of $u$ as \[r(u) = \chi(u) + \sum_{i = 1}^k \frac{1}{4}\left(m_i - 2\right),\] where $m_i$ is the number of quadrants covered by the $i$-th corner. Extend this to define the \textit{rotation number} of a formal capping surface by \[r\left(\sum_{i = 1}^j c_i u_i\right) := \sum_{i = 1}^j c_i r(u_i).\]
\end{definition}

\begin{definition}\label{dfn:grading}
	Let $c$ be a generator of $\widehat{\mathcal{A}}$. Pick a capping path $\gamma_c$ for $c$ and pick a formal capping surface $S_c$ for $\gamma_c$ following \Cref{con:Seifert's algorithm}. Define 
    \begin{align*}
        |c_\pm| =& \ 2r(S_{c_\pm}) - \frac{1}{2} \mod 2\rot(\Lambda)\\
        |c_\Gamma| =&\  2r(S_{c_\Gamma}) - 1 \ \mod 2\rot(\Lambda)
    \end{align*} and $|\widehat{c}_\Gamma| = |c_\Gamma| + 1$ for generators $c_\pm$ of $\mathcal{A}_\pm$ and $c_\Gamma$ of $\mathcal{A}_\Gamma$.
\end{definition}

\Cref{dfn:grading} is morally identical to that in the standard Legendrian contact homology of $\mathbb{R}^3$, see \cite[Sec. 3.1]{chekanov2002dga}, which we comment further on in \Cref{lem:rotation in T2} and \Cref{rem:geometric grading}. \Cref{con:Seifert's algorithm} is heavily inspired by that in \cite[Sec. 4.2]{licata2013seifert}, with the distinction that our grading takes values in $\Z$ rather than $\Q$. The grading in \cite[Sec. 4.2]{licata2013seifert} takes rational values since the base space of their projection map is an orbifold and not necessarily a smooth surface. Consequently, their grading makes use of a \textit{rational} Euler characteristic and so-called \textit{defects}. Our projection has no orbifold points, and so we do not require a rational grading. More directly, our construction yields integer values since for a chord $c_{\pm}$, the resolution of \Cref{con:Seifert's algorithm} produces a capping surface with an odd number of corners, while for $c_{\Gamma}$ it produces a capping surface with an even number of corners. Hence $r(S_{c_\pm})$ is an odd multiple of $1/4$, while $r(S_{c_\Gamma})$ is an even multiple of $1/4$. Here we make use of the fact that we can perturb the diagram locally so that crossings in $\Sigma_\pm$ and intersections of $\Lambda$ and $\Gamma$ are orthogonal.

\begin{lemma}\label{lem:gradingwelldefined}
	The grading $|\cdot |$ is well-defined.
\end{lemma}

\begin{proof}
	We want to show that for any $c\in \widehat{\mathcal{A}}$, $|c|$ is independent of the choice of capping path $\gamma$ and formal capping surface $S$. We'll begin with the latter, so first fix a choice of $\gamma$. Any two capping surfaces differ by a multiple of $\Sigma$. We claim that $r(\Sigma) = 0$. We can write \[\Sigma = \Sigma_+ - \Sigma_-\] as a formal capping surface, which implies \[r(\Sigma) = r(\Sigma_+) - r(\Sigma_-) = \chi(\Sigma_+) - \chi(\Sigma_-) = c_1(\xi) = 0.\] Hence for a fixed capping path $|c|$ is independent of choice of the capping surface. To prove the former, first note that two capping paths differ by a multiple of $\pi(\Lambda)$. Let $S$ be a Seifert surface of $\pi(\Lambda)$. We claim that $r(S) = \rot(\Lambda)$. By definition $\rot(\Lambda)$ is the degree of $\gamma'$ for a parametrization $\gamma$ of $\Lambda$. Extend $\gamma'$ to be a vector field $Y$ on $S$ and by perturbing $Y$ generically (relative to $\partial S$) we may assume that $Y$ has finitely many isolated zeros. Then by the Poincaré-Hopf theorem we have \[\rot(\Lambda) = \ind(Y) = \chi(S) = r(S).\] Therefore, the two gradings differ by a multiple of $2r(S) = 2\rot(\Lambda)$.
\end{proof}
\begin{figure}[h]
\begin{minipage}{0.48\textwidth}
    
    \begin{center}
    	\input{images/unknot_in_T2xI}
    	\caption{The standard unknot in $\mathbb{T}^2 \times \R$.}
    	\label{stdUnknot}
    \end{center}
\end{minipage}
\hfill
\begin{minipage}{0.48\textwidth}

    \begin{center}
    	\input{images/torus_capping_surface}
    	\caption{Capping surfaces for Reeb chords $c$ and $d$.}
    	\label{fig:torus capping surface}
    \end{center}
    \end{minipage}
\end{figure}

\begin{example}
    Consider the standard unknot in \Cref{stdUnknot}. Here $\Sigma = \mathbb{T}^2$ and $\# \Gamma = 2$ with slope $0$. We pick $\mathbf{X}$ to be ``boundary'' of the square. For the Reeb chord $c$, we pick the capping path $\gamma_c$ to be the upper half of $\pi(\Lambda)$ and the capping surface $S_c$ to be the upper half-disk which is positively oriented, see \Cref{fig:torus capping surface}. Note that \[|c| = 2\left(1 - \frac{2}{4}\right) - 1 = 0,\qquad |\widehat{c}| = |c| + 1 = 1\] For the Reeb chord $d$, we pick the capping path $\gamma_d$ to be $\gamma_c$ with the opposite orientation, and pick the capping surface $S_d = \Sigma_+ - S_c$. We compute \[|d| = 2\left(0-\frac{2}{4}\right) -1= -2, \qquad |\widehat{d}| = |d| + 1 = -1.\] 
    Finally, for the Reeb chord $c*d$, we pick the capping surface to be the cylinder $\Sigma_+$, which is positively oriented. We compute    \[|c*d| = 2(0-0)-1=-1,\qquad |\widehat{c}| = |c|+1 = 0\]
    From this one can either compute similarly or infer from $|c|$ and $|d|$ the grading of all other chords.
\end{example}

\begin{figure}[h]
    \centering
    \input{images/genus_2_octagon}
    \caption{The standard unknot in $\Sigma_2\times \R$.}
    \label{fig:genus 2 octagon}
\end{figure}
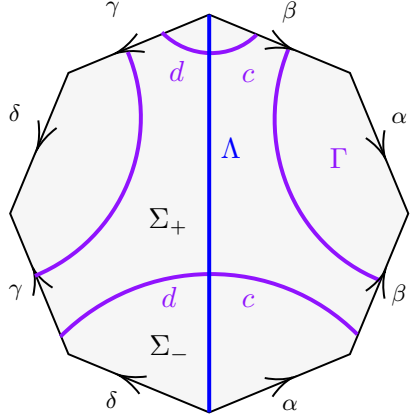

\begin{figure}[h]
    \centering
    \input{images/octagon_capping_surface}
    \caption{Capping surfaces for Reeb chords (a) $c$ and $d$, and (b) $c*d$.}
    \label{fig:octagon capping surface}
\end{figure}
\begin{example}
    Consider the example in \Cref{fig:genus 2 octagon}. Here $\Sigma$ is a genus-$2$ surface and $\Gamma$ divides $\Sigma$ into two pair-of-pants. We set $\mathbf{X} = \{a,b,\beta,\gamma\}$. For the Reeb chord $c$, we pick the capping path $\gamma_c$ to be the arc of $\pi(\Lambda)$ which does not intersect $\mathbf{X}$, and the capping surface $S_c$ to be the cylinder as in \Cref{fig:octagon capping surface} (a). We compute
\[|c| = 2\left(0-\frac{2}{4}\right)-1 = -2,\qquad |\widehat{c}|=|c|+1 = -1.\]
    For the Reeb chord $d$, we pick the capping path $\gamma_d$ to be $\gamma_c$ with reversed orientation, and pick the capping surface $S_d$ to be the cylinder as in \Cref{fig:octagon capping surface} (a). We compute
\[|d| = 2\left(0-\frac{2}{4}\right)-1 = -2,\qquad |\widehat{d}|=|d|+1 = -1.\]
    Finally, for the Reeb chord $c*d$, we choose the empty capping path, and the formal capping surface $S_{c*d}
$ to be the pair-of-pants as in \Cref{fig:octagon capping surface} (b). We compute
    \[
        |c*d| 
        = 2\left(-1+0\right)-1 = -3,\qquad
         |\widehat{c*d}|=|c*d|+1 = -2.
    \]
    From this one can either compute similarly or infer from $|c|$ and $|d|$ the grading of all other chords.
\end{example}
So far our examples have focused on knots whose capping paths require no Seifert resolutions. The following example demonstrates the procedure for a trefoil knot with a chord whose capping path \textit{does} require a resolution. 
\begin{example}
    Consider the torus knot in \Cref{fig:trefoil seifert 1}. We leave it as an exercise that it has a Legendrian lift (see \Cref{LegConstr}). Here, $\Sigma=\T^2$ and $\#\Gamma=2$ with slope $0$. As before, we pick $\mathbf{X}$ to be the ``boundary'' of the square. For the Reeb chord $c$, There are two choices of capping path $\gamma_c$: one that remains entirely in $\Sigma_+$, and one that intersects $\Gamma$ multiple times. Choose the latter path to be $\gamma_c$. We perform a Seifert resolution following \Cref{con:Seifert's algorithm} to obtain \Cref{fig:trefoil seifert 2}(a), so that $c$ admits a capping surface $S_c = -A+B+C$ as in \Cref{fig:trefoil seifert 2}(b). We compute 
\[r(A)= 1-\frac{2}{4} = \frac{1}{2}, \qquad r(B)=1-\frac{5}{4} = -\frac{1}{4},\qquad r(C)=  1+\frac{1}{4} = \frac{5}{4},\]
  so that 
  \[|c| = 2\left(-\frac{1}{2}-\frac{1}{4}+\frac{5}{2}\right)-1=0.\]
It is easy to check that this agrees with the grading given by choosing the other capping path. 
\end{example}

\begin{figure}
    \centering
    \input{images/trefoil_Seifert_resolution_1}
    \caption{A Legendrian trefoil knot in $\T^2\times \R$.}
    \label{fig:trefoil seifert 1}
\end{figure}
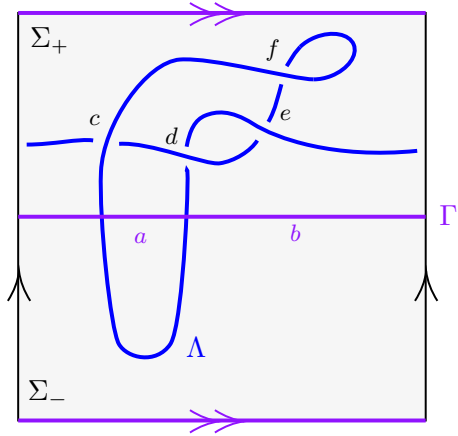

\begin{figure}
    \centering
    \input{images/trefoil_Seifert_resolution_2}
    \caption{A capping path and surface for the Reeb chord $c$.}
    \label{fig:trefoil seifert 2}
\end{figure}

The following example highlights that the grading of generators extends that of Chekanov (see proof of \cite[Lem. 6.3]{chekanov2002dga}), except that there is a $-2 \cdot \frac{1}{4}$ adjustment in $|c_\Gamma|$ accounting for the fact that we are viewing $c_\Gamma$ as a corner (the positive corner) rather than an edge of a polygon:
\begin{example}\label{lem:rotation in T2} Let $\Sigma$ is a torus $\T^2$ and $\#\Gamma =2$ with slope 0. Suppose that $\Lambda\subset \Sigma\times \R$ is a Legendrian knot and that $\pi(\Lambda)$ is null-homotopic. Let $\rot(\gamma)$ denote the rotation number of a curve $\gamma$ in $\R^2$. Then \[|c_\pm| = 2\rot(\gamma) - \frac{1}{2} \mod 2\rot(\Lambda)\] \[|c_\Gamma| = 2\rot(\gamma) - 1 \mod 2\rot(\Lambda)\] for generators $c_\pm$ of $\mathcal{A}_\pm$ and $c_\Gamma$ of $\mathcal{A}_\Gamma$. Note that the rotation number of a counterclockwise circle is $+1$ in $\Sigma_+$ and $-1$ in $\Sigma_-$.
\end{example}

\begin{lemma}\label{lem:diff=1}
	Let $u$ be an immersed polygon with positive corner $c_1$ and negative corners $c_2, \dots , c_k$. Then \[|c_1| - \sum_{i = 2}^k |c_i| = 1.\] (If $c_1$ is along $\Gamma$ we use $|\widehat{c}_1|$ instead.)
\end{lemma}

\begin{proof}
	This is the analogue of \cite[Lem. 6.3]{chekanov2002dga} and \cite[Thm 4.8]{licata2013seifert}. Pick corresponding capping surfaces $S_{c_i}$. Since both $S_{c_1}$ and $S_{c_2} + \dots  + S_{c_k} + u$ are formal capping surfaces of $c_1$, then modulo $2\rot(\Lambda)$ we have \[r(S_{c_1}) = r(S_{c_2}) + \dots  + r(S_{c_k}) + r(u).\] Assume $c_1$ is in $\Sigma_\pm$. Then \[r(u) = \chi(\mathbb{D}) - \frac{k}{4} = 1 - \frac{k}{4}.\] Hence, \begin{align*}
	    |c_1| &= 2r(S_{c_1}) - \frac{1}{2} \\
        &= 2\left(r(S_{c_2}) + \dots  + r(S_{c_k}) + 1 - \frac{k}{4}\right) - \frac{1}{2} \\
        &= |c_2| + \dots  + |c_k| + \frac{k-1}{2} + 2\left(1-\frac{k}{4}\right) - \frac{1}{2} \\
        &= |c_2| + \dots  + |c_k| + 1
	\end{align*} If $c_1$ is along $\Gamma$, then $r(u) = 1 - \frac{k+1}{4}$ and $|c_1| = 2r(S_{c_1}) - 1$, so \[|\widehat{c}_1| = |c_1| + 1 = |c_2| + \dots  + |c_k| + 1.\]
\end{proof}

\begin{corollary}\label{cor:phipmgrading}
    $\Phi_\pm: \mathcal{A}_\Gamma \rightarrow \mathcal{A}_\pm$ preserves the grading.
\end{corollary}

\begin{proof}
    This follows from \Cref{lem:diff=1} and the fact that $|\widehat{c}_\Gamma| = |c_\Gamma| + 1$.
\end{proof}

\begin{remark}\label{rem:geometric grading} One should expect the above to give a grading that is in the spirit of the standard geometric definition, see for example \cite[Sec. 1.6]{egh2000SFT}. Under the identification $H_1(\Sigma\times \R)\simeq H_1(\Sigma)$, we may choose a symplectic trivialization of $\xi$ over the lifts of each generator $X_i\in \mathbf{X}\subset H_1(\Sigma)$. For a Reeb chord $c\in \widehat{\mathcal{A}}$, let $S$ denote the formal capping surface associated to a capping path $\gamma$. The surface $S$ lifts to a surface in $\Sigma\times \R$, and there exists a trivialization of $\xi$ along $\gamma$ that extends over $\xi|_S$ and agrees with the chosen trivializations along the lifts of the $X_i$. With respect to this trivialization, one can associate to $\gamma$ and hence to $c$ a well-defined Conley–Zehnder index.  This is analogous to \cite[Rem. 4.4]{licata2013seifert}.
\end{remark}
\subsubsection{Differential}

Our definition of the differential coincides with \cite[Lem. 2.3]{ekholm2016nonloose} when $\pi(\Lambda) \subset \Sigma$ is embedded (see also \cite[Sec. 3.4, 3.5]{ekholm2016cobordisms}). For generators $c_\pm$ of $\mathcal{A}_\pm$, we use the same grading as in $\mathcal{A}_\pm$ and define \[\widehat{\partial}(c_\pm) = \partial_\pm c_\pm.\] For generators $c_\Gamma$ of $\mathcal{A}_\Gamma$, we increase the grading by $1$ (using the notation $\widehat{c}_\Gamma$) and define \[\widehat{\partial}(\widehat{c}_\Gamma) = \Omega(\partial_\Gamma c_\Gamma) + \Phi_+(c_\Gamma) + \Phi_-(c_\Gamma),\] where \[\Omega(0) = \Omega(1) = 0, \Omega(c_\Gamma) = \widehat{c}_\Gamma,\] and extend $\Omega$ so that it is a chain homotopy between $\Phi_+$ and $\Phi_-$, i.e. \[\Omega(c_\Gamma^1c_\Gamma^2) = \Phi_+(c_\Gamma^1)\widehat{c}_\Gamma^2 + \widehat{c}_\Gamma^1\Phi_-(c_\Gamma^2)\] for generators $c_\Gamma^1, c_\Gamma^2$ of $\mathcal{A}_\Gamma$, and more generally, \[\Omega(c^1_\Gamma \dots c^n_\Gamma) = \Phi_+(c^1_\Gamma \dots c^{n-1}_\Gamma)\widehat{c}^n_\Gamma + \Phi_+(c^1_\Gamma \dots c^{n-2}_\Gamma)\widehat{c}^{n-1}_\Gamma\Phi_-(c^n_\Gamma) + \dots + \widehat{c}^1_\Gamma \Phi_-(c^2_\Gamma \dots c^n_\Gamma)\] for generators $c_\Gamma^1, \dots, c_\Gamma^n$ of $\mathcal{A}_\Gamma$. We remark that in general, arbitrarily long Reeb chords can have nontrivial differential, see \Cref{exm:arbitrary long correct}. In \Cref{prop:homotopy finite} and \Cref{cor:homotopy finite} we give a homotopy condition on $\Lambda$ such that $\widehat{\partial}(c)=0$ for $H(c)$ sufficiently large.

\begin{figure}[h]
    \centering
    \input{images/legendrian_trefoil_surface}
    \caption{A genus-2 convex surface with dividing set $\Gamma$.}
    \label{fig:legendrian trefoil surface}
\end{figure}

\begin{example}\label{exm:arbitrary long correct}
    We construct a Legendrian trefoil knot in $\Sigma_2\times \R$ that is nullhomologous, but not nullhomotopic, whose differential is nontrivial for arbitrarily long Reeb chords. Here, $\Sigma_2$ denotes the genus-2 surface, and $\Gamma$ is given as in \Cref{fig:legendrian trefoil surface}. To construct $\Lambda$, first consider the Legendrian $\Lambda'$ in $\R^3$ given by the Lagrangian projection in \Cref{fig:legendrian trefoil 1}. It is easy to check that this indeed lifts to a \textit{closed} Legendrian using \cite[Sec. 11]{chekanov2002dga}. We perform the inverse closing operation on $\Lambda'$ of \cite[Sec. 12]{chekanov2002dga} to construct a long Legendrian knot at the crossing $c_l$, obtaining \Cref{fig:legendrian trefoil 2}. In particular, we may arrange that the vertical $\R$-coordinates of the asymptotic ends are equal. 
    \begin{figure}[h]
        \begin{minipage}{0.48\textwidth}
            \centering
        \input{images/legendrian_trefoil_1}
        \caption{A Legendrian trefoil knot in $\R^3$.}
        \label{fig:legendrian trefoil 1}
        \end{minipage}
        \hfill
        \begin{minipage}{0.48\textwidth}
            \centering
        \input{images/legendrian_trefoil_2}
        \caption{A long Legendrian trefoil knot in $\R^3$.}
        \label{fig:legendrian trefoil 2}
        \end{minipage}
    \end{figure}

    We embed the long Legendrian knot in a cylindrical neighborhood of the dividing set, such that its image is contained in a neighborhood $D$ as in \Cref{fig:legendrian trefoil 3} that 1) is sufficiently far away from the dividing set to allow us to identify the contact structure with $\ker(\beta+\de t)$, and 2) the distance between the image of the two asymptotic ends of $\Lambda^l$ is small. In $\Sigma_2$, we define $\Lambda$ to be the concatenation of the image of $\Lambda^l$ with a small arc $\delta$ as in \Cref{fig:legendrian trefoil 3}. By \Cref{lem:legendrian lift}, this lifts to a Legendrian arc in $\Sigma_2\times \R$. Since $\delta$ is small, we may introduce a small perturbation of $\Lambda^l$ so that $\Lambda$ actually lifts to a \textit{closed} Legendrian, arranging the $\R$-coordinates of the asymptotic ends of $\Lambda^l$ to account for the height difference from $\delta$. Finally, we perform a flipping operation of \Cref{LegConstr} so that $\Lambda$ intersects $\Gamma$, arriving at a local picture of \Cref{fig:legendrian trefoil 4}.

    \begin{figure}[h]
    \begin{minipage}{0.48\textwidth}
        \centering
        \input{images/legendrian_trefoil_3}
        \caption{We embed the projection of the long Legendrian $\Lambda^l$ into $\Sigma_2^+$, and close it by a small arc $\delta$.}
        \label{fig:legendrian trefoil 3}
        \end{minipage}
        \hfill
        \begin{minipage}{0.48\textwidth}
        \centering
        \input{images/legendrian_trefoil_4}
        \caption{A local picture of the Legendrian trefoil $\Lambda$ intersecting $\Gamma$.}
        \label{fig:legendrian trefoil 4}
    \end{minipage}
    \end{figure}

    There are two indecomposable chords along $\Gamma$, $a$ and $b$. We can readily read off for example that $\<\widehat{\partial}(b*(ab)^{*n}),de^{n+1}c\>=1$. We have thus shown that $\Lambda$ has arbitrarily long Reeb chords whose differential is nontrivial. 
\end{example}




\begin{lemma}\label{prop:homotopy finite}
    Assume that $c_\Gamma$ winds around $\Gamma$ for $n$ times (but not necessarily closed) and $\Phi_+(c_\Gamma) \neq 0$ (or $\Phi_-(c_\Gamma) \neq 0$). Then there exists a (piecewise-smooth) arc $\gamma_1$ in $\pi(\Lambda)$ and a non-self-overlapping arc $\gamma_2$ along $\Gamma$ such that $\gamma_1 \cdot \gamma_2$ is closed and winds around $\Gamma$ for $n$ times, where ``$\cdot$'' denotes concatenation of two arcs.
\end{lemma}

\begin{proof}
    Let $\gamma_1$ be the concatenation of all sides of a polygon counted in $\Phi_\pm(c_\Gamma)$ but $c_\Gamma$. Then $\gamma_1$ is homotopic to $c_\Gamma$. Find a non-self-overlapping arc $\gamma_2$ along $\Gamma$ such that $c_\Gamma \cdot \gamma_2$ is closed and winds around $\Gamma$ for $n$ times. Then $\gamma_1 \cdot \gamma_2$ is homotopic to $c_\Gamma \cdot \gamma_2$.
\end{proof}

\begin{corollary}\label{cor:homotopy finite}
    If $\pi(\Lambda)$ is nullhomotopic, then there exists an $l>0$ such that for all generator $c$ of $\mathcal{A}_\Gamma$, $H(c)>l$ implies $\Phi_\pm(c)=0$.
\end{corollary}

\begin{proof}
    Equivalently, $\pi(\Lambda)$ is contained in a disk in $\Sigma$. So in the lemma above $\gamma_1 \cdot \gamma_2$ is a trivial loop in $\Gamma$.
\end{proof}



\section{Linearization and examples}\label{sec:linearization}

In this section, we compute several examples of linearized Legendrian contact homology induced by augmentations. This is the counterpart of \cite[Sec. 4,5]{chekanov2002dga}.

\begin{example}(the standard unknot)
	Consider the standard unknot $\Lambda$ (i.e. $\tb = -1$, $\rot = 0$) in \Cref{stdUnknot}. To simplify the notation we suppress all ``widehat'' symbols above the Reeb chords along $\Gamma$. The generators of $\widehat{\mathcal{A}}$ are \[c *(*^n (d * c)), d *(*^n (c * d)), *^{n+1}(c*d), *^{n+1}(d*c),\] where $n \geq 0$. The gradings are given by \[|c *(*^n (d * c))| = 1, |d *(*^n (c * d))| = -1, |*^{n+1}(c*d)| = |*^{n+1}(d*c)| = 0.\] The differential is given by \[\widehat{\partial} (c *(*^n (d * c))) = (c*d)^n + (d*c)^n,\] \[\widehat{\partial} (d *(*^n (c * d))) = 0,\] \[\widehat{\partial}(*^{n+1}(c*d)) = \widehat{\partial}(*^{n+1}(d*c)) = d *(*^n (c * d)).\] There are infinitely many augmentations given by augmenting $*^{n+1}(c*d)$ and $*^{n+1}(d*c)$ in pairs. More precisely, there exists a bijection \[\{I \subset \mathbb{Z}_{\geq 0}\} \rightarrow \{\epsilon:\widehat{\mathcal{A}} \rightarrow \Z_2 : \epsilon \text{ is a DGA map}\}\] \[I \mapsto \epsilon_I,\] where \[\epsilon_I(*^{n+1}(c*d)) = \epsilon_I(*^{n+1}(d*c)) = 1\] for all $n \in I$ and $\epsilon_I(e) = 0$ for all other chords $e$. The linearized differentials satisfy \[\widehat{\partial}_{\epsilon_I} = \widehat{\partial}\] since $(c * d)^n$ and $(d * c)^n$ always appear in pairs in any $\widehat{\partial} e$. The linearized Legendrian contact homologies are given by \[H^{\epsilon_I}_k(\Lambda) = \begin{cases}
		\Z_2 & k = 1 \\
		0 & \text{otherwise}
	\end{cases}\] for all $I$, where $H^{\epsilon_I}_1(\Lambda) = \langle c \rangle$. The set of all Chekanov polynomials is \[I(\widehat{\mathcal{A}}, \widehat{\partial}) = \{t\},\] which agrees with that of the standard unknot computed in the ``standard Lagrangian projection''. Indeed, the two diagrams are related by the sequence of Reidemeister moves shown in \Cref{fig:R moves unknot}.
\end{example}

\begin{figure}[h]
    \centering
    \input{images/R_moves_unknot}
    \caption{A sequence of an RI move followed by an RII move isotopes the max tb unknot from its standard Lagrangian projection in the interior of $\Sigma_+$ to one that intersects $\Gamma$.}
    \label{fig:R moves unknot}
\end{figure}

\begin{example}(the Chekanov $5_2$ knots)
	We construct a pair of Chekanov $5_2$ knots $\Lambda_1$, $\Lambda_2$ that are null-homologous but not null-homotopic. This is motivated by \cite[Sec. 5.2]{licata2013seifert} using long Legendrian knots defined in \cite[Sec. 12]{chekanov2002dga}. We take the Chekanov $5_2$ knots and perform the inverse closing operation at $a_3$ to obtain long Legendrian knots $\Lambda_1^l$, $\Lambda_2^l$, see \Cref{fig:5_2 knot diagrams}. Let $\Sigma$ be a genus-$2$ surface with dividing set $\Gamma$ as shown in  \Cref{chekanov5_2_2}. See \Cref{chekanov5_2} for diagrams of $\Sigma_\pm$. Placing $\Lambda_1^l, \Lambda_2^l$ into the box $K$ we get $\Lambda_1$, $\Lambda_2$ which are Legendrian knots in $\Sigma \times \R$. To justify the existence of Legendrian lifts of $\Lambda_1$, $\Lambda_2$ rigorously, we can apply \Cref{subsec:examplesOfLegendrians} by starting with the actual Chekanov $5_2$ knots in $\R^3$, isotoping it to add two cusps along $\Gamma$ so that half of it contains $a_3$ and the other half is the rest of the knot, and then ``flipping'' the half containing $a_3$ to $\Sigma_-$.
    
    For both $\Lambda_1$ and $\Lambda_2$, this procedure decreases the tb number by $1$ since a positive self-intersection in $\Sigma_+$ is replaced by a positive self-intersection in $\Sigma_-$ plus two intersection points with $\Gamma$. The rotation number remains the same since the rotation number of the standard unknot shown in right of \Cref{chekanov5_2_2} is $0$ by symmetry. Hence, $\Lambda_1$ and $\Lambda_2$ have the same classical invariants $\tb(\Lambda_1) = \tb(\Lambda_2) = 1 - 1 = 0$ and $\rot(\Lambda_1) = \rot(\Lambda_2) = 0$.
    
    Now we compute their Chekanov polynomials $P_{\Lambda_1}(t), P_{\Lambda_2}(t)$. Observe that $\widehat{\partial}$ vanishes on all generators $c, d, c*d, d*c, \dots $ along $\Gamma$ which contributes $2 \sum_{k = 1}^\infty t^k$ to both $P_{\Lambda_1}(t)$ and $P_{\Lambda_2}(t)$. The differential on all other generators remains the same except \[\widehat{\partial} a_3 = 1\] for both $\Lambda_1, \Lambda_2$. Note that $a_5, a_6$ still contribute $t^{-2}, t^2$ and $t^0, t^0$ to $P_{\Lambda_1}(t), P_{\Lambda_2}(t)$, respectively. However, $H_1^\epsilon(\Lambda_1)$ vanishes since by a similar computation as in \cite{chekanov2002dga} we have \begin{align*}
        \widehat{\partial}^\epsilon a_1 &= a_7 \\
        \widehat{\partial}^\epsilon a_2 &= a_9 \\
        \widehat{\partial}^\epsilon a_3 &= 1 \\
        \widehat{\partial}^\epsilon a_4 &= a_8 + a_9
    \end{align*} and \[c_1 a_7 + c_2 a_9 + c_3 + c_4(a_8 + a_9) = 0 \implies c_1 = c_2 = c_3 = c_4 = 0.\] One can check similarly that $H_1^\epsilon(\Lambda_2)$ also vanishes. Therefore \[P_{\Lambda_1}(t) = t^{-2} + t^2 + 2 \sum_{k = 1}^\infty t^k \neq 2t^0 + 2 \sum_{k = 1}^\infty t^k = P_{\Lambda_2}(t).\] Hence, $\Lambda_1$ and $\Lambda_2$ are not Legendrian isotopic. 
\end{example}

\begin{figure}[h]
    \begin{center}
    	\input{images/genus_2_unknot}
    	\caption{A genus-$2$ convex surface with a standard unknot.}
    	\label{chekanov5_2_2}
    \end{center}
\end{figure}
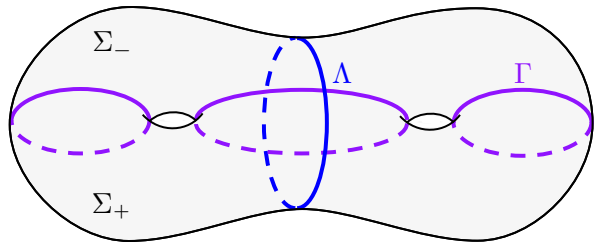

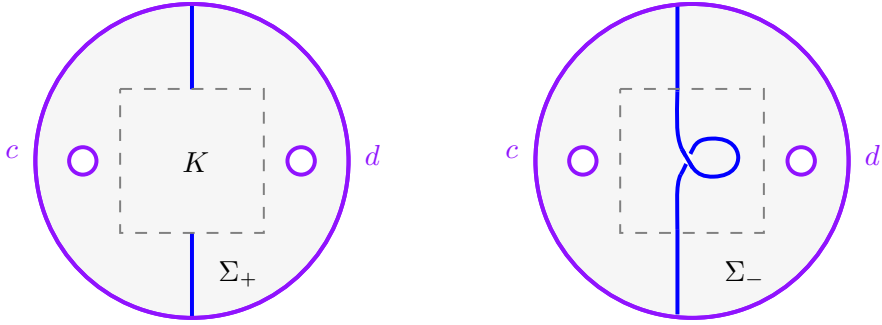
\begin{figure}[h]
    \begin{center}
    	\input{images/five_two_schematic}
    	\caption{A convex surface with the Chekanov $5_2$ knots.}
    	\label{chekanov5_2}
    \end{center}
\end{figure}

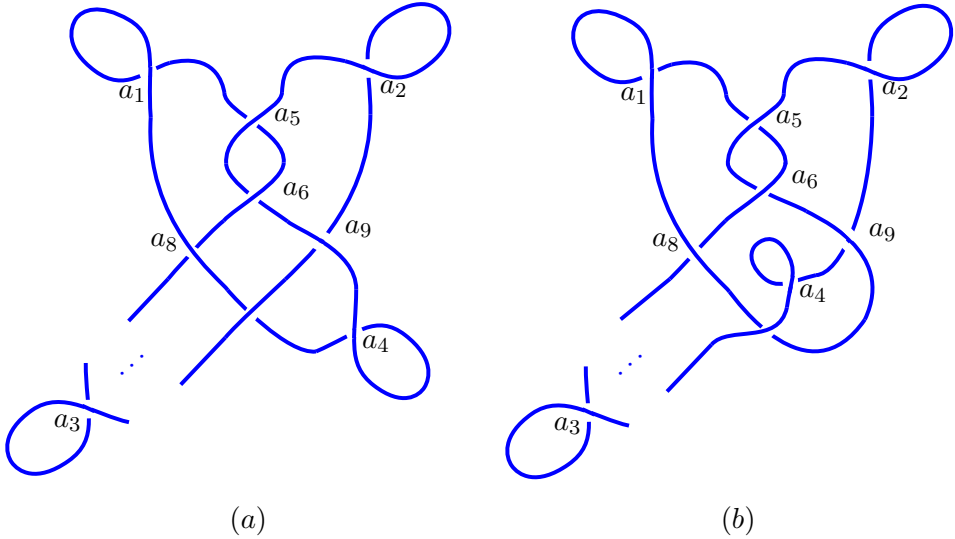
\begin{figure}[h]
    \centering
    \input{images/five_two_knots}
    \caption{A knot diagram and labeling of (a) $\Lambda_1^l$ and (b) $\Lambda_2^l$. }
    \label{fig:5_2 knot diagrams}
\end{figure}
\section{Proof of $\widehat{\partial}^2=0$}\label{sec:d2 is 0}

In this section we prove that $\Phi_{\pm}: \mathcal{A}_\Gamma \rightarrow \mathcal{A}_\pm$ are DGA maps and thus $\widehat{\partial}^2 = 0$. Our proof uses a broken-heart-type argument analogous to \cite[Sec. 7]{chekanov2002dga}. First we prove $\widehat{\partial}^2 = 0$ assuming that $\Phi_\pm$ are DGA maps.

\begin{theorem} \it
	$\widehat{\partial}^2 = 0$.
\end{theorem}

\begin{proof}
	Note that \[\widehat{\partial}^2(c_\Gamma) = \widehat{\partial}(\Omega(\partial_\Gamma c_\Gamma) + \Phi_+(c_\Gamma) + \Phi_-(c_\Gamma)) = (\widehat{\partial} \circ \Omega + \Phi_+ + \Phi_-)(\partial_\Gamma c_\Gamma).\] We claim that for all $w \in \mathcal{A}_\Gamma$, \[(\widehat{\partial} \circ \Omega + \Phi_+ + \Phi_-)(w) = \Omega \circ \partial_\Gamma (w).\] Since $\widehat{\partial}, \Gamma, \Phi_\pm$ are additive, it suffices to prove this when $w$ is a monomial. Now we use induction on the length of $w$. For a base case, let $a$ be a generator of $\widehat{\mathcal{A}}$. Note that \begin{align*}
		(\widehat{\partial} \circ \Omega + \Phi_+ + \Phi_-)(a)
		&= \widehat{\partial}(a) + \Phi_+(a) + \Phi_-(a) \\
		&= \Omega(\partial_\Gamma a) + \Phi_+(a) + \Phi_-(a) + \Phi_+(a) + \Phi_-(a) \\
		&= \Omega(\partial_\Gamma a).
	\end{align*} Assume the statement holds for all $w$ of length $n-1$. Then \begin{align*}
		&(\widehat{\partial} \circ \Omega + \Phi_+ + \Phi_-)(a_1\dots a_n) \\
		=& \widehat{\partial}(\Omega(a_1\dots a_{n-1})\Phi_-(a_n) + \Phi_+(a_1\dots a_{n-1})\widehat{a}_n) + \Phi_+(a_1\dots a_n) + \Phi_-(a_1\dots a_n) \\
		=& \widehat{\partial} \circ \Omega(a_1\dots a_{n-1}) \Phi_-(a_n) + \Omega(a_1\dots a_{n-1})\Phi_-(\partial_\Gamma a_n) + \Phi_+(\partial_\Gamma(a_1\dots a_{n-1}))\widehat{a}_n \\
		+& \Phi_+(a_1\dots a_{n-1})(\Omega(\partial_\Gamma a_n) + \Phi_+(a_n) + \Phi_-(a_n)) + \Phi_+(a_1\dots a_n) + \Phi_-(a_1\dots a_n).
	\end{align*} By the induction hypothesis, \[\widehat{\partial} \circ \Omega(a_1\dots a_{n-1}) = \Omega(\partial_\Gamma (a_1\dots a_{n-1})) + \Phi_+(a_1\dots a_{n-1}) + \Phi_-(a_1\dots a_{n-1}).\] Hence, \begin{align*}
		&(\widehat{\partial} \circ \Omega + \Phi_+ + \Phi_-)(a_1\dots a_n) \\
		=& \Omega(\partial_\Gamma (a_1\dots a_{n-1}))\Phi_-(a_n) + \Phi_+(a_1\dots a_{n-1})\Phi_-(a_n) + \Phi_-(a_1\dots a_n) \\
		+& \Omega(a_1\dots a_{n-1})\Phi_-(\partial_\Gamma a_n) + \Phi_+(\partial_\Gamma(a_1\dots a_{n-1}))\widehat{a}_n \\
		+& \Phi_+(a_1\dots a_{n-1})(\Omega(\partial_\Gamma a_n) + \Phi_+(a_n) + \Phi_-(a_n)) + \Phi_+(a_1\dots a_n) + \Phi_-(a_1\dots a_n) \\
		=& \Omega(\partial_\Gamma (a_1\dots a_{n-1}))\Phi_-(a_n) + \Phi_+(\partial_\Gamma(a_1\dots a_{n-1}))\widehat{a}_n \\
		+& \Omega(a_1\dots a_{n-1})\Phi_-(\partial_\Gamma a_n) + \Phi_+(a_1\dots a_{n-1})(\Omega(\partial_\Gamma a_n)) \\
		=& \Omega(\partial_\Gamma (a_1\dots a_{n-1})a_n) + \Omega(a_1\dots a_{n-1}(\partial_\Gamma a_n)) \\
		=& \Omega(\partial_\Gamma(a_1\dots a_n)).
	\end{align*} Therefore, \[(\widehat{\partial} \circ \Omega + \Phi_+ + \Phi_-)(\partial_\Gamma c_\Gamma) = \Omega(\partial_\Gamma^2 c_\Gamma) = \Omega(0) = 0.\]
\end{proof}

\begin{proposition}\label{prop:phipmDGA} 
	$\Phi_{\pm}$ are DGA maps.
\end{proposition}

\begin{proof}
    It suffices to show that $\Phi_+$ is a DGA map, as the proof for $\Phi_-$ is identical. We already saw in \Cref{cor:phipmgrading} that $\Phi_+$ preserves the grading. Let $w\in \mathcal{A}_\Gamma$ be a generator, which we think of as a concatenation of some nonempty compatible set of indecomposable Reeb chords. We'd like to show that 
    \[\partial_+\circ \Phi_+(w) = \Phi_+\circ \partial_\Gamma (w).\tag{\dag}\]

    Suppose that $v\in \mathcal{A}_+$ is some monomial appearing in the right hand side of $(\dag)$. Then the coefficient associated to $v$ counts the number of disks of the form in \Cref{fig:decomp 1}. That is, it is a count of two disks with common boundary $\gamma$ whose concatenation has positive end $w$ and negative end $v$. Without loss of generality, we assume that $a$ is a negatively asymptotic endpoint of the disk with positive end $w'$. We also assume for now that the beginning point of $w'$ is distinct from the end point of $w''$. We will address the case that they coincide in \Cref{rem: why hypertight}.

    \begin{figure}[h]
        \centering
        \input{images/hatdsquare_=0/decomposition_1}
        \caption{A collection of disks corresponding to $v$ in $\Phi_+\circ \partial_\Gamma$.}
        \label{fig:decomp 1}
    \end{figure}
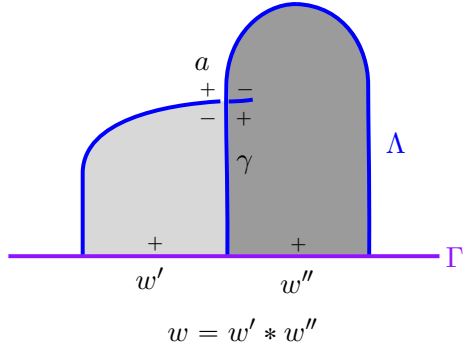

We'd like to show that $v$ appears as a monomial on the right hand side with the same $\Z_2$ coefficient as on the left hand side of $(\dag)$. Notice that concatenating the two disks along their common boundary $\gamma$ produces a disk with a positive end at $w'*w'' = w$, one obtuse corner negatively asymptotic to some $a\in \mathcal{A}_+$, and all other corners acute and negatively asymptotic to chords in $\mathcal{A}_+$. We may then consider the disks in \Cref{fig:decomp 1} as a degeneration of such a disk. Degenerating along the other strand produces the configuration in \Cref{fig:decomp 2}.

    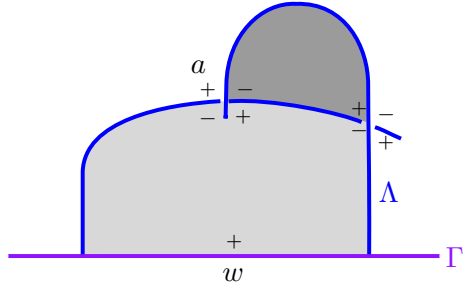
\begin{figure}[h]
        \centering
        \input{images/hatdsquare_=0/decomposition_2}
        \caption{Degeneration of disk along second strand, corresponding to $v$ in $\partial_+\circ \Phi_+$. }
        \label{fig:decomp 2}
    \end{figure}

    These disks contribute $+1$ to the coefficient of $v$ in the left hand side of $(\dag)$.  Note that we assumed a particular position of the Legendrian knot relative to the dividing set $\Gamma$. In particular, we assumed that one strand at $a$ intersects $\Gamma$, and the other intersects $\Lambda$.  In the event that the knot $\Lambda$ is not in this position, then without loss of generality $\Lambda$ must occupy one of the other three configurations in \Cref{fig:decomp 3}. In each of these cases, following a similar degeneration argument along each strand would contribute $2\equiv 0$ (mod $2$) to either the left hand side or right hand side of $(\dag)$. 

    \begin{figure}[h]
        \centering
        \input{images/hatdsquare_=0/decomposition_3}
        \caption{Other strand configurations at obtuse corner $a$. Each row encodes the two possible degenerations of a disk with obtuse corner a, that has a strands with ends on $\Lambda$ or on $\Gamma$.}
        \label{fig:decomp 3}
    \end{figure}

\end{proof}

\begin{remark}\label{rem: why hypertight} Suppose that $\text{start}(w) = \text{end}(w)$. Consider a monomial $v\in \Phi_+\circ \partial_\Gamma(w)$. Then the coefficient of $v$ is a count of rigid polygons of the form in \Cref{fig:hypertight 1}. Note in particular that the two disks share two boundary arcs, $\gamma$ and $\gamma'$, which  may coincide. If they do not coincide, then at least one disk must have a negative asymptotic end in $\Sigma_+$. In this case, the argument proceeds identically to that in the proof of \cref{prop:phipmDGA}. If $\gamma$ and $\gamma'$ do coincide, then $\Lambda$ must have the configuration shown in \cref{fig:hypertight 3}. We may think of this case as a degeneration of a disk along an obtuse corner on $\Gamma$ at $\text{start}(w') = \text{end}(w'')$. For such a disk there is only one possible degeneration, which would imply that $\Phi_+$ is not a chain map. We can circumvent this problem by noting that the concatenation of the two disks is a disk whose boundary is $w$, so that $w$ a contractible Reeb orbit. Restricting to the case that $\xi$ is hypertight eliminates this possibility. 

\begin{figure}[h]
    \begin{minipage}{0.48\textwidth}
        \centering
    \input{images/hypertight_1}
    \caption{Generic configuration of disks counted in $\Phi_+(w')\Phi_+(w'')$ for a closed Reeb orbit $w$, where at least one disk has a nontrivial negative end.}
    \label{fig:hypertight 1}
    \end{minipage}
    \hfill
    \begin{minipage}{0.48\textwidth}
        \centering
    \input{images/hypertight_3}
    \caption{Generic configuration of disks counted in $\Phi_+(w')\Phi_+(w'')$ for a closed Reeb orbit $w$, where neither disk has a nontrivial negative end.}
    \label{fig:hypertight 3}
    \end{minipage}
\end{figure}

\begin{example}
    Consider the standard unknot in \Cref{fig:S2xI}. Here, $\Sigma = S^2$ and $\Gamma$ is the equator. Clearly $\xi$ is not hypertight, as $\Gamma$ itself is a contractible Reeb orbit. We compute that $\Phi_+(a*b)=1$, $\Phi_+(a)=1$ and $\Phi_+(b)=1$, so 
    \[\partial_+\circ\Phi_+(a*b)=0 \neq 1 =\Phi_+\circ\partial_\Gamma(a*b). \] So $\Phi_+$ is not
    a chain map. 
\end{example}

\begin{figure}[h]
    \centering
    \input{images/S2xI}
    \caption{Standard Legendrian unknot in $S^2$.}
    \label{fig:S2xI}
\end{figure}
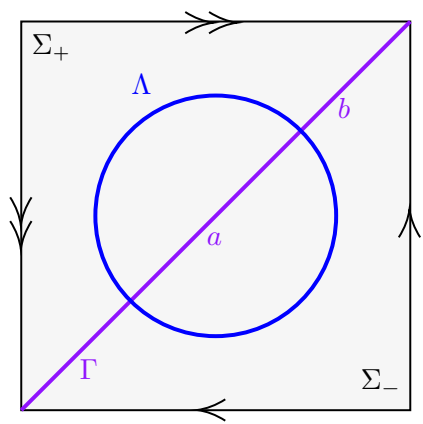
\end{remark}







\section{Proof of invariance under Legendrian isotopy}\label{sec:invariance}

In this section we prove that $(\widehat{\mathcal{A}}, \widehat{\partial})$ is invariant under Legendrian isotopy by showing that it is invariant under the four types of Legendrian Reidemeister moves for the convex surface projection.

\subsection{Reidemeister IV(a) move} \text{ }

\begin{figure}[h]
    \begin{center}
    	\input{images/RIV_images}
    	\caption{The Reidemeister IV moves, with important Reeb chords labeled.}
    	\label{Reidemeister IV}
    \end{center}
\end{figure}
The proof can be divided into two steps: first we make the two DGAs before and after the move agree on all generators in $\Sigma_\pm$ together with $a$ as shown in \Cref{Reidemeister IV}(a). This step proceeds in a similar way as in \cite[Sec. 8.4]{chekanov2002dga}. Second we make the two DGAs agree on the rest of the generators, namely the generators along $\Gamma$ other than $a$. Both steps involve conjugating the differential by a sequence of elementary automorphisms corresponding to the generators.

\subsubsection*{Step 1:} 
Denote by $(\mathcal{A}^-, \partial^-)$ and $(\mathcal{A}', \partial')$ the DGA before and after the move, respectively. We have $\partial^{-}(a) = b + v$ for some $v$. Conjugate $\partial^{-}$ by the elementary automorphism sending $a \mapsto a$, $b \mapsto b + v$ and still denote the differential by $\partial^{-}$. This has the effect that $\partial^{-}(a) = b$ and $\partial^{-}(b) = 0$. Similarly, we have $\partial'(a) = b + w$ for some $w$. Conjugate $\partial'$ by the elementary automorphism sending $a \mapsto a$, $b \mapsto b + w$ and still denote the differential by $\partial'$. This has the effect that $\partial'(a) = b$ and $\partial'(b) = 0$. Still denote by $\partial^{-}$ the conjugate $s \circ \partial^{-} \circ s^{-1}$, where $s : A^{-} \to A'$ is the identification of generators given by \Cref{Reidemeister IV}(a), meaning that generators in $\Sigma_\pm$ and indecomposable Reeb chords along $\Gamma$ are identified as in the figure and decomposable Reeb chords along $\Gamma$ are identified as follows: in $\mathcal{A}^-$ and $\mathcal{A}'$ Reeb chords along $\Gamma$ (excluding $a$) come in pairs $(c_\Gamma^-, c_\Gamma')$ for which there is a natural correspondence between polygons with a positive corner at $c_\Gamma^-$ and $c_\Gamma'$, modulo those with negative corners at $b$. For example, if $\Gamma$ has only two indecomposable chords $a$ and $c$ (i.e. $c$ and $d$ in \Cref{Reidemeister IV}(a) are the same chord), then \[(c_\Gamma^-, c_\Gamma') = (c, a * c * a), (a * c, c * a), (c * a, a * c), (a * c * a, c),\text{ etc.}\] are such pairs, see \Cref{fig:polygon correspondence}. Also note that any Reeb chord along a different component of $\Gamma$ than the one containing $a$ is paired with itself.

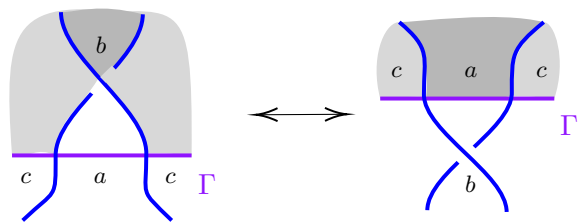
\begin{figure}[h]
    \centering
    \input{images/polygon_correspondence}
    \caption{A correspondence between polygons for $c$ and $a*c*a$.}
    \label{fig:polygon correspondence}
\end{figure}

Now $\partial^{-}$ and $\partial'$ are both differentials on $A'$ that agree on $a, b \in A'$. Order and label the Reeb chords in $\Sigma_\pm$ and $a, b$ by height as \[H(a_l) \geq \dots \geq H(a_1) \geq H(a) > H(b) \geq H(b_1) \geq \dots  \geq H(b_m).\] By the height property of polygons, $\partial^-$ and $\partial'$ also agree on $b_1, \dots, b_m$. Let $\mathcal{A}_{[i]}$, $0 \leq i \leq l$, be the sub-DGA generated by $a_i, \dots, a_1, a, b, b_1, \dots, b_m$. Then $\partial^-$ and $\partial'$ agree on $\mathcal{A}_{[0]}$. This is the analogue of \cite[Lemma 8.2(a)]{chekanov2002dga}.

Define a graded linear map \[G: \mathcal{A}' \rightarrow \mathcal{A}'\] \[ybz \mapsto yaz\] of degree $1$ where $y$ is a word not containing $a$ or $b$ and $G$ sends everything else to $0$. This map is denoted by $h$ in \cite[Lem. 2.2]{chekanov2002dga}. For $1 \leq i \leq l$, define a sequence of elementary automorphisms by \[g_i: \mathcal{A}' \rightarrow \mathcal{A}'\] \[a_i \mapsto a_i + G \circ (\partial' + \partial_{[i-1]})(a_i)\] \[c \mapsto c \text{ for any other generator $c$},\] where $\partial_{[0]} := \partial^-$ and $\partial_{[i]} = g_i \circ \partial_{[i-1]} \circ g_i^{-1}$. We claim that $\partial_{[l]} = \partial'$ on $\mathcal{A}_{[l]}$. By the discussion in \cite{chekanov2002dga} after Lemma 8.2, it suffices to prove an analog of Lemma 8.2(b). Namely, we claim that \[\tau \circ \partial^-(a_i) = \tau \circ \partial'(a_i)\] for all $1 \leq i \leq l$, where the DGA homomorphism \begin{align*}
    \tau: \mathcal{A}' &\rightarrow \mathcal{A}' \\
    a, b &\mapsto 0 \\
    c &\mapsto c \text{ if $c \neq a, b$}.
\end{align*} This follows by observing that $\partial^-(a_i)$ and $\partial'(a_i)$ agree after killing all polygons with negative corners at $b$. Therefore $\partial_{[l]} = \partial'$ on $\mathcal{A}_{[l]}$.

\subsubsection*{Step 2:}  It remains to make $\partial_{[l]}$ agree with $\partial'$ on all Reeb chords along $\Gamma$ except $a$ (on which they already agree). The proof is similar to Step 1, except that there are infinitely many such Reeb chords. Enumerate the $c_\Gamma'$'s by $c'_{\Gamma, i}$, where $i \in \mathbb{N}$, so that \[H(c'_{\Gamma, 1}) \leq H(c'_{\Gamma, 2}) \leq \dots\] Define an infinite sequence of elementary automorphisms by \[h_i: \mathcal{A}' \rightarrow \mathcal{A}'\] \[c'_{\Gamma, i} \mapsto c'_{\Gamma, i} + G \circ (\partial' + \partial_{[l+i-1]})(c'_{\Gamma, i})\] \[c \mapsto c \text{ for any other generator $c$},\] where $\partial_{[l+i]} = h_{i} \circ \partial_{[l+i-1]} \circ h_{i}^{-1}$. We claim that $h := \dots \circ h_2 \circ h_1$ is a tame automorphism. Order the generators of $\mathcal{A}'$ by \[b_m < \dots < b_1 < b < a < a_1 < \dots < a_l < c'_{\Gamma, 1} < c'_{\Gamma, 2} < \dots\] and let $\mathcal{A}_{[l+i]}$ be the sub-DGA generated by $b_m, \dots, b_1, b, a, a_1, \dots, a_l, c'_{\Gamma, 1}, \dots c'_{\Gamma, i}$. It suffices to show $G \circ (\partial' + \partial_{[l+i-1]})(c'_{\Gamma, i}) \in \mathcal{A}_{[l+i-1]}$. Since $G$ simply replaces $b$'s with $a$'s and $\partial'(c'_{\Gamma, i}) \in \mathcal{A}_{[l+i-1]}$ by the definition of height and the ordering, it suffices to show $\partial_{[l+i-1]} c'_{\Gamma, i} \in \mathcal{A}_{[l+i-1]}$. In fact we now prove by induction that $\partial_{[l+k-1]} c'_{\Gamma, i} \in \mathcal{A}_{[l+i-1]}$ for all $1 \leq k \leq i$. Note that \begin{align*}
    \partial_{[l+k-1]}(c'_{\Gamma, i}) &= h_{k-1} \circ \dots \circ h_1 \circ \partial_{[l]}(c'_{\Gamma, i}).
\end{align*} The $i = 1$ case is trivial. Assume the statement holds for $1, 2, \dots, i-1$. Then $\partial_{[l]}(c'_{\Gamma, 1}) \in \mathcal{A}_{[l]}$ implies $h_1 \circ \partial_{[l]}(c'_{\Gamma, i}) \in \mathcal{A}_{[l+i-1]}$, \dots, $\partial_{[l+i-2]}(c'_{\Gamma, i-1}) \in \mathcal{A}_{[l+i-2]}$ implies $h_{i-1} \circ \dots \circ h_1 \circ \partial_{[l]}(c'_{\Gamma, i}) \in \mathcal{A}_{[l+i-1]}$.  Thus, $h$ is a tame automorphism. The rest again follows from an analog of Lemma 8.2(b) and the discussion that follows. More specifically, we claim that \[(\partial_{[l]})^h(c'_{\Gamma, i}) = \partial'(c'_{\Gamma, i})\] for all $i$. Note that \begin{align*}
    (\partial_{[l]})^h(c'_{\Gamma, i}) &= (\dots \circ h_{i+1} \circ h_i) \circ \partial_{[l+i-1]} (c'_{\Gamma, i} + G \circ (\partial' + \partial_{[l+i-1]})(c'_{\Gamma, i})) \\
    &= \partial_{[l+i-1]} (c'_{\Gamma, i}) + \partial' (G \circ (\partial' + \partial_{[l+i-1]})(c'_{\Gamma, i})) \\
    &= \partial_{[l+i-1]} (c'_{\Gamma, i}) + (\Id + \tau + G \circ \partial') (\partial' + \partial_{[l+i-1]})(c'_{\Gamma, i}).
\end{align*} By induction, $\partial'$ and $\partial_{[l+i-1]}$ agree on $\mathcal{A}_{[l+i-1]}$ (since $\partial_{[l+i-1]}|_{\mathcal{A}_{[l+i-1]}} = (\partial_{[l]})^h|_{\mathcal{A}_{[l+i-1]}}$) so $\partial_{[l+i-1]}(c'_{\Gamma, i}) \in \mathcal{A}_{[l+i-1]}$ implies $\partial' \circ \partial_{[l+i-1]}(c'_{\Gamma, i}) = \partial_{[l+i-1]}^2 (c'_{\Gamma, i}) = 0$. Hence, \begin{align*}
    (\partial_{[l]})^h(c'_{\Gamma, i}) &= \partial'(c'_{\Gamma, i}) + \tau \circ \partial'(c'_{\Gamma, i}) + \tau \circ \partial_{[l+i-1]}(c'_{\Gamma, i}) \\
    &= \partial'(c'_{\Gamma, i}) + \tau \circ \partial'(c'_{\Gamma, i}) + \tau \circ h_{i-1} \circ \dots \circ h_1 \circ \partial_{[l]}(c'_{\Gamma, i}) \\
    &= \partial'(c'_{\Gamma, i}) + \tau \circ \partial'(c'_{\Gamma, i}) + \tau \circ \partial_{[l]}(c'_{\Gamma, i}) \text{ since } \tau \circ h_k = \tau \text{ for any } k.
\end{align*}

Assume the following analogue of \cite[Lemma 8.2(b)]{chekanov2002dga} still holds: \begin{lemma}\label{projectionLemma}\vphantom{}
    $\tau \circ \partial_{[l]}(c_\Gamma') = \tau \circ \partial'(c_\Gamma')$.
\end{lemma} Then $(\partial_{[l]})^h(c'_{\Gamma, i}) = \partial'(c'_{\Gamma, i})$ and we are done. \\
	
\noindent \begin{proof}(of \Cref{projectionLemma}) For clarity in this proof we will use the notation $\partial_{[l]}(c_\Gamma^-)$ in place of $\partial_{[l]}(c_\Gamma')$.

\tcbox[size=fbox, colback=gray!30]{Case 1: Pairs of Reeb chords starting at the left or right vertices of $a$.} Let $\gamma_l$ and $\gamma_r$ denote full loops starting at the left and right vertices of $a$, respectively. Denote the Reeb chord to the right and left of $a$ by $c$ and $d$, respectively. See \Cref{Reidemeister IV}. Then the corresponding pairs are \begin{align*}
        (c_\Gamma^-, c_\Gamma') = &(\gamma_l^n, \gamma_r^n), (\gamma_l^{n+1} * a, \gamma_r^n * c * \dots  * d), (\gamma_l^n * a * c, \gamma_r^n * c), \\
        &(\gamma_r^n, \gamma_l^n), (\gamma_r^n * c * \dots  * d, \gamma_l^{n+1} * a), (\gamma_r^n * c, \gamma_l^n * a * c),        
    \end{align*} where $n \geq 0$ and $c, d$ are not necessarily indecomposible. \\
    \tcbox[size=fbox, colback=gray!10]{Case 1a: $(\gamma_l^n, \gamma_r^n)$.}
     Being a pair means $\Phi_i(\gamma_l^n) = \Phi_i(\gamma_r^n)$, $i = \pm$. The remaining terms in $\partial_{[l]}(\gamma_l^n)$ and $\partial'(\gamma_r^n)$ are \begin{align*}
        &\Phi_+(a) c * \dots  * d + a\Phi_-(c * \dots  * d) \\
        +& \Phi_+(a * c_1) c_2 * \dots  * d + a*c_1 \Phi_-(c_2 * \dots  * d) \\
        +& \ldots \\
        +& \Phi_+(a * \dots  * d_{m-1}) d_m + a * \dots  * d_{m-1} \Phi_-(d_m)
    \end{align*} and \begin{align*}
        &\Phi_+(c_1) c_2 * \dots  * d * a + c_1\Phi_-(c_2 * \dots  * d * a) \\
        +& \ldots \\
        +& \Phi_+(c * \dots  * d_{m-1}) d_m * a + c * \dots  * d_{m-1} \Phi_-(d_m * a) \\
        +& \Phi_+(c * \dots  * d)a + c*\dots *d\Phi_-(a),
    \end{align*} where $c_i, d_i$'s are indecomposable. 
    
    First note that the first line in the former expression is killed by $\tau$ since $\Phi_+(a) = b$ in $\partial_{[l]}(\gamma^n_l)$ and $\tau(a) = \tau(b) = 0$. Similarly, the last line in the latter expression is killed by $\tau$ since $\Phi_-(a) = b$ in $\partial'(\gamma^n_r)$. Next the pairing of $(a * c_1, c_1)$ and $(c_2 * \dots  * d, c_2 * \dots  * d * a)$ implies that the second line and first line of the former and latter expressions, respectively, agree after applying $\tau$ to kill all polygons with negative corners at $b$ that may appear in $\Phi_+(a * c_1)$ in $\partial_{[l]}(\gamma_l^n)$ and $\Phi_-(c_2 * \dots  * d * a)$ in $\partial'(\gamma_r^n)$, respectively.
    
    In general there is line-by-line correspondence until the pairs $(a * \dots  * d_{m-1}, c * \dots  * d_{m-1})$ and $(d_m, d_m*a)$ which implies the last line and second-to-last line of the remaining terms in $\partial_{[l]}(\gamma_l^n)$ and $\partial'(\gamma_r^n)$, respectively, agree (again after killing all polygons with negative corners at $b$). This proves the claim for the pair $(\gamma_l^n, \gamma_r^n)$. \\
    \tcbox[size=fbox, colback=gray!10]{Case 1b: $(\gamma_l^{n+1} * a, \gamma_r^n * c * \dots  * d)$.} The corresponding terms are \begin{align*}
        &\Phi_+(a) c * \dots  * d * a + a \Phi_-(c * \dots  * d * a) \\
        +& \Phi_+(a * c_1) c_2 * \dots  * d * a + a * c_1 \Phi_-(c_2 * \dots  * d * a) \\
        +& \ldots \\
        +& \Phi_+(a * c * \dots  * d_{m-1}) d_m * a + a * c \dots  * d_{m-1} \Phi_-(d_m * a) \\
        +& \Phi_+(\gamma_l^{n+1}) a + \gamma_l^{n+1} \Phi_-(a)
    \end{align*} and \begin{align*}
        & \Phi_+(c * \dots  * d_{m-1}) d_m + c * \dots  * d_{m-1} \Phi_-(d_m) \\
        +& \ldots \\
        +& \Phi_+(c_1)c_2 * \dots  * d + c_1 \Phi_-(c_2 * \dots  * d).
    \end{align*} By the height property of Legendrian diagrams we have that $\Phi_-(a) = 0 \text{ in } \partial_{[l]}(\gamma_l^n)$. Indeed, if there were such a polygon, then after the move we have a polygon with all corners being negative, which contradicts the height property. Note that $\tau(\Phi_+(a) c * \dots  * d * a + a \Phi_-(c * \dots  * d * a) + \Phi_+(\gamma_l^{n+1}) a) = 0$ since $\Phi_+(a) = b$. Thus $\tau$ kills the first and last lines of the former expression. The second through second-to-last lines of the former expression agree with the latter in reverse order (again after killing all polygons with negative corners at $b$). \\
\tcbox[size=fbox, colback=gray!10]{Case 1c: $(\gamma_l^n * a * c, \gamma_r^n * c)$.}
 The corresponding terms are \begin{align*}
        &\Phi_+(a) c * \dots  * c + a \Phi_-(c * \dots  * c) \\
        +& \Phi_+(a * c_1) c_2 * \dots  * c + a * c_1 \Phi_-(c_2 * \dots  * c) \\
        +& \ldots \\
        +& \Phi_+(a * c * \dots  * c_{m-1}) c_m + a * \dots  * c_{m-1} \Phi_-(c_m)
    \end{align*} and \begin{align*}
        & \Phi_+(c_1) c_2 * \dots  * c + c_1 \Phi_-(c_2 * \dots  * c) \\
        +& \Phi_+(c_1 * c_2) c_3 * \dots  * c + c_1 * c_2 \Phi_-(c_3 * \dots  * c) \\
        +& \ldots \\
        +& \Phi_+(c * \dots  * c_{m-1}) c_m + c * \dots  * c_{m-1} \Phi_-(c_m).
    \end{align*} Note that the first line of the former expression is killed by $\tau$. The $i$-th line of the former corresponds to the $(i-1)$-th line of the latter, where $i \geq 2$. \\

\tcbox[size=fbox, colback=gray!10]{Case 1d: $(\gamma_r^n, \gamma_l^n)$.} The corresponding terms are the same as Case 1a reversed. Note that $\Phi_-(a) = 0$ in $\partial_{[l]}(\gamma_r^n)$ and $\Phi_+(a) = 0$ in $\partial'(\gamma_l^n)$ by the same height property argument as in Case 1b. Also $\tau(\Phi_+(c * \dots  * d)a) = \tau(a \Phi_-(c * \dots  * d)) = 0$. The line-by-line correspondence is the same as in Case 1a. \\

    \tcbox[size=fbox, colback=gray!10]{Case 1e: $(\gamma_r^n * c * \dots  * d, \gamma_l^{n+1} * a)$.} The corresponding terms are the same as Case 1b reversed. Note that $\Phi_+(a) = 0$ in $\partial'(\gamma_l^{n+1} * a))$ by the same height property argument as in Case 1b. Also $\tau(a \Phi_-(c * \dots  * d * a) + \Phi_+(\gamma_l^{n+1})a + \gamma_l^{n+1} \Phi_-(a)) = 0$. The line-by-line correspondence is the same as in Case 1b. \\

    \tcbox[size=fbox, colback=gray!10]{Case 1f: $(\gamma_r^n * c, \gamma_l^n * a * c)$.} The corresponding terms are the same as Case 1c reversed. Note that $\Phi_+(a) = 0$ in $\partial'(\gamma^n_l * a * c)$ by the same height property argument as in Case 1b. The line-by-line correspondence is the same as in Case 1c. \\
    \tcbox[size=fbox, colback=gray!30]{Case 2: Pairs of Reeb chords ending at the left or right vertices of $a$ excluding those in Case 1.} This case is similar to Case 1. \\
    
    \tcbox[size=fbox, colback=gray!30]{\begin{minipage}{\textwidth}
        Case 3: Pairs of Reeb chords that start and end outside of the local disk or pairs of Reeb chords along a different component of $\Gamma$.
    \end{minipage}} In this case the two sides could differ by some polygons with negative corners at $b$, which will be killed by $\tau$, respectively.
\end{proof}

\subsection{Reidemeister IV(b) move} \text{ }

The proof follows the same steps as that of the Type IV(a) move.

\subsubsection*{Step 1:} Again denote by $(\mathcal{A}^-, \partial^-)$ and $(\mathcal{A}', \partial')$ the DGA before and after the move, respectively. Write $\partial^-(a) = v$ and $\partial^-(b) = w$. Then $\partial'(a) = v$ and $\partial'(b) = w$. Denote by $\partial^-$ the conjugate $s \circ \partial^- \circ s^{-1}$, where $s: \mathcal{A}^- \rightarrow \mathcal{A}'$ is the identification of generators given by \Cref{Reidemeister IV}. The pairing is the almost the same as the IV(a) case except that $a$ is replaced by $b$ for Reeb chords along $\Gamma$ in $\mathcal{A}'$. For example, if $\Gamma$ has only two indecomposable chords $a$ and $c$ (i.e. $c$ and $d$ in \Cref{Reidemeister IV} are the same chord), then \[(c_\Gamma^-, c_\Gamma') = (c, b * c * b), (a * c, c * b), (c * a, b * c), (a * c * a, c),\text{ etc.}\] are such pairs. Now following the same argument as in the IV(a) case we have $\partial_{[l]} = \partial'$ on $\mathcal{A}_{[l]}$.

\subsubsection*{Step 2:}  It remains to make $\partial_{[l]}$ and $\partial'$ agree on all Reeb chords along $\Gamma$. Here the proof is different from that of the Type IV(a) move since $a$ and $b$ do not form a cancelling pair anymore, i.e. they are not of the form ``$\partial a = b, \partial b = 0$''.

\begin{lemma}
    $\partial_{[l]}(c_\Gamma^-) = g \circ \partial'(c_\Gamma') \circ g^{-1}$ for some tame automorphism $g: \mathcal{A}' \rightarrow \mathcal{A}'$.
\end{lemma}

\begin{proof}
    We divide the proof into three cases depending on whether the pair $(c_\Gamma^-, c_\Gamma')$ consists of chords starting at the left and right vertices of $a$, chords ending at the left and right vertices of $a$, or the other chords. The main idea is to construct explicit elementary automorphisms.
    
    \tcbox[size=fbox, colback=gray!30]{Case 1: Pairs of Reeb chords starting at the left or right vertices of $a$.} Let $\gamma_l$ and $\gamma_r$ denote full loops starting at the left and right vertices of $a$, respectively. The corresponding pairs are \begin{align*}
        (c_\Gamma^-, c_\Gamma') = &(\gamma_l^n, \gamma_r^n), (\gamma_l^{n+1} * a, \gamma_r^n * c * \dots  * d), (\gamma_l^n * a * c, \gamma_r^n * c), \\
        &(\gamma_r^n, \gamma_l^n), (\gamma_r^n * c * \dots  * d, \gamma_l^{n+1} * b), (\gamma_r^n * c, \gamma_l^n * b * c),        
    \end{align*} where $n \geq 0$ and $c, d$ may or may not be indecomposible. \\
    
    \tcbox[size=fbox, colback=gray!10]{Case 1a: $(\gamma_l^n, \gamma_r^n)$.} Since $\Phi_+(a) = 0$ in $\partial_{[l]}(\gamma_l^n)$ and $\Phi_-(b) = 0$ in $\partial'(\gamma_r^n)$, then \begin{align*}
        \partial_{[l]}(\gamma_l^n) &= \Phi_+(\gamma_l^n) + \Phi_-(\gamma_l^n) + a\Phi_-(c * \dots  * d) \\
        &+ \Phi_+(a * c_1) c_2 * \dots  * d + a*c_1 \Phi_-(c_2 * \dots  * d) \\
        &+ \ldots \\
        &+ \Phi_+(a * \dots  * d_{m-1}) d_m + a * \dots  * d_{m-1} \Phi_-(d_m)
    \end{align*} and \begin{align*}
        \partial'(\gamma_r^n) &= \Phi_+(\gamma_r^n) + \Phi_-(\gamma_r^n) \\
        &+ \Phi_+(c_1) c_2 * \dots  * d * b + c_1\Phi_-(c_2 * \dots  * d * b) \\
        &+ \ldots \\
        &+ \Phi_+(c * \dots  * d_{m-1}) d_m * b + c * \dots  * d_{m-1} \Phi_-(d_m * b) \\
        &+ \Phi_+(c * \dots  * d)b
    \end{align*} where $c_i, d_i$'s are indecomposable. Observe that we have the following correspondence:
    \begin{align*}
        \Phi_+(\gamma_l^n) &\leftrightarrow \Phi_+(c * \dots  * d)b \\
        \Phi_-(\gamma_l^n) + a\Phi_-(c * \dots  * d) &\leftrightarrow \Phi_-(\gamma_r^n) \\
        i\text{th line in } \partial_{[l]}(\gamma_l^n) &\leftrightarrow i\text{th line in } \partial'(\gamma_r^n), 2 \leq i \leq m.
    \end{align*} This proves Case 1a. \\

    \tcbox[size=fbox, colback=gray!10]{Case 1b: $(\gamma_l^{n+1} * a, \gamma_r^n * c * \dots * d)$.} Note that \begin{align*}
        \partial_{[l]}(\gamma_l^{n+1} * a) &= \Phi_+(\gamma_l^{n+1} * a) + \Phi_-(\gamma_l^{n+1} * a) + a \Phi_-(c * \dots * d * a) \\
        &+ \Phi_+(a * c_1) c_2 * \dots * d * a + a * c_1 \Phi_-(c_2 * \dots * d * a) \\
        &+ \ldots \\
        &+ \Phi_+(a * c * \dots * d_{m-1}) d_m * a + a * c \dots * d_{m-1} \Phi_-(d_m * a) \\
        &+ \Phi_+(\gamma_l^{n+1}) a + \gamma_l^{n+1} \Phi_-(a)
    \end{align*} and \begin{align*}
        \partial'(\gamma_r^n * c * \dots * d) &= \Phi_+(\gamma_r^n * c * \dots * d) + \Phi_-(\gamma_r^n * c * \dots * d) \\
        &+ \Phi_+(c * \dots * d_{m-1}) d_m + c  \dots * d_{m-1} \Phi_-(d_m) \\
        &+ \ldots \\
        &+ \Phi_+(c_1)c_2 * \dots * d + c_1 \Phi_-(c_2 * \dots * d).
    \end{align*}
    We claim that $\partial_{[l]}(\gamma_l^{n+1} * a)$ and $\partial'(\gamma_r^n * c * \dots * d)$ agree after conjugating $\partial'$ by the following elementary automorphisms: \begin{align*}
        \gamma_r^n * c * \dots * d &\mapsto \gamma_r^n * c * \dots * d + (\gamma_r^n * c * \dots * d * b)a \\
        d_m &\mapsto d_m + (d_m * b) a \\
        &\ldots \\
        c_2 * \dots * d &\mapsto c_2 * \dots * d + (c_2 * \dots * d * b)a.
    \end{align*} Indeed, after the conjugation $\partial'(\gamma_r^n * c * \dots * d)$ is given by \begin{align*}
        &\Phi_+(\gamma_r^n * c * \dots * d) + \Phi_-(\gamma_r^n * c * \dots * d) + \Phi_-(\gamma_r^n * c * \dots * d * b) a \\
        &+ \Phi_+(c * \dots * d_{m-1}) d_m + c  \dots * d_{m-1} \Phi_-(d_m) + c * \dots * d_{m-1} \Phi_-(d_m * b) a \\
        &+ \ldots \\
        &+ \Phi_+(c_1)c_2 * \dots * d + c_1 \Phi_-(c_2 * \dots * d) + c_1 \Phi_-(c_2 * \dots * d * b)a \\
        &+ \Phi_+(\gamma_r^n * c * \dots * d * b)a + \Phi_+(\gamma_r^n * c * \dots * d) ba + (\gamma_r^n * c * \dots * d * b) \Phi_-(a) \\
        &+ (\gamma_r^n * c * \dots * d) \Phi_-(b) a + (\gamma_r^n * c * \dots * d * b) \Phi_+(a).
    \end{align*} (Note that the second through third-to-last lines have two copies of $\Phi_+(c * \dots * d_{m-1}) (d_m * b) a, \dots, \Phi_+(c_1)(c_2 * \dots * d * b)a$, respectively, that are canceled.) The first line corresponds to that of $\partial_{[l]}(\gamma_l^{n+1} * a)$. The second through third-to-last lines correspond to the second through second-to-last lines of $\partial_{[l]}(\gamma_l^{n+1} * a)$. The second-to-last line corresponds to the last line of $\partial_{[l]}(\gamma_l^{n+1} * a)$. The last line evaluates to $0$ since $\Phi_-(b) = \Phi_+(a) = 0$. This proves Case 1b. \\

    \tcbox[size=fbox, colback=gray!10]{Case 1c: $(\gamma_l^n * a * c, \gamma_r^n * c)$.} Since $\Phi_+(a) = 0$, then \begin{align*}
        \partial_{[l]} (\gamma_l^n * a * c) &= \Phi_+(\gamma_l^n * a * c) + \Phi_-(\gamma_l^n * a * c) \\ 
        &+ a \Phi_-(c * \dots  * c) \\
        &+ \Phi_+(a * c_1) c_2 * \dots  * c + a * c_1 \Phi_-(c_2 * \dots  * c) \\
        &+ \ldots \\
        &+ \Phi_+(a * c * \dots  * c_{m-1}) c_m + a * \dots  * c_{m-1} \Phi_-(c_m)
    \end{align*} and \begin{align*}
        \partial' (\gamma_r^n * c) &= \Phi_+(\gamma_r^n * c) + \Phi_-(\gamma_r^n * c) \\
        &+ \Phi_+(c_1) c_2 * \dots  * c + c_1 \Phi_-(c_2 * \dots  * c) \\
        &+ \Phi_+(c_1 * c_2) c_3 * \dots  * c + c_1 * c_2 \Phi_-(c_3 * \dots  * c) \\
        &+ \ldots \\
        &+ \Phi_+(c * \dots  * c_{m-1}) c_m + c * \dots  * c_{m-1} \Phi_-(c_m).
    \end{align*} Observe that we have the following correspondence:
    \begin{align*}
        \Phi_+(\gamma_l^n * a * c) &\leftrightarrow \Phi_+(\gamma_r^n * c) \\
        \Phi_-(\gamma_l^n * a * c) + a \Phi_-(c * \dots  * c) &\leftrightarrow \Phi_-(\gamma_r^n * c) \\
        i\text{th line in } \partial_{[l]}(\gamma_l^n) &\leftrightarrow (i-1)\text{th line in } \partial'(\gamma_r^n), 2 \leq i \leq m.
    \end{align*} This proves Case 1c. \\

    \tcbox[size=fbox, colback=gray!10]{Case 1d: $(\gamma_r^n, \gamma_l^n)$.} We have \begin{align*}
        \partial_{[l]}(\gamma_r^n) &= \Phi_+(\gamma_r^n) + \Phi_-(\gamma_r^n) \\
        &+ \Phi_+(c_1) c_2 * \dots  * d * a + c_1\Phi_-(c_2 * \dots  * d * a) \\
        &+ \ldots \\
        &+ \Phi_+(c * \dots  * d_{m-1}) d_m * a + c * \dots  * d_{m-1} \Phi_-(d_m * a) \\
        &+ \Phi_+(c * \dots  * d)a + c * \dots  * d \Phi_-(a)
    \end{align*} and \begin{align*}
        \partial'(\gamma_l^n) &= \Phi_+(\gamma_l^n) + \Phi_-(\gamma_l^n) \\
        &+ \Phi_+(b) c * \dots  * d + b\Phi_-(c * \dots  * d) \\
        &+ \Phi_+(b * c_1) c_2 * \dots  * d + b*c_1 \Phi_-(c_2 * \dots  * d) \\
        &+ \ldots \\
        &+ \Phi_+(b * \dots  * d_{m-1}) d_m + b * \dots  * d_{m-1} \Phi_-(d_m).
    \end{align*}

    We have the following correspondence: \begin{align*}
        \Phi_+(\gamma_r^n) &\leftrightarrow \Phi_+(\gamma_l^n) + b\Phi_+(c * \dots * d) \\
        \Phi_+(c_1) &\leftrightarrow \Phi_+(b * c_1) + b\Phi_+(c_1) \\
        &\dots \\
        \Phi_+(c * \dots * d_{m-1}) &\leftrightarrow \Phi_+(b * \dots * d_{m-1}) + b\Phi_+(c * \dots * d_{m-1}) \\
        \Phi_+(c * \cdots * d) &\leftrightarrow \Phi_+(b * \dots * d)b + b \Phi_+(c * \dots * d) b
    \end{align*} \begin{align*}
        \Phi_-(\gamma_r^n) + \Phi_-(c * \dots * d) a &\leftrightarrow \Phi_-(\gamma_l^n) \\
        \Phi_-(c_2 * \dots * d * a) + \Phi_-(c_2 * \dots * d)a &\leftrightarrow \Phi_-(c_2 * \dots * d) \\
        &\dots \\
        \Phi_-(d_m * a) + \Phi_-(d_m)a &\leftrightarrow \Phi_-(d_m) \\
        \Phi_-(a) &\leftrightarrow \Phi_-(a).
    \end{align*} Conjugate $\partial_{[l]}$ by the following elementary automorphisms: \begin{align*}
        \gamma_r^n &\mapsto \gamma_r^n + (\gamma_r^{n-1} * c * \dots * d) a \\
        c_2 * \dots * d * a &\mapsto c_2 * \dots * d * a + (c_2 * \dots * d) a \\
        &\dots \\
        d_m * a &\mapsto d_m * a + d_m a.
    \end{align*} Note that the first elementary automorphism introduces the terms \[\Phi_+(c * \dots * d)a + \Phi_-(c * \dots * d)a + c_1 \Phi_-(c_2 * \dots * d)a + \dots + c * \dots * d_{m-1} \Phi_-(d_m)a + c * \dots * d \Phi_-(a)\] while the terms \[\Phi_+(c_1) (c_2 * \dots * d) a + \dots + \Phi_+(c * \dots * d_{m-1})d_m a\] are canceled by the other elementary automorphisms. (Here we also used the fact that $c * \dots * d \Phi_+(a) = (c * \dots * d) \cdot 0 = 0$.) Conjugate $\partial'$ by the following elementary automorphisms: \begin{align*}
        \gamma_l^n &\mapsto \gamma_l^n + b (c * \dots * d * \gamma_l^{n-1}) \\
        b * c_1 &\mapsto b * c_1 + bc_1 \\
        &\dots \\
        b * c * \dots * d_{m-1} &\mapsto b * c * \dots * d_{m-1} + b (c * \dots * d_{m-1}).
    \end{align*} Note that the first elementary automorphism introduces the terms \[\Phi_+(b) c * \dots * d + b \Phi_-(c * \dots * d) + b \Phi_+(c * \dots * d) + b \Phi_+(c_1) c_2 * \dots * d + \dots + b \Phi_+(c * \dots * d_{m-1}) d_m\] while the terms \[b c_1\Phi_-(c_2 * \dots * d) + \dots + b (c * \dots * d_{m-1}) \Phi_-(d_m)\] are canceled by the other elementary automorphisms. (Here we also used the fact that $\Phi_-(b) c * \dots * d = 0 \cdot (c * \dots * d) = 0$.) To finish the proof of Case 1d it remains to observe that both sides with the newly introduced terms represent the same polygons under the correspondence listed above. \\

    \tcbox[size=fbox, colback=gray!10]{Case 1e $(\gamma_r^n * c * \dots  * d, \gamma_l^{n+1} * b)$.} Since $\Phi_-(b) = 0$, we have \begin{align*}
        \partial_{[l]}(\gamma_r^n * c * \dots * d) &= \Phi_+(\gamma_r^n * c * \dots * d) + \Phi_-(\gamma_r^n * c * \dots * d) \\
        &+ \Phi_+(c * \dots  * d_{m-1}) d_m + c * \dots  * d_{m-1} \Phi_-(d_m) \\
        &+ \ldots \\
        &+ \Phi_+(c_1)c_2 * \dots  * d + c_1 \Phi_-(c_2 * \dots  * d)
    \end{align*} and \begin{align*}
        \partial'(\gamma_l^{n+1} * b) &= \Phi_+(\gamma_l^{n+1} * b) + \Phi_-(\gamma_l^{n+1} * b) \\
        &+ \Phi_+(b) c * \dots  * d * b + b \Phi_-(c * \dots  * d * b) \\
        &+ \Phi_+(b * c_1) c_2 * \dots  * d * b + b * c_1 \Phi_-(c_2 * \dots  * d * b) \\
        &+ \ldots \\
        &+ \Phi_+(b * c * \dots  * d_{m-1}) d_m * b + b * c \dots  * d_{m-1} \Phi_-(d_m * b) \\
        &+ \Phi_+(\gamma_l^{n+1}) b.
    \end{align*} We have the following correspondence: \begin{align*}
        \Phi_+(\gamma_r^n * c * \dots * d) &\leftrightarrow \Phi_+(\gamma_l^{n+1} * b) + b \Phi_+(c * \dots * d * b) + \Phi_+(b * c * \dots * d)b + b\Phi_+(c * \dots * d)b \\
        \Phi_+(c * \dots * d_{m-1}) &\leftrightarrow \Phi_+(b * c * \dots * d_{m-1}) + b\Phi_+(c * \dots * d_{m-1}) \\
        &\dots \\
        \Phi_+(c_1) &\leftrightarrow \Phi_+(b * c_1) + b \Phi_+(c_1)
    \end{align*} \begin{align*}
        \Phi_-(\gamma_r^n * c * \dots * d) &\leftrightarrow \Phi_-(\gamma_l^{n+1} * b) \\
        \Phi_-(d_m) &\leftrightarrow \Phi_-(d_m * b) \\
        &\dots \\
        \Phi_-(c_2 * \dots * d) &\leftrightarrow \Phi_-(c_2 * \dots * d * b)
    \end{align*} Conjugate $\partial'$ by the elementary automorphisms \begin{align*}
        \gamma_l^{n+1} * b &\mapsto \gamma_l^{n+1} * b + b (\gamma_r^n * c * \dots * d * b) \\
        b * c_1 &\mapsto b * c_1 + b c_1 \\
        &\dots \\
        b * \gamma_r^n * c * \dots * d &\mapsto b * \gamma_r^n * c * \dots * d + b(\gamma_r^n * c * \dots * d).
    \end{align*} The first elementary automorphism introduces the terms \[\Phi_+(b) (\gamma_l^n * c * \dots * d * b) + b \Phi_+(\gamma_r^n * c * \dots * d * b) + b \Phi_-(\gamma_r^n * c * \dots * d * b) + b \Phi_+(\gamma_r^n * c * \dots * d) b,\] of which $\Phi_+(b) (\gamma_l^n * c * \dots * d * b) + b \Phi_-(\gamma_r^n * c * \dots * d * b)$ cancel with the second line of the original $\partial'(\gamma_l^{n+1} * b)$, the terms \[b\Phi_+(\gamma_r^n * d * \dots * d_{m-1}) d_m * b + \dots + b\Phi_+(c_1)c_2 * \dots * d * b,\] and the terms \[b (\gamma_r^n * c * \dots * d_{m-1}) \Phi_-(d_m * b) + \dots + bc_1\Phi_-(c_2 * \dots * d * b),\] which cancel with the terms introduced by the other elementary automorphisms. (Here we used the fact that $\Phi_-(b) (\gamma_r^n * c * \dots * d * b) = b (\gamma_r^n * c * \dots * d) \Phi_-(b) = 0$ since $\Phi_-(b) = 0$.) Now Case 1e follows from the correspondence listed above. \\

    \tcbox[size=fbox, colback=gray!10]{Case 1f: $(\gamma_r^n * c, \gamma_l^n * b * c)$.} We have \begin{align*}
        \partial_{[l]}(\gamma_r^n * c) &= \Phi_+(\gamma_r^n * c) + \Phi_-(\gamma_r^n * c) \\
        &+ \Phi_+(c_1) c_2 * \dots  * c + c_1 \Phi_-(c_2 * \dots  * c) \\
        &+ \Phi_+(c_1 * c_2) c_3 * \dots  * c + c_1 * c_2 \Phi_-(c_3 * \dots  * c) \\
        &+ \ldots \\
        &+ \Phi_+(c * \dots  * c_{m-1}) c_m + c * \dots  * c_{m-1} \Phi_-(c_m)
    \end{align*} and \begin{align*}
        \partial'(\gamma_l^n * b * c) &= \Phi_+(\gamma_l^n * b * c) + \Phi_-(\gamma_l^n * b * c) \\
        &+ \Phi_+(b) c * \dots  * c + b \Phi_-(c * \dots  * c) \\
        &+ \Phi_+(b * c_1) c_2 * \dots  * c + b * c_1 \Phi_-(c_2 * \dots  * c) \\
        &+ \ldots \\
        &+ \Phi_+(b * c * \dots  * c_{m-1}) c_m + b * \dots  * c_{m-1} \Phi_-(c_m).
    \end{align*} we have the following correspondence: \begin{align*}
        \Phi_+(\gamma_r^n * c) &\leftrightarrow \Phi_+(\gamma_l^n * b * c) + b\Phi_+(\gamma_r^n * c) \\
        \Phi_+(c_1) &\leftrightarrow \Phi_+(b * c_1) + b\Phi_+(c_1) \\
        &\dots \\
        \Phi_+(c * \dots * c_{m-1}) &\leftrightarrow \Phi_+(b * c * \dots * c_{m-1}) + b\Phi_+(c * \dots * c_{m-1})
    \end{align*} \begin{align*}
        \Phi_-(\gamma_r^n * c) &\leftrightarrow \Phi_-(\gamma_l^n * b * c) \\
        \Phi_-(c_2 * \dots * c) &\leftrightarrow \Phi_-(c_2 * \dots * c) \\
        &\dots \\
        \Phi_-(c_m) &\leftrightarrow \Phi_-(c_m).
    \end{align*} Conjugate $\partial'$ by the following elementary automorphisms: \begin{align*}
        \gamma_l^n * b * c &\mapsto \gamma_l^n * b * c + b (\gamma_r^n * c) \\
        b * c_1 &\mapsto b * c_1 + b c_1 \\
        &\dots \\
        b * c * \dots * c_{m-1} &\mapsto b * c * \dots * c_{m-1} + b(c * \dots * c_{m-1}).
    \end{align*} The first elementary automorphism introduces the terms \[\Phi_+(b)(\gamma_r^n * c) + b\Phi_+(\gamma_r^n * c) + b\Phi_-(\gamma_r^n * c),\] of which the first and the third cancel with the second line of $\partial'(\gamma_l^n * b * c)$, the terms \[b\Phi_+(c_1) c_2 * \dots * c + \dots + b \Phi_+(c * \dots * c_{m-1})c_m,\] and the terms \[bc_1\Phi_-(c_2 * \dots * c) + b(c * \dots * c_{m-1}) \Phi_-(c_m),\] which cancel with the terms introduced by the other elementary automorphisms. Now Case 1f follows from the correspondence listed above. This also finishes Case 1. \\

    \tcbox[size=fbox, colback=gray!30]{Case 2: Pairs of Reeb chords ending at the left or right vertices of $a$ excluding those in Case 1.} This case is similar to Case 1. \\
    
    \tcbox[size=fbox, colback=gray!30]{\begin{minipage}{\textwidth}
        Case 3: Pairs of Reeb chords that start and end outside of the local disk or pairs of Reeb chords along a different component of $\Gamma$.
    \end{minipage}} In this case the two sides coincide.
\end{proof}

\subsection{Reidemeister I move} \text{ }
Label the chords in a neighborhood of the Reidemeister I move as shown in \Cref{reidILabels}. Without loss of generality we assume that $b$ is in $\Sigma_+$. Recall that in \cite{chekanov2002dga}, the stabilization of a DGA $\mathcal{A}$ is the free product $\mathcal{A} \sqcup E$ of $\mathcal{A}$ where $E = T(a, b)$ with $\partial_E a = b$ and $\partial_E b = 0$. In our case, the Reidemeister I move introduces infinitely many Reeb chords along the dividing set, so we need infinitely many stabilization operations. Then as in the proof of the Type IV(a) move, we use an analog of \cite[Lemma 8.2]{chekanov2002dga} and the discussion that follows to finish the proof.

    \begin{figure}[h]
        \begin{center}
            \input{images/RI_invariance}
            \caption{The Reidemeister I move, with important Reeb chords labeled.}
            \label{reidILabels}
        \end{center}
    \end{figure}
    
\begin{lemma}\label{invReidI}
	Again denote by $(\mathcal{A}^-, \partial^-)$ and $(\mathcal{A}', \partial')$ the DGA before and after the move, respectively. One can perform infinitely many stabilization operations on $\mathcal{A}'$ so that $\mathcal{A}'$ is tame isomorphic to $\mathcal{A}^-$.
\end{lemma}

\begin{proof}
    Our proof consists of two steps: in Step 1 we group the newly introduced Reeb chords into three groups of canceling pairs $\{c_i, c_i'\}_{i \in \mathbb{N}}, \{d_i, d_i'\}_{i \in \mathbb{N}}, \{a, b\}$ and conjugate $\partial^-$ by infinitely many elementary automorphisms (denoted by $s$), arranging that $\partial^s c_i = c_i', \partial^s d_i = d_i', \partial^s a = b$. This allows us perform infinitely many stabilization operations to $\partial'$, so that it agrees with $\partial^s$ on all Reeb chords affected by the move. In Step 2 we further conjugate $\partial^s$ by elementary automorphisms so that the differentials before and after the move also agree on all remaining Reeb chords. \\
    
    \textit{Step 1}. The following Stabilizations I, II, and III  stabilize respectively the pairs $\{c_i, c_i'\}_{i \in \mathbb{N}}$, $\{d_i, d_i'\}_{i \in \mathbb{N}}$, and $\{a, b\}$ as labeled in \Cref{reidILabels}.

	\tcbox[size=fbox, colback=gray!30]{Step 1-1. Stabilization I}

	Denote the Reeb chords that end at the left and right vertices of $a$ by $\{c_i', c_i\}_{i \in \mathbb{N}}$, respectively, so that they are sorted by increasing numbers of indecomposable chords. We claim that there exists a tame automorphism $g$ such that $\partial^g := g \circ \partial^- \circ g^{-1}$ satisfies \[\partial^g c_i = c'_i \text{ (and thus } \partial^g c'_i = 0)\] for all $i$. Define a sequence of elementary automorphisms $g_i$ by \[g_{i}(c_{i}') = \partial^- c_{i}.\] We want to show that $\partial^- c_{i} =  c_i' + w_i$ for some $w_i$ containing no terms involving $c_j'$. Indeed, note that \[\partial^- c_{i} = \Phi_+(c_{i}) + \Phi_-(c_{i}) + \sum_{e_1 * e_2 = c_{i}} \Phi_+(e_1)e_2 + e_1\Phi_-(e_2),\] where the only term involving $c_j'$ is when $e_2 = a$ and $e_1\Phi_-(e_2) = c_{i}'$. Clearly $\partial^- c_i$ cannot produce any Reeb chord longer than $c_i$. In fact, if we are close enough to the moment of bifurcation $a$ is short enough so that any Reeb chord in $\partial^- c_i$ is strictly shorter than $c_i'$ except $c_i'$ itself. But we already showed that $w_i$ does not involve $c_i'$. As before we pick any ordering such that all generators along $\Gamma$ have higher ranks than all generators from $\Sigma_\pm$ and the generators along $\Gamma$ are ranked by height. Then \[g := \dots  \circ g_2 \circ g_1\] is a tame automorphism and \begin{align*}
		\partial^g(c_i) &= g \circ \partial^- \circ g^{-1} (c_i) \\
		&= g \circ \partial^- c_i \\
		&= g(c_i' + w_i) \\
		&= \partial^-(c_i)+ w_i \text{ since } w_i \text{ contains no } c_j' \text{ terms} \\
		&= c_i'.
	\end{align*}

	\tcbox[size=fbox, colback=gray!30]{Step 1-2. Stabilization II} 
	
	Denote the Reeb chords that start at the left and right vertices of $a$ (but are not already labeled by $\{c_i, c_i'\}_{i\in \mathbb{N}}$) by $\{d_i, d_i'\}_{i \in \mathbb{N}}$. Define a sequence of elementary automorphisms $h_i^1, h_i^2$ inductively by \[h_i^1(d_i) = d_i + a d_i'\] \[h_i^2(d_i') = \partial^{h_{i, \dots , 1}^1} d_i,\] where $h_{i, \dots , 1}^1 := h_i^1 \circ h_{i-1}^1 \circ \dots \circ h_1^1$. We want to show that $\partial^{h_{i, \dots , 1}^1} d_i = d_i' + v_i$ for some $v_i$ containing no terms involving $d_j'$. Indeed, note that \begin{align*}
        \partial^{h_{i, \dots , 1}^1} d_i &= h_{i, \dots , 1}^1 \circ \partial^-(d_i + a d_i') \\
        &= h_{i, \dots , 1}^1(\Phi_+(d_i) + \Phi_-(d_i) + bd_i' + a\Phi_-(d_i') + \sum_j (\Phi_+(d_j)e_j + d_j\Phi_-(e_j)) \\
        &+ (b+1)d_i' + a\partial^- d_i'),
    \end{align*} where $e_j$ are old chords and \begin{align*}
        \partial^- d_i' &= \Phi_+(d_i') + \Phi_-(d_i') + \sum_j(\Phi_+(d_j')e_j + d_j'\Phi_-(e_j)).
    \end{align*} Thus, \[\partial^{h_{i, \dots , 1}^1} d_i = h_{i, \dots , 1}^1(d_i' + \Phi_+(d_i) + a\Phi_+(d_i') + \sum(\Phi_+(d_j) e_j + d_j \Phi_-(e_j)) + a\sum(\Phi_+(d_j') e_j + d_j' \Phi_-(e_j))).\] The above computation of $\partial^{h_{i, \dots , 1}^1} d_i$ is actually a partial lie since some of the $d_j$'s may have already been labeled by $c_i, c_i'$'s. If this does not happen, then applying $h_{i, \dots , 1}^1$ replaces $d_j$ with $d_j + ad_j'$, so we get an extra term $a \sum d_j'\Phi_-(e_j)$, which cancels with the corresponding term in $a\partial d_i'$ and no $d_j'$ terms exist (other than $d_i'$). Otherwise $\Phi_-(e_j) = 0$ by local configuration of the move and $d_j \Phi_-(e_j) = d_j' \Phi_-(e_j) = 0$ so again no $d_j'$ terms exist. Clearly $a$ and $d_i'$ are shorter than $d_i$. Similar to Step 1-1, $\partial^{h_{i, \dots , 1}^1} d_i$ cannot produce any Reeb chord longer than $d_i$, and moreover if we are close enough to the moment of bifurcation $a$ is short enough so that any Reeb chord in $\partial^{h_{i, \dots , 1}^1} d_i$ is strictly shorter than $d_i'$ except $d_i'$ itself. But we already showed that $v_i$ does not involve $d_i'$. Hence \[h := (\dots  \circ h_2^2 \circ h_1^2) \circ (\dots \circ h_2^1 \circ h_1^1)\] is a tame automorphism. Note that \[\partial^{h \circ g} c_i = h \circ \partial^g \circ h^{-1} (c_i) = h \circ \partial^g (c_i) = c_i'\] and \begin{align*}
		\partial^{h \circ g} d_i &= h \circ g \circ \partial^-(d_i + ad_i') \\
		&= h \circ \partial^-(d_i + ad_i') \\
        &\text{ since whenever } \partial^-(d_i + ad_i') \text{ has a } c_i' \text{ term  it is of the form } c_i' \Phi_-(c_\Gamma) = c_i' \cdot 0 = 0 \\
		&= (\dots  \circ h_2^2 \circ h_1^2) \circ \partial^{h_{i, \dots , 1}^1} d_i \\
		&= (\dots  \circ h_{i+1}^2) \circ h_i^2 (d_i' + v_i) \\
		&= \dots  \circ h_{i+1}^2 (d_i') \\
		&= d_i'.
	\end{align*}
	
	\tcbox[size=fbox, colback=gray!30]{Step 1-3. Stabilization III} 
	
	Note that $\partial^- a = b + 1$ and $\partial^- b = 0$. Let $k$ be the elementary automorphism $b \mapsto b + 1$ and let $s := k \circ h \circ g$. Then $\partial^s a = b$. Clearly we still have $\partial^s (c_i) = c_i'$ and $\partial^s (d_i) = d_i'$. This finishes the stabilization part of the proof. \\
	
	\subsubsection*{Step 2.} The final step to show invariance under the RI move is to conjugate $\partial^s$ by a tame automorphism so that it also agrees with $\partial'$ on the remaining chords.
	
	\begin{lemma}\label{keyProjection}
		Let $y$ be a generator of $\mathcal{A}^-$ that is preserved by (i.e. happens away from) the move. Let $\tau: \mathcal{A}^- \rightarrow \mathcal{A}^-$ be the projection map that sends all new chords $a, b, c_j, c_j', d_j, d_j'$ to $0$. Then \[\tau \circ \partial^s(y) = \tau \circ \partial'(y).\]
	\end{lemma}
	
	\begin{proof}
        This is again an analog of \cite[Lem. 8.2(b)]{chekanov2002dga}.

		\tcbox[size=fbox, colback=gray!30]{Case 1: $y$ is of the form $c_i' * a * d_i'$.} To avoid confusion we denote $\Phi_\pm$ in $\partial^s$ by $\Phi^s_\pm$ and $\Phi_\pm$ in $\partial'$ by $\Phi'_{\pm}$. Note that \begin{align*}
			\partial^s (y) &= s \circ \partial (y) \\
			&= s (\Phi_+^s(y) + \Phi_-^s(y) + \Phi_+^s(c_i) d_i' + \Phi_+^s(c_i') d_i + w) \\
			&= k(\Phi_+^s(y) + \Phi_-^s(y) + w) + k \circ h(\Phi_+^s(c_i) d_i' + \Phi_+^s(c_i') d_i),
		\end{align*} where $w$ collects all terms in $\Omega(\partial_\Gamma (y))$ not involving $c_j, c_j', d_j, d_j'$. Note that \begin{align*}
			k \circ h(\Phi_+^s(c_i) d_i' + \Phi_+^s(c_i') d_i) &= k \circ (\dots \circ h_1^2) \circ (\dots \circ h_1^1)(\Phi_+^s(c_i) d_i' + \Phi_+^s(c_i') d_i) \\
			&= k \circ (\dots \circ h_1^2)(\Phi_+^s(c_i) d_i' + \Phi_+^s(c_i') (d_i + a d_i')) \\
			&= k(\Phi_+^s(c_i) (\partial^{h_{i, \dots , 1}^1} d_i) + \Phi_+^s(c_i')(d_i + a(\partial^{h_{i, \dots , 1}^1} d_i))),
		\end{align*} where by Step 1-2 we have \begin{align*}
			\partial^{h_{i, \dots , 1}^1} d_i &= d_i' + \Phi_+^s(d_i) + a\Phi_+^s(d_i') + \sum_j (\Phi_+^s(d_j)e_j + d_j\Phi_-^s(e_j)) + a\sum_j (\Phi_+^s(d_j')e_j).
		\end{align*} Thus, \[\tau \circ k \circ h (\Phi_+^s(c_i) d_i' + \Phi_+^s(c_i') d_i) = \tau \circ k(\Phi_+^s(c_i)\Phi_+^s(d_i) + \Phi_+^s(c_i)\sum_j \Phi_+^s(d_j) e_j).\] At the same time, we have $\partial'(y) = \Phi'_+(y) + \Phi'_-(y) + \Omega(\partial_\Gamma y)$, where \[\Phi'_+(y) = \tau \circ k(\Phi_+^s(y) + \Phi_+^s(c_i)\Phi_+^s(d_i))\] \[\Phi'_-(y) = \Phi_-^s(y)\] \[\Omega(\partial_\Gamma y) = \tau \circ k (w + \Phi_+^s(c_i)\sum_j \Phi_+^s(d_j) e_j).\] See \Cref{reidICorr} for the case when $\Phi^s_+(c_i), \Phi^s_+(d_i), \Phi^s_+(d_j)$ do not have negative corners at $b$. Now both $\tau \circ \partial^s(y)$ and $\tau \circ \partial'(y)$ equal to \[\tau \circ k(\Phi_+^s(y) + \Phi_+^s(c_i)\Phi_+^s(d_i) + w + \Phi_+^s(c_i)\sum_j \Phi_+^s(d_j) e_j) + \Phi_-^s(y).\]
 \begin{figure}[h]
            \begin{center}
                \input{images/RI_invariance_2}
                \caption{Correspondence of polygons before and after a Reidemeister I move.}
                \label{reidICorr}
            \end{center}
        \end{figure}

		\tcbox[size=fbox, colback=gray!30]{Case 2: $y$ is an old chord along $\Gamma$ that does not involve $a$ or an old chord in $\Sigma_\pm$.} If $y$ is in $\Sigma_-$, then $\partial^s(y) = \partial'(y)$ since we assumed $b$ is in $\Sigma_+$. In general, $\partial^s(y)$ replaces all $b$'s in $\partial^-(y)$ by $b + 1$ so that after $\tau$ kills all terms involving $b, c_i, c_i', d_i, d_i'$ we have $\tau \circ \partial^s(y) = \tau \circ \partial'(y)$. See \Cref{reidICorr2} for the case when $y$ is an old chord along $\Gamma$.

         \begin{figure}[h]
            \begin{center}
                \input{images/RI_invariance_3}
                \caption{Correspondence of polygons before and after a Reidemeister I move.}
                \label{reidICorr2}
            \end{center}
        \end{figure}
	\end{proof}
	
	\text{ }\\ As before we finish the proof using a similar argument as in the discussion that follows \cite[Lemma 8.2]{chekanov2002dga}. Relabel the new chords by $\{a_i, x_i\}_{i \in \mathbb{N}}$, where $\partial^s a_i = x_i$, $\partial^s x_i = 0$, $a_1 = a$, $x_1 = b$, $H(x_1) < H(a_1) < H(x_2) < H(a_2) < \ldots$. Indeed, by taking a time sufficiently close to the bifurcation time we may assume (a) $H(b)$ is the smallest and $H(a)$ is the second smallest and (b) $H(a_i) - H(x_i) = H(a)$ ($i \geq 2$) is small enough so that no old chord has height between $H(a_i)$ and $H(x_i)$. Then we have \[H(x_1) < H(a_1) < H(b_1^1) < \ldots < H(b_{n_1}^1) < H(x_2) < H(a_2) < H(b_1^2) < \ldots < H(b_{n_2}^2) < \ldots,\] where $b_i^j$'s denote the old chords. Define an infinite sequence of elementary automorphisms \[t_i^j: b_i^j \mapsto b_i^j + q_i^j,\] where \[q_i^j = G \circ (\partial_{[i-1]}^j + \partial')(b_i^j),\] $\partial_{[0]}^1 = \partial^s$, $\partial_{[i]}^j = t_i^j \circ \partial_{[i-1]}^j \circ (t_i^j)^{-1}$, $\partial_{[0]}^j = \partial^{j-1}_{[n_{j-1}]}$, and \[G(y x_i z) = y a_i z\] for all $i$ where $y$ is a word not containing $a_i$ or $x_i$ and $G$ sends everything else to $0$.
    
    Let $\mathcal{A}_{[i]}^j$ be the subalgebra of $\mathcal{A}^-$ generated by $x_1, a_1, \ldots, b_i^j$. First we claim that $q_i^j \in \mathcal{A}_{[i-1]}^j$, so that the ordering given by the above ranking by height guarantees that the infinite composition of $t^j_i$'s is a tame automorphism. Base case: $q_1^1 = G \circ (\partial^s + \partial')(b_1^1)$. Since $g(c'_i) = \partial^- c_i$, and there is no chord with height in between that of $c_i$ and $c'_i$, then applying $g$ does not introduce any chords with larger heights. Similar statements hold for $h$ and $k$. So $\partial^s(b_1^1) = s \circ \partial^- (b_1^1) \in \mathcal{A}_{[0]}^1$, which implies $q_1^1 \in \mathcal{A}_{[0]}^1$ and $t_1^1(\mathcal{A}_{[1]}^1) \subset \mathcal{A}_{[1]}^1$. Then $\partial_{[1]}^1(b_2^1) = t_1^1 \circ s \circ \partial^- (b_2^1) \in \mathcal{A}_{[1]}^1$, which implies $q_2^1 \in \mathcal{A}_{[1]}^1$ and $t_2^1(\mathcal{A}_{[2]}^1) \subset \mathcal{A}_{[2]}^1$. Proceed by induction and get $q_{i}^1 \in \mathcal{A}_{[i-1]}^1$ for all $1 \leq i \leq n_1$. Now we have $\partial_{[0]}^2(b_1^2) = \partial_{[n_1]}^1(b_1^2) \in \mathcal{A}_{[n_1]}^1 = \mathcal{A}_{[0]}^2$ so $q_1^2 \in \mathcal{A}_{[0]}^2$ and $t_1^2(\mathcal{A}_{[1]}^2) \subset \mathcal{A}_{[1]}^2$. Proceed by induction again and get $q_{i}^2 \in \mathcal{A}_{[i-1]}^2$ for all $1 \leq i \leq n_2$. Eventually we get $q_i^j \in \mathcal{A}_{[i-1]}^j$ for all $i, j$.
	
	Next we claim that $\partial'$ and $\partial_{[i]}^j$ agree on $\mathcal{A}_{[i]}^j$. Observe that $\partial_{[i]}^j(b_{i+1}^j) = t_i^j \circ \dots  \circ t_1^1 \circ s \circ \partial^-(b_{i+1}^j) \in \mathcal{A}_{[i]}^j$ by the previous paragraph. Base case: $i = 0$, $j = 1$. $\partial'$ and $\partial^s$ agree on $x_1 = b, a_1 = a$. Indeed, both $\partial'$ and $\partial^s$ send $a \mapsto b, b \mapsto 0$. Assume $\partial'$ and $\partial_{[i]}^1$ agree on $\mathcal{A}_{[i]}^1$. Then \begin{align*}
		\partial^1_{[i+1]}(b_{i+1}^1) &= t_{i+1}^1 \circ \partial^1_{[i]} (b_{i+1}^1 + q_{i+1}^1) \\
		&= t_{i+1}^1 (\partial^1_{[i]}(b_{i+1}^1) + \partial'(q_{i+1}^1)) \text{ by the induction hypothesis}\\
		&= t_{i+1}^1 (\partial^1_{[i]}(b_{i+1}^1) + \partial' \circ G \circ (\partial^1_{[i]} + \partial')(b_{i+1}^1)) \\
		&= t_{i+1}^1 (\partial^1_{[i]}(b_{i+1}^1) + (\Id + \tau + G \circ \partial') (\partial^1_{[i]} + \partial')(b_{i+1}^1)) \\
		&= t_{i+1}^1(\partial'(b^1_{i+1}) + \tau \circ (\partial^1_{[i]} + \partial')(b_{i+1}^1)) \\
		&\text{ since $\partial^1_{[i]}(b_{i+1}^1) \in \mathcal{A}_{[i]}^1$ implies $\partial'(\partial^1_{[i]}(b_{i+1}^1)) = (\partial^1_{[i]})^2(b_{i+1}^1) = 0$} \\
		&= \partial'(b^1_{i+1}) + \tau \circ (\partial^1_{[i]} + \partial')(b_{i+1}^1) \\
		&= \partial'(b^1_{i+1}) \text{ by \Cref{keyProjection} and that } \tau \circ t_i^j = \tau.
	\end{align*}
    
    Hence, $\partial^1_{[i+1]}$ and $\partial'$ agree on $\mathcal{A}^1_{[i+1]}$. Proceed by induction and get $\partial^1_{[n_1]}$ and $\partial'$ agree on $\mathcal{A}^1_{[n_1]}$. Note that $\partial^2_{[0]} = \partial^1_{[n_1]}$ and $\partial'$ also agree on $a_2, x_2$ and thus on $\mathcal{A}^2_{[0]}$. Then proceed by induction again. This finishes the proof of \Cref{invReidI}.
\end{proof}

\subsection{Reidemeister III and II moves} \text{ }

We adapt Chekanov's proof for the Reidemeister III and II moves in Lagrangian projection to our case. All notations used in this subsection come from \cite[Sec 8.2-8.4]{chekanov2002dga} which we recommend that the reader review before moving on. Without loss of generality we assume the moves take place in $\Sigma_+$.

The proof for Move IIIa is the same since all we need is an identification of the double points before and after the move as shown in Figure 14. For Move IIIb, the key is to prove \[\partial_-(s) = g\partial_+(s).\] If $s$ is a generator of $\mathcal{A}_\pm$ then the same argument applies since $\widehat{\partial}(s) = \partial_\pm(s)$ by our definition. Assume $s$ is a Reeb chord along $\Gamma$. Then the same argument gives \[\Phi_+^-(s) = g \Phi_+(s),\] where $\Phi_+^-$ and $\Phi_+$ are the DGA maps before and after the move. More generally, we have $\Phi_+^-(s^1_j) = g\Phi_+(s^1_j)$, where $\partial_\Gamma s = \sum_j s^1_j s^2_j$. Hence $\Phi_+^-(s^1_j)s^2_j = g(\Phi_+(s^1_j)s^2_j)$. On the other hand, we have $s^1_j\Phi_-^-(s^2_j) = g(s^1_j\Phi_-(s^2_j))$ since $\Phi_-^-$ agrees with $\Phi_-$. Putting these together gives $\partial_-(s) = g\partial_+(s)$, i.e. $\partial_-(s) = \partial_+^g(s)$. 

For Move II, in contrast to Chekanov's case we have infinitely many Reeb chords ranked by \[\dots  \geq H(a_2) \geq H(a_1) \geq H(a) > H(b) \geq H(b_1) \geq \dots  \geq H(b_m).\] The key is to prove Lemma 8.2(b), and more specifically that \[\zeta(\partial^-(a_j)) = \partial a_j.\] If $a_j$ is a Reeb chord in $\Sigma_\pm$ then the same argument applies. Assume $a_j$ is a Reeb chord along $\Gamma$. Then the same argument gives \[\zeta(\Phi_+^-(a_j)) = \Phi_+(a_j),\] where again $\Phi_+^-$ and $\Phi_+$ denote the DGA maps before and after the move. More generally, we have $\zeta(\Phi^-_+(a_{j, i}^1)) = \Phi_+(a_{j, i}^1)$, where $\partial_\Gamma a_j = \sum_i a_{j, i}^1 a_{j, i}^2$. Thus, $\zeta(\Phi^-_+(a_{j, i}^1)a_{j, i}^2) = \Phi_+(a_{j, i}^1)a_{j, i}^2$. On the other hand, we have $\zeta(a_{j, i}^1\Phi^-_-(a_{j, i}^2)) = a_{j, i}^1\Phi_-(a_{j, i}^2)$ since $\Phi^-_-$ agrees with $\Phi_-$. Putting these together gives $\zeta(\partial^-(a_j)) = \partial a_j$.

\nocite{*}
{
\printbibliography}
\end{document}

%% file: preamble.tex
\usepackage[utf8]{inputenc}
\usepackage[twoside,includefoot,footskip=25pt, margin=1in]{geometry}
\usepackage{amsmath,amssymb,amsthm,enumerate,enumitem,hyperref,algorithm,algorithmic, tikz, tikz-cd, faktor, hyperref, aliascnt,  mdframed, longtable, mathabx, tcolorbox, etoolbox}
\usepackage[nameinlink]{cleveref}
\usepackage{fancyhdr}
\usepackage{url}
\usepackage[mathscr]{euscript}
\usepackage[bb=boondox]{mathalfa}
\usepackage[skip=0.5cm]{caption}
\usepackage[symbol]{footmisc}
\usepackage[
backend=biber,
style=alphabetic,
]{biblatex}

\let\cref\Cref
\usepackage{multicol}
\usepackage{placeins}
\FloatBarrier

\usepackage{xcolor}
\definecolor{myblue}{RGB}{30,30,210}
\definecolor{myteal}{RGB}{0,120,120}
\hypersetup{
    colorlinks=true,
    linkcolor=myblue,
    citecolor=myteal,
    urlcolor=myteal
}

\newtheoremstyle{mystyle}{}{}{}{}{\bfseries}{.}{ }{}
\newtheoremstyle{cstyle}{}{}{}{}{\bfseries}{.}{ }{\it }
\makeatletter
\renewenvironment{proof}[1][\proofname] {\par\pushQED{\qed}{\normalfont\bfseries\topsep6\p@\@plus6\p@\relax #1\@addpunct{.} }}{\popQED\endtrivlist\@endpefalse}
\makeatother

\theoremstyle{mystyle}

\newtheorem{definition}{Definition}[section]

\newaliascnt{theorem}{definition}
\newtheorem{theorem}[theorem]{Theorem}
\aliascntresetthe{theorem}

\newaliascnt{lemma}{definition}
\newtheorem{lemma}[lemma]{Lemma}
\aliascntresetthe{lemma}

\newaliascnt{conjecture}{definition}
\newtheorem{conjecture}[conjecture]{Conjecture}
\aliascntresetthe{conjecture}

\newaliascnt{proposition}{definition}
\newtheorem{proposition}[proposition]{Proposition}
\aliascntresetthe{proposition}

\newaliascnt{corollary}{definition}
\newtheorem{corollary}[corollary]{Corollary}
\aliascntresetthe{corollary}

\newaliascnt{example}{definition}
\newtheorem{example}[example]{Example}
\aliascntresetthe{example}

\newaliascnt{remark}{definition}
\newtheorem{remark}[remark]{Remark}
\aliascntresetthe{remark}

\newaliascnt{construction}{definition}
\newtheorem{construction}[construction]{Construction}
\aliascntresetthe{construction}

\newcommand{\N}{\mathbb{N}}
\newcommand{\Q}{\mathbb{Q}}
\newcommand{\R}{\mathbb{R}}
\newcommand{\T}{\mathbb{T}}

\newcommand{\Z}{\mathbb{Z}}

\makeatletter
\newcommand{\tpitchfork}{%
  \vbox{
    \baselineskip\z@skip
    \lineskip-.52ex
    \lineskiplimit\maxdimen
    \m@th
    \ialign{##\crcr\hidewidth\smash{$-$}\hidewidth\crcr$\pitchfork$\crcr}
  }%
}
\makeatother

\newcommand{\de}{\text{d}}

\newcommand{\<}{\langle}
\renewcommand{\>}{\rangle}

\DeclareMathOperator{\ind}{\text{ind}}

\DeclareMathOperator{\rot}{\text{rot}}

\DeclareMathOperator{\Id}{Id}

\DeclareMathOperator{\tb}{tb}

\DeclareMathOperator{\writhe}{writhe}

\usepackage{todonotes}

\bibliography{references}

\allowdisplaybreaks

%% file: images/Chekanov_example_combined.tex
\begin{tikzpicture}[x=0.70pt,y=0.70pt,yscale=-1,xscale=1]

\draw  [color={rgb, 255:red, 144; green, 19; blue, 254 }  ,draw opacity=1 ][fill={rgb, 255:red, 246; green, 246; blue, 246 }  ,fill opacity=1 ][line width=1.5]  (363.16,165.06) .. controls (363.16,107.38) and (410.59,60.63) .. (469.09,60.63) .. controls (527.6,60.63) and (575.03,107.38) .. (575.03,165.06) .. controls (575.03,222.74) and (527.6,269.49) .. (469.09,269.49) .. controls (410.59,269.49) and (363.16,222.74) .. (363.16,165.06) -- cycle ;
\draw [color={rgb, 255:red, 0; green, 0; blue, 255 }  ,draw opacity=1 ][line width=1.5]    (462.07,134.97) .. controls (462.14,140.28) and (481.6,146.19) .. (480.85,154.38) ;
\draw  [color={rgb, 255:red, 144; green, 19; blue, 254 }  ,draw opacity=1 ][fill={rgb, 255:red, 246; green, 246; blue, 246 }  ,fill opacity=1 ][line width=1.5]  (42.65,406.27) .. controls (42.65,348.73) and (90.64,302.09) .. (149.84,302.09) .. controls (209.04,302.09) and (257.03,348.73) .. (257.03,406.27) .. controls (257.03,463.81) and (209.04,510.45) .. (149.84,510.45) .. controls (90.64,510.45) and (42.65,463.81) .. (42.65,406.27) -- cycle ;
\draw  [color={rgb, 255:red, 144; green, 19; blue, 254 }  ,draw opacity=1 ][fill={rgb, 255:red, 255; green, 255; blue, 255 }  ,fill opacity=1 ][line width=1.5]  (65.75,406.27) .. controls (65.75,401.24) and (69.94,397.17) .. (75.12,397.17) .. controls (80.29,397.17) and (84.48,401.24) .. (84.48,406.27) .. controls (84.48,411.3) and (80.29,415.37) .. (75.12,415.37) .. controls (69.94,415.37) and (65.75,411.3) .. (65.75,406.27) -- cycle ;
\draw  [color={rgb, 255:red, 144; green, 19; blue, 254 }  ,draw opacity=1 ][fill={rgb, 255:red, 255; green, 255; blue, 255 }  ,fill opacity=1 ][line width=1.5]  (215.21,406.27) .. controls (215.21,401.24) and (219.4,397.17) .. (224.57,397.17) .. controls (229.74,397.17) and (233.94,401.24) .. (233.94,406.27) .. controls (233.94,411.3) and (229.74,415.37) .. (224.57,415.37) .. controls (219.4,415.37) and (215.21,411.3) .. (215.21,406.27) -- cycle ;

\draw [color={rgb, 255:red, 0; green, 0; blue, 255 }  ,draw opacity=1 ][line width=1.5]    (150.13,303.54) -- (150.13,360.18) ;
\draw [color={rgb, 255:red, 0; green, 0; blue, 255 }  ,draw opacity=1 ][line width=1.5]    (150.13,451.53) -- (150.13,508.17) ;
\draw [color={rgb, 255:red, 0; green, 0; blue, 255 }  ,draw opacity=1 ][line width=1.5]    (153.17,398.54) .. controls (159.66,408.72) and (161.01,417.21) .. (175.01,416.7) .. controls (189.02,416.19) and (191.87,407.89) .. (191.6,403.46) .. controls (191.32,399.03) and (187.97,390.95) .. (173.45,391.44) .. controls (158.93,391.93) and (161.75,399.36) .. (152.45,415.31) ;
\draw [color={rgb, 255:red, 246; green, 246; blue, 246 }  ,draw opacity=1 ][line width=3]    (155.27,407.89) -- (158.74,401.48) ;
\draw [color={rgb, 255:red, 0; green, 0; blue, 255 }  ,draw opacity=1 ][line width=1.5]    (160.56,410.93) -- (153.17,398.54) ;
\draw [color={rgb, 255:red, 0; green, 0; blue, 255 }  ,draw opacity=1 ][line width=1.5]    (150.09,360.18) .. controls (149.92,376.74) and (148.65,388.78) .. (153.17,398.54) ;
\draw [color={rgb, 255:red, 0; green, 0; blue, 255 }  ,draw opacity=1 ][line width=1.5]    (150.09,451.53) .. controls (150.28,439.66) and (148.59,421.4) .. (152.45,415.31) ;
\draw  [color={rgb, 255:red, 144; green, 19; blue, 254 }  ,draw opacity=1 ][fill={rgb, 255:red, 255; green, 255; blue, 255 }  ,fill opacity=1 ][line width=1.5]  (371.69,165.06) .. controls (371.69,160.02) and (375.83,155.93) .. (380.95,155.93) .. controls (386.06,155.93) and (390.2,160.02) .. (390.2,165.06) .. controls (390.2,170.1) and (386.06,174.18) .. (380.95,174.18) .. controls (375.83,174.18) and (371.69,170.1) .. (371.69,165.06) -- cycle ;
\draw  [color={rgb, 255:red, 144; green, 19; blue, 254 }  ,draw opacity=1 ][fill={rgb, 255:red, 255; green, 255; blue, 255 }  ,fill opacity=1 ][line width=1.5]  (547.99,165.06) .. controls (547.99,160.02) and (552.13,155.93) .. (557.24,155.93) .. controls (562.35,155.93) and (566.5,160.02) .. (566.5,165.06) .. controls (566.5,170.1) and (562.35,174.18) .. (557.24,174.18) .. controls (552.13,174.18) and (547.99,170.1) .. (547.99,165.06) -- cycle ;
\draw [color={rgb, 255:red, 0; green, 0; blue, 255 }  ,draw opacity=1 ][line width=1.5]    (469.09,60.63) -- (469.09,75.01) ;
\draw [color={rgb, 255:red, 0; green, 0; blue, 255 }  ,draw opacity=1 ][line width=1.5]    (469.09,248.04) -- (469.09,269.49) ;
\draw [color={rgb, 255:red, 246; green, 246; blue, 246 }  ,draw opacity=1 ][line width=3]    (468.73,195) -- (465.51,192.95) ;
\draw [color={rgb, 255:red, 0; green, 0; blue, 255 }  ,draw opacity=1 ][line width=1.5]    (457.45,204.72) .. controls (415.04,233.09) and (468.64,225.51) .. (469.09,248.04) ;
\draw [color={rgb, 255:red, 0; green, 0; blue, 255 }  ,draw opacity=1 ][line width=1.5]    (444.25,186.42) .. controls (404.7,230.27) and (386.45,137.9) .. (398.6,118.82) .. controls (410.75,99.74) and (472.93,106.81) .. (469.09,75.01) ;
\draw [color={rgb, 255:red, 0; green, 0; blue, 255 }  ,draw opacity=1 ][line width=1.5]    (461.8,170.2) .. controls (457.82,173.53) and (446.95,184.15) .. (444.25,186.42) ;
\draw [color={rgb, 255:red, 0; green, 0; blue, 255 }  ,draw opacity=1 ][line width=1.5]    (490.52,185.36) .. controls (482.46,187.04) and (478.13,189.76) .. (473.41,185.68) .. controls (468.68,181.61) and (469.98,178.33) .. (471.3,177.08) .. controls (472.62,175.84) and (477.03,173.11) .. (480.93,178.58) .. controls (484.83,184.06) and (482.65,186.46) .. (481.29,194.41) ;
\draw [color={rgb, 255:red, 0; green, 0; blue, 255 }  ,draw opacity=1 ][line width=1.5]    (462.62,154.15) .. controls (462.68,159.46) and (483.95,165.46) .. (487.49,167.53) ;
\draw [color={rgb, 255:red, 0; green, 0; blue, 255 }  ,draw opacity=1 ][line width=1.5]    (463.1,189.99) .. controls (465.28,192.38) and (488.03,220.65) .. (505.99,197.74) ;
\draw [color={rgb, 255:red, 0; green, 0; blue, 255 }  ,draw opacity=1 ][line width=1.5]    (508.31,141.02) .. controls (507.88,167.19) and (498.04,185.08) .. (490.52,185.36) ;
\draw [color={rgb, 255:red, 0; green, 0; blue, 255 }  ,draw opacity=1 ][line width=1.5]    (500.27,126.08) .. controls (511.01,128.58) and (515.66,133.88) .. (525.46,127.89) .. controls (535.25,121.9) and (533.9,115.85) .. (531.9,113.36) .. controls (529.9,110.86) and (524.21,107.46) .. (514.04,113.65) .. controls (503.86,119.84) and (508.45,127.85) .. (508.31,141.02) ;
\draw [color={rgb, 255:red, 0; green, 0; blue, 255 }  ,draw opacity=1 ][line width=1.5]    (438.74,141.74) .. controls (437.55,128.1) and (441.36,119.47) .. (431.92,114.6) .. controls (422.48,109.73) and (416.94,111.79) .. (415.05,114.28) .. controls (413.15,116.78) and (411.54,122.61) .. (421.12,128.84) .. controls (430.71,135.07) and (437.01,128.58) .. (444.77,126.56) ;
\draw [color={rgb, 255:red, 0; green, 0; blue, 255 }  ,draw opacity=1 ][line width=1.5]    (444.77,126.56) .. controls (458.91,123.68) and (461.5,131.97) .. (462.07,134.97) ;
\draw [color={rgb, 255:red, 0; green, 0; blue, 255 }  ,draw opacity=1 ][line width=1.5]    (500.27,126.08) .. controls (481.75,122.36) and (480.67,129.31) .. (480.31,134.72) ;
\draw [color={rgb, 255:red, 246; green, 246; blue, 246 }  ,draw opacity=1 ][line width=3.75]    (486.18,186.93) -- (480.29,187.41) ;
\draw [color={rgb, 255:red, 246; green, 246; blue, 246 }  ,draw opacity=1 ][line width=3]    (478,164.25) -- (470,160.25) ;
\draw [color={rgb, 255:red, 246; green, 246; blue, 246 }  ,draw opacity=1 ][line width=3]    (454.46,176.36) -- (449.2,181.72) ;
\draw [color={rgb, 255:red, 246; green, 246; blue, 246 }  ,draw opacity=1 ][line width=3]    (477.6,203.43) -- (472.34,198.82) ;
\draw [color={rgb, 255:red, 246; green, 246; blue, 246 }  ,draw opacity=1 ][line width=3]    (507.45,133.08) -- (507.32,124.89) ;
\draw [color={rgb, 255:red, 246; green, 246; blue, 246 }  ,draw opacity=1 ][line width=3]    (442.52,128.15) -- (434.29,130.04) ;
\draw [color={rgb, 255:red, 0; green, 0; blue, 255 }  ,draw opacity=1 ][line width=1.5]    (438.64,124.62) -- (438.53,131.48) ;
\draw [color={rgb, 255:red, 246; green, 246; blue, 246 }  ,draw opacity=1 ][line width=3]    (467.13,140.88) -- (474.57,144.65) ;
\draw [color={rgb, 255:red, 246; green, 246; blue, 246 }  ,draw opacity=1 ][line width=3]    (498.25,177.96) -- (502.25,172.96) ;
\draw [color={rgb, 255:red, 0; green, 0; blue, 255 }  ,draw opacity=1 ][line width=1.5]    (482.11,190.29) -- (483.19,184.5) ;
\draw [color={rgb, 255:red, 0; green, 0; blue, 255 }  ,draw opacity=1 ][line width=1.5]    (511.45,129.96) -- (502.66,126.54) ;
\draw [color={rgb, 255:red, 0; green, 0; blue, 255 }  ,draw opacity=1 ][line width=1.5]    (487.49,167.53) .. controls (490.21,168.97) and (516.33,178.56) .. (505.99,197.74) ;
\draw [color={rgb, 255:red, 0; green, 0; blue, 255 }  ,draw opacity=1 ][line width=1.5]    (481.29,194.41) .. controls (478.64,205.05) and (462.86,199.29) .. (457.45,204.72) ;
\draw [color={rgb, 255:red, 0; green, 0; blue, 255 }  ,draw opacity=1 ][line width=1.5]    (438.74,141.74) .. controls (436.6,171.02) and (460.58,186.9) .. (463.1,189.99) ;
\draw [color={rgb, 255:red, 0; green, 0; blue, 255 }  ,draw opacity=1 ][line width=1.5]    (480.85,153.9) .. controls (480.78,159.21) and (464.89,167.3) .. (461.8,170.2) ;
\draw  [color={rgb, 255:red, 144; green, 19; blue, 254 }  ,draw opacity=1 ][fill={rgb, 255:red, 246; green, 246; blue, 246 }  ,fill opacity=1 ][line width=1.5]  (43.92,165.06) .. controls (43.92,107.38) and (91.34,60.63) .. (149.85,60.63) .. controls (208.36,60.63) and (255.79,107.38) .. (255.79,165.06) .. controls (255.79,222.74) and (208.36,269.49) .. (149.85,269.49) .. controls (91.34,269.49) and (43.92,222.74) .. (43.92,165.06) -- cycle ;
\draw  [color={rgb, 255:red, 144; green, 19; blue, 254 }  ,draw opacity=1 ][fill={rgb, 255:red, 255; green, 255; blue, 255 }  ,fill opacity=1 ][line width=1.5]  (52.45,165.06) .. controls (52.45,160.02) and (56.59,155.93) .. (61.7,155.93) .. controls (66.82,155.93) and (70.96,160.02) .. (70.96,165.06) .. controls (70.96,170.1) and (66.82,174.18) .. (61.7,174.18) .. controls (56.59,174.18) and (52.45,170.1) .. (52.45,165.06) -- cycle ;
\draw  [color={rgb, 255:red, 144; green, 19; blue, 254 }  ,draw opacity=1 ][fill={rgb, 255:red, 255; green, 255; blue, 255 }  ,fill opacity=1 ][line width=1.5]  (228.74,165.06) .. controls (228.74,160.02) and (232.89,155.93) .. (238,155.93) .. controls (243.11,155.93) and (247.25,160.02) .. (247.25,165.06) .. controls (247.25,170.1) and (243.11,174.18) .. (238,174.18) .. controls (232.89,174.18) and (228.74,170.1) .. (228.74,165.06) -- cycle ;
\draw [color={rgb, 255:red, 0; green, 0; blue, 255 }  ,draw opacity=1 ][line width=1.5]    (149.85,60.63) -- (149.85,75.01) ;
\draw [color={rgb, 255:red, 0; green, 0; blue, 255 }  ,draw opacity=1 ][line width=1.5]    (149.85,248.04) -- (149.85,269.49) ;
\draw [color={rgb, 255:red, 0; green, 0; blue, 255 }  ,draw opacity=1 ][line width=1.5]    (182.58,141.93) .. controls (182.19,166.3) and (161.1,181.73) .. (157.92,184.74) ;
\draw [color={rgb, 255:red, 0; green, 0; blue, 255 }  ,draw opacity=1 ][line width=1.5]    (139.96,169.56) .. controls (135.34,172.83) and (124.87,183.55) .. (123.88,184.66) ;
\draw [color={rgb, 255:red, 0; green, 0; blue, 255 }  ,draw opacity=1 ][line width=1.5]    (175.22,128) .. controls (185.06,130.34) and (189.32,135.28) .. (198.29,129.7) .. controls (207.27,124.12) and (206.04,118.48) .. (204.2,116.15) .. controls (202.37,113.83) and (197.15,110.67) .. (187.83,116.43) .. controls (178.51,122.2) and (182.71,129.66) .. (182.58,141.93) ;
\draw [color={rgb, 255:red, 246; green, 246; blue, 246 }  ,draw opacity=1 ][line width=3]    (182.08,133.34) -- (181.95,126.51) ;
\draw [color={rgb, 255:red, 0; green, 0; blue, 255 }  ,draw opacity=1 ][line width=1.5]    (187.51,131.75) -- (179.08,129.34) ;
\draw [color={rgb, 255:red, 0; green, 0; blue, 255 }  ,draw opacity=1 ][line width=1.5]    (118.82,142.6) .. controls (117.74,129.89) and (121.23,121.85) .. (112.57,117.31) .. controls (103.92,112.78) and (98.85,114.69) .. (97.11,117.02) .. controls (95.38,119.34) and (93.9,124.78) .. (102.68,130.58) .. controls (111.46,136.38) and (117.24,130.34) .. (124.35,128.45) ;
\draw [color={rgb, 255:red, 246; green, 246; blue, 246 }  ,draw opacity=1 ][line width=3]    (122.46,129.94) -- (113.66,132.69) ;
\draw [color={rgb, 255:red, 0; green, 0; blue, 255 }  ,draw opacity=1 ][line width=1.5]    (118.73,128.44) -- (118.64,134.11) ;
\draw [color={rgb, 255:red, 0; green, 0; blue, 255 }  ,draw opacity=1 ][line width=1.5]    (178.45,187.35) .. controls (178.37,199.6) and (175.38,206.73) .. (183.58,212.38) .. controls (191.77,218.04) and (196.67,215.15) .. (198.26,212.85) .. controls (199.85,210.55) and (201.12,205.15) .. (192.64,199.27) .. controls (184.16,193.38) and (178.35,199.08) .. (166.78,203.54) ;
\draw [color={rgb, 255:red, 246; green, 246; blue, 246 }  ,draw opacity=1 ][line width=3]    (174.78,200.71) -- (181.29,197.02) ;
\draw [color={rgb, 255:red, 0; green, 0; blue, 255 }  ,draw opacity=1 ][line width=1.5]    (177.81,202.16) -- (177.93,195.57) ;
\draw [color={rgb, 255:red, 0; green, 0; blue, 255 }  ,draw opacity=1 ][line width=1.5]    (140.21,136.29) .. controls (140.27,141.23) and (158.11,146.74) .. (157.42,154.37) ;
\draw [color={rgb, 255:red, 246; green, 246; blue, 246 }  ,draw opacity=1 ][line width=3]    (150.98,145.53) -- (145.76,141.48) ;
\draw [color={rgb, 255:red, 0; green, 0; blue, 255 }  ,draw opacity=1 ][line width=1.5]    (156.92,136.06) .. controls (156.86,141) and (140.15,144.95) .. (140.71,154.6) ;
\draw [color={rgb, 255:red, 0; green, 0; blue, 255 }  ,draw opacity=1 ][line width=1.5]    (140.71,154.6) .. controls (140.77,159.55) and (153.79,166.58) .. (158.5,169.67) ;
\draw [color={rgb, 255:red, 246; green, 246; blue, 246 }  ,draw opacity=1 ][line width=3]    (151.85,166.06) -- (145.63,162) ;
\draw [color={rgb, 255:red, 0; green, 0; blue, 255 }  ,draw opacity=1 ][line width=1.5]    (157.42,154.37) .. controls (157.36,159.31) and (143.32,166.13) .. (139.96,169.56) ;
\draw [color={rgb, 255:red, 246; green, 246; blue, 246 }  ,draw opacity=1 ][line width=3]    (172.43,175.01) -- (168,172.8) ;
\draw [color={rgb, 255:red, 246; green, 246; blue, 246 }  ,draw opacity=1 ][line width=3]    (132.85,175.4) -- (128.86,180.08) ;
\draw [color={rgb, 255:red, 0; green, 0; blue, 255 }  ,draw opacity=1 ][line width=1.5]    (141.15,187.54) .. controls (143.15,189.77) and (158.54,204.41) .. (166.78,203.54) ;
\draw [color={rgb, 255:red, 246; green, 246; blue, 246 }  ,draw opacity=1 ][line width=3]    (150.48,196) -- (145.26,191.95) ;
\draw [color={rgb, 255:red, 0; green, 0; blue, 255 }  ,draw opacity=1 ][line width=1.5]    (157.92,184.74) .. controls (154.12,187.99) and (143.64,196.47) .. (140.46,199.93) ;
\draw [color={rgb, 255:red, 0; green, 0; blue, 255 }  ,draw opacity=1 ][line width=1.5]    (118.82,142.6) .. controls (116.86,169.87) and (138.84,184.66) .. (141.15,187.54) ;
\draw [color={rgb, 255:red, 0; green, 0; blue, 255 }  ,draw opacity=1 ][line width=1.5]    (124.35,128.45) .. controls (137.31,125.77) and (139.69,133.49) .. (140.21,136.29) ;
\draw [color={rgb, 255:red, 0; green, 0; blue, 255 }  ,draw opacity=1 ][line width=1.5]    (175.22,128) .. controls (158.24,124.54) and (157.26,131.01) .. (156.92,136.06) ;
\draw [color={rgb, 255:red, 0; green, 0; blue, 255 }  ,draw opacity=1 ][line width=1.5]    (123.88,184.66) -- (112.4,195.61) ;
\draw [color={rgb, 255:red, 0; green, 0; blue, 255 }  ,draw opacity=1 ][line width=1.5]    (140.46,199.93) -- (127.53,212.16) ;
\draw [color={rgb, 255:red, 246; green, 246; blue, 246 }  ,draw opacity=1 ][line width=3]    (168.43,176.69) -- (165.22,175.08) ;
\draw [color={rgb, 255:red, 0; green, 0; blue, 255 }  ,draw opacity=1 ][line width=1.5]    (127.53,212.16) .. controls (105.8,231.16) and (149.4,225.51) .. (149.85,248.04) ;
\draw [color={rgb, 255:red, 0; green, 0; blue, 255 }  ,draw opacity=1 ][line width=1.5]    (112.4,195.61) .. controls (71.49,231.16) and (67.2,137.9) .. (79.36,118.82) .. controls (91.51,99.74) and (153.69,106.81) .. (149.85,75.01) ;
\draw [color={rgb, 255:red, 0; green, 0; blue, 255 }  ,draw opacity=1 ][line width=1.5]    (158.5,169.67) .. controls (161.27,171.94) and (178.49,177.16) .. (178.45,187.35) ;
\draw  [color={rgb, 255:red, 144; green, 19; blue, 254 }  ,draw opacity=1 ][fill={rgb, 255:red, 246; green, 246; blue, 246 }  ,fill opacity=1 ][line width=1.5]  (363.65,406.27) .. controls (363.65,348.73) and (411.64,302.09) .. (470.84,302.09) .. controls (530.04,302.09) and (578.03,348.73) .. (578.03,406.27) .. controls (578.03,463.81) and (530.04,510.45) .. (470.84,510.45) .. controls (411.64,510.45) and (363.65,463.81) .. (363.65,406.27) -- cycle ;
\draw  [color={rgb, 255:red, 144; green, 19; blue, 254 }  ,draw opacity=1 ][fill={rgb, 255:red, 255; green, 255; blue, 255 }  ,fill opacity=1 ][line width=1.5]  (386.75,406.27) .. controls (386.75,401.24) and (390.94,397.17) .. (396.12,397.17) .. controls (401.29,397.17) and (405.48,401.24) .. (405.48,406.27) .. controls (405.48,411.3) and (401.29,415.37) .. (396.12,415.37) .. controls (390.94,415.37) and (386.75,411.3) .. (386.75,406.27) -- cycle ;
\draw  [color={rgb, 255:red, 144; green, 19; blue, 254 }  ,draw opacity=1 ][fill={rgb, 255:red, 255; green, 255; blue, 255 }  ,fill opacity=1 ][line width=1.5]  (536.21,406.27) .. controls (536.21,401.24) and (540.4,397.17) .. (545.57,397.17) .. controls (550.74,397.17) and (554.94,401.24) .. (554.94,406.27) .. controls (554.94,411.3) and (550.74,415.37) .. (545.57,415.37) .. controls (540.4,415.37) and (536.21,411.3) .. (536.21,406.27) -- cycle ;

\draw [color={rgb, 255:red, 0; green, 0; blue, 255 }  ,draw opacity=1 ][line width=1.5]    (471.13,303.54) -- (471.13,360.18) ;
\draw [color={rgb, 255:red, 0; green, 0; blue, 255 }  ,draw opacity=1 ][line width=1.5]    (471.13,451.53) -- (471.13,508.17) ;
\draw [color={rgb, 255:red, 0; green, 0; blue, 255 }  ,draw opacity=1 ][line width=1.5]    (474.17,398.54) .. controls (480.66,408.72) and (482.01,417.21) .. (496.01,416.7) .. controls (510.02,416.19) and (512.87,407.89) .. (512.6,403.46) .. controls (512.32,399.03) and (508.97,390.95) .. (494.45,391.44) .. controls (479.93,391.93) and (482.75,399.36) .. (473.45,415.31) ;
\draw [color={rgb, 255:red, 246; green, 246; blue, 246 }  ,draw opacity=1 ][line width=3]    (476.27,407.89) -- (479.74,401.48) ;
\draw [color={rgb, 255:red, 0; green, 0; blue, 255 }  ,draw opacity=1 ][line width=1.5]    (481.56,410.93) -- (474.17,398.54) ;
\draw [color={rgb, 255:red, 0; green, 0; blue, 255 }  ,draw opacity=1 ][line width=1.5]    (471.09,360.18) .. controls (470.92,376.74) and (469.65,388.78) .. (474.17,398.54) ;
\draw [color={rgb, 255:red, 0; green, 0; blue, 255 }  ,draw opacity=1 ][line width=1.5]    (471.09,451.53) .. controls (471.28,439.66) and (469.59,421.4) .. (473.45,415.31) ;
\draw [color={rgb, 255:red, 0; green, 0; blue, 255 }  ,draw opacity=1 ][line width=1.5]    (480.31,134.24) .. controls (480.24,139.55) and (462.01,143.79) .. (462.62,154.15) ;

\draw (140.05,542) node [anchor=north west][inner sep=0.75pt]   [align=left] {(a)};
\draw (462.35,541) node [anchor=north west][inner sep=0.75pt]   [align=left] {(b)};
\draw (211,110.58) node [anchor=north west][inner sep=0.75pt]  [color={rgb, 255:red, 0; green, 0; blue, 255 }  ,opacity=1 ]  {$\Lambda _{1}$};
\draw (538,110.58) node [anchor=north west][inner sep=0.75pt]  [color={rgb, 255:red, 0; green, 0; blue, 255 }  ,opacity=1 ]  {$\Lambda _{2}$};
\draw (172.34,472.46) node [anchor=north west][inner sep=0.75pt]    {$\Sigma _{-}$};
\draw (167.97,231.43) node [anchor=north west][inner sep=0.75pt]    {$\Sigma _{+}$};
\draw (487.21,231.43) node [anchor=north west][inner sep=0.75pt]    {$\Sigma _{+}$};
\draw (493.34,472.46) node [anchor=north west][inner sep=0.75pt]    {$\Sigma _{-}$};
\draw (582,400.25) node [anchor=north west][inner sep=0.75pt]  [color={rgb, 255:red, 144; green, 19; blue, 254 }  ,opacity=1 ]  {$\Gamma $};
\draw (582,154.75) node [anchor=north west][inner sep=0.75pt]  [color={rgb, 255:red, 144; green, 19; blue, 254 }  ,opacity=1 ]  {$\Gamma $};
\draw (262,154.75) node [anchor=north west][inner sep=0.75pt]  [color={rgb, 255:red, 144; green, 19; blue, 254 }  ,opacity=1 ]  {$\Gamma $};
\draw (262,400.25) node [anchor=north west][inner sep=0.75pt]  [color={rgb, 255:red, 144; green, 19; blue, 254 }  ,opacity=1 ]  {$\Gamma $};

\end{tikzpicture}

%% file: images/curve_justification.tex
\begin{tikzpicture}[x=0.75pt,y=0.75pt,yscale=-1,xscale=1]

\draw  [draw opacity=0][fill={rgb, 255:red, 216; green, 216; blue, 216 }  ,fill opacity=1 ] (386.6,186.53) .. controls (391.51,186.67) and (399.51,185.67) .. (408.03,186.53) .. controls (407.51,202.67) and (407.51,204.67) .. (407.51,224.67) .. controls (405.51,246.67) and (388.51,242.67) .. (386.51,224.67) .. controls (388.51,201.67) and (386.51,199.67) .. (386.6,186.53) -- cycle ;
\draw  [fill={rgb, 255:red, 216; green, 216; blue, 216 }  ,fill opacity=1 ] (296,59) .. controls (307,59.26) and (321,57.26) .. (336,59) .. controls (337.51,90.09) and (369,194.98) .. (365.6,240.72) .. controls (350,241.98) and (355,239.98) .. (345.6,240.72) .. controls (344.51,148.09) and (324.51,161.09) .. (325.6,240.72) .. controls (315,240.98) and (316,239.98) .. (305.6,240.72) .. controls (306.51,158.09) and (288.51,177.09) .. (285.6,240.72) .. controls (279,241.98) and (273,240.98) .. (265.6,240.72) .. controls (260,203.98) and (300.51,86.09) .. (296,59) -- cycle ;
\draw [color={rgb, 255:red, 0; green, 0; blue, 255 }  ,draw opacity=1 ][line width=1.5]    (296,59) .. controls (296.51,116.68) and (266.11,162.11) .. (265.6,240.72) ;
\draw [color={rgb, 255:red, 0; green, 0; blue, 255 }  ,draw opacity=1 ][line width=1.5]    (336,59) .. controls (337.04,91.93) and (365.11,159.51) .. (365.6,240.72) ;
\draw [line width=1.5]    (296,59) -- (336,59) ;
\draw [color={rgb, 255:red, 0; green, 0; blue, 255 }  ,draw opacity=1 ][line width=1.5]    (285.6,240.72) .. controls (289.01,178.72) and (307.01,158.72) .. (305.6,240.72) ;
\draw [color={rgb, 255:red, 0; green, 0; blue, 255 }  ,draw opacity=1 ][line width=1.5]    (325.6,240.72) .. controls (325.01,172.72) and (342.01,137.72) .. (345.6,240.72) ;
\draw [line width=1.5]    (265.6,240.72) -- (285.6,240.72) ;
\draw [line width=1.5]    (305.6,240.72) -- (325.6,240.72) ;
\draw [line width=1.5]    (345.6,240.72) -- (365.6,240.72) ;
\draw [color={rgb, 255:red, 0; green, 0; blue, 255 }  ,draw opacity=1 ][line width=1.5]    (408.03,214.96) .. controls (408.56,248.29) and (387.13,248.29) .. (386.6,214.96) ;
\draw [color={rgb, 255:red, 0; green, 0; blue, 255 }  ,draw opacity=1 ][line width=1.5]    (386.6,186.53) -- (386.6,214.96) ;
\draw [color={rgb, 255:red, 0; green, 0; blue, 255 }  ,draw opacity=1 ][line width=1.5]    (408.03,186.53) -- (408.03,214.96) ;
\draw  [fill={rgb, 255:red, 216; green, 216; blue, 216 }  ,fill opacity=1 ][line width=1.5]  (386.6,186.53) .. controls (386.6,182.88) and (391.4,179.93) .. (397.32,179.93) .. controls (403.23,179.93) and (408.03,182.88) .. (408.03,186.53) .. controls (408.03,190.18) and (403.23,193.13) .. (397.32,193.13) .. controls (391.4,193.13) and (386.6,190.18) .. (386.6,186.53) -- cycle ;
\draw  [draw opacity=0][fill={rgb, 255:red, 216; green, 216; blue, 216 }  ,fill opacity=1 ] (333.71,112.69) .. controls (344.72,112.95) and (313.65,91.05) .. (328.64,92.79) .. controls (377.14,91.4) and (410.71,119.77) .. (407.31,165.51) .. controls (391.71,166.77) and (395.64,164.04) .. (386.24,164.79) .. controls (386.14,107.4) and (346.14,108.81) .. (333.71,112.69) -- cycle ;
\draw [color={rgb, 255:red, 0; green, 0; blue, 255 }  ,draw opacity=1 ][line width=1.5]    (334.07,111.96) .. controls (364.07,106.96) and (386.23,122.22) .. (386.6,164.06) ;
\draw [color={rgb, 255:red, 0; green, 0; blue, 255 }  ,draw opacity=1 ][line width=1.5]    (329,92.07) .. controls (371.07,91.96) and (407.23,108.28) .. (407.6,163.47) ;
\draw  [fill={rgb, 255:red, 216; green, 216; blue, 216 }  ,fill opacity=1 ][line width=1.5]  (386.24,164.79) .. controls (386.24,161.14) and (391.04,158.18) .. (396.95,158.18) .. controls (402.87,158.18) and (407.67,161.14) .. (407.67,164.79) .. controls (407.67,168.43) and (402.87,171.39) .. (396.95,171.39) .. controls (391.04,171.39) and (386.24,168.43) .. (386.24,164.79) -- cycle ;

\draw (415,166.27) node [anchor=north west][inner sep=0.75pt]    {$\gamma _{1}$};
\draw (267.6,247.05) node [anchor=north west][inner sep=0.75pt]    {$c_{1}$};
\draw (347.6,247.05) node [anchor=north west][inner sep=0.75pt]    {$c_{k}$};
\draw (301,248.05) node [anchor=north west][inner sep=0.75pt]    {$\cdots $};
\draw (310,39.4) node [anchor=north west][inner sep=0.75pt]    {$w$};

\end{tikzpicture}

%% file: images/orientations_on_Sigmapm.tex
\begin{tikzpicture}[x=0.65pt,y=0.65pt,yscale=-1,xscale=1]

\draw   [fill={rgb, 255:red, 246; green, 246; blue, 246 }  ,fill opacity=1 ] (198.13,104) -- (480,104) -- (429.38,200) -- (147.51,200) -- cycle ;
\draw [color={rgb, 255:red, 144; green, 19; blue, 254 }  ,draw opacity=1 ][line width=1.5]    (290.21,200) -- (319.43,140.43) -- (337.29,104) ;
\draw [color={rgb, 255:red, 255; green, 0; blue, 0 }  ,draw opacity=1 ][line width=1.5]    (243.14,152) -- (243.14,116.6) ;
\draw [shift={(243.14,113.6)}, rotate = 90] [color={rgb, 255:red, 255; green, 0; blue, 0 }  ,draw opacity=1 ][line width=1.5]    (8.53,-2.57) .. controls (5.42,-1.09) and (2.58,-0.23) .. (0,0) .. controls (2.58,0.23) and (5.42,1.09) .. (8.53,2.57)   ;
\draw [color={rgb, 255:red, 255; green, 0; blue, 0 }  ,draw opacity=1 ][line width=1.5]    (384.37,187.4) -- (384.37,152) ;
\draw [shift={(384.37,190.4)}, rotate = 270] [color={rgb, 255:red, 255; green, 0; blue, 0 }  ,draw opacity=1 ][line width=1.5]    (8.53,-2.57) .. controls (5.42,-1.09) and (2.58,-0.23) .. (0,0) .. controls (2.58,0.23) and (5.42,1.09) .. (8.53,2.57)   ;
\draw [color={rgb, 255:red, 255; green, 0; blue, 0 }  ,draw opacity=1 ][line width=1.5]    (313.75,152) -- (327.04,124.91) ;
\draw [shift={(328.36,122.22)}, rotate = 116.12] [color={rgb, 255:red, 255; green, 0; blue, 0 }  ,draw opacity=1 ][line width=1.5]    (8.53,-2.57) .. controls (5.42,-1.09) and (2.58,-0.23) .. (0,0) .. controls (2.58,0.23) and (5.42,1.09) .. (8.53,2.57)   ;
\draw  [draw opacity=0] (249.81,141.85) .. controls (257.4,143.37) and (262.87,147.62) .. (262.87,152.62) .. controls (262.87,158.91) and (254.23,164) .. (243.58,164) .. controls (232.94,164) and (224.3,158.91) .. (224.3,152.62) -- (243.58,152.62) -- cycle ; \draw    (251.8,142.32) .. controls (258.34,144.14) and (262.87,148.07) .. (262.87,152.62) .. controls (262.87,158.91) and (254.23,164) .. (243.58,164) .. controls (232.94,164) and (224.3,158.91) .. (224.3,152.62) ;  \draw [shift={(249.81,141.85)}, rotate = 19.97] [color={rgb, 255:red, 0; green, 0; blue, 0 }  ][line width=0.75]    (8.74,-2.63) .. controls (5.56,-1.12) and (2.65,-0.24) .. (0,0) .. controls (2.65,0.24) and (5.56,1.12) .. (8.74,2.63)   ;
\draw  [draw opacity=0] (403.65,152) .. controls (403.65,152) and (403.65,152) .. (403.65,152) .. controls (403.65,152) and (403.65,152) .. (403.65,152) .. controls (403.65,158.28) and (395.02,163.38) .. (384.37,163.38) .. controls (373.72,163.38) and (365.09,158.28) .. (365.09,152) .. controls (365.09,147) and (370.55,142.75) .. (378.15,141.23) -- (384.37,152) -- cycle ; \draw    (403.65,152) .. controls (403.65,152) and (403.65,152) .. (403.65,152) .. controls (403.65,158.28) and (395.02,163.38) .. (384.37,163.38) .. controls (373.72,163.38) and (365.09,158.28) .. (365.09,152) .. controls (365.09,147.43) and (369.66,143.48) .. (376.26,141.67) ; \draw [shift={(378.15,141.23)}, rotate = 158.75] [color={rgb, 255:red, 0; green, 0; blue, 0 }  ][line width=0.75]    (8.74,-2.63) .. controls (5.56,-1.12) and (2.65,-0.24) .. (0,0) .. controls (2.65,0.24) and (5.56,1.12) .. (8.74,2.63)   ; 

\draw (160,178.2) node [anchor=north west][inner sep=0.75pt]    {$\Sigma _{+}$};
\draw (434.57,108.6) node [anchor=north west][inner sep=0.75pt]    {$\Sigma _{-}$};
\draw (338.65,84.4) node [anchor=north west][inner sep=0.75pt]  [color={rgb, 255:red, 144; green, 19; blue, 254 }  ,opacity=1 ]  {$\Gamma $};
\draw (236.34,84.4) node [anchor=north west][inner sep=0.75pt]  [color={rgb, 255:red, 255; green, 0; blue, 0 }  ,opacity=1 ]  {$R$};

\end{tikzpicture}

%% file: images/addCusp.tex
\begin{tikzpicture}[x=0.65pt,y=0.65pt,yscale=-1,xscale=1]

\draw [color={rgb, 255:red, 144; green, 19; blue, 254 }  ,draw opacity=1 ][line width=1.5]    (180,135) -- (47.14,135) ;
\draw [color={rgb, 255:red, 0; green, 0; blue, 255 }  ,draw opacity=1 ][line width=1.5]    (180,45) .. controls (144,35) and (132,70.51) .. (112,70.51) .. controls (92,70.51) and (79,37) .. (45,45) ;
\draw [color={rgb, 255:red, 144; green, 19; blue, 254 }  ,draw opacity=1 ][line width=1.5]    (405,135.78) -- (272.14,135.78) ;
\draw [color={rgb, 255:red, 0; green, 0; blue, 255 }  ,draw opacity=1 ][line width=1.5]  [dash pattern={on 5.63pt off 4.5pt}]  (382,46.54) .. controls (369,48.51) and (341.01,81.75) .. (332,71.51) ;
\draw [color={rgb, 255:red, 0; green, 0; blue, 255 }  ,draw opacity=1 ][line width=1.5]    (405,45.78) .. controls (396,43.28) and (388.5,44.62) .. (382,46.54) .. controls (371,52.51) and (378,21.51) .. (360,29.51) .. controls (346,39.51) and (371,54.51) .. (347,69.51) .. controls (336.59,72.68) and (342.59,106.68) .. (336.59,105.68) .. controls (330.59,106.79) and (332.59,76.68) .. (332,71.51) .. controls (331.59,66.68) and (320,55.79) .. (309,51.51) .. controls (298,47.23) and (278.5,43.78) .. (270,45.78) ;
\draw [color={rgb, 255:red, 144; green, 19; blue, 254 }  ,draw opacity=1 ][line width=1.5]    (630,135) -- (497.14,135) ;
\draw [color={rgb, 255:red, 0; green, 0; blue, 255 }  ,draw opacity=1 ][line width=1.5]  [dash pattern={on 5.63pt off 4.5pt}]  (607,45.77) .. controls (594,47.73) and (589.01,50.75) .. (579.01,61.75) .. controls (569.01,72.75) and (563.01,71.75) .. (557,70.73) ;
\draw [color={rgb, 255:red, 0; green, 0; blue, 255 }  ,draw opacity=1 ][line width=1.5]    (630,45) .. controls (621,42.5) and (613.5,43.84) .. (607,45.77) .. controls (596,51.73) and (603,20.73) .. (585,28.73) .. controls (570.59,43.68) and (599.59,42.68) .. (572,68.73) .. controls (568.59,74.68) and (564,104.51) .. (563.57,135) .. controls (562,107.51) and (560.59,80.68) .. (557,70.73) .. controls (551,61.51) and (545,55.01) .. (534,50.73) .. controls (523,46.45) and (503.5,43) .. (495,45) ;
\draw [line width=0.75]    (247.69,90) -- (204.77,90.05) ;
\draw [shift={(249.69,90)}, rotate = 179.93] [color={rgb, 255:red, 0; green, 0; blue, 0 }  ][line width=0.75]    (10.93,-3.29) .. controls (6.95,-1.4) and (3.31,-0.3) .. (0,0) .. controls (3.31,0.3) and (6.95,1.4) .. (10.93,3.29)   ;
\draw [line width=0.75]    (471.92,90) -- (429.01,90.05) ;
\draw [shift={(473.92,90)}, rotate = 179.93] [color={rgb, 255:red, 0; green, 0; blue, 0 }  ][line width=0.75]    (10.93,-3.29) .. controls (6.95,-1.4) and (3.31,-0.3) .. (0,0) .. controls (3.31,0.3) and (6.95,1.4) .. (10.93,3.29)   ;

\draw (164.09,139.39) node [anchor=north west][inner sep=0.75pt]  [color={rgb, 255:red, 144; green, 19; blue, 254 }  ,opacity=1 ]  {$\Gamma $};
\draw (389.43,138.97) node [anchor=north west][inner sep=0.75pt]  [color={rgb, 255:red, 144; green, 19; blue, 254 }  ,opacity=1 ]  {$\Gamma $};
\draw (615,139.39) node [anchor=north west][inner sep=0.75pt]  [color={rgb, 255:red, 144; green, 19; blue, 254 }  ,opacity=1 ]  {$\Gamma $};

\end{tikzpicture}

%% file: images/stdUnknotConstruction.tex
\begin{tikzpicture}[x=0.75pt,y=0.75pt,yscale=-1,xscale=1]

\draw  [fill={rgb, 255:red, 246; green, 246; blue, 246 }  ,fill opacity=1 ][line width=0.75]  (46.65,48.13) -- (174.74,48.13) -- (174.74,176.22) -- (46.65,176.22) -- cycle ;
\draw [color={rgb, 255:red, 144; green, 19; blue, 254 }  ,draw opacity=1 ][line width=1.5]    (46.65,112.17) -- (174.74,112.17) ;
\draw  [line width=0.75]  (43,89.39) .. controls (44.96,86.29) and (46.13,83.21) .. (46.52,80.12) .. controls (46.92,83.21) and (48.09,86.29) .. (50.04,89.39) ;
\draw  [line width=0.75]  (170.96,89.39) .. controls (172.92,86.29) and (174.09,83.21) .. (174.48,80.12) .. controls (174.88,83.21) and (176.05,86.29) .. (178,89.39) ;

\draw  [color={rgb, 255:red, 144; green, 19; blue, 254 }  ,draw opacity=1 ][line width=0.75]  (100.64,45) .. controls (103.73,46.96) and (106.82,48.13) .. (109.91,48.52) .. controls (106.82,48.91) and (103.73,50.09) .. (100.64,52.04) ;
\draw  [color={rgb, 255:red, 144; green, 19; blue, 254 }  ,draw opacity=1 ][line width=0.75]  (100.64,172.96) .. controls (103.73,174.91) and (106.82,176.09) .. (109.91,176.48) .. controls (106.82,176.87) and (103.73,178.04) .. (100.64,180) ;

\draw  [color={rgb, 255:red, 144; green, 19; blue, 254 }  ,draw opacity=1 ][line width=0.75]  (108.48,45) .. controls (111.57,46.96) and (114.66,48.13) .. (117.74,48.52) .. controls (114.66,48.91) and (111.57,50.09) .. (108.48,52.04) ;
\draw  [color={rgb, 255:red, 144; green, 19; blue, 254 }  ,draw opacity=1 ][line width=0.75]  (108.48,172.96) .. controls (111.57,174.91) and (114.66,176.09) .. (117.74,176.48) .. controls (114.66,176.87) and (111.57,178.04) .. (108.48,180) ;

\draw [color={rgb, 255:red, 144; green, 19; blue, 254 }  ,draw opacity=1 ][line width=1.5]    (46.65,48.13) -- (174.74,48.13) ;
\draw [color={rgb, 255:red, 144; green, 19; blue, 254 }  ,draw opacity=1 ][line width=1.5]    (46.65,176.22) -- (174.74,176.22) ;
\draw  [draw opacity=0][line width=0.75]  (100.66,90.74) .. controls (98.17,93.35) and (94.68,94.97) .. (90.81,94.97) .. controls (83.17,94.97) and (76.94,88.65) .. (76.9,80.86) .. controls (76.85,73.07) and (83.01,66.75) .. (90.64,66.75) .. controls (94.51,66.75) and (98.02,68.37) .. (100.55,70.98) -- (90.73,80.86) -- cycle ; \draw  [color={rgb, 255:red, 0; green, 0; blue, 255 }  ,draw opacity=1 ][line width=0.75]  (100.66,90.74) .. controls (98.17,93.35) and (94.68,94.97) .. (90.81,94.97) .. controls (83.17,94.97) and (76.94,88.65) .. (76.9,80.86) .. controls (76.85,73.07) and (83.01,66.75) .. (90.64,66.75) .. controls (94.51,66.75) and (98.02,68.37) .. (100.55,70.98) ;  
\draw  [draw opacity=0][line width=0.75]  (119.85,71.02) .. controls (122.35,68.21) and (125.91,66.46) .. (129.86,66.46) .. controls (137.5,66.46) and (143.73,73) .. (143.78,81.06) .. controls (143.83,89.12) and (137.67,95.66) .. (130.03,95.66) .. controls (126.08,95.66) and (122.5,93.91) .. (119.97,91.1) -- (129.95,81.06) -- cycle ; \draw  [color={rgb, 255:red, 0; green, 0; blue, 255 }  ,draw opacity=1 ][line width=0.75]  (119.85,71.02) .. controls (122.35,68.21) and (125.91,66.46) .. (129.86,66.46) .. controls (137.5,66.46) and (143.73,73) .. (143.78,81.06) .. controls (143.83,89.12) and (137.67,95.66) .. (130.03,95.66) .. controls (126.08,95.66) and (122.5,93.91) .. (119.97,91.1) ;  
\draw [color={rgb, 255:red, 0; green, 0; blue, 255 }  ,draw opacity=1 ][line width=0.75]    (100.46,70.9) -- (120.05,91.19) ;
\draw [color={rgb, 255:red, 0; green, 0; blue, 255 }  ,draw opacity=1 ][line width=0.75]    (108.6,82.23) -- (100.58,90.82) ;
\draw [color={rgb, 255:red, 0; green, 0; blue, 255 }  ,draw opacity=1 ][line width=0.75]    (119.93,70.93) -- (112.36,78.7) ;

\draw  [fill={rgb, 255:red, 246; green, 246; blue, 246 }  ,fill opacity=1 ][line width=0.75]  (271.65,48.13) -- (399.74,48.13) -- (399.74,176.22) -- (271.65,176.22) -- cycle ;
\draw [color={rgb, 255:red, 144; green, 19; blue, 254 }  ,draw opacity=1 ][line width=1.5]    (271.65,112.17) -- (399.74,112.17) ;
\draw  [line width=0.75]  (268,89.39) .. controls (269.96,86.29) and (271.13,83.21) .. (271.52,80.12) .. controls (271.92,83.21) and (273.09,86.29) .. (275.04,89.39) ;
\draw  [line width=0.75]  (395.96,89.39) .. controls (397.92,86.29) and (399.09,83.21) .. (399.48,80.12) .. controls (399.88,83.21) and (401.05,86.29) .. (403,89.39) ;

\draw  [color={rgb, 255:red, 144; green, 19; blue, 254 }  ,draw opacity=1 ][line width=0.75]  (325.64,45) .. controls (328.73,46.96) and (331.82,48.13) .. (334.91,48.52) .. controls (331.82,48.91) and (328.73,50.09) .. (325.64,52.04) ;
\draw  [color={rgb, 255:red, 144; green, 19; blue, 254 }  ,draw opacity=1 ][line width=0.75]  (325.64,172.96) .. controls (328.73,174.91) and (331.82,176.09) .. (334.91,176.48) .. controls (331.82,176.87) and (328.73,178.04) .. (325.64,180) ;

\draw  [color={rgb, 255:red, 144; green, 19; blue, 254 }  ,draw opacity=1 ][line width=0.75]  (333.48,45) .. controls (336.57,46.96) and (339.66,48.13) .. (342.74,48.52) .. controls (339.66,48.91) and (336.57,50.09) .. (333.48,52.04) ;
\draw  [color={rgb, 255:red, 144; green, 19; blue, 254 }  ,draw opacity=1 ][line width=0.75]  (333.48,172.96) .. controls (336.57,174.91) and (339.66,176.09) .. (342.74,176.48) .. controls (339.66,176.87) and (336.57,178.04) .. (333.48,180) ;

\draw [color={rgb, 255:red, 144; green, 19; blue, 254 }  ,draw opacity=1 ][line width=1.5]    (271.65,48.13) -- (399.74,48.13) ;
\draw [color={rgb, 255:red, 144; green, 19; blue, 254 }  ,draw opacity=1 ][line width=1.5]    (271.65,176.22) -- (399.74,176.22) ;
\draw  [fill={rgb, 255:red, 246; green, 246; blue, 246 }  ,fill opacity=1 ][line width=0.75]  (496.65,48.13) -- (624.74,48.13) -- (624.74,176.22) -- (496.65,176.22) -- cycle ;
\draw [color={rgb, 255:red, 144; green, 19; blue, 254 }  ,draw opacity=1 ][line width=1.5]    (496.65,112.17) -- (624.74,112.17) ;
\draw  [line width=0.75]  (493,89.39) .. controls (494.96,86.29) and (496.13,83.21) .. (496.52,80.12) .. controls (496.92,83.21) and (498.09,86.29) .. (500.04,89.39) ;
\draw  [line width=0.75]  (620.96,89.39) .. controls (622.92,86.29) and (624.09,83.21) .. (624.48,80.12) .. controls (624.88,83.21) and (626.05,86.29) .. (628,89.39) ;

\draw  [color={rgb, 255:red, 144; green, 19; blue, 254 }  ,draw opacity=1 ][line width=0.75]  (550.64,45) .. controls (553.73,46.96) and (556.82,48.13) .. (559.91,48.52) .. controls (556.82,48.91) and (553.73,50.09) .. (550.64,52.04) ;
\draw  [color={rgb, 255:red, 144; green, 19; blue, 254 }  ,draw opacity=1 ][line width=0.75]  (550.64,172.96) .. controls (553.73,174.91) and (556.82,176.09) .. (559.91,176.48) .. controls (556.82,176.87) and (553.73,178.04) .. (550.64,180) ;

\draw  [color={rgb, 255:red, 144; green, 19; blue, 254 }  ,draw opacity=1 ][line width=0.75]  (558.48,45) .. controls (561.57,46.96) and (564.66,48.13) .. (567.74,48.52) .. controls (564.66,48.91) and (561.57,50.09) .. (558.48,52.04) ;
\draw  [color={rgb, 255:red, 144; green, 19; blue, 254 }  ,draw opacity=1 ][line width=0.75]  (558.48,172.96) .. controls (561.57,174.91) and (564.66,176.09) .. (567.74,176.48) .. controls (564.66,176.87) and (561.57,178.04) .. (558.48,180) ;

\draw [color={rgb, 255:red, 144; green, 19; blue, 254 }  ,draw opacity=1 ][line width=1.5]    (496.65,48.13) -- (624.74,48.13) ;
\draw [color={rgb, 255:red, 144; green, 19; blue, 254 }  ,draw opacity=1 ][line width=1.5]    (496.65,176.22) -- (624.74,176.22) ;
\draw    (244.69,113.04) -- (201.77,113.09) ;
\draw [shift={(246.69,113.04)}, rotate = 179.93] [color={rgb, 255:red, 0; green, 0; blue, 0 }  ][line width=0.75]    (10.93,-3.29) .. controls (6.95,-1.4) and (3.31,-0.3) .. (0,0) .. controls (3.31,0.3) and (6.95,1.4) .. (10.93,3.29)   ;
\draw    (472.92,113.04) -- (430.01,113.09) ;
\draw [shift={(474.92,113.04)}, rotate = 179.93] [color={rgb, 255:red, 0; green, 0; blue, 0 }  ][line width=0.75]    (10.93,-3.29) .. controls (6.95,-1.4) and (3.31,-0.3) .. (0,0) .. controls (3.31,0.3) and (6.95,1.4) .. (10.93,3.29)   ;
\draw [color={rgb, 255:red, 0; green, 0; blue, 255 }  ,draw opacity=1 ][line width=0.75]    (317,111.3) .. controls (314.92,93.44) and (309.8,97.14) .. (302.92,90) .. controls (296.04,82.86) and (304.08,69.44) .. (315,68.44) .. controls (325.92,67.44) and (340.62,90.57) .. (345.05,93.19) .. controls (349.48,95.81) and (355.92,99.44) .. (355.59,111.04) ;
\draw [color={rgb, 255:red, 0; green, 0; blue, 255 }  ,draw opacity=1 ][line width=0.75]    (355.65,111.57) .. controls (357.74,93.7) and (362.87,97.41) .. (369.76,90.27) .. controls (376.65,83.12) and (370,68.93) .. (357.66,68.7) .. controls (345.32,68.47) and (343.76,76.85) .. (339.32,79.47) ;
\draw [color={rgb, 255:red, 0; green, 0; blue, 255 }  ,draw opacity=1 ][line width=0.75]    (541.4,112.3) .. controls (539.31,94.44) and (534.19,98.14) .. (527.31,91) .. controls (520.43,83.86) and (528.48,70.44) .. (539.4,69.44) .. controls (550.31,68.44) and (565.01,91.57) .. (569.45,94.19) .. controls (573.88,96.81) and (580.31,100.44) .. (579.99,112.04) ;
\draw [color={rgb, 255:red, 0; green, 0; blue, 255 }  ,draw opacity=1 ][line width=0.75]    (579.99,112.04) .. controls (582.07,129.91) and (587.19,126.2) .. (594.07,133.34) .. controls (600.95,140.48) and (592.91,153.91) .. (581.99,154.91) .. controls (571.07,155.91) and (556.37,132.78) .. (551.94,130.15) .. controls (547.5,127.53) and (541.07,123.91) .. (541.4,112.3) ;
\draw [color={rgb, 255:red, 0; green, 0; blue, 255 }  ,draw opacity=1 ][line width=0.75]    (334.32,89) .. controls (332.32,91.47) and (317.67,101.7) .. (318,113.3) ;

\draw (160.63,114.88) node [anchor=north west][inner sep=0.75pt]  [color={rgb, 255:red, 144; green, 19; blue, 254 }  ,opacity=1 ]  {$\Gamma $};
\draw (48.89,50.48) node [anchor=north west][inner sep=0.75pt]    {$\Sigma _{+}$};
\draw (49.31,157.84) node [anchor=north west][inner sep=0.75pt]    {$\Sigma _{-}$};
\draw (385.63,114.88) node [anchor=north west][inner sep=0.75pt]  [color={rgb, 255:red, 144; green, 19; blue, 254 }  ,opacity=1 ]  {$\Gamma $};
\draw (273.89,50.48) node [anchor=north west][inner sep=0.75pt]    {$\Sigma _{+}$};
\draw (274.31,157.84) node [anchor=north west][inner sep=0.75pt]    {$\Sigma _{-}$};
\draw (611.53,114.88) node [anchor=north west][inner sep=0.75pt]  [color={rgb, 255:red, 144; green, 19; blue, 254 }  ,opacity=1 ]  {$\Gamma $};
\draw (498.89,50.48) node [anchor=north west][inner sep=0.75pt]    {$\Sigma _{+}$};
\draw (499.31,157.84) node [anchor=north west][inner sep=0.75pt]    {$\Sigma _{-}$};

\end{tikzpicture}

%% file: images/unknotConstruction.tex
\begin{tikzpicture}[x=0.75pt,y=0.75pt,yscale=-1,xscale=1]

\draw  [fill={rgb, 255:red, 246; green, 246; blue, 246 }  ,fill opacity=1 ][line width=0.75]  (47.65,47.13) -- (175.74,47.13) -- (175.74,175.22) -- (47.65,175.22) -- cycle ;
\draw [color={rgb, 255:red, 144; green, 19; blue, 254 }  ,draw opacity=1 ][line width=1.5]    (47.65,111.17) -- (175.74,111.17) ;
\draw  [line width=0.75]  (44,88.39) .. controls (45.96,85.29) and (47.13,82.21) .. (47.52,79.12) .. controls (47.92,82.21) and (49.09,85.29) .. (51.04,88.39) ;
\draw  [line width=0.75]  (171.96,88.39) .. controls (173.92,85.29) and (175.09,82.21) .. (175.48,79.12) .. controls (175.88,82.21) and (177.05,85.29) .. (179,88.39) ;

\draw  [color={rgb, 255:red, 144; green, 19; blue, 254 }  ,draw opacity=1 ][line width=0.75]  (101.64,44) .. controls (104.73,45.96) and (107.82,47.13) .. (110.91,47.52) .. controls (107.82,47.91) and (104.73,49.09) .. (101.64,51.04) ;
\draw  [color={rgb, 255:red, 144; green, 19; blue, 254 }  ,draw opacity=1 ][line width=0.75]  (101.64,171.96) .. controls (104.73,173.91) and (107.82,175.09) .. (110.91,175.48) .. controls (107.82,175.87) and (104.73,177.04) .. (101.64,179) ;

\draw  [color={rgb, 255:red, 144; green, 19; blue, 254 }  ,draw opacity=1 ][line width=0.75]  (109.48,44) .. controls (112.57,45.96) and (115.66,47.13) .. (118.74,47.52) .. controls (115.66,47.91) and (112.57,49.09) .. (109.48,51.04) ;
\draw  [color={rgb, 255:red, 144; green, 19; blue, 254 }  ,draw opacity=1 ][line width=0.75]  (109.48,171.96) .. controls (112.57,173.91) and (115.66,175.09) .. (118.74,175.48) .. controls (115.66,175.87) and (112.57,177.04) .. (109.48,179) ;

\draw [color={rgb, 255:red, 144; green, 19; blue, 254 }  ,draw opacity=1 ][line width=1.5]    (47.65,47.13) -- (175.74,47.13) ;
\draw [color={rgb, 255:red, 144; green, 19; blue, 254 }  ,draw opacity=1 ][line width=1.5]    (47.65,175.22) -- (175.74,175.22) ;
\draw  [draw opacity=0][line width=0.75]  (103.84,73.99) .. controls (99.98,72.36) and (97.43,69.53) .. (97.43,66.31) .. controls (97.43,61.23) and (103.75,57.09) .. (111.54,57.05) .. controls (119.33,57.02) and (125.65,61.11) .. (125.65,66.19) .. controls (125.65,69.42) and (123.1,72.27) .. (119.25,73.92) -- (111.54,66.25) -- cycle ; \draw  [color={rgb, 255:red, 0; green, 0; blue, 255 }  ,draw opacity=1 ][line width=0.75]  (103.84,73.99) .. controls (99.98,72.36) and (97.43,69.53) .. (97.43,66.31) .. controls (97.43,61.23) and (103.75,57.09) .. (111.54,57.05) .. controls (119.33,57.02) and (125.65,61.11) .. (125.65,66.19) .. controls (125.65,69.42) and (123.1,72.27) .. (119.25,73.92) ;  
\draw  [draw opacity=0][line width=0.75]  (119.12,84.52) .. controls (123.22,86.13) and (125.94,89) .. (125.94,92.27) .. controls (125.94,97.35) and (119.4,101.5) .. (111.34,101.53) .. controls (103.27,101.56) and (96.74,97.47) .. (96.74,92.39) .. controls (96.74,89.11) and (99.46,86.22) .. (103.56,84.58) -- (111.34,92.33) -- cycle ; \draw  [color={rgb, 255:red, 0; green, 0; blue, 255 }  ,draw opacity=1 ][line width=0.75]  (119.12,84.52) .. controls (123.22,86.13) and (125.94,89) .. (125.94,92.27) .. controls (125.94,97.35) and (119.4,101.5) .. (111.34,101.53) .. controls (103.27,101.56) and (96.74,97.47) .. (96.74,92.39) .. controls (96.74,89.11) and (99.46,86.22) .. (103.56,84.58) ;  
\draw [color={rgb, 255:red, 0; green, 0; blue, 255 }  ,draw opacity=1 ][line width=0.75]    (121.5,72.72) -- (101.21,85.75) ;
\draw [color={rgb, 255:red, 0; green, 0; blue, 255 }  ,draw opacity=1 ][line width=0.75]    (108,76.86) -- (101.58,72.8) ;
\draw [color={rgb, 255:red, 0; green, 0; blue, 255 }  ,draw opacity=1 ][line width=0.75]    (121.47,85.67) -- (113.7,80.63) ;
\draw  [fill={rgb, 255:red, 246; green, 246; blue, 246 }  ,fill opacity=1 ][line width=0.75]  (272.65,47.13) -- (400.74,47.13) -- (400.74,175.22) -- (272.65,175.22) -- cycle ;
\draw [color={rgb, 255:red, 144; green, 19; blue, 254 }  ,draw opacity=1 ][line width=1.5]    (272.65,111.17) -- (400.74,111.17) ;
\draw  [line width=0.75]  (269,88.39) .. controls (270.96,85.29) and (272.13,82.21) .. (272.52,79.12) .. controls (272.92,82.21) and (274.09,85.29) .. (276.04,88.39) ;
\draw  [line width=0.75]  (396.96,88.39) .. controls (398.92,85.29) and (400.09,82.21) .. (400.48,79.12) .. controls (400.88,82.21) and (402.05,85.29) .. (404,88.39) ;

\draw  [color={rgb, 255:red, 144; green, 19; blue, 254 }  ,draw opacity=1 ][line width=0.75]  (326.64,44) .. controls (329.73,45.96) and (332.82,47.13) .. (335.91,47.52) .. controls (332.82,47.91) and (329.73,49.09) .. (326.64,51.04) ;
\draw  [color={rgb, 255:red, 144; green, 19; blue, 254 }  ,draw opacity=1 ][line width=0.75]  (326.64,171.96) .. controls (329.73,173.91) and (332.82,175.09) .. (335.91,175.48) .. controls (332.82,175.87) and (329.73,177.04) .. (326.64,179) ;

\draw  [color={rgb, 255:red, 144; green, 19; blue, 254 }  ,draw opacity=1 ][line width=0.75]  (334.48,44) .. controls (337.57,45.96) and (340.66,47.13) .. (343.74,47.52) .. controls (340.66,47.91) and (337.57,49.09) .. (334.48,51.04) ;
\draw  [color={rgb, 255:red, 144; green, 19; blue, 254 }  ,draw opacity=1 ][line width=0.75]  (334.48,171.96) .. controls (337.57,173.91) and (340.66,175.09) .. (343.74,175.48) .. controls (340.66,175.87) and (337.57,177.04) .. (334.48,179) ;

\draw [color={rgb, 255:red, 144; green, 19; blue, 254 }  ,draw opacity=1 ][line width=1.5]    (272.65,47.13) -- (400.74,47.13) ;
\draw [color={rgb, 255:red, 144; green, 19; blue, 254 }  ,draw opacity=1 ][line width=1.5]    (272.65,175.22) -- (400.74,175.22) ;
\draw  [fill={rgb, 255:red, 246; green, 246; blue, 246 }  ,fill opacity=1 ][line width=0.75]  (497.65,47.13) -- (625.74,47.13) -- (625.74,175.22) -- (497.65,175.22) -- cycle ;
\draw [color={rgb, 255:red, 144; green, 19; blue, 254 }  ,draw opacity=1 ][line width=1.5]    (497.65,111.17) -- (625.74,111.17) ;
\draw  [line width=0.75]  (494,88.39) .. controls (495.96,85.29) and (497.13,82.21) .. (497.52,79.12) .. controls (497.92,82.21) and (499.09,85.29) .. (501.04,88.39) ;
\draw  [line width=0.75]  (621.96,88.39) .. controls (623.92,85.29) and (625.09,82.21) .. (625.48,79.12) .. controls (625.88,82.21) and (627.05,85.29) .. (629,88.39) ;

\draw  [color={rgb, 255:red, 144; green, 19; blue, 254 }  ,draw opacity=1 ][line width=0.75]  (551.64,44) .. controls (554.73,45.96) and (557.82,47.13) .. (560.91,47.52) .. controls (557.82,47.91) and (554.73,49.09) .. (551.64,51.04) ;
\draw  [color={rgb, 255:red, 144; green, 19; blue, 254 }  ,draw opacity=1 ][line width=0.75]  (551.64,171.96) .. controls (554.73,173.91) and (557.82,175.09) .. (560.91,175.48) .. controls (557.82,175.87) and (554.73,177.04) .. (551.64,179) ;

\draw  [color={rgb, 255:red, 144; green, 19; blue, 254 }  ,draw opacity=1 ][line width=0.75]  (559.48,44) .. controls (562.57,45.96) and (565.66,47.13) .. (568.74,47.52) .. controls (565.66,47.91) and (562.57,49.09) .. (559.48,51.04) ;
\draw  [color={rgb, 255:red, 144; green, 19; blue, 254 }  ,draw opacity=1 ][line width=0.75]  (559.48,171.96) .. controls (562.57,173.91) and (565.66,175.09) .. (568.74,175.48) .. controls (565.66,175.87) and (562.57,177.04) .. (559.48,179) ;

\draw [color={rgb, 255:red, 144; green, 19; blue, 254 }  ,draw opacity=1 ][line width=1.5]    (497.65,47.13) -- (625.74,47.13) ;
\draw [color={rgb, 255:red, 144; green, 19; blue, 254 }  ,draw opacity=1 ][line width=1.5]    (497.65,175.22) -- (625.74,175.22) ;
\draw    (245.69,112.04) -- (202.77,112.09) ;
\draw [shift={(247.69,112.04)}, rotate = 179.93] [color={rgb, 255:red, 0; green, 0; blue, 0 }  ][line width=0.75]    (10.93,-3.29) .. controls (6.95,-1.4) and (3.31,-0.3) .. (0,0) .. controls (3.31,0.3) and (6.95,1.4) .. (10.93,3.29)   ;
\draw    (473.92,112.04) -- (431.01,112.09) ;
\draw [shift={(475.92,112.04)}, rotate = 179.93] [color={rgb, 255:red, 0; green, 0; blue, 0 }  ][line width=0.75]    (10.93,-3.29) .. controls (6.95,-1.4) and (3.31,-0.3) .. (0,0) .. controls (3.31,0.3) and (6.95,1.4) .. (10.93,3.29)   ;
\draw [color={rgb, 255:red, 0; green, 0; blue, 255 }  ,draw opacity=1 ][line width=0.75]    (336.69,47.13) .. controls (337.97,53.62) and (338.97,54.62) .. (346.97,58.62) .. controls (354.97,62.62) and (359.97,60.62) .. (359,68.62) .. controls (358.03,76.62) and (342.97,75.12) .. (335.97,76.62) .. controls (328.97,78.12) and (313.03,78.38) .. (314,89) .. controls (314.97,99.62) and (322.97,92.23) .. (330.97,97.62) .. controls (338.97,103) and (335.97,107.62) .. (336.69,111.17) ;
\draw [color={rgb, 255:red, 0; green, 0; blue, 255 }  ,draw opacity=1 ][line width=0.75]    (336.48,47.13) .. controls (335.22,53.41) and (334.23,54.38) .. (326.3,58.26) .. controls (318.38,62.13) and (313.42,60.2) .. (314.38,67.95) .. controls (315.34,75.7) and (328.5,69.74) .. (333.2,73.7) ;
\draw [color={rgb, 255:red, 0; green, 0; blue, 255 }  ,draw opacity=1 ][line width=0.75]    (341.2,78.7) .. controls (345.5,82.74) and (359.92,77.41) .. (358.96,87.69) .. controls (358,97.98) and (350.07,90.82) .. (342.15,96.04) .. controls (334.23,101.26) and (337.2,105.73) .. (336.48,109.17) ;
\draw [color={rgb, 255:red, 0; green, 0; blue, 255 }  ,draw opacity=1 ][line width=0.75]    (561.69,47.13) .. controls (562.97,53.62) and (563.97,54.62) .. (571.97,58.62) .. controls (579.97,62.62) and (584.97,60.62) .. (584,68.62) .. controls (583.03,76.62) and (567.97,75.12) .. (560.97,76.62) .. controls (553.97,78.12) and (538.03,78.38) .. (539,89) .. controls (539.97,99.62) and (547.97,92.23) .. (555.97,97.62) .. controls (563.97,103) and (560.97,107.62) .. (561.69,111.17) ;
\draw [color={rgb, 255:red, 0; green, 0; blue, 255 }  ,draw opacity=1 ][line width=0.75]    (561.69,175.22) .. controls (560.43,168.73) and (559.43,167.73) .. (551.48,163.73) .. controls (543.53,159.73) and (538.56,161.73) .. (539.53,153.73) .. controls (540.49,145.73) and (555.46,147.23) .. (562.41,145.73) .. controls (569.37,144.23) and (585.21,143.96) .. (584.25,133.35) .. controls (583.28,122.73) and (575.33,130.11) .. (567.38,124.73) .. controls (559.43,119.35) and (562.41,114.73) .. (561.69,111.17) ;

\draw (161.63,113.88) node [anchor=north west][inner sep=0.75pt]  [color={rgb, 255:red, 144; green, 19; blue, 254 }  ,opacity=1 ]  {$\Gamma $};
\draw (49.89,49.48) node [anchor=north west][inner sep=0.75pt]    {$\Sigma _{+}$};
\draw (50.31,156.84) node [anchor=north west][inner sep=0.75pt]    {$\Sigma _{-}$};
\draw (386.63,113.88) node [anchor=north west][inner sep=0.75pt]  [color={rgb, 255:red, 144; green, 19; blue, 254 }  ,opacity=1 ]  {$\Gamma $};
\draw (274.89,49.48) node [anchor=north west][inner sep=0.75pt]    {$\Sigma _{+}$};
\draw (275.31,156.84) node [anchor=north west][inner sep=0.75pt]    {$\Sigma _{-}$};
\draw (612.53,113.88) node [anchor=north west][inner sep=0.75pt]  [color={rgb, 255:red, 144; green, 19; blue, 254 }  ,opacity=1 ]  {$\Gamma $};
\draw (499.89,49.48) node [anchor=north west][inner sep=0.75pt]    {$\Sigma _{+}$};
\draw (500.31,156.84) node [anchor=north west][inner sep=0.75pt]    {$\Sigma _{-}$};

\end{tikzpicture}

%% file: images/Reidemeister_moves.tex
\begin{tikzpicture}[x=0.75pt,y=0.75pt,yscale=-1,xscale=1]

\draw [color={rgb, 255:red, 0; green, 0; blue, 255 }  ,draw opacity=1 ][line width=1.5]    (35.79,253.57) -- (123.41,253.57) ;
\draw [color={rgb, 255:red, 255; green, 255; blue, 255 }  ,draw opacity=1 ][line width=3]    (104.49,253.6) -- (97.73,253.6) ;
\draw [color={rgb, 255:red, 255; green, 255; blue, 255 }  ,draw opacity=1 ][line width=3]    (61.05,253.6) -- (54.3,253.6) ;
\draw [color={rgb, 255:red, 0; green, 0; blue, 255 }  ,draw opacity=1 ][line width=1.5]    (47.18,269.53) -- (91.49,203.4) ;
\draw [color={rgb, 255:red, 255; green, 255; blue, 255 }  ,draw opacity=1 ][line width=3]    (82.33,217.87) -- (77.5,224.43) ;
\draw [color={rgb, 255:red, 0; green, 0; blue, 255 }  ,draw opacity=1 ][line width=1.5]    (68.7,203.4) -- (111.64,269.9) ;

\draw  [color={rgb, 255:red, 0; green, 0; blue, 0 }  ,draw opacity=1 ][dash pattern={on 0.84pt off 2.51pt}] (31.38,244.56) .. controls (31.38,218.54) and (52.48,197.44) .. (78.51,197.44) .. controls (104.53,197.44) and (125.63,218.54) .. (125.63,244.56) .. controls (125.63,270.59) and (104.53,291.69) .. (78.51,291.69) .. controls (52.48,291.69) and (31.38,270.59) .. (31.38,244.56) -- cycle ;

\draw [color={rgb, 255:red, 0; green, 0; blue, 255 }  ,draw opacity=1 ][line width=1.5]    (287.4,230.73) -- (199.78,230.73) ;
\draw [color={rgb, 255:red, 255; green, 255; blue, 255 }  ,draw opacity=1 ][line width=3]    (225.04,231.09) -- (218.29,231.09) ;
\draw [color={rgb, 255:red, 255; green, 255; blue, 255 }  ,draw opacity=1 ][line width=3]    (268.48,231.09) -- (261.72,231.09) ;
\draw [color={rgb, 255:red, 0; green, 0; blue, 255 }  ,draw opacity=1 ][line width=1.5]    (276.01,214.77) -- (231.7,280.9) ;
\draw [color={rgb, 255:red, 255; green, 255; blue, 255 }  ,draw opacity=1 ][line width=3]    (245.84,260.12) -- (241.45,266.72) ;
\draw [color={rgb, 255:red, 0; green, 0; blue, 255 }  ,draw opacity=1 ][line width=1.5]    (254.48,280.9) -- (211.54,214.4) ;

\draw  [color={rgb, 255:red, 0; green, 0; blue, 0 }  ,draw opacity=1 ][dash pattern={on 0.84pt off 2.51pt}] (196.38,244.56) .. controls (196.38,218.54) and (217.48,197.44) .. (243.51,197.44) .. controls (269.53,197.44) and (290.63,218.54) .. (290.63,244.56) .. controls (290.63,270.59) and (269.53,291.69) .. (243.51,291.69) .. controls (217.48,291.69) and (196.38,270.59) .. (196.38,244.56) -- cycle ;

\draw    (186.04,256.15) -- (142.81,256.38) ;
\draw [shift={(140.81,256.39)}, rotate = 359.69] [color={rgb, 255:red, 0; green, 0; blue, 0 }  ][line width=0.75]    (10.93,-3.29) .. controls (6.95,-1.4) and (3.31,-0.3) .. (0,0) .. controls (3.31,0.3) and (6.95,1.4) .. (10.93,3.29)   ;
\draw [shift={(188.04,256.14)}, rotate = 179.69] [color={rgb, 255:red, 0; green, 0; blue, 0 }  ][line width=0.75]    (10.93,-3.29) .. controls (6.95,-1.4) and (3.31,-0.3) .. (0,0) .. controls (3.31,0.3) and (6.95,1.4) .. (10.93,3.29)   ;

\draw [color={rgb, 255:red, 0; green, 0; blue, 255 }  ,draw opacity=1 ][line width=1.5]    (450.51,253.57) -- (361.41,253.57) ;
\draw [color={rgb, 255:red, 255; green, 255; blue, 255 }  ,draw opacity=1 ][line width=3]    (380.65,253.6) -- (387.52,253.6) ;
\draw [color={rgb, 255:red, 255; green, 255; blue, 255 }  ,draw opacity=1 ][line width=3]    (424.81,253.6) -- (431.68,253.6) ;
\draw [color={rgb, 255:red, 0; green, 0; blue, 255 }  ,draw opacity=1 ][line width=1.5]    (438.92,269.53) -- (393.87,203.4) ;
\draw [color={rgb, 255:red, 255; green, 255; blue, 255 }  ,draw opacity=1 ][line width=3]    (403.18,217.87) -- (408.09,224.43) ;
\draw [color={rgb, 255:red, 0; green, 0; blue, 255 }  ,draw opacity=1 ][line width=1.5]    (417.03,203.4) -- (373.38,269.9) ;

\draw  [color={rgb, 255:red, 0; green, 0; blue, 0 }  ,draw opacity=1 ][dash pattern={on 0.84pt off 2.51pt}] (358.38,244.56) .. controls (358.38,218.54) and (379.48,197.44) .. (405.51,197.44) .. controls (431.53,197.44) and (452.63,218.54) .. (452.63,244.56) .. controls (452.63,270.59) and (431.53,291.69) .. (405.51,291.69) .. controls (379.48,291.69) and (358.38,270.59) .. (358.38,244.56) -- cycle ;

\draw [color={rgb, 255:red, 0; green, 0; blue, 255 }  ,draw opacity=1 ][line width=1.5]    (532.51,231.73) -- (620.51,231.73) ;
\draw [color={rgb, 255:red, 255; green, 255; blue, 255 }  ,draw opacity=1 ][line width=3]    (595.13,232.09) -- (601.91,232.09) ;
\draw [color={rgb, 255:red, 255; green, 255; blue, 255 }  ,draw opacity=1 ][line width=3]    (551.51,232.09) -- (558.29,232.09) ;
\draw [color={rgb, 255:red, 0; green, 0; blue, 255 }  ,draw opacity=1 ][line width=1.5]    (543.95,215.77) -- (588.45,281.9) ;
\draw [color={rgb, 255:red, 255; green, 255; blue, 255 }  ,draw opacity=1 ][line width=3]    (574.25,261.12) -- (578.65,267.72) ;
\draw [color={rgb, 255:red, 0; green, 0; blue, 255 }  ,draw opacity=1 ][line width=1.5]    (565.57,281.9) -- (608.69,215.4) ;

\draw  [color={rgb, 255:red, 0; green, 0; blue, 0 }  ,draw opacity=1 ][dash pattern={on 0.84pt off 2.51pt}] (529.38,244.56) .. controls (529.38,218.54) and (550.48,197.44) .. (576.51,197.44) .. controls (602.53,197.44) and (623.63,218.54) .. (623.63,244.56) .. controls (623.63,270.59) and (602.53,291.69) .. (576.51,291.69) .. controls (550.48,291.69) and (529.38,270.59) .. (529.38,244.56) -- cycle ;

\draw    (514,254.15) -- (470.77,254.38) ;
\draw [shift={(468.77,254.39)}, rotate = 359.69] [color={rgb, 255:red, 0; green, 0; blue, 0 }  ][line width=0.75]    (10.93,-3.29) .. controls (6.95,-1.4) and (3.31,-0.3) .. (0,0) .. controls (3.31,0.3) and (6.95,1.4) .. (10.93,3.29)   ;
\draw [shift={(516,254.14)}, rotate = 179.69] [color={rgb, 255:red, 0; green, 0; blue, 0 }  ][line width=0.75]    (10.93,-3.29) .. controls (6.95,-1.4) and (3.31,-0.3) .. (0,0) .. controls (3.31,0.3) and (6.95,1.4) .. (10.93,3.29)   ;

\draw [color={rgb, 255:red, 0; green, 0; blue, 255 }  ,draw opacity=1 ][line width=1.5]    (448.97,105.4) .. controls (428.82,67) and (380.62,67.82) .. (362.22,106.22) ;
\draw [color={rgb, 255:red, 255; green, 255; blue, 255 }  ,draw opacity=1 ][line width=3]    (435.78,88.89) -- (442.25,94.2) ;
\draw [color={rgb, 255:red, 255; green, 255; blue, 255 }  ,draw opacity=1 ][line width=3]    (374.98,88.89) -- (369.86,94.2) ;
\draw [color={rgb, 255:red, 0; green, 0; blue, 255 }  ,draw opacity=1 ][line width=1.5]    (362.22,77.83) .. controls (382.38,116.23) and (430.57,115.41) .. (448.97,77.01) ;

\draw  [color={rgb, 255:red, 0; green, 0; blue, 0 }  ,draw opacity=1 ][dash pattern={on 0.84pt off 2.51pt}] (358.38,96.31) .. controls (358.38,70.28) and (379.48,49.19) .. (405.51,49.19) .. controls (431.53,49.19) and (452.63,70.28) .. (452.63,96.31) .. controls (452.63,122.34) and (431.53,143.44) .. (405.51,143.44) .. controls (379.48,143.44) and (358.38,122.34) .. (358.38,96.31) -- cycle ;

\draw [color={rgb, 255:red, 0; green, 0; blue, 255 }  ,draw opacity=1 ][line width=1.5]    (540.51,67.38) .. controls (557.11,97.96) and (596.8,97.31) .. (611.96,66.73) ;
\draw [color={rgb, 255:red, 0; green, 0; blue, 255 }  ,draw opacity=1 ][line width=1.5]    (611.96,122.2) .. controls (595.36,91.62) and (555.66,92.27) .. (540.51,122.85) ;

\draw  [color={rgb, 255:red, 0; green, 0; blue, 0 }  ,draw opacity=1 ][dash pattern={on 0.84pt off 2.51pt}] (529.38,96.31) .. controls (529.38,70.28) and (550.48,49.19) .. (576.51,49.19) .. controls (602.53,49.19) and (623.63,70.28) .. (623.63,96.31) .. controls (623.63,122.34) and (602.53,143.44) .. (576.51,143.44) .. controls (550.48,143.44) and (529.38,122.34) .. (529.38,96.31) -- cycle ;

\draw    (514,105.56) -- (470.77,105.8) ;
\draw [shift={(468.77,105.81)}, rotate = 359.69] [color={rgb, 255:red, 0; green, 0; blue, 0 }  ][line width=0.75]    (10.93,-3.29) .. controls (6.95,-1.4) and (3.31,-0.3) .. (0,0) .. controls (3.31,0.3) and (6.95,1.4) .. (10.93,3.29)   ;
\draw [shift={(516,105.55)}, rotate = 179.69] [color={rgb, 255:red, 0; green, 0; blue, 0 }  ][line width=0.75]    (10.93,-3.29) .. controls (6.95,-1.4) and (3.31,-0.3) .. (0,0) .. controls (3.31,0.3) and (6.95,1.4) .. (10.93,3.29)   ;

\draw [color={rgb, 255:red, 144; green, 19; blue, 254 }  ,draw opacity=1 ][line width=1.5]    (123.32,99.31) -- (33.69,99.31) ;
\draw [color={rgb, 255:red, 0; green, 0; blue, 255 }  ,draw opacity=1 ][line width=1.5]    (58.08,61.7) .. controls (80.31,75.86) and (96.91,79.13) .. (96.91,99.56) .. controls (96.91,119.98) and (85.57,124.88) .. (79.43,124.88) .. controls (73.3,124.88) and (61.96,120.73) .. (61.91,99.56) .. controls (61.86,78.38) and (80.31,75.86) .. (101.89,60.89) ;
\draw [color={rgb, 255:red, 255; green, 255; blue, 255 }  ,draw opacity=1 ][line width=3]    (83.74,71.62) -- (75.05,77.25) ;
\draw [color={rgb, 255:red, 74; green, 74; blue, 255 }  ,draw opacity=1 ][line width=1.5]    (82.02,75.56) -- (77.19,72.75) ;
\draw  [color={rgb, 255:red, 0; green, 0; blue, 0 }  ,draw opacity=1 ][dash pattern={on 0.84pt off 2.51pt}] (31.38,96.31) .. controls (31.38,70.28) and (52.48,49.19) .. (78.51,49.19) .. controls (104.53,49.19) and (125.63,70.28) .. (125.63,96.31) .. controls (125.63,122.34) and (104.53,143.44) .. (78.51,143.44) .. controls (52.48,143.44) and (31.38,122.34) .. (31.38,96.31) -- cycle ;

\draw [color={rgb, 255:red, 144; green, 19; blue, 254 }  ,draw opacity=1 ][line width=1.5]    (288.78,99.03) -- (198.16,99.03) ;
\draw [color={rgb, 255:red, 0; green, 0; blue, 255 }  ,draw opacity=1 ][line width=1.5]    (202.51,74.07) .. controls (221.69,91.16) and (267.57,90.8) .. (285.09,73.71) ;
\draw  [color={rgb, 255:red, 0; green, 0; blue, 0 }  ,draw opacity=1 ][dash pattern={on 0.84pt off 2.51pt}] (196.38,96.31) .. controls (196.38,70.28) and (217.48,49.19) .. (243.51,49.19) .. controls (269.53,49.19) and (290.63,70.28) .. (290.63,96.31) .. controls (290.63,122.34) and (269.53,143.44) .. (243.51,143.44) .. controls (217.48,143.44) and (196.38,122.34) .. (196.38,96.31) -- cycle ;

\draw    (186.04,105.76) -- (142.81,105.99) ;
\draw [shift={(140.81,106)}, rotate = 359.69] [color={rgb, 255:red, 0; green, 0; blue, 0 }  ][line width=0.75]    (10.93,-3.29) .. controls (6.95,-1.4) and (3.31,-0.3) .. (0,0) .. controls (3.31,0.3) and (6.95,1.4) .. (10.93,3.29)   ;
\draw [shift={(188.04,105.75)}, rotate = 179.69] [color={rgb, 255:red, 0; green, 0; blue, 0 }  ][line width=0.75]    (10.93,-3.29) .. controls (6.95,-1.4) and (3.31,-0.3) .. (0,0) .. controls (3.31,0.3) and (6.95,1.4) .. (10.93,3.29)   ;

\draw [color={rgb, 255:red, 144; green, 19; blue, 254 }  ,draw opacity=1 ][line width=1.5]    (362.08,397.83) -- (452.04,397.83) ;
\draw [color={rgb, 255:red, 0; green, 0; blue, 255 }  ,draw opacity=1 ][line width=1.5]    (388.42,356.5) .. controls (388.54,370.05) and (424.25,384.55) .. (424.25,397.82) .. controls (424.25,411.09) and (423.76,413.19) .. (437.78,422.67) ;
\draw [color={rgb, 255:red, 255; green, 255; blue, 255 }  ,draw opacity=1 ][line width=3]    (408.09,379.09) -- (401.61,373.46) ;
\draw [color={rgb, 255:red, 0; green, 0; blue, 255 }  ,draw opacity=1 ][line width=1.5]    (422.06,355.87) .. controls (421.94,369.42) and (386.98,383.91) .. (386.98,397.19) .. controls (386.98,410.46) and (390.26,413.19) .. (376.53,422.67) ;
\draw  [color={rgb, 255:red, 0; green, 0; blue, 0 }  ,draw opacity=1 ][dash pattern={on 0.84pt off 2.51pt}] (358.38,393.7) .. controls (358.38,368.19) and (379.72,347.51) .. (406.05,347.51) .. controls (432.38,347.51) and (453.72,368.19) .. (453.72,393.7) .. controls (453.72,419.21) and (432.38,439.89) .. (406.05,439.89) .. controls (379.72,439.89) and (358.38,419.21) .. (358.38,393.7) -- cycle ;

\draw [color={rgb, 255:red, 144; green, 19; blue, 254 }  ,draw opacity=1 ][line width=1.5]    (533.18,387.63) -- (621.35,387.63) ;
\draw [color={rgb, 255:red, 0; green, 0; blue, 255 }  ,draw opacity=1 ][line width=1.5]    (594.2,431.43) .. controls (594.08,417.88) and (558.37,403.38) .. (558.37,390.11) .. controls (558.37,376.84) and (558.86,374.73) .. (544.84,365.25) ;
\draw [color={rgb, 255:red, 255; green, 255; blue, 255 }  ,draw opacity=1 ][line width=3]    (580.91,414.79) -- (574.43,409.15) ;
\draw [color={rgb, 255:red, 0; green, 0; blue, 255 }  ,draw opacity=1 ][line width=1.5]    (560.56,432.06) .. controls (560.69,418.51) and (595.64,404.01) .. (595.64,390.74) .. controls (595.64,377.47) and (592.36,374.73) .. (606.09,365.25) ;
\draw  [color={rgb, 255:red, 0; green, 0; blue, 0 }  ,draw opacity=1 ][dash pattern={on 0.84pt off 2.51pt}] (529.38,393.7) .. controls (529.38,368.19) and (550.72,347.51) .. (577.05,347.51) .. controls (603.38,347.51) and (624.72,368.19) .. (624.72,393.7) .. controls (624.72,419.21) and (603.38,439.89) .. (577.05,439.89) .. controls (550.72,439.89) and (529.38,419.21) .. (529.38,393.7) -- cycle ;

\draw [color={rgb, 255:red, 144; green, 19; blue, 254 }  ,draw opacity=1 ][line width=1.5]    (120.99,397.26) -- (32.95,397.26) ;
\draw [color={rgb, 255:red, 0; green, 0; blue, 255 }  ,draw opacity=1 ][line width=1.5]    (96.82,353.93) .. controls (96.69,368.14) and (58.33,383.33) .. (58.33,397.25) .. controls (58.33,411.16) and (58.85,413.36) .. (43.79,423.3) ;
\draw [color={rgb, 255:red, 255; green, 255; blue, 255 }  ,draw opacity=1 ][line width=3]    (75.68,377.62) -- (82.64,371.71) ;
\draw [color={rgb, 255:red, 0; green, 0; blue, 255 }  ,draw opacity=1 ][line width=1.5]    (60.68,353.27) .. controls (60.81,367.47) and (98.37,382.67) .. (98.37,396.58) .. controls (98.37,410.49) and (94.85,413.36) .. (109.59,423.3) ;
\draw  [color={rgb, 255:red, 0; green, 0; blue, 0 }  ,draw opacity=1 ][dash pattern={on 0.84pt off 2.51pt}] (31.38,393.7) .. controls (31.38,367.64) and (52.18,346.51) .. (77.85,346.51) .. controls (103.51,346.51) and (124.31,367.64) .. (124.31,393.7) .. controls (124.31,419.76) and (103.51,440.89) .. (77.85,440.89) .. controls (52.18,440.89) and (31.38,419.76) .. (31.38,393.7) -- cycle ;

\draw [color={rgb, 255:red, 144; green, 19; blue, 254 }  ,draw opacity=1 ][line width=1.5]    (284.71,387.2) -- (200.84,387.2) ;
\draw [color={rgb, 255:red, 0; green, 0; blue, 255 }  ,draw opacity=1 ][line width=1.5]    (224.25,433.12) .. controls (224.38,418.91) and (261.58,403.71) .. (261.58,389.8) .. controls (261.58,375.89) and (261.07,373.69) .. (275.68,363.75) ;
\draw [color={rgb, 255:red, 255; green, 255; blue, 255 }  ,draw opacity=1 ][line width=3]    (238.11,415.67) -- (244.86,409.77) ;
\draw [color={rgb, 255:red, 0; green, 0; blue, 255 }  ,draw opacity=1 ][line width=1.5]    (259.3,433.78) .. controls (259.17,419.57) and (222.76,404.38) .. (222.76,390.47) .. controls (222.76,376.55) and (226.17,373.69) .. (211.87,363.75) ;
\draw  [color={rgb, 255:red, 0; green, 0; blue, 0 }  ,draw opacity=1 ][dash pattern={on 0.84pt off 2.51pt}] (196.38,393.7) .. controls (196.38,367.64) and (217.18,346.51) .. (242.85,346.51) .. controls (268.51,346.51) and (289.31,367.64) .. (289.31,393.7) .. controls (289.31,419.76) and (268.51,440.89) .. (242.85,440.89) .. controls (217.18,440.89) and (196.38,419.76) .. (196.38,393.7) -- cycle ;

\draw    (186.04,405.29) -- (142.81,405.52) ;
\draw [shift={(140.81,405.53)}, rotate = 359.69] [color={rgb, 255:red, 0; green, 0; blue, 0 }  ][line width=0.75]    (10.93,-3.29) .. controls (6.95,-1.4) and (3.31,-0.3) .. (0,0) .. controls (3.31,0.3) and (6.95,1.4) .. (10.93,3.29)   ;
\draw [shift={(188.04,405.28)}, rotate = 179.69] [color={rgb, 255:red, 0; green, 0; blue, 0 }  ][line width=0.75]    (10.93,-3.29) .. controls (6.95,-1.4) and (3.31,-0.3) .. (0,0) .. controls (3.31,0.3) and (6.95,1.4) .. (10.93,3.29)   ;

\draw    (514,405.29) -- (470.77,405.52) ;
\draw [shift={(468.77,405.53)}, rotate = 359.69] [color={rgb, 255:red, 0; green, 0; blue, 0 }  ][line width=0.75]    (10.93,-3.29) .. controls (6.95,-1.4) and (3.31,-0.3) .. (0,0) .. controls (3.31,0.3) and (6.95,1.4) .. (10.93,3.29)   ;
\draw [shift={(516,405.28)}, rotate = 179.69] [color={rgb, 255:red, 0; green, 0; blue, 0 }  ][line width=0.75]    (10.93,-3.29) .. controls (6.95,-1.4) and (3.31,-0.3) .. (0,0) .. controls (3.31,0.3) and (6.95,1.4) .. (10.93,3.29)   ;

\draw (488.41,89.23) node [anchor=north west][inner sep=0.75pt]  [font=\footnotesize]  {$II$};
\draw (160.04,89.04) node [anchor=north west][inner sep=0.75pt]  [font=\footnotesize]  {$I$};
\draw (154.57,235.15) node [anchor=north west][inner sep=0.75pt]  [font=\footnotesize]  {$IIIa$};
\draw (480.53,237.15) node [anchor=north west][inner sep=0.75pt]  [font=\footnotesize]  {$IIIb$};
\draw (626.13,386.99) node [anchor=north west][inner sep=0.75pt]  [color={rgb, 255:red, 144; green, 19; blue, 254 }  ,opacity=1 ]  {$\Gamma $};
\draw (454.91,395.67) node [anchor=north west][inner sep=0.75pt]  [color={rgb, 255:red, 144; green, 19; blue, 254 }  ,opacity=1 ]  {$\Gamma $};
\draw (290.74,387.63) node [anchor=north west][inner sep=0.75pt]  [color={rgb, 255:red, 144; green, 19; blue, 254 }  ,opacity=1 ]  {$\Gamma $};
\draw (125.32,395.88) node [anchor=north west][inner sep=0.75pt]  [color={rgb, 255:red, 144; green, 19; blue, 254 }  ,opacity=1 ]  {$\Gamma $};
\draw (291.78,102.43) node [anchor=north west][inner sep=0.75pt]  [color={rgb, 255:red, 144; green, 19; blue, 254 }  ,opacity=1 ]  {$\Gamma $};
\draw (126.32,99.71) node [anchor=north west][inner sep=0.75pt]  [color={rgb, 255:red, 144; green, 19; blue, 254 }  ,opacity=1 ]  {$\Gamma $};
\draw (153.43,384.29) node [anchor=north west][inner sep=0.75pt]  [font=\footnotesize]  {$IVa$};
\draw (482.53,384.29) node [anchor=north west][inner sep=0.75pt]  [font=\footnotesize]  {$IVb$};

\end{tikzpicture}

%% file: images/sign_of_a_crossing.tex
\begin{tikzpicture}[x=0.65pt,y=0.65pt,yscale=-1,xscale=1]

\draw [color={rgb, 255:red, 0; green, 0; blue, 255 }  ,draw opacity=1 ][line width=1.5]    (152.12,102.12) -- (250,200) ;
\draw [shift={(150,100)}, rotate = 45] [color={rgb, 255:red, 0; green, 0; blue, 255 }  ,draw opacity=1 ][line width=1.5]    (14.21,-4.28) .. controls (9.04,-1.82) and (4.3,-0.39) .. (0,0) .. controls (4.3,0.39) and (9.04,1.82) .. (14.21,4.28)   ;
\draw [color={rgb, 255:red, 0; green, 0; blue, 255 }  ,draw opacity=1 ][line width=1.5]    (497.88,102.12) -- (400,200) ;
\draw [shift={(500,100)}, rotate = 135] [color={rgb, 255:red, 0; green, 0; blue, 255 }  ,draw opacity=1 ][line width=1.5]    (14.21,-4.28) .. controls (9.04,-1.82) and (4.3,-0.39) .. (0,0) .. controls (4.3,0.39) and (9.04,1.82) .. (14.21,4.28)   ;
\draw [color={rgb, 255:red, 0; green, 200; blue, 200 }  ,draw opacity=1 ][line width=1.5]    (402.12,102.12) -- (440,140) ;
\draw [shift={(400,100)}, rotate = 45] [color={rgb, 255:red, 0; green, 200; blue, 200 }  ,draw opacity=1 ][line width=1.5]    (14.21,-4.28) .. controls (9.04,-1.82) and (4.3,-0.39) .. (0,0) .. controls (4.3,0.39) and (9.04,1.82) .. (14.21,4.28)   ;
\draw [color={rgb, 255:red, 0; green, 200; blue, 200 }  ,draw opacity=1 ][line width=1.5]    (460,160) -- (500,200) ;
\draw [color={rgb, 255:red, 0; green, 200; blue, 200 }  ,draw opacity=1 ][line width=1.5]    (247.88,102.12) -- (210,140) ;
\draw [shift={(250,100)}, rotate = 135] [color={rgb, 255:red, 0; green, 200; blue, 200 }  ,draw opacity=1 ][line width=1.5]    (14.21,-4.28) .. controls (9.04,-1.82) and (4.3,-0.39) .. (0,0) .. controls (4.3,0.39) and (9.04,1.82) .. (14.21,4.28)   ;
\draw [color={rgb, 255:red, 0; green, 200; blue, 200 }  ,draw opacity=1 ][line width=1.5]    (190,160) -- (150,200) ;

\draw (190,226.6) node [anchor=north west][inner sep=0.75pt]    {$-1$};
\draw (439.6,226.6) node [anchor=north west][inner sep=0.75pt]    {$+1$};
\draw (134,82.4) node [anchor=north west][inner sep=0.75pt]  [color={rgb, 255:red, 0; green, 0; blue, 255 }  ,opacity=1 ]  {$\Lambda $};
\draw (251,82.4) node [anchor=north west][inner sep=0.75pt]  [color={rgb, 255:red, 0; green, 200; blue, 200 }  ,opacity=1 ]  {$\Lambda '$};
\draw (501,81.4) node [anchor=north west][inner sep=0.75pt]  [color={rgb, 255:red, 0; green, 0; blue, 255 }  ,opacity=1 ]  {$\Lambda $};
\draw (379,81.4) node [anchor=north west][inner sep=0.75pt]  [color={rgb, 255:red, 0; green, 200; blue, 200 }  ,opacity=1 ]  {$\Lambda '$};

\end{tikzpicture}

%% file: images/tb_invariant.tex
\begin{tikzpicture}[x=0.65pt,y=0.65pt,yscale=-1,xscale=1]

\draw [color={rgb, 255:red, 0; green, 0; blue, 255 }  ,draw opacity=1 ][line width=1.5]    (144.45,81.45) -- (144.45,132) ;
\draw [shift={(144.45,78.45)}, rotate = 90] [color={rgb, 255:red, 0; green, 0; blue, 255 }  ,draw opacity=1 ][line width=1.5]    (8.53,-2.57) .. controls (5.42,-1.09) and (2.58,-0.23) .. (0,0) .. controls (2.58,0.23) and (5.42,1.09) .. (8.53,2.57)   ;
\draw [color={rgb, 255:red, 0; green, 200; blue, 200 }  ,draw opacity=1 ][line width=1.5]    (242,144.43) .. controls (232.7,155.21) and (220.03,157.15) .. (217.67,197.48) ;
\draw [shift={(217.54,200)}, rotate = 274.48] [color={rgb, 255:red, 0; green, 200; blue, 200 }  ,draw opacity=1 ][line width=1.5]    (8.53,-2.57) .. controls (5.42,-1.09) and (2.58,-0.23) .. (0,0) .. controls (2.58,0.23) and (5.42,1.09) .. (8.53,2.57)   ;
\draw [color={rgb, 255:red, 144; green, 19; blue, 254 }  ,draw opacity=1 ][line width=1.5]    (85.26,139.22) -- (310,139.22) ;
\draw [color={rgb, 255:red, 0; green, 200; blue, 200 }  ,draw opacity=1 ][line width=1.5]    (114.19,81.52) .. controls (115.66,121.3) and (175.72,158.29) .. (177.77,200) ;
\draw [shift={(114.2,78.45)}, rotate = 88.76] [color={rgb, 255:red, 0; green, 200; blue, 200 }  ,draw opacity=1 ][line width=1.5]    (8.53,-2.57) .. controls (5.42,-1.09) and (2.58,-0.23) .. (0,0) .. controls (2.58,0.23) and (5.42,1.09) .. (8.53,2.57)   ;
\draw [color={rgb, 255:red, 0; green, 0; blue, 255 }  ,draw opacity=1 ][line width=1.5]    (144.51,145.43) -- (144.45,200) ;
\draw [color={rgb, 255:red, 0; green, 0; blue, 255 }  ,draw opacity=1 ][line width=1.5]    (246.82,78.45) -- (246.82,197) ;
\draw [shift={(246.82,200)}, rotate = 270] [color={rgb, 255:red, 0; green, 0; blue, 255 }  ,draw opacity=1 ][line width=1.5]    (8.53,-2.57) .. controls (5.42,-1.09) and (2.58,-0.23) .. (0,0) .. controls (2.58,0.23) and (5.42,1.09) .. (8.53,2.57)   ;
\draw [color={rgb, 255:red, 0; green, 200; blue, 200 }  ,draw opacity=1 ][line width=1.5]    (277.3,78.45) .. controls (273.51,116.43) and (263.51,118.43) .. (250.82,135.22) ;
\draw [color={rgb, 255:red, 0; green, 0; blue, 255 }  ,draw opacity=1 ][line width=1.5]    (440,120) -- (530,210) ;
\draw [color={rgb, 255:red, 0; green, 0; blue, 255 }  ,draw opacity=1 ][line width=1.5]    (470,160) -- (440,190) ;
\draw [color={rgb, 255:red, 0; green, 0; blue, 255 }  ,draw opacity=1 ][line width=1.5]    (530,100) -- (511.67,118.33) -- (500,130) ;
\draw [color={rgb, 255:red, 0; green, 200; blue, 200 }  ,draw opacity=1 ][line width=1.5]    (460,100) -- (550,190) ;
\draw [color={rgb, 255:red, 0; green, 200; blue, 200 }  ,draw opacity=1 ][line width=1.5]    (490,180) -- (460,210) ;
\draw [color={rgb, 255:red, 0; green, 200; blue, 200 }  ,draw opacity=1 ][line width=1.5]    (550,120) -- (520,150) ;
\draw [color={rgb, 255:red, 0; green, 200; blue, 200 }  ,draw opacity=1 ][line width=1.5]    (510,160) -- (500,170) ;
\draw [color={rgb, 255:red, 0; green, 0; blue, 255 }  ,draw opacity=1 ][line width=1.5]    (490,140) -- (480,150) ;

\draw (143.26,56.26) node [anchor=north west][inner sep=0.75pt]  [color={rgb, 255:red, 0; green, 0; blue, 255 }  ,opacity=1 ]  {$\Lambda $};
\draw (102.26,56.26) node [anchor=north west][inner sep=0.75pt]  [color={rgb, 255:red, 0; green, 200; blue, 200 }  ,opacity=1 ]  {$\Lambda '$};
\draw (296.6,143.16) node [anchor=north west][inner sep=0.75pt]  [color={rgb, 255:red, 144; green, 19; blue, 254 }  ,opacity=1 ]  {$\Gamma $};
\draw (241.26,202.4) node [anchor=north west][inner sep=0.75pt]  [color={rgb, 255:red, 0; green, 0; blue, 255 }  ,opacity=1 ]  {$\Lambda $};
\draw (205.26,202.4) node [anchor=north west][inner sep=0.75pt]  [color={rgb, 255:red, 0; green, 200; blue, 200 }  ,opacity=1 ]  {$\Lambda '$};
\draw (170,242) node [anchor=north west][inner sep=0.75pt]   [align=left] {(a)};
\draw (480,242) node [anchor=north west][inner sep=0.75pt]   [align=left] {(b)};
\draw (421,102.4) node [anchor=north west][inner sep=0.75pt]  [color={rgb, 255:red, 0; green, 0; blue, 255 }  ,opacity=1 ]  {$\Lambda $};
\draw (441,82.4) node [anchor=north west][inner sep=0.75pt]  [color={rgb, 255:red, 0; green, 200; blue, 200 }  ,opacity=1 ]  {$\Lambda '$};

\end{tikzpicture}

%% file: images/corner_orientations.tex
\begin{tikzpicture}[x=0.75pt,y=0.75pt,yscale=-1,xscale=1]

\draw [color={rgb, 255:red, 0; green, 0; blue, 255 }  ,draw opacity=1 ][line width=1.5]    (250,101) -- (350,201) ;
\draw [color={rgb, 255:red, 0; green, 0; blue, 255 }  ,draw opacity=1 ][line width=1.5]    (350,101) -- (310,141) ;
\draw [color={rgb, 255:red, 0; green, 0; blue, 255 }  ,draw opacity=1 ][line width=1.5]    (290,161) -- (250,201) ;

\draw (329.84,142.48) node [anchor=north west][inner sep=0.75pt]    {$+$};
\draw (257.67,142.48) node [anchor=north west][inner sep=0.75pt]    {$+$};
\draw (295.36,175.4) node [anchor=north west][inner sep=0.75pt]    {$-$};
\draw (295.36,110.21) node [anchor=north west][inner sep=0.75pt]    {$-$};

\end{tikzpicture}

%% file: images/immersed_polygons.tex
\begin{tikzpicture}[x=0.65pt,y=0.65pt,yscale=-1,xscale=1]

\draw  [draw opacity=0][fill={rgb, 255:red, 216; green, 216; blue, 216 }  ,fill opacity=1 ] (465.67,80) .. controls (475.67,109.11) and (507.67,120.11) .. (497.84,198) .. controls (485.84,203) and (411.84,212) .. (382.84,196) .. controls (374.67,135.11) and (400.67,118.11) .. (413.67,80) .. controls (435.67,78.11) and (445.67,81.11) .. (465.67,80) -- cycle ;
\draw  [draw opacity=0][fill={rgb, 255:red, 216; green, 216; blue, 216 }  ,fill opacity=1 ] (159,84.62) .. controls (176.84,86.28) and (224.84,151.28) .. (216.84,186.28) .. controls (204.84,191.28) and (130.84,200.28) .. (101.84,184.28) .. controls (98.84,149.28) and (129.84,91.28) .. (159,84.62) -- cycle ;
\draw [color={rgb, 255:red, 0; green, 0; blue, 255 }  ,draw opacity=1 ][line width=1.5]    (189.84,69.61) .. controls (133.84,87.61) and (98.01,132.48) .. (102.01,199.48) ;
\draw [color={rgb, 255:red, 0; green, 0; blue, 255 }  ,draw opacity=1 ][line width=1.5]    (163.84,87.61) .. controls (178.84,96.61) and (212.84,132.61) .. (217.14,176.02) ;
\draw [color={rgb, 255:red, 0; green, 0; blue, 255 }  ,draw opacity=1 ][line width=1.5]    (108.86,186.11) .. controls (122,194.05) and (207.43,203.31) .. (240,174.02) ;
\draw [color={rgb, 255:red, 0; green, 0; blue, 255 }  ,draw opacity=1 ][line width=1.5]    (217.14,191.22) -- (217.14,200) ;
\draw [color={rgb, 255:red, 0; green, 0; blue, 255 }  ,draw opacity=1 ][line width=1.5]    (78,174.02) .. controls (86,178.65) and (90,180.82) .. (96.29,183.47) ;
\draw [color={rgb, 255:red, 0; green, 0; blue, 255 }  ,draw opacity=1 ][line width=1.5]    (152.58,82.31) .. controls (142.84,77.61) and (132.24,73.44) .. (125.74,73.11) ;
\draw [color={rgb, 255:red, 0; green, 0; blue, 255 }  ,draw opacity=1 ][line width=1.5]    (415,52) .. controls (416.84,126.28) and (367.67,106.8) .. (385.84,212) ;
\draw [color={rgb, 255:red, 0; green, 0; blue, 255 }  ,draw opacity=1 ][line width=1.5]    (465,52) .. controls (460.84,111.28) and (504.84,102.28) .. (499.14,188.02) ;
\draw [color={rgb, 255:red, 0; green, 0; blue, 255 }  ,draw opacity=1 ][line width=1.5]    (390.86,198.11) .. controls (404,206.05) and (489.43,212.31) .. (522,183.02) ;
\draw [color={rgb, 255:red, 0; green, 0; blue, 255 }  ,draw opacity=1 ][line width=1.5]    (499.14,203.22) -- (499.14,212) ;
\draw [color={rgb, 255:red, 0; green, 0; blue, 255 }  ,draw opacity=1 ][line width=1.5]    (360,186.02) .. controls (368,190.65) and (372,192.82) .. (378.29,195.47) ;
\draw [color={rgb, 255:red, 144; green, 19; blue, 254 }  ,draw opacity=1 ][line width=1.5]    (379.83,80) -- (499.84,80) ;

\draw (153,90.9) node [anchor=north west][inner sep=0.75pt]    {$c_\pm$};
\draw (195,166.4) node [anchor=north west][inner sep=0.75pt]    {$d_{1}$};
\draw (107,166.4) node [anchor=north west][inner sep=0.75pt]    {$d_{2}$};
\draw (478,175.4) node [anchor=north west][inner sep=0.75pt]    {$d_{1}$};
\draw (390,175.4) node [anchor=north west][inner sep=0.75pt]    {$d_{2}$};
\draw (434.84,84.9) node [anchor=north west][inner sep=0.75pt]    {$c_\Gamma$};
\draw (150.67,222) node [anchor=north west][inner sep=0.75pt]   [align=left] {(a)};
\draw (430.67,222) node [anchor=north west][inner sep=0.75pt]   [align=left] {(b)};
\draw (505,72.4) node [anchor=north west][inner sep=0.75pt]  [color={rgb, 255:red, 144; green, 19; blue, 254 }  ,opacity=1 ]  {$\Gamma $};

\end{tikzpicture}

%% file: images/Seifert_resolution_1.tex
\begin{tikzpicture}[x=0.55pt,y=0.55pt,yscale=-1,xscale=1]

\draw [color={rgb, 255:red, 0; green, 0; blue, 255 }  ,draw opacity=1 ][line width=1.5]    (94.91,118.14) -- (197.73,215.74) ;
\draw [shift={(92.73,116.07)}, rotate = 43.51] [color={rgb, 255:red, 0; green, 0; blue, 255 }  ,draw opacity=1 ][line width=1.5]    (14.21,-6.37) .. controls (9.04,-2.99) and (4.3,-0.87) .. (0,0) .. controls (4.3,0.87) and (9.04,2.99) .. (14.21,6.37)   ;
\draw [color={rgb, 255:red, 0; green, 0; blue, 255 }  ,draw opacity=1 ][line width=1.5]    (196.59,118.8) -- (149.91,162.66) ;
\draw [shift={(198.78,116.74)}, rotate = 136.78] [color={rgb, 255:red, 0; green, 0; blue, 255 }  ,draw opacity=1 ][line width=1.5]    (14.21,-6.37) .. controls (9.04,-2.99) and (4.3,-0.87) .. (0,0) .. controls (4.3,0.87) and (9.04,2.99) .. (14.21,6.37)   ;
\draw [color={rgb, 255:red, 0; green, 0; blue, 255 }  ,draw opacity=1 ][line width=1.5]    (140.47,170.66) -- (94.29,216.66) ;
\draw [color={rgb, 255:red, 0; green, 0; blue, 255 }  ,draw opacity=1 ][line width=1.5]    (537.68,117.16) -- (437.01,214.74) ;
\draw [shift={(539.84,115.07)}, rotate = 135.89] [color={rgb, 255:red, 0; green, 0; blue, 255 }  ,draw opacity=1 ][line width=1.5]    (14.21,-6.37) .. controls (9.04,-2.99) and (4.3,-0.87) .. (0,0) .. controls (4.3,0.87) and (9.04,2.99) .. (14.21,6.37)   ;
\draw [color={rgb, 255:red, 0; green, 0; blue, 255 }  ,draw opacity=1 ][line width=1.5]    (438.14,117.82) -- (483.84,161.66) ;
\draw [shift={(435.98,115.74)}, rotate = 43.81] [color={rgb, 255:red, 0; green, 0; blue, 255 }  ,draw opacity=1 ][line width=1.5]    (14.21,-6.37) .. controls (9.04,-2.99) and (4.3,-0.87) .. (0,0) .. controls (4.3,0.87) and (9.04,2.99) .. (14.21,6.37)   ;
\draw [color={rgb, 255:red, 0; green, 0; blue, 255 }  ,draw opacity=1 ][line width=1.5]    (493.09,169.66) -- (538.31,215.66) ;
\draw  [draw opacity=0][line width=1.5]  (351.12,219.05) .. controls (351.12,219.05) and (351.12,219.05) .. (351.12,219.05) .. controls (322.39,191.67) and (323.71,146.02) .. (354.08,117.08) -- (406.1,166.66) -- cycle ; \draw  [color={rgb, 255:red, 0; green, 0; blue, 255 }  ,draw opacity=1 ][line width=1.5]  (351.12,219.05) .. controls (351.12,219.05) and (351.12,219.05) .. (351.12,219.05) .. controls (322.39,191.67) and (323.71,146.02) .. (354.08,117.08) ;  
\draw  [draw opacity=0][line width=1.5]  (284.31,116.49) .. controls (313.35,144.17) and (312.25,190.09) .. (281.86,219.05) -- (229.27,168.94) -- cycle ; \draw  [color={rgb, 255:red, 0; green, 0; blue, 255 }  ,draw opacity=1 ][line width=1.5]  (284.31,116.49) .. controls (313.35,144.17) and (312.25,190.09) .. (281.86,219.05) ;  
\draw [line width=0.75]    (253.9,163.69) -- (206.34,163.69) ;
\draw [shift={(255.9,163.69)}, rotate = 180] [color={rgb, 255:red, 0; green, 0; blue, 0 }  ][line width=0.75]    (10.93,-3.29) .. controls (6.95,-1.4) and (3.31,-0.3) .. (0,0) .. controls (3.31,0.3) and (6.95,1.4) .. (10.93,3.29)   ;
\draw    (421.7,162.69) -- (374.14,162.75) ;
\draw [shift={(372.14,162.75)}, rotate = 359.93] [color={rgb, 255:red, 0; green, 0; blue, 0 }  ][line width=0.75]    (10.93,-3.29) .. controls (6.95,-1.4) and (3.31,-0.3) .. (0,0) .. controls (3.31,0.3) and (6.95,1.4) .. (10.93,3.29)   ;
\draw [color={rgb, 255:red, 0; green, 0; blue, 255 }  ,draw opacity=1 ][line width=1.5]    (357.51,113.34) -- (354.08,117.09) ;
\draw [shift={(359.54,111.13)}, rotate = 132.54] [color={rgb, 255:red, 0; green, 0; blue, 255 }  ,draw opacity=1 ][line width=1.5]    (14.21,-6.37) .. controls (9.04,-2.99) and (4.3,-0.87) .. (0,0) .. controls (4.3,0.87) and (9.04,2.99) .. (14.21,6.37)   ;
\draw [color={rgb, 255:red, 0; green, 0; blue, 255 }  ,draw opacity=1 ][line width=1.5]    (280.57,112.42) -- (284.31,116.49) ;
\draw [shift={(278.54,110.21)}, rotate = 47.48] [color={rgb, 255:red, 0; green, 0; blue, 255 }  ,draw opacity=1 ][line width=1.5]    (14.21,-6.37) .. controls (9.04,-2.99) and (4.3,-0.87) .. (0,0) .. controls (4.3,0.87) and (9.04,2.99) .. (14.21,6.37)   ;

\end{tikzpicture}

%% file: images/dividing_curves_along_Gamma.tex
\begin{tikzpicture}[x=0.45pt,y=0.45pt,yscale=-1,xscale=1]

\draw  [color={rgb, 255:red, 0; green, 0; blue, 255 }  ,draw opacity=1 ][line width=1.5]  (94.67,206.98) .. controls (122.57,222.6) and (133.57,222.6) .. (158.57,222.6) .. controls (183.57,222.6) and (230.39,215.42) .. (246.43,196.29) .. controls (262.48,177.17) and (273.66,117.1) .. (247.42,93.59) .. controls (221.18,70.09) and (190.67,96.59) .. (156.5,97.48) .. controls (122.33,98.36) and (79.29,80.29) .. (69.5,106.48) .. controls (59.71,132.66) and (66.77,191.37) .. (94.67,206.98) -- cycle ;
\draw  [color={rgb, 255:red, 0; green, 0; blue, 255 }  ,draw opacity=1 ][line width=1.5]  (487.2,222.6) .. controls (495.45,222.56) and (561.67,214.46) .. (577.72,195.33) .. controls (593.77,176.21) and (604.95,116.14) .. (578.71,92.63) .. controls (552.47,69.13) and (521.96,95.63) .. (487.79,96.52) .. controls (488.11,123.6) and (487.11,193.6) .. (487.2,222.6) -- cycle ;
\draw [color={rgb, 255:red, 0; green, 0; blue, 255 }  ,draw opacity=1 ][line width=1.5]    (404.67,207.41) .. controls (432.57,223.02) and (442.2,221.6) .. (467.2,221.6) .. controls (466.83,207.18) and (467.73,117.99) .. (467.79,95.52) .. controls (454.11,94.6) and (389.29,80.72) .. (379.5,106.9) .. controls (369.71,133.09) and (381.45,195.37) .. (404.67,207.41) -- cycle ;
\draw [color={rgb, 255:red, 144; green, 19; blue, 254 }  ,draw opacity=1 ][line width=1.5]    (157.19,80.16) -- (157.19,235.16) ;
\draw    (343.22,151.69) -- (297.99,151.75) ;
\draw [shift={(345.22,151.69)}, rotate = 179.93] [color={rgb, 255:red, 0; green, 0; blue, 0 }  ][line width=0.75]    (10.93,-3.29) .. controls (6.95,-1.4) and (3.31,-0.3) .. (0,0) .. controls (3.31,0.3) and (6.95,1.4) .. (10.93,3.29)   ;
\draw  [color={rgb, 255:red, 0; green, 0; blue, 255 }  ,draw opacity=1 ][line width=0.75]  (269.24,135.12) .. controls (265.45,141.52) and (263.18,147.91) .. (262.42,154.31) .. controls (261.66,147.91) and (259.39,141.52) .. (255.6,135.12) ;
\draw  [color={rgb, 255:red, 0; green, 0; blue, 255 }  ,draw opacity=1 ][line width=0.75]  (600,135.12) .. controls (596.21,141.51) and (593.94,147.91) .. (593.18,154.31) .. controls (592.42,147.91) and (590.15,141.51) .. (586.36,135.12) ;
\draw  [color={rgb, 255:red, 0; green, 0; blue, 255 }  ,draw opacity=1 ][line width=0.75]  (369.35,154.31) .. controls (373.14,147.91) and (375.41,141.51) .. (376.17,135.12) .. controls (376.93,141.51) and (379.2,147.91) .. (382.99,154.31) ;
\draw  [color={rgb, 255:red, 0; green, 0; blue, 255 }  ,draw opacity=1 ][line width=0.75]  (59.23,154.31) .. controls (63.02,147.91) and (65.29,141.51) .. (66.05,135.12) .. controls (66.81,141.51) and (69.08,147.91) .. (72.87,154.31) ;
\draw  [color={rgb, 255:red, 0; green, 0; blue, 255 }  ,draw opacity=1 ][line width=0.75]  (474.06,135.12) .. controls (470.27,141.51) and (468,147.91) .. (467.24,154.31) .. controls (466.48,147.91) and (464.21,141.51) .. (460.42,135.12) ;
\draw  [color={rgb, 255:red, 0; green, 0; blue, 255 }  ,draw opacity=1 ][line width=0.75]  (480.74,154.31) .. controls (484.53,147.91) and (486.8,141.51) .. (487.56,135.12) .. controls (488.32,141.51) and (490.59,147.91) .. (494.38,154.31) ;

\draw (46,82.4) node [anchor=north west][inner sep=0.75pt]  [color={rgb, 255:red, 0; green, 0; blue, 255 }  ,opacity=1 ]  {$\gamma _{i}$};
\draw (153,57.4) node [anchor=north west][inner sep=0.75pt]  [color={rgb, 255:red, 144; green, 19; blue, 254 }  ,opacity=1 ]  {$\Gamma $};
\draw (446,57.82) node [anchor=north west][inner sep=0.75pt]  [color={rgb, 255:red, 0; green, 0; blue, 255 }  ,opacity=1 ]  {$\widetilde{\gamma _{i}}$};
\draw (490,55.82) node [anchor=north west][inner sep=0.75pt]  [color={rgb, 255:red, 0; green, 0; blue, 255 }  ,opacity=1 ]  {$\widetilde{\gamma _{i}} '$};

\end{tikzpicture}

%% file: images/unknot_in_T2xI.tex
\begin{tikzpicture}[x=0.75pt,y=0.75pt,yscale=-1,xscale=1]

\draw  [fill={rgb, 255:red, 246; green, 246; blue, 246 }  ,fill opacity=1 ][line width=0.75]  (194.59,51) -- (390.79,51) -- (390.79,247.19) -- (194.59,247.19) -- cycle ;
\draw [color={rgb, 255:red, 144; green, 19; blue, 254 }  ,draw opacity=1 ][line width=1.5]    (194.59,149.1) -- (390.79,149.1) ;
\draw  [color={rgb, 255:red, 0; green, 0; blue, 255 }  ,draw opacity=1 ][line width=1.5]  (231.97,149.1) .. controls (231.97,115.56) and (259.15,88.37) .. (292.69,88.37) .. controls (326.23,88.37) and (353.41,115.56) .. (353.41,149.1) .. controls (353.41,182.63) and (326.23,209.82) .. (292.69,209.82) .. controls (259.15,209.82) and (231.97,182.63) .. (231.97,149.1) -- cycle ;
\draw  [line width=0.75]  (189,114.19) .. controls (192,109.46) and (193.8,104.73) .. (194.39,100) .. controls (195,104.73) and (196.79,109.46) .. (199.79,114.19) ;
\draw  [line width=0.75]  (385,114.19) .. controls (388,109.46) and (389.8,104.73) .. (390.39,100) .. controls (391,104.73) and (392.79,109.46) .. (395.79,114.19) ;

\draw  [color={rgb, 255:red, 144; green, 19; blue, 254 }  ,draw opacity=1 ][line width=0.75]  (277.3,46.2) .. controls (282.03,49.2) and (286.76,51) .. (291.49,51.6) .. controls (286.76,52.2) and (282.03,53.99) .. (277.3,56.99) ;
\draw  [color={rgb, 255:red, 144; green, 19; blue, 254 }  ,draw opacity=1 ][line width=0.75]  (277.3,242.2) .. controls (282.03,245.2) and (286.76,247) .. (291.49,247.6) .. controls (286.76,248.2) and (282.03,249.99) .. (277.3,252.99) ;

\draw  [color={rgb, 255:red, 144; green, 19; blue, 254 }  ,draw opacity=1 ][line width=0.75]  (289.3,46.2) .. controls (294.03,49.2) and (298.76,51) .. (303.49,51.6) .. controls (298.76,52.2) and (294.03,53.99) .. (289.3,56.99) ;
\draw  [color={rgb, 255:red, 144; green, 19; blue, 254 }  ,draw opacity=1 ][line width=0.75]  (289.3,242.2) .. controls (294.03,245.2) and (298.76,247) .. (303.49,247.6) .. controls (298.76,248.2) and (294.03,249.99) .. (289.3,252.99) ;

\draw [color={rgb, 255:red, 144; green, 19; blue, 254 }  ,draw opacity=1 ][line width=1.5]    (194.59,51) -- (390.79,51) ;
\draw [color={rgb, 255:red, 144; green, 19; blue, 254 }  ,draw opacity=1 ][line width=1.5]    (194.59,247.19) -- (390.79,247.19) ;

\draw (287,152.4) node [anchor=north west][inner sep=0.75pt]  [color={rgb, 255:red, 144; green, 19; blue, 254 }  ,opacity=1 ]  {$c$};
\draw (365,149.4) node [anchor=north west][inner sep=0.75pt]  [color={rgb, 255:red, 144; green, 19; blue, 254 }  ,opacity=1 ]  {$d$};
\draw (395.91,142.01) node [anchor=north west][inner sep=0.75pt]  [color={rgb, 255:red, 144; green, 19; blue, 254 }  ,opacity=1 ]  {$\Gamma $};
\draw (334.99,82.4) node [anchor=north west][inner sep=0.75pt]  [color={rgb, 255:red, 0; green, 0; blue, 255 }  ,opacity=1 ]  {$\Lambda $};
\draw (198.59,56.4) node [anchor=north west][inner sep=0.75pt]    {$\Sigma _{+}$};
\draw (197.59,226.4) node [anchor=north west][inner sep=0.75pt]    {$\Sigma _{-}$};

\end{tikzpicture}

%% file: images/torus_capping_surface.tex
\begin{tikzpicture}[x=0.75pt,y=0.75pt,yscale=-1,xscale=1]
\draw [fill={rgb, 255:red, 246; green, 246; blue, 246 }  ,fill opacity=1 ] [line width=0.75]  (194.59,51) -- (390.79,51) -- (390.79,247.19) -- (194.59,247.19) -- cycle ;
\draw  [draw opacity=0][fill={rgb, 255:red, 216; green, 216; blue, 216 }  ,fill opacity=1 ] (194.59,51) -- (390.79,51) -- (390.79,149.1) -- (194.59,149.1) -- cycle ;

\draw [color={rgb, 255:red, 144; green, 19; blue, 254 }  ,draw opacity=1 ][line width=1.5]    (194.59,149.1) -- (390.79,149.1) ;
\draw  [color={rgb, 255:red, 0; green, 0; blue, 255 }  ,draw opacity=1 ][line width=1.5]  (231.97,149.1) .. controls (231.97,115.56) and (259.15,88.37) .. (292.69,88.37) .. controls (326.23,88.37) and (353.41,115.56) .. (353.41,149.1) .. controls (353.41,182.63) and (326.23,209.82) .. (292.69,209.82) .. controls (259.15,209.82) and (231.97,182.63) .. (231.97,149.1) -- cycle ;
\draw  [line width=0.75]  (189,114.19) .. controls (192,109.46) and (193.8,104.73) .. (194.39,100) .. controls (195,104.73) and (196.79,109.46) .. (199.79,114.19) ;
\draw  [line width=0.75]  (385,114.19) .. controls (388,109.46) and (389.8,104.73) .. (390.39,100) .. controls (391,104.73) and (392.79,109.46) .. (395.79,114.19) ;
\draw  [color={rgb, 255:red, 144; green, 19; blue, 254 }  ,draw opacity=1 ][line width=0.75]  (277.3,46.2) .. controls (282.03,49.2) and (286.76,51) .. (291.49,51.6) .. controls (286.76,52.2) and (282.03,53.99) .. (277.3,56.99) ;
\draw  [color={rgb, 255:red, 144; green, 19; blue, 254 }  ,draw opacity=1 ][line width=0.75]  (277.3,242.2) .. controls (282.03,245.2) and (286.76,247) .. (291.49,247.6) .. controls (286.76,248.2) and (282.03,249.99) .. (277.3,252.99) ;

\draw  [color={rgb, 255:red, 144; green, 19; blue, 254 }  ,draw opacity=1 ][line width=0.75]  (289.3,46.2) .. controls (294.03,49.2) and (298.76,51) .. (303.49,51.6) .. controls (298.76,52.2) and (294.03,53.99) .. (289.3,56.99) ;
\draw  [color={rgb, 255:red, 144; green, 19; blue, 254 }  ,draw opacity=1 ][line width=0.75]  (289.3,242.2) .. controls (294.03,245.2) and (298.76,247) .. (303.49,247.6) .. controls (298.76,248.2) and (294.03,249.99) .. (289.3,252.99) ;

\draw [color={rgb, 255:red, 144; green, 19; blue, 254 }  ,draw opacity=1 ][line width=1.5]    (194.59,51) -- (390.79,51) ;
\draw [color={rgb, 255:red, 144; green, 19; blue, 254 }  ,draw opacity=1 ][line width=1.5]    (194.59,247.19) -- (390.79,247.19) ;
\draw  [draw opacity=0][fill={rgb, 255:red, 155; green, 155; blue, 155 }  ,fill opacity=1 ][line width=1.5]  (231.97,148.48) .. controls (231.97,148.48) and (231.97,148.48) .. (231.97,148.48) .. controls (231.97,148.48) and (231.97,148.48) .. (231.97,148.48) .. controls (231.97,115.28) and (259.15,88.37) .. (292.69,88.37) .. controls (326.23,88.37) and (353.41,115.28) .. (353.41,148.48) -- (292.69,148.48) -- cycle ; \draw  [color={rgb, 255:red, 255; green, 0; blue, 0 }  ,draw opacity=1 ][line width=1.5]  (231.97,148.48) .. controls (231.97,148.48) and (231.97,148.48) .. (231.97,148.48) .. controls (231.97,148.48) and (231.97,148.48) .. (231.97,148.48) .. controls (231.97,115.28) and (259.15,88.37) .. (292.69,88.37) .. controls (326.23,88.37) and (353.41,115.28) .. (353.41,148.48) ;  
\draw  [color={rgb, 255:red, 255; green, 0; blue, 0 }  ,draw opacity=1 ][line width=0.75]  (298.49,93.49) .. controls (293.76,90.49) and (289.03,88.69) .. (284.3,88.1) .. controls (289.03,87.49) and (293.76,85.7) .. (298.49,82.7) ;

\draw (287,149.4) node [anchor=north west][inner sep=0.75pt]  [color={rgb, 255:red, 144; green, 19; blue, 254 }  ,opacity=1 ]  {$c$};
\draw (365,149.4) node [anchor=north west][inner sep=0.75pt]  [color={rgb, 255:red, 144; green, 19; blue, 254 }  ,opacity=1 ]  {$d$};
\draw (395.91,142.01) node [anchor=north west][inner sep=0.75pt]  [color={rgb, 255:red, 144; green, 19; blue, 254 }  ,opacity=1 ]  {$\Gamma $};
\draw (346.99,94.4) node [anchor=north west][inner sep=0.75pt]  [color={rgb, 255:red, 255; green, 0; blue, 0 }  ,opacity=1 ]  {$\gamma _{c}$};
\draw (287.59,116.4) node [anchor=north west][inner sep=0.75pt]    {$S_{c}$};
\draw (222.59,70.4) node [anchor=north west][inner sep=0.75pt]    {$S_{d}$};

\draw  [line width=0.75]  (194.59,51) -- (390.79,51) -- (390.79,247.19) -- (194.59,247.19) -- cycle ;
\end{tikzpicture}

%% file: images/genus_2_octagon.tex
\begin{tikzpicture}[x=0.75pt,y=0.75pt,yscale=-1,xscale=1]
\draw [fill={rgb, 255:red, 246; green, 246; blue, 246 }  ,fill opacity=1 ] [line width=0.75]  (414.79,145.65) -- (385.38,216.64) -- (314.39,246.05) -- (243.4,216.64) -- (214,145.65) -- (243.4,74.67) -- (314.39,45.26) -- (385.38,74.67) -- cycle ;
\draw  [draw opacity=0][line width=1.5]  (338.28,55.41) .. controls (325.36,68.22) and (304.01,67.69) .. (290.52,54.21) -- (313.89,30.84) -- cycle ; \draw  [color={rgb, 255:red, 144; green, 19; blue, 254 }  ,draw opacity=1 ][line width=1.5]  (338.28,55.41) .. controls (325.36,68.22) and (304.01,67.69) .. (290.52,54.21) ;  
\draw  [draw opacity=0][line width=1.5]  (400.58,179.34) .. controls (356.31,161.59) and (335.58,109.44) .. (354.26,62.87) -- (434.41,95.02) -- cycle ; \draw  [color={rgb, 255:red, 144; green, 19; blue, 254 }  ,draw opacity=1 ][line width=1.5]  (400.58,179.34) .. controls (356.31,161.59) and (335.58,109.44) .. (354.26,62.87) ;  
\draw  [draw opacity=0][line width=1.5]  (239.32,207.81) .. controls (239.32,207.81) and (239.32,207.81) .. (239.32,207.81) .. controls (281,166.13) and (348.05,165.61) .. (389.08,206.64) -- (313.62,282.11) -- cycle ; \draw  [color={rgb, 255:red, 144; green, 19; blue, 254 }  ,draw opacity=1 ][line width=1.5]  (239.32,207.81) .. controls (239.32,207.81) and (239.32,207.81) .. (239.32,207.81) .. controls (281,166.13) and (348.05,165.61) .. (389.08,206.64) ;  
\draw [color={rgb, 255:red, 0; green, 0; blue, 255 }  ,draw opacity=1 ][line width=1.5]    (314.39,45.26) -- (314.39,246.05) ;
\draw  [line width=0.75]  (228.2,98.65) .. controls (228.97,104.04) and (228.72,109.01) .. (227.44,113.56) .. controls (229.75,109.44) and (233.08,105.74) .. (237.44,102.48) ;
\draw  [line width=0.75]  (393.39,184.69) .. controls (397.67,181.78) and (400.9,178.43) .. (403.08,174.65) .. controls (401.95,178.87) and (401.86,183.52) .. (402.83,188.6) ;
\draw  [line width=0.75]  (278.03,54.88) .. controls (274.86,59.2) and (271.27,62.5) .. (267.25,64.76) .. controls (271.69,63.52) and (276.56,63.31) .. (281.86,64.13) ;
\draw  [line width=0.75]  (342.4,228.66) .. controls (347.48,229.62) and (352.14,229.54) .. (356.35,228.41) .. controls (352.57,230.59) and (349.22,233.82) .. (346.31,238.1) ;
\draw  [line width=0.75]  (346.31,53.15) .. controls (349.22,57.43) and (352.57,60.65) .. (356.35,62.83) .. controls (352.14,61.71) and (347.49,61.62) .. (342.4,62.59) ;
\draw  [line width=0.75]  (282.85,227.88) .. controls (277.64,228.58) and (272.83,228.3) .. (268.43,227.05) .. controls (272.43,229.28) and (276.02,232.49) .. (279.21,236.67) ;
\draw  [line width=0.75]  (402.83,103.64) .. controls (401.87,108.72) and (401.95,113.38) .. (403.08,117.59) .. controls (400.9,113.81) and (397.67,110.46) .. (393.39,107.55) ;
\draw  [line width=0.75]  (235.39,183.07) .. controls (231.08,180.35) and (227.85,177.19) .. (225.72,173.57) .. controls (226.76,177.64) and (226.73,182.16) .. (225.59,187.13) ;
\draw  [draw opacity=0][line width=1.5]  (273.31,64.15) .. controls (291.61,108.19) and (270.65,158.77) .. (226.49,177.13) -- (193.34,97.38) -- cycle ; \draw  [color={rgb, 255:red, 144; green, 19; blue, 254 }  ,draw opacity=1 ][line width=1.5]  (273.31,64.15) .. controls (291.61,108.19) and (270.65,158.77) .. (226.49,177.13) ;

\draw (374,111.36) node [anchor=north west][inner sep=0.75pt]  [color={rgb, 255:red, 144; green, 19; blue, 254 }  ,opacity=1 ]  {$\Gamma $};
\draw (319,106.36) node [anchor=north west][inner sep=0.75pt]  [color={rgb, 255:red, 0; green, 0; blue, 255 }  ,opacity=1 ]  {$\Lambda $};
\draw (329.5,185.53) node [anchor=north west][inner sep=0.75pt]  [color={rgb, 255:red, 144; green, 19; blue, 254 }  ,opacity=1 ]  {$c$};
\draw (329.5,71.53) node [anchor=north west][inner sep=0.75pt]  [color={rgb, 255:red, 144; green, 19; blue, 254 }  ,opacity=1 ]  {$c$};
\draw (293,67.53) node [anchor=north west][inner sep=0.75pt]  [color={rgb, 255:red, 144; green, 19; blue, 254 }  ,opacity=1 ]  {$d$};
\draw (289,181.53) node [anchor=north west][inner sep=0.75pt]  [color={rgb, 255:red, 144; green, 19; blue, 254 }  ,opacity=1 ]  {$d$};
\draw (405,93.4) node [anchor=north west][inner sep=0.75pt]  [font=\footnotesize]  {$\alpha $};
\draw (350,238.4) node [anchor=north west][inner sep=0.75pt]  [font=\footnotesize]  {$\alpha $};
\draw (405,180.4) node [anchor=north west][inner sep=0.75pt]  [font=\footnotesize]  {$\beta $};
\draw (350,37.4) node [anchor=north west][inner sep=0.75pt]  [font=\footnotesize]  {$\beta $};
\draw (261,234.4) node [anchor=north west][inner sep=0.75pt]  [font=\footnotesize]  {$\delta $};
\draw (212,89.4) node [anchor=north west][inner sep=0.75pt]  [font=\footnotesize]  {$\delta $};
\draw (212,180.4) node [anchor=north west][inner sep=0.75pt]  [font=\footnotesize]  {$\gamma $};
\draw (261,37.4) node [anchor=north west][inner sep=0.75pt]  [font=\footnotesize]  {$\gamma $};
\draw (282.59,141.9) node [anchor=north west][inner sep=0.75pt]    {$\Sigma _{+}$};
\draw (283.09,205.9) node [anchor=north west][inner sep=0.75pt]    {$\Sigma _{-}$};

\end{tikzpicture}

%% file: images/octagon_capping_surface.tex
\begin{tikzpicture}[x=0.75pt,y=0.75pt,yscale=-1,xscale=1]
\draw  [fill={rgb, 255:red, 246; green, 246; blue, 246 }  ,fill opacity=1 ][line width=0.75]  (569.79,134.79) -- (540.38,205.78) -- (469.39,235.18) -- (398.4,205.78) -- (369,134.79) -- (398.4,63.8) -- (469.39,34.4) -- (540.38,63.8) -- cycle ;
\draw  [color={rgb, 255:red, 144; green, 19; blue, 254 }  ,draw opacity=1 ][line width=0.75]  (507.09,125.74) .. controls (507.59,120.6) and (507.09,115.97) .. (505.58,111.88) .. controls (508.09,115.44) and (511.61,118.49) .. (516.13,121) ;
\draw  [color={rgb, 255:red, 144; green, 19; blue, 254 }  ,draw opacity=1 ][line width=0.75]  (461.91,57.78) .. controls (458.84,53.61) and (455.37,50.52) .. (451.51,48.48) .. controls (455.76,49.45) and (460.42,49.36) .. (465.46,48.19) ;
\draw  [color={rgb, 255:red, 144; green, 19; blue, 254 }  ,draw opacity=1 ][line width=0.75]  (499.66,164.11) .. controls (503.09,167.98) and (506.83,170.74) .. (510.86,172.42) .. controls (506.54,171.84) and (501.92,172.36) .. (497,173.98) ;
\draw  [color={rgb, 255:red, 144; green, 19; blue, 254 }  ,draw opacity=1 ][line width=0.75]  (435.82,115.8) .. controls (431.53,118.69) and (428.28,122.02) .. (426.08,125.78) .. controls (427.23,121.57) and (427.35,116.92) .. (426.41,111.84) ;

\draw  [fill={rgb, 255:red, 246; green, 246; blue, 246 }  ,fill opacity=1 ] [line width=0.75]  (259.79,134.79) -- (230.38,205.78) -- (159.39,235.18) -- (88.4,205.78) -- (59,134.79) -- (88.4,63.8) -- (159.39,34.4) -- (230.38,63.8) -- cycle ;

\draw  [draw opacity=0][fill={rgb, 255:red, 155; green, 155; blue, 155 }  ,fill opacity=1 ] (447.38,44.55) .. controls (433.23,49.66) and (438.5,47.94) .. (428.31,52.29) .. controls (435.85,68.39) and (437.85,95.39) .. (431.85,110.39) .. controls (422.85,140.39) and (403.16,158.44) .. (382.49,166.27) .. controls (387.93,180.44) and (389.58,183.66) .. (395.81,196.78) .. controls (420.03,170.66) and (450.55,165.66) .. (471.04,164.4) .. controls (471.87,145.66) and (470.97,64.4) .. (470.97,54.4) .. controls (459.62,52.66) and (455.56,50.66) .. (447.38,44.55) -- cycle ;
\draw  [draw opacity=0][fill={rgb, 255:red, 155; green, 155; blue, 155 }  ,fill opacity=1 ] (494.28,44.55) .. controls (508.22,49.66) and (500.22,47.66) .. (510.26,52) .. controls (505.22,68.66) and (500.22,84.66) .. (506.04,109.4) .. controls (515.22,141.66) and (536.22,160.66) .. (556.58,168.48) .. controls (551.22,182.66) and (551.22,182.66) .. (545.08,195.78) .. controls (520.51,172.79) and (491.22,165.66) .. (471.04,164.4) .. controls (470.22,145.66) and (471.04,64.4) .. (471.04,54.4) .. controls (482.22,52.66) and (486.22,50.66) .. (494.28,44.55) -- cycle ;
\draw  [draw opacity=0][fill={rgb, 255:red, 155; green, 155; blue, 155 }  ,fill opacity=1 ] (136.38,44.55) .. controls (122.23,49.66) and (126.51,46.79) .. (118.31,52.29) .. controls (125.85,68.39) and (126.85,95.39) .. (120.85,110.39) .. controls (111.85,140.39) and (92.16,158.44) .. (71.49,166.27) .. controls (76.93,180.44) and (78.58,182.66) .. (84.81,195.78) .. controls (112.51,172.79) and (139.55,165.66) .. (160.04,164.4) .. controls (160.87,145.66) and (159.97,64.4) .. (159.97,54.4) .. controls (148.62,52.66) and (144.56,50.66) .. (136.38,44.55) -- cycle ;
\draw  [draw opacity=0][fill={rgb, 255:red, 216; green, 216; blue, 216 }  ,fill opacity=1 ] (183.28,44.55) .. controls (197.22,49.66) and (189.22,47.66) .. (199.26,52) .. controls (194.22,68.66) and (189.22,84.66) .. (195.04,109.4) .. controls (204.22,141.66) and (225.22,160.66) .. (245.58,168.48) .. controls (240.22,182.66) and (240.22,182.66) .. (234.08,195.78) .. controls (210.22,169.66) and (180.22,165.66) .. (160.04,164.4) .. controls (159.22,145.66) and (160.04,64.4) .. (160.04,54.4) .. controls (171.22,52.66) and (175.22,50.66) .. (183.28,44.55) -- cycle ;
\draw  [draw opacity=0][line width=1.5]  (245.58,168.48) .. controls (245.58,168.48) and (245.58,168.48) .. (245.58,168.48) .. controls (201.31,150.72) and (180.58,98.57) .. (199.26,52) -- (279.41,84.16) -- cycle ; \draw  [color={rgb, 255:red, 144; green, 19; blue, 254 }  ,draw opacity=1 ][line width=1.5]  (245.58,168.48) .. controls (245.58,168.48) and (245.58,168.48) .. (245.58,168.48) .. controls (201.31,150.72) and (180.58,98.57) .. (199.26,52) ;  
\draw [color={rgb, 255:red, 0; green, 0; blue, 255 }  ,draw opacity=1 ][line width=1.5]    (159.39,34.4) -- (159.39,235.18) ;
\draw  [line width=0.75]  (73.2,87.79) .. controls (73.98,93.18) and (73.72,98.15) .. (72.44,102.69) .. controls (74.74,98.57) and (78.08,94.88) .. (82.44,91.62) ;
\draw  [line width=0.75]  (238.39,173.83) .. controls (242.67,170.92) and (245.9,167.57) .. (248.08,163.79) .. controls (246.94,168) and (246.87,172.66) .. (247.83,177.74) ;
\draw  [line width=0.75]  (123.03,44.02) .. controls (119.86,48.34) and (116.27,51.64) .. (112.25,53.9) .. controls (116.69,52.66) and (121.57,52.45) .. (126.86,53.26) ;
\draw  [line width=0.75]  (187.4,217.79) .. controls (192.49,218.76) and (197.14,218.67) .. (201.35,217.55) .. controls (197.57,219.73) and (194.22,222.95) .. (191.31,227.23) ;
\draw  [line width=0.75]  (191.31,41.28) .. controls (194.22,45.57) and (197.57,48.79) .. (201.35,50.97) .. controls (197.14,49.85) and (192.49,49.76) .. (187.4,50.72) ;
\draw  [line width=0.75]  (127.85,217.01) .. controls (122.64,217.72) and (117.83,217.44) .. (113.43,216.19) .. controls (117.42,218.42) and (121.02,221.62) .. (124.21,225.8) ;
\draw  [line width=0.75]  (247.83,92.78) .. controls (246.87,97.86) and (246.95,102.51) .. (248.08,106.73) .. controls (245.9,102.95) and (242.67,99.6) .. (238.39,96.69) ;
\draw  [line width=0.75]  (80.39,172.21) .. controls (76.07,169.49) and (72.86,166.32) .. (70.72,162.71) .. controls (71.77,166.77) and (71.72,171.29) .. (70.59,176.27) ;
\draw  [draw opacity=0][line width=1.5]  (118.31,53.29) .. controls (136.61,97.33) and (115.65,147.91) .. (71.49,166.27) -- (38.34,86.52) -- cycle ; \draw  [color={rgb, 255:red, 144; green, 19; blue, 254 }  ,draw opacity=1 ][line width=1.5]  (118.31,53.29) .. controls (136.61,97.33) and (115.65,147.91) .. (71.49,166.27) ;  
\draw [color={rgb, 255:red, 255; green, 0; blue, 0 }  ,draw opacity=1 ][line width=1.5]    (159.5,52.13) -- (159.5,166.13) ;
\draw  [draw opacity=0][line width=1.5]  (84.32,196.95) .. controls (84.32,196.95) and (84.32,196.95) .. (84.32,196.95) .. controls (126,155.27) and (193.05,154.75) .. (234.08,195.78) -- (158.62,271.24) -- cycle ; \draw  [color={rgb, 255:red, 144; green, 19; blue, 254 }  ,draw opacity=1 ][line width=1.5]  (84.32,196.95) .. controls (84.32,196.95) and (84.32,196.95) .. (84.32,196.95) .. controls (126,155.27) and (193.05,154.75) .. (234.08,195.78) ;  
\draw  [draw opacity=0][line width=1.5]  (183.28,44.55) .. controls (170.36,57.36) and (149.01,56.83) .. (135.52,43.34) -- (158.89,19.97) -- cycle ; \draw  [color={rgb, 255:red, 144; green, 19; blue, 254 }  ,draw opacity=1 ][line width=1.5]  (183.28,44.55) .. controls (170.36,57.36) and (149.01,56.83) .. (135.52,43.34) ;  

\draw  [color={rgb, 255:red, 144; green, 19; blue, 254 }  ,draw opacity=1 ][line width=0.75]  (197.09,125.74) .. controls (197.59,120.6) and (197.09,115.97) .. (195.58,111.88) .. controls (198.09,115.44) and (201.61,118.49) .. (206.13,121) ;
\draw  [color={rgb, 255:red, 144; green, 19; blue, 254 }  ,draw opacity=1 ][line width=0.75]  (151.91,57.78) .. controls (148.84,53.61) and (145.37,50.52) .. (141.51,48.48) .. controls (145.76,49.45) and (150.42,49.36) .. (155.46,48.19) ;
\draw  [color={rgb, 255:red, 144; green, 19; blue, 254 }  ,draw opacity=1 ][line width=0.75]  (189.66,164.11) .. controls (193.09,167.98) and (196.83,170.74) .. (200.86,172.42) .. controls (196.54,171.84) and (191.92,172.36) .. (187,173.98) ;
\draw  [color={rgb, 255:red, 144; green, 19; blue, 254 }  ,draw opacity=1 ][line width=0.75]  (125.82,115.8) .. controls (121.53,118.69) and (118.28,122.02) .. (116.08,125.78) .. controls (117.23,121.57) and (117.35,116.92) .. (116.41,111.84) ;
\draw  [draw opacity=0][line width=1.5]  (555.58,168.48) .. controls (555.58,168.48) and (555.58,168.48) .. (555.58,168.48) .. controls (511.31,150.72) and (490.58,98.57) .. (509.26,52) -- (589.41,84.16) -- cycle ; \draw  [color={rgb, 255:red, 144; green, 19; blue, 254 }  ,draw opacity=1 ][line width=1.5]  (555.58,168.48) .. controls (555.58,168.48) and (555.58,168.48) .. (555.58,168.48) .. controls (511.31,150.72) and (490.58,98.57) .. (509.26,52) ;  
\draw [color={rgb, 255:red, 0; green, 0; blue, 255 }  ,draw opacity=1 ][line width=1.5]    (469.39,34.4) -- (469.39,235.18) ;
\draw  [line width=0.75]  (383.2,87.79) .. controls (383.98,93.18) and (383.72,98.15) .. (382.44,102.69) .. controls (384.74,98.57) and (388.08,94.88) .. (392.44,91.62) ;
\draw  [line width=0.75]  (548.39,173.83) .. controls (552.67,170.92) and (555.9,167.57) .. (558.08,163.79) .. controls (556.94,168) and (556.87,172.66) .. (557.83,177.74) ;
\draw  [line width=0.75]  (433.03,44.02) .. controls (429.86,48.34) and (426.27,51.64) .. (422.25,53.9) .. controls (426.69,52.66) and (431.57,52.45) .. (436.86,53.26) ;
\draw  [line width=0.75]  (497.4,217.79) .. controls (502.49,218.76) and (507.14,218.67) .. (511.35,217.55) .. controls (507.57,219.73) and (504.22,222.95) .. (501.31,227.23) ;
\draw  [line width=0.75]  (501.31,41.28) .. controls (504.22,45.57) and (507.57,48.79) .. (511.35,50.97) .. controls (507.14,49.85) and (502.49,49.76) .. (497.4,50.72) ;
\draw  [line width=0.75]  (437.85,217.01) .. controls (432.64,217.72) and (427.83,217.44) .. (423.43,216.19) .. controls (427.42,218.42) and (431.02,221.62) .. (434.21,225.8) ;
\draw  [line width=0.75]  (557.83,92.78) .. controls (556.87,97.86) and (556.95,102.51) .. (558.08,106.73) .. controls (555.9,102.95) and (552.67,99.6) .. (548.39,96.69) ;
\draw  [line width=0.75]  (390.39,172.21) .. controls (386.07,169.49) and (382.86,166.32) .. (380.72,162.71) .. controls (381.77,166.77) and (381.72,171.29) .. (380.59,176.27) ;
\draw  [draw opacity=0][line width=1.5]  (428.31,53.29) .. controls (446.61,97.33) and (425.65,147.91) .. (381.49,166.27) -- (348.34,86.52) -- cycle ; \draw  [color={rgb, 255:red, 144; green, 19; blue, 254 }  ,draw opacity=1 ][line width=1.5]  (428.31,53.29) .. controls (446.61,97.33) and (425.65,147.91) .. (381.49,166.27) ;  
\draw  [draw opacity=0][line width=1.5]  (394.32,196.95) .. controls (394.32,196.95) and (394.32,196.95) .. (394.32,196.95) .. controls (436,155.27) and (503.05,154.75) .. (544.08,195.78) -- (468.62,271.24) -- cycle ; \draw  [color={rgb, 255:red, 144; green, 19; blue, 254 }  ,draw opacity=1 ][line width=1.5]  (394.32,196.95) .. controls (394.32,196.95) and (394.32,196.95) .. (394.32,196.95) .. controls (436,155.27) and (503.05,154.75) .. (544.08,195.78) ;  
\draw  [draw opacity=0][line width=1.5]  (493.28,44.55) .. controls (480.36,57.36) and (459.01,56.83) .. (445.52,43.34) -- (468.89,19.97) -- cycle ; \draw  [color={rgb, 255:red, 144; green, 19; blue, 254 }  ,draw opacity=1 ][line width=1.5]  (493.28,44.55) .. controls (480.36,57.36) and (459.01,56.83) .. (445.52,43.34) ;

\draw (484.5,174.67) node [anchor=north west][inner sep=0.75pt]  [color={rgb, 255:red, 144; green, 19; blue, 254 }  ,opacity=1 ]  {$c$};
\draw (484.5,60.67) node [anchor=north west][inner sep=0.75pt]  [color={rgb, 255:red, 144; green, 19; blue, 254 }  ,opacity=1 ]  {$c$};
\draw (448,56.67) node [anchor=north west][inner sep=0.75pt]  [color={rgb, 255:red, 144; green, 19; blue, 254 }  ,opacity=1 ]  {$d$};
\draw (444,170.67) node [anchor=north west][inner sep=0.75pt]  [color={rgb, 255:red, 144; green, 19; blue, 254 }  ,opacity=1 ]  {$d$};
\draw (560,82.54) node [anchor=north west][inner sep=0.75pt]  [font=\footnotesize]  {$\alpha $};
\draw (505,227.54) node [anchor=north west][inner sep=0.75pt]  [font=\footnotesize]  {$\alpha $};
\draw (560,169.54) node [anchor=north west][inner sep=0.75pt]  [font=\footnotesize]  {$\beta $};
\draw (505,26.54) node [anchor=north west][inner sep=0.75pt]  [font=\footnotesize]  {$\beta $};
\draw (416,223.54) node [anchor=north west][inner sep=0.75pt]  [font=\footnotesize]  {$\delta $};
\draw (367,78.54) node [anchor=north west][inner sep=0.75pt]  [font=\footnotesize]  {$\delta $};
\draw (367,169.54) node [anchor=north west][inner sep=0.75pt]  [font=\footnotesize]  {$\gamma $};
\draw (416,26.54) node [anchor=north west][inner sep=0.75pt]  [font=\footnotesize]  {$\gamma $};
\draw (181,141.4) node [anchor=north west][inner sep=0.75pt]    {$S_{c}$};
\draw (124,141.4) node [anchor=north west][inner sep=0.75pt]    {$S_{d}$};
\draw (481,141.4) node [anchor=north west][inner sep=0.75pt]    {$S_{c*d}$};
\draw (150,256.74) node [anchor=north west][inner sep=0.75pt]   [align=left] {(a)};
\draw (460,256.74) node [anchor=north west][inner sep=0.75pt]   [align=left] {(b)};
\draw (163,109.49) node [anchor=north west][inner sep=0.75pt]  [color={rgb, 255:red, 255; green, 0; blue, 0 }  ,opacity=1 ]  {$\gamma _{c}$};
\draw (174.5,174.67) node [anchor=north west][inner sep=0.75pt]  [color={rgb, 255:red, 144; green, 19; blue, 254 }  ,opacity=1 ]  {$c$};
\draw (174.5,60.67) node [anchor=north west][inner sep=0.75pt]  [color={rgb, 255:red, 144; green, 19; blue, 254 }  ,opacity=1 ]  {$c$};
\draw (138,56.67) node [anchor=north west][inner sep=0.75pt]  [color={rgb, 255:red, 144; green, 19; blue, 254 }  ,opacity=1 ]  {$d$};
\draw (134,170.67) node [anchor=north west][inner sep=0.75pt]  [color={rgb, 255:red, 144; green, 19; blue, 254 }  ,opacity=1 ]  {$d$};
\draw (250,82.54) node [anchor=north west][inner sep=0.75pt]  [font=\footnotesize]  {$\alpha $};
\draw (195,227.54) node [anchor=north west][inner sep=0.75pt]  [font=\footnotesize]  {$\alpha $};
\draw (250,169.54) node [anchor=north west][inner sep=0.75pt]  [font=\footnotesize]  {$\beta $};
\draw (195,26.54) node [anchor=north west][inner sep=0.75pt]  [font=\footnotesize]  {$\beta $};
\draw (106,223.54) node [anchor=north west][inner sep=0.75pt]  [font=\footnotesize]  {$\delta $};
\draw (57,78.54) node [anchor=north west][inner sep=0.75pt]  [font=\footnotesize]  {$\delta $};
\draw (57,169.54) node [anchor=north west][inner sep=0.75pt]  [font=\footnotesize]  {$\gamma $};
\draw (106,26.54) node [anchor=north west][inner sep=0.75pt]  [font=\footnotesize]  {$\gamma $};
\draw [line width=0.75]  (569.79,134.79) -- (540.38,205.78) -- (469.39,235.18) -- (398.4,205.78) -- (369,134.79) -- (398.4,63.8) -- (469.39,34.4) -- (540.38,63.8) -- cycle ;

\draw [line width=0.75]  (259.79,134.79) -- (230.38,205.78) -- (159.39,235.18) -- (88.4,205.78) -- (59,134.79) -- (88.4,63.8) -- (159.39,34.4) -- (230.38,63.8) -- cycle ;

\end{tikzpicture}

%% file: images/trefoil_Seifert_resolution_1.tex
\begin{tikzpicture}[x=0.75pt,y=0.75pt,yscale=-1,xscale=1]

\draw  [fill={rgb, 255:red, 246; green, 246; blue, 246 }  ,fill opacity=1 ][line width=1.5]  (218.59,47) -- (414.79,47) -- (414.79,243.19) -- (218.59,243.19) -- cycle ;
\draw [color={rgb, 255:red, 144; green, 19; blue, 254 }  ,draw opacity=1 ][line width=1.5]    (218.59,145.1) -- (414.79,145.1) ;
\draw  [line width=0.75]  (213,110.19) .. controls (216,105.46) and (217.8,100.73) .. (218.39,96) .. controls (219,100.73) and (220.79,105.46) .. (223.79,110.19) ;
\draw  [line width=0.75]  (409,110.19) .. controls (412,105.46) and (413.8,100.73) .. (414.39,96) .. controls (415,100.73) and (416.79,105.46) .. (419.79,110.19) ;
\draw  [color={rgb, 255:red, 144; green, 19; blue, 254 }  ,draw opacity=1 ][line width=0.75]  (301.3,42.2) .. controls (306.03,45.2) and (310.76,47) .. (315.49,47.6) .. controls (310.76,48.2) and (306.03,49.99) .. (301.3,52.99) ;
\draw  [color={rgb, 255:red, 144; green, 19; blue, 254 }  ,draw opacity=1 ][line width=0.75]  (301.3,238.2) .. controls (306.03,241.2) and (310.76,243) .. (315.49,243.6) .. controls (310.76,244.2) and (306.03,245.99) .. (301.3,248.99) ;

\draw  [color={rgb, 255:red, 144; green, 19; blue, 254 }  ,draw opacity=1 ][line width=0.75]  (313.3,42.2) .. controls (318.03,45.2) and (322.76,47) .. (327.49,47.6) .. controls (322.76,48.2) and (318.03,49.99) .. (313.3,52.99) ;
\draw  [color={rgb, 255:red, 144; green, 19; blue, 254 }  ,draw opacity=1 ][line width=0.75]  (313.3,238.2) .. controls (318.03,241.2) and (322.76,243) .. (327.49,243.6) .. controls (322.76,244.2) and (318.03,245.99) .. (313.3,248.99) ;

\draw [color={rgb, 255:red, 144; green, 19; blue, 254 }  ,draw opacity=1 ][line width=1.5]    (218.59,47) -- (414.79,47) ;
\draw [color={rgb, 255:red, 144; green, 19; blue, 254 }  ,draw opacity=1 ][line width=1.5]    (218.59,243.19) -- (414.79,243.19) ;
\draw [color={rgb, 255:red, 0; green, 0; blue, 255 }  ,draw opacity=1 ][line width=1.5]    (271.21,162.52) -- (294.66,190.36) ;
\draw [color={rgb, 255:red, 0; green, 0; blue, 255 }  ,draw opacity=1 ][line width=1.5]    (280.48,179.69) -- (271.27,190.64) ;
\draw  [draw opacity=0][line width=1.5]  (293.74,163.69) .. controls (295.31,161.46) and (297.76,160.03) .. (300.54,160.03) .. controls (303.05,160.03) and (305.32,161.21) .. (306.91,163.09) -- (300.58,169.45) -- cycle ; \draw  [color={rgb, 255:red, 0; green, 0; blue, 255 }  ,draw opacity=1 ][line width=1.5]  (293.74,163.69) .. controls (295.31,161.46) and (297.76,160.03) .. (300.54,160.03) .. controls (303.05,160.03) and (305.32,161.21) .. (306.91,163.09) ;  
\draw [color={rgb, 255:red, 0; green, 0; blue, 255 }  ,draw opacity=1 ][line width=1.5]    (294.66,162.52) -- (285.45,173.47) ;
\draw  [draw opacity=0][line width=1.5]  (306.96,189.79) .. controls (305.39,191.67) and (303.13,192.85) .. (300.62,192.85) .. controls (298.02,192.85) and (295.68,191.59) .. (294.09,189.61) -- (300.57,183.44) -- cycle ; \draw  [color={rgb, 255:red, 0; green, 0; blue, 255 }  ,draw opacity=1 ][line width=1.5]  (306.96,189.79) .. controls (305.39,191.67) and (303.13,192.85) .. (300.62,192.85) .. controls (298.02,192.85) and (295.68,191.59) .. (294.09,189.61) ;  
\draw [color={rgb, 255:red, 0; green, 0; blue, 255 }  ,draw opacity=1 ][line width=1.5]    (306.38,162.52) -- (329.82,190.36) ;
\draw [color={rgb, 255:red, 0; green, 0; blue, 255 }  ,draw opacity=1 ][line width=1.5]    (315.64,179.42) -- (306.44,190.36) ;
\draw [color={rgb, 255:red, 0; green, 0; blue, 255 }  ,draw opacity=1 ][line width=1.5]    (329.82,162.52) -- (320.62,173.47) ;
\draw  [draw opacity=0][line width=1.5]  (328.91,163.69) .. controls (330.47,161.46) and (332.93,160.03) .. (335.7,160.03) .. controls (338.22,160.03) and (340.49,161.21) .. (342.07,163.09) -- (335.75,169.45) -- cycle ; \draw  [color={rgb, 255:red, 0; green, 0; blue, 255 }  ,draw opacity=1 ][line width=1.5]  (328.91,163.69) .. controls (330.47,161.46) and (332.93,160.03) .. (335.7,160.03) .. controls (338.22,160.03) and (340.49,161.21) .. (342.07,163.09) ;  
\draw  [draw opacity=0][line width=1.5]  (342.12,189.79) .. controls (340.55,191.67) and (338.3,192.85) .. (335.78,192.85) .. controls (333.18,192.85) and (330.85,191.59) .. (329.26,189.61) -- (335.73,183.44) -- cycle ; \draw  [color={rgb, 255:red, 0; green, 0; blue, 255 }  ,draw opacity=1 ][line width=1.5]  (342.12,189.79) .. controls (340.55,191.67) and (338.3,192.85) .. (335.78,192.85) .. controls (333.18,192.85) and (330.85,191.59) .. (329.26,189.61) ;  
\draw [color={rgb, 255:red, 0; green, 0; blue, 255 }  ,draw opacity=1 ][line width=1.5]    (341.54,162.52) -- (364.99,190.36) ;
\draw [color={rgb, 255:red, 0; green, 0; blue, 255 }  ,draw opacity=1 ][line width=1.5]    (350.8,179.42) -- (341.6,190.36) ;
\draw [color={rgb, 255:red, 0; green, 0; blue, 255 }  ,draw opacity=1 ][line width=1.5]    (364.99,162.52) -- (355.78,173.47) ;

\draw [color={rgb, 255:red, 0; green, 0; blue, 255 }  ,draw opacity=1 ][line width=1.5]    (271.21,162.52) .. controls (260.9,151.51) and (259.4,130.54) .. (271.21,119.77) ;
\draw [color={rgb, 255:red, 0; green, 0; blue, 255 }  ,draw opacity=1 ][line width=1.5]    (364.99,162.52) .. controls (375.87,147.32) and (375.12,130.54) .. (364.99,119.77) ;
\draw [color={rgb, 255:red, 0; green, 0; blue, 255 }  ,draw opacity=1 ][line width=1.5]    (364.99,190.36) .. controls (374.93,203.52) and (401.93,161.58) .. (383.74,119.77) ;
\draw [color={rgb, 255:red, 0; green, 0; blue, 255 }  ,draw opacity=1 ][line width=1.5]    (271.27,190.64) .. controls (261.28,206.87) and (233.52,161.16) .. (252.46,119.77) ;
\draw [color={rgb, 255:red, 0; green, 0; blue, 255 }  ,draw opacity=1 ][line width=1.5]    (271.21,119.77) .. controls (289.41,98.67) and (346.8,99.09) .. (364.99,119.77) ;
\draw [color={rgb, 255:red, 0; green, 0; blue, 255 }  ,draw opacity=1 ][line width=1.5]    (252.46,119.77) .. controls (270.65,78.11) and (365.18,78.11) .. (383.74,119.77) ;

\draw (313,126.4) node [anchor=north west][inner sep=0.75pt]  [color={rgb, 255:red, 144; green, 19; blue, 254 }  ,opacity=1 ]  {$c$};
\draw (419.91,138.01) node [anchor=north west][inner sep=0.75pt]  [color={rgb, 255:red, 144; green, 19; blue, 254 }  ,opacity=1 ]  {$\Gamma $};
\draw (368,76.4) node [anchor=north west][inner sep=0.75pt]  [color={rgb, 255:red, 0; green, 0; blue, 255 }  ,opacity=1 ]  {$\Lambda $};
\draw (222.59,52.4) node [anchor=north west][inner sep=0.75pt]    {$\Sigma _{+}$};
\draw (221.59,222.4) node [anchor=north west][inner sep=0.75pt]    {$\Sigma _{-}$};

\end{tikzpicture}

%% file: images/trefoil_Seifert_resolution_2.tex
\begin{tikzpicture}[x=0.75pt,y=0.75pt,yscale=-1,xscale=1]

\draw  [fill={rgb, 255:red, 246; green, 246; blue, 246 }  ,fill opacity=1 ][line width=0.75]  (371.69,51) -- (567.88,51) -- (567.88,247.19) -- (371.69,247.19) -- cycle ;
\draw  [draw opacity=0][fill={rgb, 255:red, 155; green, 155; blue, 155 }  ,fill opacity=1 ] (470.09,184.23) .. controls (467.09,188.23) and (459.01,195.32) .. (454.01,197.32) .. controls (448.01,195.32) and (440.09,186.23) .. (436.09,184.23) .. controls (432.09,183.23) and (423.82,202.18) .. (416.01,197.32) .. controls (402.01,185.32) and (397.49,170.47) .. (399.01,148.98) .. controls (443.01,147.32) and (501.01,147.22) .. (542.01,149.98) .. controls (541.99,173.48) and (537.88,188.63) .. (524.01,197.32) .. controls (517.01,194.32) and (507.09,184.23) .. (506.09,184.23) .. controls (505.09,184.23) and (492.55,197.49) .. (488.01,198.32) .. controls (482.01,194.32) and (475.09,185.23) .. (470.09,184.23) -- cycle ;
\draw  [draw opacity=0][fill={rgb, 255:red, 74; green, 74; blue, 74 }  ,fill opacity=1 ] (471.09,178.23) .. controls (466.09,176.23) and (457.78,162.45) .. (454.01,163.32) .. controls (449.01,163.32) and (442.09,173.23) .. (436.09,177.23) .. controls (429.09,169.23) and (415.87,159.12) .. (417.02,149.72) .. controls (450.25,149) and (494.05,148.96) .. (525.01,150.16) .. controls (525,160.44) and (513.09,172.23) .. (507.09,177.23) .. controls (501.09,176.23) and (497.01,167.32) .. (491.01,163.32) .. controls (483.01,163.32) and (477.09,172.23) .. (471.09,178.23) -- cycle ;
\draw  [color={rgb, 255:red, 211; green, 211; blue, 211 }  ,draw opacity=1 ][fill={rgb, 255:red, 216; green, 216; blue, 216 }  ,fill opacity=1 ] (467.01,91.98) .. controls (495.83,89.22) and (545.01,106.98) .. (542.01,149.98) .. controls (499.01,148.98) and (438.01,147.98) .. (399.01,148.98) .. controls (399.01,101.98) and (444.01,91.98) .. (467.01,91.98) -- cycle ;
\draw [color={rgb, 255:red, 144; green, 19; blue, 254 }  ,draw opacity=1 ][line width=1.5]    (371.69,149.1) -- (567.88,149.1) ;
\draw  [line width=0.75]  (366.09,114.19) .. controls (369.09,109.46) and (370.89,104.73) .. (371.49,100) .. controls (372.09,104.73) and (373.88,109.46) .. (376.88,114.19) ;
\draw  [line width=0.75]  (562.09,114.19) .. controls (565.09,109.46) and (566.89,104.73) .. (567.49,100) .. controls (568.09,104.73) and (569.88,109.46) .. (572.88,114.19) ;

\draw  [color={rgb, 255:red, 144; green, 19; blue, 254 }  ,draw opacity=1 ][line width=0.75]  (454.39,46.2) .. controls (459.12,49.2) and (463.85,51) .. (468.58,51.6) .. controls (463.85,52.2) and (459.12,53.99) .. (454.39,56.99) ;
\draw  [color={rgb, 255:red, 144; green, 19; blue, 254 }  ,draw opacity=1 ][line width=0.75]  (454.39,242.2) .. controls (459.12,245.2) and (463.85,247) .. (468.58,247.6) .. controls (463.85,248.2) and (459.12,249.99) .. (454.39,252.99) ;

\draw  [color={rgb, 255:red, 144; green, 19; blue, 254 }  ,draw opacity=1 ][line width=0.75]  (466.39,46.2) .. controls (471.12,49.2) and (475.85,51) .. (480.58,51.6) .. controls (475.85,52.2) and (471.12,53.99) .. (466.39,56.99) ;
\draw  [color={rgb, 255:red, 144; green, 19; blue, 254 }  ,draw opacity=1 ][line width=0.75]  (466.39,242.2) .. controls (471.12,245.2) and (475.85,247) .. (480.58,247.6) .. controls (475.85,248.2) and (471.12,249.99) .. (466.39,252.99) ;

\draw [color={rgb, 255:red, 144; green, 19; blue, 254 }  ,draw opacity=1 ][line width=1.5]    (371.69,51) -- (567.88,51) ;
\draw [color={rgb, 255:red, 144; green, 19; blue, 254 }  ,draw opacity=1 ][line width=1.5]    (371.69,247.19) -- (567.88,247.19) ;
\draw  [draw opacity=0][line width=1.5]  (446.84,167.69) .. controls (448.4,165.46) and (450.86,164.03) .. (453.63,164.03) .. controls (456.15,164.03) and (458.42,165.21) .. (460,167.09) -- (453.68,173.45) -- cycle ; \draw  [color={rgb, 255:red, 255; green, 0; blue, 0 }  ,draw opacity=1 ][line width=1.5]  (446.84,167.69) .. controls (448.4,165.46) and (450.86,164.03) .. (453.63,164.03) .. controls (456.15,164.03) and (458.42,165.21) .. (460,167.09) ;  
\draw  [draw opacity=0][line width=1.5]  (460.05,193.79) .. controls (458.48,195.67) and (456.23,196.85) .. (453.71,196.85) .. controls (451.11,196.85) and (448.78,195.59) .. (447.19,193.61) -- (453.66,187.44) -- cycle ; \draw  [color={rgb, 255:red, 255; green, 0; blue, 0 }  ,draw opacity=1 ][line width=1.5]  (460.05,193.79) .. controls (458.48,195.67) and (456.23,196.85) .. (453.71,196.85) .. controls (451.11,196.85) and (448.78,195.59) .. (447.19,193.61) ;  
\draw  [draw opacity=0][line width=1.5]  (482,167.69) .. controls (483.56,165.46) and (486.02,164.03) .. (488.8,164.03) .. controls (491.31,164.03) and (493.58,165.21) .. (495.17,167.09) -- (488.84,173.45) -- cycle ; \draw  [color={rgb, 255:red, 255; green, 0; blue, 0 }  ,draw opacity=1 ][line width=1.5]  (482,167.69) .. controls (483.56,165.46) and (486.02,164.03) .. (488.8,164.03) .. controls (491.31,164.03) and (493.58,165.21) .. (495.17,167.09) ;  
\draw  [draw opacity=0][line width=1.5]  (494.7,194.36) .. controls (493.13,196.24) and (490.87,197.42) .. (488.36,197.42) .. controls (485.76,197.42) and (483.42,196.16) .. (481.83,194.18) -- (488.31,188.01) -- cycle ; \draw  [color={rgb, 255:red, 255; green, 0; blue, 0 }  ,draw opacity=1 ][line width=1.5]  (494.7,194.36) .. controls (493.13,196.24) and (490.87,197.42) .. (488.36,197.42) .. controls (485.76,197.42) and (483.42,196.16) .. (481.83,194.18) ;  
\draw [color={rgb, 255:red, 255; green, 0; blue, 0 }  ,draw opacity=1 ][line width=1.5]    (424.31,166.52) .. controls (414,155.51) and (417.01,149.98) .. (417.01,149.98) ;
\draw [color={rgb, 255:red, 255; green, 0; blue, 0 }  ,draw opacity=1 ][line width=1.5]    (518.08,166.52) .. controls (528.96,151.32) and (524.01,148.98) .. (524.01,148.98) ;
\draw [color={rgb, 255:red, 255; green, 0; blue, 0 }  ,draw opacity=1 ][line width=1.5]    (518.08,194.36) .. controls (528.02,207.52) and (555.03,165.58) .. (536.84,123.77) ;
\draw [color={rgb, 255:red, 255; green, 0; blue, 0 }  ,draw opacity=1 ][line width=1.5]    (424.37,194.64) .. controls (414.37,210.87) and (386.61,165.16) .. (405.55,123.77) ;
\draw [color={rgb, 255:red, 255; green, 0; blue, 0 }  ,draw opacity=1 ][line width=1.5]    (405.55,123.77) .. controls (417.52,96.37) and (462.53,86.99) .. (497.21,95.64) .. controls (515.24,100.14) and (530.48,109.52) .. (536.84,123.77) ;
\draw  [color={rgb, 255:red, 255; green, 0; blue, 0 }  ,draw opacity=1 ][line width=0.75]  (421.93,171.89) .. controls (422.07,167.65) and (421.35,164.17) .. (419.75,161.43) .. controls (422.22,163.42) and (425.55,164.68) .. (429.75,165.2) ;
\draw  [color={rgb, 255:red, 255; green, 0; blue, 0 }  ,draw opacity=1 ][line width=0.75]  (397.31,176) .. controls (399.09,172.16) and (399.78,168.67) .. (399.38,165.52) .. controls (400.87,168.32) and (403.45,170.78) .. (407.12,172.89) ;
\draw  [color={rgb, 255:red, 255; green, 0; blue, 0 }  ,draw opacity=1 ][line width=0.75]  (467.12,87.84) .. controls (470.24,90.7) and (473.36,92.41) .. (476.48,92.98) .. controls (473.36,93.56) and (470.24,95.27) .. (467.12,98.13) ;
\draw  [fill={rgb, 255:red, 246; green, 246; blue, 246 }  ,fill opacity=1 ][line width=0.75]  (103.59,53) -- (299.79,53) -- (299.79,249.19) -- (103.59,249.19) -- cycle ;
\draw [color={rgb, 255:red, 144; green, 19; blue, 254 }  ,draw opacity=1 ][line width=1.5]    (103.59,151.1) -- (299.79,151.1) ;
\draw  [line width=0.75]  (98,116.19) .. controls (101,111.46) and (102.8,106.73) .. (103.39,102) .. controls (104,106.73) and (105.79,111.46) .. (108.79,116.19) ;
\draw  [line width=0.75]  (294,116.19) .. controls (297,111.46) and (298.8,106.73) .. (299.39,102) .. controls (300,106.73) and (301.79,111.46) .. (304.79,116.19) ;
\draw  [color={rgb, 255:red, 144; green, 19; blue, 254 }  ,draw opacity=1 ][line width=0.75]  (186.3,48.2) .. controls (191.03,51.2) and (195.76,53) .. (200.49,53.6) .. controls (195.76,54.2) and (191.03,55.99) .. (186.3,58.99) ;
\draw  [color={rgb, 255:red, 144; green, 19; blue, 254 }  ,draw opacity=1 ][line width=0.75]  (186.3,244.2) .. controls (191.03,247.2) and (195.76,249) .. (200.49,249.6) .. controls (195.76,250.2) and (191.03,251.99) .. (186.3,254.99) ;

\draw  [color={rgb, 255:red, 144; green, 19; blue, 254 }  ,draw opacity=1 ][line width=0.75]  (198.3,48.2) .. controls (203.03,51.2) and (207.76,53) .. (212.49,53.6) .. controls (207.76,54.2) and (203.03,55.99) .. (198.3,58.99) ;
\draw  [color={rgb, 255:red, 144; green, 19; blue, 254 }  ,draw opacity=1 ][line width=0.75]  (198.3,244.2) .. controls (203.03,247.2) and (207.76,249) .. (212.49,249.6) .. controls (207.76,250.2) and (203.03,251.99) .. (198.3,254.99) ;

\draw [color={rgb, 255:red, 144; green, 19; blue, 254 }  ,draw opacity=1 ][line width=1.5]    (103.59,249.19) -- (299.79,249.19) ;
\draw [color={rgb, 255:red, 255; green, 0; blue, 0 }  ,draw opacity=1 ][line width=1.5]    (156.21,168.52) -- (179.66,196.36) ;
\draw [color={rgb, 255:red, 255; green, 0; blue, 0 }  ,draw opacity=1 ][line width=1.5]    (165.48,185.69) -- (156.27,196.64) ;
\draw  [draw opacity=0][line width=1.5]  (178.74,169.69) .. controls (180.31,167.46) and (182.76,166.03) .. (185.54,166.03) .. controls (188.05,166.03) and (190.32,167.21) .. (191.91,169.09) -- (185.58,175.45) -- cycle ; \draw  [color={rgb, 255:red, 255; green, 0; blue, 0 }  ,draw opacity=1 ][line width=1.5]  (178.74,169.69) .. controls (180.31,167.46) and (182.76,166.03) .. (185.54,166.03) .. controls (188.05,166.03) and (190.32,167.21) .. (191.91,169.09) ;  
\draw [color={rgb, 255:red, 255; green, 0; blue, 0 }  ,draw opacity=1 ][line width=1.5]    (179.66,168.52) -- (170.45,179.47) ;
\draw  [draw opacity=0][line width=1.5]  (191.96,195.79) .. controls (190.39,197.67) and (188.13,198.85) .. (185.62,198.85) .. controls (183.02,198.85) and (180.68,197.59) .. (179.09,195.61) -- (185.57,189.44) -- cycle ; \draw  [color={rgb, 255:red, 255; green, 0; blue, 0 }  ,draw opacity=1 ][line width=1.5]  (191.96,195.79) .. controls (190.39,197.67) and (188.13,198.85) .. (185.62,198.85) .. controls (183.02,198.85) and (180.68,197.59) .. (179.09,195.61) ;  
\draw [color={rgb, 255:red, 255; green, 0; blue, 0 }  ,draw opacity=1 ][line width=1.5]    (191.38,168.52) -- (214.82,196.36) ;
\draw [color={rgb, 255:red, 255; green, 0; blue, 0 }  ,draw opacity=1 ][line width=1.5]    (200.64,185.42) -- (191.44,196.36) ;
\draw [color={rgb, 255:red, 255; green, 0; blue, 0 }  ,draw opacity=1 ][line width=1.5]    (214.82,168.52) -- (205.62,179.47) ;
\draw  [draw opacity=0][line width=1.5]  (213.91,169.69) .. controls (215.47,167.46) and (217.93,166.03) .. (220.7,166.03) .. controls (223.22,166.03) and (225.49,167.21) .. (227.07,169.09) -- (220.75,175.45) -- cycle ; \draw  [color={rgb, 255:red, 255; green, 0; blue, 0 }  ,draw opacity=1 ][line width=1.5]  (213.91,169.69) .. controls (215.47,167.46) and (217.93,166.03) .. (220.7,166.03) .. controls (223.22,166.03) and (225.49,167.21) .. (227.07,169.09) ;  
\draw  [draw opacity=0][line width=1.5]  (227.12,195.79) .. controls (225.55,197.67) and (223.3,198.85) .. (220.78,198.85) .. controls (218.18,198.85) and (215.85,197.59) .. (214.26,195.61) -- (220.73,189.44) -- cycle ; \draw  [color={rgb, 255:red, 255; green, 0; blue, 0 }  ,draw opacity=1 ][line width=1.5]  (227.12,195.79) .. controls (225.55,197.67) and (223.3,198.85) .. (220.78,198.85) .. controls (218.18,198.85) and (215.85,197.59) .. (214.26,195.61) ;  
\draw [color={rgb, 255:red, 255; green, 0; blue, 0 }  ,draw opacity=1 ][line width=1.5]    (226.54,168.52) -- (249.99,196.36) ;
\draw [color={rgb, 255:red, 255; green, 0; blue, 0 }  ,draw opacity=1 ][line width=1.5]    (235.8,185.42) -- (226.6,196.36) ;
\draw [color={rgb, 255:red, 255; green, 0; blue, 0 }  ,draw opacity=1 ][line width=1.5]    (249.99,168.52) -- (240.78,179.47) ;

\draw [color={rgb, 255:red, 255; green, 0; blue, 0 }  ,draw opacity=1 ][line width=1.5]    (156.21,168.52) .. controls (149.67,158.24) and (149.67,152.24) .. (149.67,151.24) ;
\draw [color={rgb, 255:red, 255; green, 0; blue, 0 }  ,draw opacity=1 ][line width=1.5]    (249.99,168.52) .. controls (256.67,156.24) and (255.67,151.24) .. (255.67,151.24) ;
\draw [color={rgb, 255:red, 255; green, 0; blue, 0 }  ,draw opacity=1 ][line width=1.5]    (249.99,196.36) .. controls (259.93,209.52) and (286.93,167.58) .. (268.74,125.77) ;
\draw [color={rgb, 255:red, 255; green, 0; blue, 0 }  ,draw opacity=1 ][line width=1.5]    (156.27,196.64) .. controls (146.28,212.87) and (118.52,167.16) .. (137.46,125.77) ;
\draw [color={rgb, 255:red, 0; green, 0; blue, 255 }  ,draw opacity=1 ][line width=1.5]    (149.67,151.24) .. controls (145.67,99.24) and (257.67,95.24) .. (255.67,151.24) ;
\draw [color={rgb, 255:red, 255; green, 0; blue, 0 }  ,draw opacity=1 ][line width=1.5]    (137.46,125.77) .. controls (155.65,84.11) and (250.18,84.11) .. (268.74,125.77) ;
\draw  [color={rgb, 255:red, 255; green, 0; blue, 0 }  ,draw opacity=1 ][line width=0.75]  (256.62,166.61) .. controls (252.56,167.82) and (249.49,169.62) .. (247.39,172) .. controls (248.5,169.03) and (248.64,165.47) .. (247.79,161.32) ;
\draw  [color={rgb, 255:red, 255; green, 0; blue, 0 }  ,draw opacity=1 ][line width=0.75]  (225.48,203.13) .. controls (222.36,200.27) and (219.24,198.56) .. (216.12,197.98) .. controls (219.24,197.41) and (222.36,195.7) .. (225.48,192.84) ;
\draw  [color={rgb, 255:red, 255; green, 0; blue, 0 }  ,draw opacity=1 ][line width=0.75]  (190.48,203.13) .. controls (187.36,200.27) and (184.24,198.56) .. (181.12,197.98) .. controls (184.24,197.41) and (187.36,195.7) .. (190.48,192.84) ;
\draw  [color={rgb, 255:red, 255; green, 0; blue, 0 }  ,draw opacity=1 ][line width=0.75]  (154.3,174.78) .. controls (154.83,170.57) and (154.43,167.04) .. (153.09,164.16) .. controls (155.36,166.38) and (158.56,167.93) .. (162.7,168.83) ;
\draw  [color={rgb, 255:red, 255; green, 0; blue, 0 }  ,draw opacity=1 ][line width=0.75]  (197.12,89.84) .. controls (200.24,92.7) and (203.36,94.41) .. (206.48,94.98) .. controls (203.36,95.56) and (200.24,97.27) .. (197.12,100.13) ;
\draw  [color={rgb, 255:red, 255; green, 0; blue, 0 }  ,draw opacity=1 ][line width=0.75]  (279.91,156.27) .. controls (277.07,159.41) and (275.38,162.54) .. (274.83,165.67) .. controls (274.23,162.55) and (272.5,159.44) .. (269.62,156.33) ;
\draw  [color={rgb, 255:red, 255; green, 0; blue, 0 }  ,draw opacity=1 ][line width=0.75]  (126.79,165.81) .. controls (129.56,162.61) and (131.18,159.44) .. (131.66,156.3) .. controls (132.33,159.41) and (134.13,162.48) .. (137.07,165.52) ;
\draw  [color={rgb, 255:red, 255; green, 0; blue, 0 }  ,draw opacity=1 ][line width=0.75]  (225.48,172.13) .. controls (222.36,169.27) and (219.24,167.56) .. (216.12,166.98) .. controls (219.24,166.41) and (222.36,164.7) .. (225.48,161.84) ;
\draw  [color={rgb, 255:red, 255; green, 0; blue, 0 }  ,draw opacity=1 ][line width=0.75]  (190.48,172.13) .. controls (187.36,169.27) and (184.24,167.56) .. (181.12,166.98) .. controls (184.24,166.41) and (187.36,164.7) .. (190.48,161.84) ;
\draw [color={rgb, 255:red, 144; green, 19; blue, 254 }  ,draw opacity=1 ][line width=1.5]    (103.59,53) -- (299.79,53) ;
\draw [color={rgb, 255:red, 255; green, 0; blue, 0 }  ,draw opacity=1 ][line width=1.5]    (424.31,166.52) .. controls (442.09,182.23) and (431.09,181.23) .. (446.84,167.69) ;
\draw [color={rgb, 255:red, 255; green, 0; blue, 0 }  ,draw opacity=1 ][line width=1.5]    (495.17,167.09) .. controls (512.95,182.8) and (502.34,180.06) .. (518.08,166.52) ;
\draw [color={rgb, 255:red, 255; green, 0; blue, 0 }  ,draw opacity=1 ][line width=1.5]    (460,167.09) .. controls (477.78,182.8) and (466.26,181.23) .. (482,167.69) ;
\draw [color={rgb, 255:red, 255; green, 0; blue, 0 }  ,draw opacity=1 ][line width=1.5]    (424.37,194.64) .. controls (441.8,178.44) and (430.09,179.23) .. (447.19,193.61) ;
\draw [color={rgb, 255:red, 255; green, 0; blue, 0 }  ,draw opacity=1 ][line width=1.5]    (460.05,193.79) .. controls (477.48,177.59) and (464.74,179.8) .. (481.83,194.18) ;
\draw [color={rgb, 255:red, 255; green, 0; blue, 0 }  ,draw opacity=1 ][line width=1.5]    (494.7,194.36) .. controls (512.13,178.16) and (500.99,179.99) .. (518.08,194.36) ;

\draw (188.95,276) node [anchor=north west][inner sep=0.75pt]   [align=left] {(a)};
\draw (466.05,276) node [anchor=north west][inner sep=0.75pt]   [align=left] {(b)};
\draw (198,132.4) node [anchor=north west][inner sep=0.75pt]  [color={rgb, 255:red, 144; green, 19; blue, 254 }  ,opacity=1 ]  {$c$};
\draw (304.91,144.01) node [anchor=north west][inner sep=0.75pt]  [color={rgb, 255:red, 144; green, 19; blue, 254 }  ,opacity=1 ]  {$\Gamma $};
\draw (107.59,58.4) node [anchor=north west][inner sep=0.75pt]    {$\Sigma _{+}$};
\draw (106.59,228.4) node [anchor=north west][inner sep=0.75pt]    {$\Sigma _{-}$};
\draw (251,83.4) node [anchor=north west][inner sep=0.75pt]  [color={rgb, 255:red, 255; green, 0; blue, 0 }  ,opacity=1 ]  {$\gamma _{c}$};
\draw (573,142.01) node [anchor=north west][inner sep=0.75pt]  [color={rgb, 255:red, 144; green, 19; blue, 254 }  ,opacity=1 ]  {$\Gamma $};
\draw (375.69,56.4) node [anchor=north west][inner sep=0.75pt]    {$\Sigma _{+}$};
\draw (374.69,226.4) node [anchor=north west][inner sep=0.75pt]    {$\Sigma _{-}$};
\draw (463,114.71) node [anchor=north west][inner sep=0.75pt]    {$A$};
\draw (464.5,150.71) node [anchor=north west][inner sep=0.75pt]  [color={rgb, 255:red, 255; green, 255; blue, 255 }  ,opacity=1 ]  {$B$};
\draw (516,173.09) node [anchor=north west][inner sep=0.75pt]    {$C$};
\draw (520,83.4) node [anchor=north west][inner sep=0.75pt]  [color={rgb, 255:red, 255; green, 0; blue, 0 }  ,opacity=1 ]  {$\gamma _{c}$};

\end{tikzpicture}

%% file: images/legendrian_trefoil_surface.tex
\begin{tikzpicture}[x=0.55pt,y=0.55pt,yscale=-1,xscale=1]

\draw  [draw opacity=0][dash pattern={on 5.63pt off 4.5pt}][line width=1.5]  (317.01,198.58) .. controls (317.01,198.58) and (317.01,198.58) .. (317.01,198.58) .. controls (317.01,198.58) and (317.01,198.58) .. (317.01,198.58) .. controls (305.05,198.58) and (295.36,172.62) .. (295.36,140.6) .. controls (295.36,108.58) and (305.05,82.62) .. (317.01,82.62) -- (317.01,140.6) -- cycle ; \draw  [color={rgb, 255:red, 144; green, 19; blue, 254 }  ,draw opacity=1 ][dash pattern={on 5.63pt off 4.5pt}][line width=1.5]  (317.01,198.58) .. controls (317.01,198.58) and (317.01,198.58) .. (317.01,198.58) .. controls (317.01,198.58) and (317.01,198.58) .. (317.01,198.58) .. controls (305.05,198.58) and (295.36,172.62) .. (295.36,140.6) .. controls (295.36,108.58) and (305.05,82.62) .. (317.01,82.62) ;  
\draw  [color={rgb, 255:red, 255; green, 255; blue, 255 }  ,draw opacity=1 ][line width=1.5] [line join = round][line cap = round] (235.99,138.65) .. controls (231.55,138.41) and (236.74,136.85) .. (234.69,136.29) .. controls (229.49,134.84) and (212.67,139.28) .. (226.35,140.98) .. controls (232.56,141.76) and (237.32,141.34) .. (242.28,140.14) .. controls (242.47,140.09) and (245.11,139.34) .. (245.01,139.15) .. controls (244.51,138.25) and (237.71,135.22) .. (231.64,136.69) .. controls (225.32,138.22) and (239.83,138.86) .. (238.25,138.77) .. controls (234.86,138.59) and (231.48,138.4) .. (228.1,138.22) ;
\draw [line width=0.75]    (216.83,139.24) .. controls (224.82,131.74) and (240.81,130.88) .. (248.83,139.19) ;
\draw [line width=0.75]    (211.99,134.94) .. controls (224.05,148.16) and (243.85,146.7) .. (253.22,134.06) ;
\draw  [color={rgb, 255:red, 255; green, 255; blue, 255 }  ,draw opacity=1 ][line width=1.5] [line join = round][line cap = round] (412.99,139.64) .. controls (408.55,139.4) and (413.74,137.84) .. (411.69,137.28) .. controls (406.49,135.83) and (389.67,140.27) .. (403.35,141.97) .. controls (409.56,142.75) and (414.32,142.33) .. (419.28,141.13) .. controls (419.47,141.08) and (422.11,140.33) .. (422.01,140.14) .. controls (421.51,139.24) and (414.71,136.21) .. (408.64,137.68) .. controls (402.32,139.21) and (416.83,139.85) .. (415.25,139.76) .. controls (411.86,139.58) and (408.48,139.39) .. (405.1,139.21) ;
\draw [line width=0.75]    (393.83,140.23) .. controls (401.82,132.73) and (417.81,131.87) .. (425.83,140.18) ;
\draw [line width=0.75]    (388.99,135.93) .. controls (401.05,149.15) and (420.85,147.69) .. (430.22,135.05) ;
\draw  [draw opacity=0][line width=1.5]  (318.22,82.79) .. controls (318.22,82.79) and (318.22,82.79) .. (318.22,82.79) .. controls (318.22,82.79) and (318.22,82.79) .. (318.22,82.79) .. controls (329.51,82.79) and (338.67,108.71) .. (338.67,140.68) .. controls (338.67,172.66) and (329.51,198.58) .. (318.22,198.58) -- (318.22,140.68) -- cycle ; \draw  [color={rgb, 255:red, 144; green, 19; blue, 254 }  ,draw opacity=1 ][line width=1.5]  (318.22,82.79) .. controls (318.22,82.79) and (318.22,82.79) .. (318.22,82.79) .. controls (318.22,82.79) and (318.22,82.79) .. (318.22,82.79) .. controls (329.51,82.79) and (338.67,108.71) .. (338.67,140.68) .. controls (338.67,172.66) and (329.51,198.58) .. (318.22,198.58) ;  
\draw  [line width=0.75]  (120.56,141.8) .. controls (120.56,180.39) and (162.56,220.97) .. (202.56,220.97) .. controls (242.56,220.97) and (286.77,199.31) .. (321.56,199.2) .. controls (356.36,199.08) and (402.56,220.97) .. (442.41,220.79) .. controls (482.25,220.61) and (521.41,180.21) .. (521.41,140.63) .. controls (521.41,101.04) and (492.97,62.6) .. (441.41,62.44) .. controls (389.84,62.28) and (360.56,82.42) .. (322.56,82.42) .. controls (284.56,82.42) and (242.56,60.64) .. (201.56,61.63) .. controls (160.56,62.62) and (120.56,103.2) .. (120.56,141.8) -- cycle ;

\draw (459.26,74.33) node [anchor=north west][inner sep=0.75pt]    {$\Sigma _{+}$};
\draw (175.76,76.48) node [anchor=north west][inner sep=0.75pt]    {$\Sigma _{-}$};
\draw (340.33,102.19) node [anchor=north west][inner sep=0.75pt]  [color={rgb, 255:red, 144; green, 19; blue, 254 }  ,opacity=1 ]  {$\Gamma $};

\end{tikzpicture}

%% file: images/legendrian_trefoil_1.tex
\begin{tikzpicture}[x=0.55pt,y=0.55pt,yscale=-1,xscale=1]

\draw [color={rgb, 255:red, 0; green, 0; blue, 255 }  ,draw opacity=1 ][line width=1.5]    (298.77,176.7) .. controls (288.13,183.46) and (271.31,177.01) .. (267.31,174.34) .. controls (263.31,171.68) and (247.32,154.98) .. (257.8,133.18) .. controls (268.28,111.38) and (289.95,94.73) .. (317.23,96.54) .. controls (344.52,98.34) and (362.06,129.01) .. (376.12,132) ;
\draw [color={rgb, 255:red, 0; green, 0; blue, 255 }  ,draw opacity=1 ][line width=1.5]    (297.33,138.35) .. controls (301.54,131.23) and (319.36,122.91) .. (329.98,137.01) .. controls (340.59,151.11) and (360.65,170.34) .. (377.53,177.36) ;
\draw [color={rgb, 255:red, 0; green, 0; blue, 255 }  ,draw opacity=1 ][line width=1.5]    (269.8,123.78) .. controls (283.1,128.93) and (301.8,150.71) .. (308.2,152.6) .. controls (314.6,154.49) and (323.32,152.93) .. (327.73,149.32) ;
\draw [color={rgb, 255:red, 0; green, 0; blue, 255 }  ,draw opacity=1 ][line width=1.5]    (258.91,119.68) .. controls (221.98,105.01) and (229.44,174.68) .. (251.38,198.68) .. controls (273.31,222.68) and (318.63,222.68) .. (335.31,206.68) .. controls (351.98,190.68) and (353.36,182.34) .. (359.31,174.34) ;
\draw  [draw opacity=0][line width=1.5]  (376.36,114.2) .. controls (378.95,113.23) and (382.01,112.7) .. (385.29,112.77) .. controls (394.51,112.97) and (401.94,117.85) .. (401.88,123.67) .. controls (401.83,129.5) and (394.32,134.05) .. (385.1,133.85) .. controls (381.82,133.78) and (378.77,133.12) .. (376.2,132.03) -- (385.19,123.31) -- cycle ; \draw  [color={rgb, 255:red, 0; green, 0; blue, 255 }  ,draw opacity=1 ][line width=1.5]  (376.36,114.2) .. controls (378.95,113.23) and (382.01,112.7) .. (385.29,112.77) .. controls (394.51,112.97) and (401.94,117.85) .. (401.88,123.67) .. controls (401.83,129.5) and (394.32,134.05) .. (385.1,133.85) .. controls (381.82,133.78) and (378.77,133.12) .. (376.2,132.03) ;  
\draw [color={rgb, 255:red, 0; green, 0; blue, 255 }  ,draw opacity=1 ][line width=1.5]    (346.36,142.85) .. controls (352.06,138.72) and (356.95,131.67) .. (359.99,127.97) ;
\draw [color={rgb, 255:red, 0; green, 0; blue, 255 }  ,draw opacity=1 ][line width=1.5]    (366.06,120.57) .. controls (373.2,115.75) and (369.45,117.73) .. (376.24,114.25) ;
\draw [color={rgb, 255:red, 0; green, 0; blue, 255 }  ,draw opacity=1 ][line width=1.5]    (293.46,146.34) .. controls (286.31,155.68) and (299.8,170.68) .. (304.56,176.01) .. controls (309.33,181.34) and (316.09,184.68) .. (322.05,186.01) .. controls (328.01,187.34) and (338.73,185.01) .. (338.73,175.68) .. controls (338.73,166.34) and (321.24,160.34) .. (309.33,169.01) ;
\draw  [draw opacity=0][line width=1.5]  (377.7,159.53) .. controls (380.28,158.56) and (383.35,158.03) .. (386.62,158.1) .. controls (395.84,158.3) and (403.27,163.18) .. (403.22,169.01) .. controls (403.16,174.83) and (395.65,179.39) .. (386.43,179.19) .. controls (383.15,179.11) and (380.1,178.45) .. (377.53,177.36) -- (386.53,168.64) -- cycle ; \draw  [color={rgb, 255:red, 0; green, 0; blue, 255 }  ,draw opacity=1 ][line width=1.5]  (377.7,159.53) .. controls (380.28,158.56) and (383.35,158.03) .. (386.62,158.1) .. controls (395.84,158.3) and (403.27,163.18) .. (403.22,169.01) .. controls (403.16,174.83) and (395.65,179.39) .. (386.43,179.19) .. controls (383.15,179.11) and (380.1,178.45) .. (377.53,177.36) ;  
\draw [color={rgb, 255:red, 0; green, 0; blue, 255 }  ,draw opacity=1 ][line width=1.5]    (367.51,165.86) .. controls (374.65,161.04) and (370.9,163.02) .. (377.7,159.53) ;

\draw (410.88,157.33) node [anchor=north west][inner sep=0.75pt]  [color={rgb, 255:red, 0; green, 0; blue, 255 }  ,opacity=1 ]  {$\Lambda '$};
\draw (364.25,178.08) node [anchor=north west][inner sep=0.75pt]    {$c_l$};

\end{tikzpicture}

%% file: images/legendrian_trefoil_2.tex
\begin{tikzpicture}[x=0.55pt,y=0.55pt,yscale=-1,xscale=1]

\draw [color={rgb, 255:red, 0; green, 0; blue, 255 }  ,draw opacity=1 ][line width=1.5]    (296.82,145.72) .. controls (298.4,137.33) and (301.01,132.09) .. (306.25,129.99) .. controls (311.49,127.89) and (318.82,126.84) .. (330.35,134.18) .. controls (341.87,141.52) and (371.2,152.01) .. (413.11,147.81) ;
\draw [color={rgb, 255:red, 0; green, 0; blue, 255 }  ,draw opacity=1 ][line width=1.5]    (262.78,144.29) .. controls (278.49,143.24) and (306.25,155.15) .. (313.59,154.1) .. controls (320.92,153.06) and (329.3,147.81) .. (332.44,142.57) ;
\draw [color={rgb, 255:red, 0; green, 0; blue, 255 }  ,draw opacity=1 ][line width=1.5]    (216.16,144.67) .. controls (229.78,144.67) and (242.35,141.52) .. (249.68,142.57) ;
\draw  [draw opacity=0][line width=1.5]  (354.11,95.31) .. controls (356.09,93.39) and (358.69,91.71) .. (361.73,90.5) .. controls (370.2,87.14) and (378.86,88.78) .. (381.08,94.16) .. controls (383.3,99.54) and (378.23,106.63) .. (369.77,109.99) .. controls (366.73,111.2) and (363.66,111.76) .. (360.88,111.74) -- (365.75,100.25) -- cycle ; \draw  [color={rgb, 255:red, 0; green, 0; blue, 255 }  ,draw opacity=1 ][line width=1.5]  (354.11,95.31) .. controls (356.09,93.39) and (358.69,91.71) .. (361.73,90.5) .. controls (370.2,87.14) and (378.86,88.78) .. (381.08,94.16) .. controls (383.3,99.54) and (378.23,106.63) .. (369.77,109.99) .. controls (366.73,111.2) and (363.66,111.76) .. (360.88,111.74) ;  
\draw [color={rgb, 255:red, 0; green, 0; blue, 255 }  ,draw opacity=1 ][line width=1.5]    (338.03,133.18) .. controls (341.61,127.19) and (343.31,118.81) .. (344.63,114.24) ;
\draw [color={rgb, 255:red, 0; green, 0; blue, 255 }  ,draw opacity=1 ][line width=1.5]    (347.27,105.09) .. controls (351.89,97.9) and (349.25,101.17) .. (354.07,95.34) ;
\draw [color={rgb, 255:red, 0; green, 0; blue, 255 }  ,draw opacity=1 ][line width=1.5]    (311.12,183.93) .. controls (304.55,194.69) and (286.61,198.31) .. (281.84,197.68) .. controls (277.07,197.04) and (255.36,189.1) .. (255.17,164.91) .. controls (254.98,140.72) and (267.1,116.24) .. (292.39,105.84) .. controls (317.68,95.44) and (346.94,115.25) .. (360.88,111.74) ;
\draw [color={rgb, 255:red, 0; green, 0; blue, 255 }  ,draw opacity=1 ][line width=1.5]    (296.02,158.71) .. controls (293.8,170.25) and (312.57,177.64) .. (319.22,180.28) .. controls (325.86,182.92) and (333.41,182.88) .. (339.33,181.4) .. controls (345.25,179.93) and (353.8,173.05) .. (349.63,164.7) .. controls (345.46,156.36) and (327.13,158.81) .. (320.35,171.89) ;

\draw (419.4,140.48) node [anchor=north west][inner sep=0.75pt]  [color={rgb, 255:red, 0; green, 0; blue, 255 }  ,opacity=1 ]  {$\Lambda ^{l}$};

\end{tikzpicture}

%% file: images/legendrian_trefoil_3.tex
\begin{tikzpicture}[x=0.65pt,y=0.65pt,yscale=-1,xscale=1]

\draw  [draw opacity=0][dash pattern={on 5.63pt off 4.5pt}][line width=1.5]  (403.95,370.7) .. controls (404.06,359.09) and (367.32,349.22) .. (321.88,348.66) .. controls (276.44,348.11) and (239.52,357.07) .. (239.4,368.69) -- (321.67,369.69) -- cycle ; \draw  [color={rgb, 255:red, 144; green, 19; blue, 254 }  ,draw opacity=1 ][dash pattern={on 5.63pt off 4.5pt}][line width=1.5]  (403.95,370.7) .. controls (404.06,359.09) and (367.32,349.22) .. (321.88,348.66) .. controls (276.44,348.11) and (239.52,357.07) .. (239.4,368.69) ;  
\draw  [draw opacity=0][fill={rgb, 255:red, 188; green, 188; blue, 188 }  ,fill opacity=1 ] (234.73,242.05) .. controls (239.81,221.98) and (283.81,215.31) .. (305.16,213.75) .. controls (307.31,238.76) and (304.65,265.76) .. (305.12,296.83) .. controls (274.48,303.39) and (247.81,306.39) .. (239.31,325.42) .. controls (240.65,286.76) and (245.15,268.64) .. (234.73,242.05) -- cycle ;
\draw  [draw opacity=0][fill={rgb, 255:red, 188; green, 188; blue, 188 }  ,fill opacity=1 ] (408.98,244.81) .. controls (400.65,221.09) and (353.98,215.76) .. (332.94,215.52) .. controls (330.61,240.52) and (333.49,267.52) .. (332.97,298.59) .. controls (360.65,301.09) and (396.76,304.52) .. (403.95,327.19) .. controls (404.65,289.42) and (399.81,287.5) .. (408.98,244.81) -- cycle ;
\draw  [draw opacity=0][fill={rgb, 255:red, 216; green, 216; blue, 216 }  ,fill opacity=1 ] (233.73,242.05) .. controls (259.66,274.41) and (298.8,270.09) .. (317.98,270.42) .. controls (319.31,294.42) and (318.39,326.77) .. (319.31,358.42) .. controls (280.65,353.09) and (255.81,349.98) .. (238.48,332.64) .. controls (234.98,293.2) and (237.15,259.31) .. (233.73,242.05) -- cycle ;
\draw  [draw opacity=0][fill={rgb, 255:red, 216; green, 216; blue, 216 }  ,fill opacity=1 ] (408.98,244.81) .. controls (392.15,275.59) and (339.48,270.09) .. (318.31,270.42) .. controls (318.39,292.11) and (319.36,327.35) .. (318.84,358.42) .. controls (363.72,353.44) and (406.33,353.08) .. (406.18,327.04) .. controls (407.51,288.93) and (407.81,264.64) .. (408.98,244.81) -- cycle ;
\draw  [color={rgb, 255:red, 255; green, 255; blue, 255 }  ,draw opacity=1 ][line width=1.5] [line join = round][line cap = round] (323.69,157.29) .. controls (317.38,156.99) and (324.71,154.08) .. (321.79,153.14) .. controls (314.37,150.73) and (290.6,159.09) .. (310.05,161.71) .. controls (318.89,162.9) and (325.63,162.02) .. (332.64,159.75) .. controls (332.92,159.66) and (336.66,158.25) .. (336.5,157.9) .. controls (335.77,156.32) and (326.06,151.16) .. (317.47,153.95) .. controls (308.54,156.84) and (329.15,157.55) .. (326.89,157.44) .. controls (322.09,157.21) and (317.29,156.98) .. (312.49,156.75) ;
\draw [line width=0.75]    (296.52,158.9) .. controls (307.67,145.36) and (330.35,143.37) .. (341.92,157.87) ;
\draw [line width=0.75]    (289.54,151.43) .. controls (306.97,174.49) and (335.03,171.32) .. (348.03,148.65) ;
\draw  [draw opacity=0][line width=1.5]  (239.99,367.86) .. controls (239.99,367.86) and (239.99,367.86) .. (239.99,367.86) .. controls (239.99,367.86) and (239.99,367.86) .. (239.99,367.86) .. controls (239.82,385.35) and (276.46,399.97) .. (321.84,400.52) .. controls (367.21,401.07) and (404.13,387.35) .. (404.31,369.87) -- (322.15,368.87) -- cycle ; \draw  [color={rgb, 255:red, 144; green, 19; blue, 254 }  ,draw opacity=1 ][line width=1.5]  (239.99,367.86) .. controls (239.99,367.86) and (239.99,367.86) .. (239.99,367.86) .. controls (239.99,367.86) and (239.99,367.86) .. (239.99,367.86) .. controls (239.82,385.35) and (276.46,399.97) .. (321.84,400.52) .. controls (367.21,401.07) and (404.13,387.35) .. (404.31,369.87) ;  
\draw  [draw opacity=0][line width=0.75]  (305.16,214.75) .. controls (265.65,216.88) and (235.77,228.1) .. (235.63,241.99) .. controls (235.47,257.81) and (273.89,271.1) .. (321.44,271.68) .. controls (368.99,272.26) and (407.66,259.91) .. (407.82,244.1) .. controls (407.96,229.49) and (375.17,217.03) .. (332.67,214.76) -- (321.73,243.05) -- cycle ; \draw  [color={rgb, 255:red, 128; green, 128; blue, 128 }  ,draw opacity=1 ][line width=0.75]  (305.16,214.75) .. controls (265.65,216.88) and (235.77,228.1) .. (235.63,241.99) .. controls (235.47,257.81) and (273.89,271.1) .. (321.44,271.68) .. controls (368.99,272.26) and (407.66,259.91) .. (407.82,244.1) .. controls (407.96,229.49) and (375.17,217.03) .. (332.67,214.76) ;  
\draw  [draw opacity=0][line width=0.75]  (238.05,327.78) .. controls (237.9,343.35) and (275.4,356.44) .. (321.82,357) .. controls (368.24,357.57) and (405.99,345.41) .. (406.15,329.83) -- (322.1,328.81) -- cycle ; \draw  [color={rgb, 255:red, 128; green, 128; blue, 128 }  ,draw opacity=1 ][line width=0.75]  (238.05,327.78) .. controls (237.9,343.35) and (275.4,356.44) .. (321.82,357) .. controls (368.24,357.57) and (405.99,345.41) .. (406.15,329.83) ;  
\draw [color={rgb, 255:red, 128; green, 128; blue, 128 }  ,draw opacity=1 ]   (305.16,214.75) -- (305.98,271.5) ;
\draw [color={rgb, 255:red, 128; green, 128; blue, 128 }  ,draw opacity=1 ][line width=0.75]    (332.59,215.15) -- (332.77,272.59) ;
\draw [color={rgb, 255:red, 0; green, 200; blue, 200 }  ,draw opacity=1 ][line width=1.5]    (304.3,256.95) -- (332.68,256.87) ;
\draw  [draw opacity=0][line width=1.5]  (236.88,284.75) .. controls (237.85,300.29) and (275.49,313.23) .. (321.9,313.79) .. controls (366.25,314.34) and (402.81,303.4) .. (407.1,288.89) -- (322.19,285.07) -- cycle ; \draw  [color={rgb, 255:red, 0; green, 0; blue, 255 }  ,draw opacity=1 ][line width=1.5]  (236.88,284.75) .. controls (237.85,300.29) and (275.49,313.23) .. (321.9,313.79) .. controls (366.25,314.34) and (402.81,303.4) .. (407.1,288.89) ;  
\draw  [draw opacity=0][dash pattern={on 5.63pt off 4.5pt}][line width=1.5]  (303.85,255.95) .. controls (265.44,258.33) and (236.57,269.66) .. (236.43,283.64) .. controls (236.27,299.75) and (274.34,313.28) .. (321.46,313.85) .. controls (368.59,314.43) and (406.92,301.84) .. (407.08,285.72) .. controls (407.22,270.82) and (374.66,258.13) .. (332.48,255.85) -- (321.75,284.68) -- cycle ; \draw  [color={rgb, 255:red, 0; green, 0; blue, 255 }  ,draw opacity=1 ][dash pattern={on 5.63pt off 4.5pt}][line width=1.5]  (303.85,255.95) .. controls (265.44,258.33) and (236.57,269.66) .. (236.43,283.64) .. controls (236.27,299.75) and (274.34,313.28) .. (321.46,313.85) .. controls (368.59,314.43) and (406.92,301.84) .. (407.08,285.72) .. controls (407.22,270.82) and (374.66,258.13) .. (332.48,255.85) ;  
\draw  [color={rgb, 255:red, 0; green, 0; blue, 255 }  ,draw opacity=1 ][fill={rgb, 255:red, 216; green, 216; blue, 216 }  ,fill opacity=1 ][dash pattern={on 5.63pt off 4.5pt}][line width=1.5]  (305.12,296.83) -- (337.99,296.83) -- (337.99,332.83) -- (305.12,332.83) -- cycle ;
\draw [line width=0.75]    (405.13,368.19) .. controls (410.48,177.98) and (405.15,213.98) .. (419.81,161.98) .. controls (434.48,109.98) and (355.81,91.31) .. (319.81,92.64) .. controls (283.81,93.98) and (212.7,112.57) .. (224.26,157.94) .. controls (235.81,203.31) and (233.15,180.64) .. (239.4,368.69) ;

\draw (390.95,135.76) node [anchor=north west][inner sep=0.75pt]    {$\Sigma _{+}$};
\draw (412.85,363.47) node [anchor=north west][inner sep=0.75pt]  [color={rgb, 255:red, 144; green, 19; blue, 254 }  ,opacity=1 ]  {$\Gamma $};
\draw (311.91,302.4) node [anchor=north west][inner sep=0.75pt]  [color={rgb, 255:red, 0; green, 0; blue, 255 }  ,opacity=1 ]  {$\Lambda ^{l}$};
\draw (313.25,237.07) node [anchor=north west][inner sep=0.75pt]  [color={rgb, 255:red, 0; green, 200; blue, 200 }  ,opacity=1 ]  {$\delta $};
\draw (386,318.07) node [anchor=north west][inner sep=0.75pt]    {$D$};

\end{tikzpicture}

%% file: images/legendrian_trefoil_4.tex
\begin{tikzpicture}[x=0.75pt,y=0.75pt,yscale=-1,xscale=1]

\draw  [line width=0.75]  (401.33,224.4) .. controls (404.47,219.44) and (406.35,214.48) .. (406.98,209.52) .. controls (407.61,214.48) and (409.49,219.44) .. (412.63,224.4) ;
\draw [color={rgb, 255:red, 0; green, 0; blue, 255 }  ,draw opacity=1 ][line width=1.5]    (286.16,159.16) .. controls (287.9,149.64) and (284.92,235.99) .. (277.88,246.32) .. controls (270.83,256.64) and (256.16,253.49) .. (251.97,246.15) .. controls (247.78,238.81) and (241.91,183.33) .. (242.96,159.21) .. controls (244,135.1) and (265.89,103.82) .. (282.76,102.59) .. controls (299.64,101.37) and (339.6,112.59) .. (350.16,112.59) ;
\draw [color={rgb, 255:red, 0; green, 0; blue, 255 }  ,draw opacity=1 ][line width=1.5]    (286.16,146.57) .. controls (287.73,138.19) and (290.35,132.94) .. (295.59,130.85) .. controls (300.82,128.75) and (308.16,127.7) .. (319.68,135.04) .. controls (331.2,142.38) and (360.54,152.87) .. (402.44,148.67) ;
\draw [color={rgb, 255:red, 0; green, 0; blue, 255 }  ,draw opacity=1 ][line width=1.5]    (252.11,145.15) .. controls (267.82,144.1) and (295.59,156.01) .. (302.92,154.96) .. controls (310.25,153.91) and (318.63,148.67) .. (321.78,143.43) ;
\draw [color={rgb, 255:red, 0; green, 0; blue, 255 }  ,draw opacity=1 ][line width=1.5]    (205.49,145.53) .. controls (219.11,145.53) and (231.68,142.38) .. (239.02,143.43) ;
\draw [color={rgb, 255:red, 144; green, 19; blue, 254 }  ,draw opacity=1 ][line width=1.5]    (201.2,181.93) -- (401.73,181.93) ;
\draw  [draw opacity=0][line width=1.5]  (343.44,96.17) .. controls (345.42,94.25) and (348.02,92.56) .. (351.06,91.36) .. controls (359.53,88) and (368.19,89.64) .. (370.41,95.02) .. controls (372.63,100.4) and (367.57,107.49) .. (359.1,110.85) .. controls (356.06,112.06) and (353,112.62) .. (350.21,112.59) -- (355.08,101.1) -- cycle ; \draw  [color={rgb, 255:red, 0; green, 0; blue, 255 }  ,draw opacity=1 ][line width=1.5]  (343.44,96.17) .. controls (345.42,94.25) and (348.02,92.56) .. (351.06,91.36) .. controls (359.53,88) and (368.19,89.64) .. (370.41,95.02) .. controls (372.63,100.4) and (367.57,107.49) .. (359.1,110.85) .. controls (356.06,112.06) and (353,112.62) .. (350.21,112.59) ;  
\draw [color={rgb, 255:red, 0; green, 0; blue, 255 }  ,draw opacity=1 ][line width=1.5]    (327.36,134.04) .. controls (330.94,128.04) and (332.64,119.67) .. (333.96,115.09) ;
\draw [color={rgb, 255:red, 0; green, 0; blue, 255 }  ,draw opacity=1 ][line width=1.5]    (336.6,105.95) .. controls (341.22,98.76) and (338.58,102.03) .. (343.41,96.2) ;

\draw (258.48,184.62) node [anchor=north west][inner sep=0.75pt]  [font=\footnotesize,color={rgb, 255:red, 144; green, 19; blue, 254 }  ,opacity=1 ]  {$a$};
\draw (337.05,184.62) node [anchor=north west][inner sep=0.75pt]  [font=\footnotesize,color={rgb, 255:red, 144; green, 19; blue, 254 }  ,opacity=1 ]  {$b$};
\draw (406.38,174.87) node [anchor=north west][inner sep=0.75pt]  [color={rgb, 255:red, 144; green, 19; blue, 254 }  ,opacity=1 ]  {$\Gamma $};
\draw (404.44,152.07) node [anchor=north west][inner sep=0.75pt]  [color={rgb, 255:red, 0; green, 0; blue, 255 }  ,opacity=1 ]  {$\Lambda $};
\draw (205.94,85.16) node [anchor=north west][inner sep=0.75pt]    {$\Sigma _{+}$};
\draw (204.86,263.4) node [anchor=north west][inner sep=0.75pt]    {$\Sigma _{-}$};
\draw (235.75,129.17) node [anchor=north west][inner sep=0.75pt]  [font=\footnotesize,color={rgb, 255:red, 0; green, 0; blue, 0 }  ,opacity=1 ]  {$c$};
\draw (274.08,135.72) node [anchor=north west][inner sep=0.75pt]  [font=\footnotesize,color={rgb, 255:red, 0; green, 0; blue, 0 }  ,opacity=1 ]  {$d$};
\draw (331.9,126.56) node [anchor=north west][inner sep=0.75pt]  [font=\footnotesize,color={rgb, 255:red, 0; green, 0; blue, 0 }  ,opacity=1 ]  {$e$};
\draw (324.62,91.2) node [anchor=north west][inner sep=0.75pt]  [font=\footnotesize,color={rgb, 255:red, 0; green, 0; blue, 0 }  ,opacity=1 ]  {$f$};

\end{tikzpicture}

%% file: images/R_moves_unknot.tex
\begin{tikzpicture}[x=0.65pt,y=0.65pt,yscale=-1,xscale=1]

\draw  [fill={rgb, 255:red, 246; green, 246; blue, 246 }  ,fill opacity=1 ][line width=1.5]  (20.31,81.65) -- (161.98,81.65) -- (161.98,222.95) -- (20.31,222.95) -- cycle ;
\draw [color={rgb, 255:red, 144; green, 19; blue, 254 }  ,draw opacity=1 ][line width=1.5]    (20.31,152.3) -- (161.98,152.3) ;
\draw  [line width=0.75]  (16.27,127.16) .. controls (18.43,123.76) and (19.73,120.35) .. (20.16,116.94) .. controls (20.6,120.35) and (21.89,123.76) .. (24.06,127.16) ;
\draw  [line width=0.75]  (157.8,127.16) .. controls (159.96,123.76) and (161.26,120.35) .. (161.69,116.94) .. controls (162.13,120.35) and (163.42,123.76) .. (165.59,127.16) ;

\draw  [color={rgb, 255:red, 144; green, 19; blue, 254 }  ,draw opacity=1 ][line width=0.75]  (80.03,78.2) .. controls (83.44,80.35) and (86.86,81.65) .. (90.28,82.08) .. controls (86.86,82.51) and (83.44,83.81) .. (80.03,85.96) ;
\draw  [color={rgb, 255:red, 144; green, 19; blue, 254 }  ,draw opacity=1 ][line width=0.75]  (80.03,219.36) .. controls (83.44,221.52) and (86.86,222.81) .. (90.28,223.25) .. controls (86.86,223.68) and (83.44,224.97) .. (80.03,227.13) ;

\draw  [color={rgb, 255:red, 144; green, 19; blue, 254 }  ,draw opacity=1 ][line width=0.75]  (88.69,78.2) .. controls (92.11,80.35) and (95.52,81.65) .. (98.94,82.08) .. controls (95.52,82.51) and (92.11,83.81) .. (88.69,85.96) ;
\draw  [color={rgb, 255:red, 144; green, 19; blue, 254 }  ,draw opacity=1 ][line width=0.75]  (88.69,219.36) .. controls (92.11,221.52) and (95.52,222.81) .. (98.94,223.25) .. controls (95.52,223.68) and (92.11,224.97) .. (88.69,227.13) ;

\draw [color={rgb, 255:red, 144; green, 19; blue, 254 }  ,draw opacity=1 ][line width=1.5]    (20.31,81.65) -- (161.98,81.65) ;
\draw [color={rgb, 255:red, 144; green, 19; blue, 254 }  ,draw opacity=1 ][line width=1.5]    (20.31,222.95) -- (161.98,222.95) ;
\draw  [draw opacity=0][line width=0.75]  (81.14,135.28) .. controls (78.38,138.16) and (74.53,139.94) .. (70.26,139.94) .. controls (61.81,139.94) and (54.92,132.97) .. (54.87,124.38) .. controls (54.82,115.78) and (61.62,108.81) .. (70.07,108.81) .. controls (74.34,108.81) and (78.22,110.59) .. (81.01,113.47) -- (70.17,124.38) -- cycle ; \draw  [color={rgb, 255:red, 0; green, 0; blue, 255 }  ,draw opacity=1 ][line width=0.75]  (81.14,135.28) .. controls (78.38,138.16) and (74.53,139.94) .. (70.26,139.94) .. controls (61.81,139.94) and (54.92,132.97) .. (54.87,124.38) .. controls (54.82,115.78) and (61.62,108.81) .. (70.07,108.81) .. controls (74.34,108.81) and (78.22,110.59) .. (81.01,113.47) ;  
\draw  [draw opacity=0][line width=0.75]  (102.38,113.51) .. controls (105.15,110.42) and (109.08,108.49) .. (113.44,108.49) .. controls (121.89,108.49) and (128.78,115.7) .. (128.84,124.6) .. controls (128.89,133.49) and (122.09,140.71) .. (113.64,140.71) .. controls (109.28,140.71) and (105.32,138.78) .. (102.52,135.69) -- (113.54,124.6) -- cycle ; \draw  [color={rgb, 255:red, 0; green, 0; blue, 255 }  ,draw opacity=1 ][line width=0.75]  (102.38,113.51) .. controls (105.15,110.42) and (109.08,108.49) .. (113.44,108.49) .. controls (121.89,108.49) and (128.78,115.7) .. (128.84,124.6) .. controls (128.89,133.49) and (122.09,140.71) .. (113.64,140.71) .. controls (109.28,140.71) and (105.32,138.78) .. (102.52,135.69) ;  
\draw [color={rgb, 255:red, 0; green, 0; blue, 255 }  ,draw opacity=1 ][line width=0.75]    (80.93,113.39) -- (102.6,135.77) ;
\draw [color={rgb, 255:red, 0; green, 0; blue, 255 }  ,draw opacity=1 ][line width=0.75]    (89.93,125.89) -- (81.06,135.36) ;
\draw [color={rgb, 255:red, 0; green, 0; blue, 255 }  ,draw opacity=1 ][line width=0.75]    (102.46,113.42) -- (94.09,122) ;

\draw  [fill={rgb, 255:red, 246; green, 246; blue, 246 }  ,fill opacity=1 ][line width=1.5]  (255.05,81.65) -- (396.72,81.65) -- (396.72,222.95) -- (255.05,222.95) -- cycle ;
\draw [color={rgb, 255:red, 144; green, 19; blue, 254 }  ,draw opacity=1 ][line width=1.5]    (255.05,152.3) -- (396.72,152.3) ;
\draw  [line width=0.75]  (251.02,127.16) .. controls (253.18,123.75) and (254.48,120.35) .. (254.91,116.94) .. controls (255.35,120.35) and (256.64,123.75) .. (258.8,127.16) ;
\draw  [line width=0.75]  (392.54,127.16) .. controls (394.7,123.76) and (396,120.35) .. (396.44,116.94) .. controls (396.87,120.35) and (398.16,123.76) .. (400.33,127.16) ;

\draw  [color={rgb, 255:red, 144; green, 19; blue, 254 }  ,draw opacity=1 ][line width=0.75]  (314.77,78.2) .. controls (318.19,80.35) and (321.6,81.65) .. (325.02,82.08) .. controls (321.6,82.51) and (318.19,83.81) .. (314.77,85.96) ;
\draw  [color={rgb, 255:red, 144; green, 19; blue, 254 }  ,draw opacity=1 ][line width=0.75]  (314.77,219.36) .. controls (318.19,221.52) and (321.6,222.81) .. (325.02,223.25) .. controls (321.6,223.68) and (318.19,224.97) .. (314.77,227.13) ;

\draw  [color={rgb, 255:red, 144; green, 19; blue, 254 }  ,draw opacity=1 ][line width=0.75]  (323.44,78.2) .. controls (326.85,80.35) and (330.27,81.65) .. (333.69,82.08) .. controls (330.27,82.51) and (326.85,83.81) .. (323.44,85.96) ;
\draw  [color={rgb, 255:red, 144; green, 19; blue, 254 }  ,draw opacity=1 ][line width=0.75]  (323.44,219.36) .. controls (326.85,221.52) and (330.27,222.81) .. (333.69,223.25) .. controls (330.27,223.68) and (326.85,224.97) .. (323.44,227.13) ;

\draw [color={rgb, 255:red, 144; green, 19; blue, 254 }  ,draw opacity=1 ][line width=1.5]    (255.05,81.65) -- (396.72,81.65) ;
\draw [color={rgb, 255:red, 144; green, 19; blue, 254 }  ,draw opacity=1 ][line width=1.5]    (255.05,222.95) -- (396.72,222.95) ;
\draw  [draw opacity=0][line width=0.75]  (342.92,99.6) .. controls (345.78,96.45) and (349.82,94.49) .. (354.31,94.49) .. controls (363.04,94.49) and (370.16,101.89) .. (370.21,111.02) .. controls (370.27,120.16) and (363.24,127.56) .. (354.52,127.56) .. controls (350.03,127.56) and (345.96,125.6) .. (343.06,122.45) -- (354.42,111.02) -- cycle ; \draw  [color={rgb, 255:red, 0; green, 0; blue, 255 }  ,draw opacity=1 ][line width=0.75]  (342.92,99.6) .. controls (345.78,96.45) and (349.82,94.49) .. (354.31,94.49) .. controls (363.04,94.49) and (370.16,101.89) .. (370.21,111.02) .. controls (370.27,120.16) and (363.24,127.56) .. (354.52,127.56) .. controls (350.03,127.56) and (345.96,125.6) .. (343.06,122.45) ;  
\draw [color={rgb, 255:red, 0; green, 0; blue, 255 }  ,draw opacity=1 ][line width=0.75]    (320.72,99.52) -- (343.11,122.5) ;
\draw [color={rgb, 255:red, 0; green, 0; blue, 255 }  ,draw opacity=1 ][line width=0.75]    (342.96,99.55) -- (334.32,108.35) ;
\draw  [draw opacity=0][line width=0.75]  (293.93,142.43) .. controls (290.34,145.2) and (288.04,149.29) .. (287.98,153.88) .. controls (287.87,162.43) and (295.54,169.5) .. (305.12,169.67) .. controls (314.7,169.84) and (322.56,163.05) .. (322.67,154.5) .. controls (322.73,149.9) and (320.55,145.74) .. (317.03,142.84) -- (305.33,154.19) -- cycle ; \draw  [color={rgb, 255:red, 0; green, 0; blue, 255 }  ,draw opacity=1 ][line width=0.75]  (293.93,142.43) .. controls (290.34,145.2) and (288.04,149.29) .. (287.98,153.88) .. controls (287.87,162.43) and (295.54,169.5) .. (305.12,169.67) .. controls (314.7,169.84) and (322.56,163.05) .. (322.67,154.5) .. controls (322.73,149.9) and (320.55,145.74) .. (317.03,142.84) ;  
\draw [color={rgb, 255:red, 0; green, 0; blue, 255 }  ,draw opacity=1 ][line width=0.75]    (293.7,121.02) -- (317.51,143.25) ;
\draw [color={rgb, 255:red, 0; green, 0; blue, 255 }  ,draw opacity=1 ][line width=0.75]    (308.74,130.29) -- (327.74,113.29) ;
\draw [color={rgb, 255:red, 0; green, 0; blue, 255 }  ,draw opacity=1 ][line width=0.75]    (293.44,142.82) -- (302.79,134.46) ;
\draw [color={rgb, 255:red, 0; green, 0; blue, 255 }  ,draw opacity=1 ][line width=0.75]    (293.7,121.02) .. controls (280.44,108.8) and (305.64,84.11) .. (320.72,99.52) ;
\draw  [fill={rgb, 255:red, 246; green, 246; blue, 246 }  ,fill opacity=1 ][line width=1.5]  (489.8,81.65) -- (631.47,81.65) -- (631.47,222.95) -- (489.8,222.95) -- cycle ;
\draw [color={rgb, 255:red, 144; green, 19; blue, 254 }  ,draw opacity=1 ][line width=1.5]    (489.8,152.3) -- (631.47,152.3) ;
\draw  [color={rgb, 255:red, 0; green, 0; blue, 255 }  ,draw opacity=1 ][line width=0.75]  (516.79,152.3) .. controls (516.79,128.15) and (536.42,108.57) .. (560.64,108.57) .. controls (584.85,108.57) and (604.48,128.15) .. (604.48,152.3) .. controls (604.48,176.46) and (584.85,196.04) .. (560.64,196.04) .. controls (536.42,196.04) and (516.79,176.46) .. (516.79,152.3) -- cycle ;
\draw  [line width=0.75]  (485.77,127.16) .. controls (487.93,123.75) and (489.23,120.35) .. (489.66,116.94) .. controls (490.1,120.35) and (491.39,123.75) .. (493.55,127.16) ;
\draw  [line width=0.75]  (627.29,127.16) .. controls (629.45,123.76) and (630.75,120.35) .. (631.19,116.94) .. controls (631.62,120.35) and (632.91,123.76) .. (635.08,127.16) ;

\draw  [color={rgb, 255:red, 144; green, 19; blue, 254 }  ,draw opacity=1 ][line width=0.75]  (549.52,78.2) .. controls (552.94,80.35) and (556.36,81.65) .. (559.77,82.08) .. controls (556.36,82.51) and (552.94,83.81) .. (549.52,85.96) ;
\draw  [color={rgb, 255:red, 144; green, 19; blue, 254 }  ,draw opacity=1 ][line width=0.75]  (549.52,219.36) .. controls (552.94,221.52) and (556.36,222.81) .. (559.77,223.25) .. controls (556.36,223.68) and (552.94,224.97) .. (549.52,227.13) ;

\draw  [color={rgb, 255:red, 144; green, 19; blue, 254 }  ,draw opacity=1 ][line width=0.75]  (558.19,78.2) .. controls (561.6,80.35) and (565.02,81.65) .. (568.44,82.08) .. controls (565.02,82.51) and (561.6,83.81) .. (558.19,85.96) ;
\draw  [color={rgb, 255:red, 144; green, 19; blue, 254 }  ,draw opacity=1 ][line width=0.75]  (558.19,219.36) .. controls (561.6,221.52) and (565.02,222.81) .. (568.44,223.25) .. controls (565.02,223.68) and (561.6,224.97) .. (558.19,227.13) ;

\draw [color={rgb, 255:red, 144; green, 19; blue, 254 }  ,draw opacity=1 ][line width=1.5]    (489.8,81.65) -- (631.47,81.65) ;
\draw [color={rgb, 255:red, 144; green, 19; blue, 254 }  ,draw opacity=1 ][line width=1.5]    (489.8,222.95) -- (631.47,222.95) ;
\draw    (228.49,152.69) -- (183.26,152.75) ;
\draw [shift={(230.49,152.69)}, rotate = 179.93] [color={rgb, 255:red, 0; green, 0; blue, 0 }  ][line width=0.75]    (10.93,-3.29) .. controls (6.95,-1.4) and (3.31,-0.3) .. (0,0) .. controls (3.31,0.3) and (6.95,1.4) .. (10.93,3.29)   ;
\draw    (468.49,152.69) -- (423.26,152.75) ;
\draw [shift={(470.49,152.69)}, rotate = 179.93] [color={rgb, 255:red, 0; green, 0; blue, 0 }  ][line width=0.75]    (10.93,-3.29) .. controls (6.95,-1.4) and (3.31,-0.3) .. (0,0) .. controls (3.31,0.3) and (6.95,1.4) .. (10.93,3.29)   ;

\draw (147.01,156.07) node [anchor=north west][inner sep=0.75pt]  [color={rgb, 255:red, 144; green, 19; blue, 254 }  ,opacity=1 ]  {$\Gamma $};
\draw (24,85.13) node [anchor=north west][inner sep=0.75pt]    {$\Sigma _{+}$};
\draw (24.42,203.57) node [anchor=north west][inner sep=0.75pt]    {$\Sigma _{-}$};
\draw (381.75,156.07) node [anchor=north west][inner sep=0.75pt]  [color={rgb, 255:red, 144; green, 19; blue, 254 }  ,opacity=1 ]  {$\Gamma $};
\draw (258.75,85.13) node [anchor=north west][inner sep=0.75pt]    {$\Sigma _{+}$};
\draw (259.16,203.57) node [anchor=north west][inner sep=0.75pt]    {$\Sigma _{-}$};
\draw (617.5,156.07) node [anchor=north west][inner sep=0.75pt]  [color={rgb, 255:red, 144; green, 19; blue, 254 }  ,opacity=1 ]  {$\Gamma $};
\draw (493.5,85.13) node [anchor=north west][inner sep=0.75pt]    {$\Sigma _{+}$};
\draw (493.91,203.57) node [anchor=north west][inner sep=0.75pt]    {$\Sigma _{-}$};
\draw (195.27,127.4) node [anchor=north west][inner sep=0.75pt]    {$RI$};
\draw (431.27,127.4) node [anchor=north west][inner sep=0.75pt]    {$RII$};

\end{tikzpicture}

%% file: images/genus_2_unknot.tex
\begin{tikzpicture}[x=0.55pt,y=0.55pt,yscale=-1,xscale=1]
\draw  [fill={rgb, 255:red, 246; green, 246; blue, 246 }  ,fill opacity=1 ][line width=0.75]  (121.9,186.31) .. controls (121.9,225.31) and (163.9,266.31) .. (203.9,266.31) .. controls (243.9,266.31) and (288.1,244.42) .. (322.9,244.31) .. controls (357.69,244.19) and (403.9,266.31) .. (443.74,266.12) .. controls (483.58,265.94) and (522.74,225.12) .. (522.74,185.12) .. controls (522.74,145.12) and (494.3,106.28) .. (442.74,106.12) .. controls (391.17,105.96) and (361.9,126.31) .. (323.9,126.31) .. controls (285.9,126.31) and (243.9,104.31) .. (202.9,105.31) .. controls (161.9,106.31) and (121.9,147.31) .. (121.9,186.31) -- cycle ;
\draw  [draw opacity=0][dash pattern={on 5.63pt off 4.5pt}][line width=1.5]  (318.34,243.68) .. controls (318.34,243.68) and (318.34,243.68) .. (318.34,243.68) .. controls (318.34,243.68) and (318.34,243.68) .. (318.34,243.68) .. controls (306.38,243.68) and (296.69,217.45) .. (296.69,185.1) .. controls (296.69,152.74) and (306.38,126.51) .. (318.34,126.51) -- (318.34,185.1) -- cycle ; \draw  [color={rgb, 255:red, 0; green, 0; blue, 255 }  ,draw opacity=1 ][dash pattern={on 5.63pt off 4.5pt}][line width=1.5]  (318.34,243.68) .. controls (318.34,243.68) and (318.34,243.68) .. (318.34,243.68) .. controls (318.34,243.68) and (318.34,243.68) .. (318.34,243.68) .. controls (306.38,243.68) and (296.69,217.45) .. (296.69,185.1) .. controls (296.69,152.74) and (306.38,126.51) .. (318.34,126.51) ;  
\draw  [draw opacity=0][line width=1.5]  (123,184.18) .. controls (123,184.18) and (123,184.18) .. (123,184.18) .. controls (123,171.75) and (144.3,161.68) .. (170.58,161.68) .. controls (196.86,161.68) and (218.17,171.75) .. (218.17,184.18) -- (170.58,184.18) -- cycle ; \draw  [color={rgb, 255:red, 144; green, 19; blue, 254 }  ,draw opacity=1 ][line width=1.5]  (123,184.18) .. controls (123,184.18) and (123,184.18) .. (123,184.18) .. controls (123,171.75) and (144.3,161.68) .. (170.58,161.68) .. controls (196.86,161.68) and (218.17,171.75) .. (218.17,184.18) ;  
\draw  [draw opacity=0][line width=1.5]  (427.17,185.16) .. controls (427.17,185.16) and (427.17,185.16) .. (427.17,185.16) .. controls (427.17,172.52) and (448.41,162.27) .. (474.62,162.27) .. controls (500.83,162.27) and (522.08,172.52) .. (522.08,185.16) .. controls (522.08,185.16) and (522.08,185.16) .. (522.08,185.16) -- (474.62,185.16) -- cycle ; \draw  [color={rgb, 255:red, 144; green, 19; blue, 254 }  ,draw opacity=1 ][line width=1.5]  (427.17,185.16) .. controls (427.17,185.16) and (427.17,185.16) .. (427.17,185.16) .. controls (427.17,172.52) and (448.41,162.27) .. (474.62,162.27) .. controls (500.83,162.27) and (522.08,172.52) .. (522.08,185.16) .. controls (522.08,185.16) and (522.08,185.16) .. (522.08,185.16) ;  
\draw  [draw opacity=0][line width=1.5]  (250.08,182.95) .. controls (250.08,171.53) and (282.56,162.27) .. (322.62,162.27) .. controls (362.69,162.27) and (395.17,171.53) .. (395.17,182.95) -- (322.62,182.95) -- cycle ; \draw  [color={rgb, 255:red, 144; green, 19; blue, 254 }  ,draw opacity=1 ][line width=1.5]  (250.08,182.95) .. controls (250.08,171.53) and (282.56,162.27) .. (322.62,162.27) .. controls (362.69,162.27) and (395.17,171.53) .. (395.17,182.95) ;  
\draw  [draw opacity=0][dash pattern={on 5.63pt off 4.5pt}][line width=1.5]  (522,184.26) .. controls (522,184.26) and (522,184.26) .. (522,184.26) .. controls (522,184.26) and (522,184.26) .. (522,184.26) .. controls (522,197.2) and (500.73,207.68) .. (474.49,207.68) .. controls (448.26,207.68) and (426.99,197.2) .. (426.99,184.26) -- (474.49,184.26) -- cycle ; \draw  [color={rgb, 255:red, 144; green, 19; blue, 254 }  ,draw opacity=1 ][dash pattern={on 5.63pt off 4.5pt}][line width=1.5]  (522,184.26) .. controls (522,184.26) and (522,184.26) .. (522,184.26) .. controls (522,184.26) and (522,184.26) .. (522,184.26) .. controls (522,197.2) and (500.73,207.68) .. (474.49,207.68) .. controls (448.26,207.68) and (426.99,197.2) .. (426.99,184.26) ;  
\draw  [draw opacity=0][dash pattern={on 5.63pt off 4.5pt}][line width=1.5]  (217.99,183.16) .. controls (217.99,183.16) and (217.99,183.16) .. (217.99,183.16) .. controls (217.99,195.8) and (196.74,206.05) .. (170.53,206.05) .. controls (144.32,206.05) and (123.08,195.8) .. (123.08,183.16) -- (170.53,183.16) -- cycle ; \draw  [color={rgb, 255:red, 144; green, 19; blue, 254 }  ,draw opacity=1 ][dash pattern={on 5.63pt off 4.5pt}][line width=1.5]  (217.99,183.16) .. controls (217.99,183.16) and (217.99,183.16) .. (217.99,183.16) .. controls (217.99,195.8) and (196.74,206.05) .. (170.53,206.05) .. controls (144.32,206.05) and (123.08,195.8) .. (123.08,183.16) ;  
\draw  [draw opacity=0][dash pattern={on 5.63pt off 4.5pt}][line width=1.5]  (395.08,185.36) .. controls (395.08,196.79) and (362.6,206.05) .. (322.53,206.05) .. controls (282.47,206.05) and (249.99,196.79) .. (249.99,185.36) -- (322.53,185.36) -- cycle ; \draw  [color={rgb, 255:red, 144; green, 19; blue, 254 }  ,draw opacity=1 ][dash pattern={on 5.63pt off 4.5pt}][line width=1.5]  (395.08,185.36) .. controls (395.08,196.79) and (362.6,206.05) .. (322.53,206.05) .. controls (282.47,206.05) and (249.99,196.79) .. (249.99,185.36) ;  
\draw  [color={rgb, 255:red, 255; green, 255; blue, 255 }  ,draw opacity=1 ][line width=1.5] [line join = round][line cap = round] (237.33,183.13) .. controls (232.88,182.88) and (238.07,181.31) .. (236.03,180.74) .. controls (230.82,179.28) and (214,183.76) .. (227.68,185.48) .. controls (233.89,186.27) and (238.65,185.84) .. (243.61,184.63) .. controls (243.8,184.58) and (246.45,183.82) .. (246.34,183.63) .. controls (245.85,182.72) and (239.05,179.66) .. (232.97,181.15) .. controls (226.66,182.69) and (241.17,183.34) .. (239.58,183.25) .. controls (236.2,183.06) and (232.82,182.87) .. (229.44,182.69) ;
\draw [line width=0.75]    (218.17,183.73) .. controls (226.15,176.14) and (242.15,175.27) .. (250.17,183.67) ;
\draw [line width=0.75]    (213.32,179.38) .. controls (225.39,192.73) and (245.18,191.25) .. (254.55,178.49) ;
\draw  [color={rgb, 255:red, 255; green, 255; blue, 255 }  ,draw opacity=1 ][line width=1.5] [line join = round][line cap = round] (414.33,184.13) .. controls (409.88,183.88) and (415.07,182.31) .. (413.03,181.74) .. controls (407.82,180.28) and (391,184.76) .. (404.68,186.48) .. controls (410.89,187.27) and (415.65,186.84) .. (420.61,185.63) .. controls (420.8,185.58) and (423.45,184.82) .. (423.34,184.63) .. controls (422.85,183.72) and (416.05,180.66) .. (409.97,182.15) .. controls (403.66,183.69) and (418.17,184.34) .. (416.58,184.25) .. controls (413.2,184.06) and (409.82,183.87) .. (406.44,183.69) ;
\draw [line width=0.75]    (395.17,184.73) .. controls (403.15,177.14) and (419.15,176.27) .. (427.17,184.67) ;
\draw [line width=0.75]    (390.32,180.38) .. controls (402.39,193.73) and (422.18,192.25) .. (431.55,179.49) ;
\draw  [draw opacity=0][line width=1.5]  (319.55,126.68) .. controls (319.55,126.68) and (319.55,126.68) .. (319.55,126.68) .. controls (319.55,126.68) and (319.55,126.68) .. (319.55,126.68) .. controls (330.84,126.68) and (340,152.87) .. (340,185.18) .. controls (340,217.49) and (330.84,243.68) .. (319.55,243.68) -- (319.55,185.18) -- cycle ; \draw  [color={rgb, 255:red, 0; green, 0; blue, 255 }  ,draw opacity=1 ][line width=1.5]  (319.55,126.68) .. controls (319.55,126.68) and (319.55,126.68) .. (319.55,126.68) .. controls (319.55,126.68) and (319.55,126.68) .. (319.55,126.68) .. controls (330.84,126.68) and (340,152.87) .. (340,185.18) .. controls (340,217.49) and (330.84,243.68) .. (319.55,243.68) ;  
\draw  [line width=0.75]  (121.9,186.31) .. controls (121.9,225.31) and (163.9,266.31) .. (203.9,266.31) .. controls (243.9,266.31) and (288.1,244.42) .. (322.9,244.31) .. controls (357.69,244.19) and (403.9,266.31) .. (443.74,266.12) .. controls (483.58,265.94) and (522.74,225.12) .. (522.74,185.12) .. controls (522.74,145.12) and (494.3,106.28) .. (442.74,106.12) .. controls (391.17,105.96) and (361.9,126.31) .. (323.9,126.31) .. controls (285.9,126.31) and (243.9,104.31) .. (202.9,105.31) .. controls (161.9,106.31) and (121.9,147.31) .. (121.9,186.31) -- cycle ;

\draw (176.59,231.4) node [anchor=north west][inner sep=0.75pt]    {$\Sigma _{+}$};
\draw (177.09,120.4) node [anchor=north west][inner sep=0.75pt]    {$\Sigma _{-}$};
\draw (467,142.33) node [anchor=north west][inner sep=0.75pt]  [color={rgb, 255:red, 144; green, 19; blue, 254 }  ,opacity=1 ]  {$\Gamma $};
\draw (342,142.33) node [anchor=north west][inner sep=0.75pt]  [color={rgb, 255:red, 0; green, 0; blue, 255 }  ,opacity=1 ]  {$\Lambda $};

\end{tikzpicture}

%% file: images/five_two_schematic.tex
\begin{tikzpicture}[x=0.65pt,y=0.65pt,yscale=-1,xscale=1]
\draw  [color={rgb, 255:red, 144; green, 19; blue, 254 }  ,draw opacity=1 ][line width=1.5]  [fill={rgb, 255:red, 246; green, 246; blue, 246 }  ,fill opacity=1 ] (368.48,142.39) .. controls (368.48,91.91) and (409.28,50.99) .. (459.61,50.99) .. controls (509.94,50.99) and (550.74,91.91) .. (550.74,142.39) .. controls (550.74,192.86) and (509.94,233.79) .. (459.61,233.79) .. controls (409.28,233.79) and (368.48,192.86) .. (368.48,142.39) -- cycle ;

\draw  [color={rgb, 255:red, 144; green, 19; blue, 254 }  ,draw opacity=1 ][line width=1.5] [fill={rgb, 255:red, 246; green, 246; blue, 246 }  ,fill opacity=1 ] (77.48,142.39) .. controls (77.48,91.91) and (118.28,50.99) .. (168.61,50.99) .. controls (218.94,50.99) and (259.74,91.91) .. (259.74,142.39) .. controls (259.74,192.86) and (218.94,233.79) .. (168.61,233.79) .. controls (118.28,233.79) and (77.48,192.86) .. (77.48,142.39) -- cycle ;

\draw  [color={rgb, 255:red, 144; green, 19; blue, 254 }  ,draw opacity=1 ][line width=1.5] [fill={rgb, 255:red, 255; green, 255; blue, 255 }  ,fill opacity=1 ] (97.12,142.39) .. controls (97.12,137.98) and (100.68,134.4) .. (105.08,134.4) .. controls (109.48,134.4) and (113.04,137.98) .. (113.04,142.39) .. controls (113.04,146.8) and (109.48,150.37) .. (105.08,150.37) .. controls (100.68,150.37) and (97.12,146.8) .. (97.12,142.39) -- cycle ;
\draw  [color={rgb, 255:red, 144; green, 19; blue, 254 }  ,draw opacity=1 ][line width=1.5] [fill={rgb, 255:red, 255; green, 255; blue, 255 }  ,fill opacity=1 ] (224.18,142.39) .. controls (224.18,137.98) and (227.75,134.4) .. (232.15,134.4) .. controls (236.54,134.4) and (240.11,137.98) .. (240.11,142.39) .. controls (240.11,146.8) and (236.54,150.37) .. (232.15,150.37) .. controls (227.75,150.37) and (224.18,146.8) .. (224.18,142.39) -- cycle ;

\draw [color={rgb, 255:red, 0; green, 0; blue, 255 }  ,draw opacity=1 ][line width=1.5]    (168.61,50.99) -- (168.61,100.68) ;
\draw [color={rgb, 255:red, 0; green, 0; blue, 255 }  ,draw opacity=1 ][line width=1.5]    (168.61,184.09) -- (168.61,233.79) ;
\draw  [color={rgb, 255:red, 128; green, 128; blue, 128 }  ,draw opacity=1 ][dash pattern={on 4.5pt off 4.5pt}][line width=0.75]  (126.81,100.46) -- (210.42,100.46) -- (210.42,184.31) -- (126.81,184.31) -- cycle ;
\draw  [color={rgb, 255:red, 144; green, 19; blue, 254 }  ,draw opacity=1 ][line width=1.5]   (77.48,142.39) .. controls (77.48,91.91) and (118.28,50.99) .. (168.61,50.99) .. controls (218.94,50.99) and (259.74,91.91) .. (259.74,142.39) .. controls (259.74,192.86) and (218.94,233.79) .. (168.61,233.79) .. controls (118.28,233.79) and (77.48,192.86) .. (77.48,142.39) -- cycle ;
\draw  [color={rgb, 255:red, 144; green, 19; blue, 254 }  ,draw opacity=1 ][line width=1.5]  [fill={rgb, 255:red, 255; green, 255; blue, 255 }  ,fill opacity=1 ]  (388.12,142.39) .. controls (388.12,137.98) and (391.68,134.4) .. (396.08,134.4) .. controls (400.48,134.4) and (404.04,137.98) .. (404.04,142.39) .. controls (404.04,146.8) and (400.48,150.37) .. (396.08,150.37) .. controls (391.68,150.37) and (388.12,146.8) .. (388.12,142.39) -- cycle ;
\draw  [color={rgb, 255:red, 144; green, 19; blue, 254 }  ,draw opacity=1 ][line width=1.5] [fill={rgb, 255:red, 255; green, 255; blue, 255 }  ,fill opacity=1 ] (515.18,142.39) .. controls (515.18,137.98) and (518.75,134.4) .. (523.15,134.4) .. controls (527.54,134.4) and (531.11,137.98) .. (531.11,142.39) .. controls (531.11,146.8) and (527.54,150.37) .. (523.15,150.37) .. controls (518.75,150.37) and (515.18,146.8) .. (515.18,142.39) -- cycle ;

\draw [color={rgb, 255:red, 0; green, 0; blue, 255 }  ,draw opacity=1 ][line width=1.5]    (451.13,52.26) -- (451.13,101.95) ;
\draw [color={rgb, 255:red, 0; green, 0; blue, 255 }  ,draw opacity=1 ][line width=1.5]    (451.13,182.09) -- (451.13,231.79) ;
\draw [color={rgb, 255:red, 0; green, 0; blue, 255 }  ,draw opacity=1 ][line width=1.5]    (453.72,135.61) .. controls (459.25,144.54) and (460.4,151.98) .. (472.35,151.54) .. controls (484.29,151.09) and (486.72,143.81) .. (486.49,139.92) .. controls (486.26,136.04) and (483.39,128.95) .. (471.01,129.38) .. controls (458.63,129.81) and (461.04,136.32) .. (453.1,150.31) ;
\draw [color={rgb, 255:red, 255; green, 255; blue, 255 }  ,draw opacity=1 ][line width=3]    (455.51,143.81) -- (458.47,138.19) ;
\draw [color={rgb, 255:red, 0; green, 0; blue, 255 }  ,draw opacity=1 ][line width=1.5]    (460.02,146.47) -- (453.72,135.61) ;
\draw  [color={rgb, 255:red, 144; green, 19; blue, 254 }  ,draw opacity=1 ][line width=1.5]  (368.48,142.39) .. controls (368.48,91.91) and (409.28,50.99) .. (459.61,50.99) .. controls (509.94,50.99) and (550.74,91.91) .. (550.74,142.39) .. controls (550.74,192.86) and (509.94,233.79) .. (459.61,233.79) .. controls (409.28,233.79) and (368.48,192.86) .. (368.48,142.39) -- cycle ;
\draw [color={rgb, 255:red, 0; green, 0; blue, 255 }  ,draw opacity=1 ][line width=1.5]    (451.09,101.95) .. controls (450.95,116.48) and (449.87,127.05) .. (453.72,135.61) ;
\draw [color={rgb, 255:red, 0; green, 0; blue, 255 }  ,draw opacity=1 ][line width=1.5]    (451.09,182.09) .. controls (451.25,171.68) and (449.81,155.66) .. (453.1,150.31) ;
\draw  [color={rgb, 255:red, 128; green, 128; blue, 128 }  ,draw opacity=1 ][dash pattern={on 4.5pt off 4.5pt}][line width=0.75]  (417.81,100.46) -- (501.42,100.46) -- (501.42,184.31) -- (417.81,184.31) -- cycle ;

\draw (58.62,131.79) node [anchor=north west][inner sep=0.75pt]  [color={rgb, 255:red, 144; green, 19; blue, 254 }  ,opacity=1 ]  {$c$};
\draw (267.74,131.79) node [anchor=north west][inner sep=0.75pt]  [color={rgb, 255:red, 144; green, 19; blue, 254 }  ,opacity=1 ]  {$d$};
\draw (161.66,136.07) node [anchor=north west][inner sep=0.75pt]    {$K$};
\draw (182.59,199.4) node [anchor=north west][inner sep=0.75pt]    {$\Sigma _{+}$};
\draw (477.09,199.4) node [anchor=north west][inner sep=0.75pt]    {$\Sigma _{-}$};
\draw (349.62,131.79) node [anchor=north west][inner sep=0.75pt]  [color={rgb, 255:red, 144; green, 19; blue, 254 }  ,opacity=1 ]  {$c$};
\draw (558.74,131.79) node [anchor=north west][inner sep=0.75pt]  [color={rgb, 255:red, 144; green, 19; blue, 254 }  ,opacity=1 ]  {$d$};

\end{tikzpicture}

%% file: images/five_two_knots.tex
\begin{tikzpicture}[x=0.65pt,y=0.65pt,yscale=-1,xscale=1]

\draw [color={rgb, 255:red, 0; green, 0; blue, 255 }  ,draw opacity=1 ][line width=1.5]    (238.03,99.33) .. controls (237.25,153.9) and (194.96,188.45) .. (188.57,195.19) ;
\draw [color={rgb, 255:red, 0; green, 0; blue, 255 }  ,draw opacity=1 ][line width=1.5]    (152.57,161.19) .. controls (143.31,168.52) and (122.31,192.52) .. (120.32,195.01) ;
\draw [color={rgb, 255:red, 0; green, 0; blue, 255 }  ,draw opacity=1 ][line width=1.5]    (223.26,68.16) .. controls (243,73.38) and (251.54,84.44) .. (269.54,71.95) .. controls (287.55,59.46) and (285.07,46.83) .. (281.39,41.63) .. controls (277.71,36.42) and (267.25,29.34) .. (248.56,42.25) .. controls (229.86,55.15) and (238.29,71.85) .. (238.03,99.33) ;
\draw [color={rgb, 255:red, 255; green, 255; blue, 255 }  ,draw opacity=1 ][line width=3]    (237.03,77.86) -- (236.78,67.04) ;
\draw [color={rgb, 255:red, 0; green, 0; blue, 255 }  ,draw opacity=1 ][line width=1.5]    (240.13,74.42) -- (233.16,71.21) ;
\draw [color={rgb, 255:red, 0; green, 0; blue, 255 }  ,draw opacity=1 ][line width=1.5]    (110.18,100.83) .. controls (108,72.38) and (115,54.38) .. (97.65,44.22) .. controls (80.3,34.07) and (70.13,38.36) .. (66.65,43.56) .. controls (63.17,48.77) and (60.21,60.93) .. (77.82,73.93) .. controls (95.43,86.92) and (107,73.38) .. (121.26,69.16) ;
\draw [color={rgb, 255:red, 255; green, 255; blue, 255 }  ,draw opacity=1 ][line width=3]    (113.46,72.48) -- (103.85,76.41) ;
\draw [color={rgb, 255:red, 0; green, 0; blue, 255 }  ,draw opacity=1 ][line width=1.5]    (110,71.38) -- (109.81,77.35) ;
\draw [color={rgb, 255:red, 0; green, 0; blue, 255 }  ,draw opacity=1 ][line width=1.5]    (229.75,201.02) .. controls (229.59,228.45) and (223.6,244.41) .. (240.03,257.07) .. controls (256.46,269.74) and (266.3,263.26) .. (269.48,258.12) .. controls (272.67,252.97) and (275.21,240.88) .. (258.21,227.71) .. controls (241.2,214.54) and (229.54,227.29) .. (206.34,237.27) ;
\draw [color={rgb, 255:red, 255; green, 255; blue, 255 }  ,draw opacity=1 ][line width=3]    (224.4,228.7) -- (233.44,224.9) ;
\draw [color={rgb, 255:red, 0; green, 0; blue, 255 }  ,draw opacity=1 ][line width=1.5]    (228.47,229.7) -- (228.71,223.91) ;
\draw [color={rgb, 255:red, 0; green, 0; blue, 255 }  ,draw opacity=1 ][line width=1.5]    (153.07,86.71) .. controls (153.19,97.77) and (188.96,110.1) .. (187.57,127.19) ;
\draw [color={rgb, 255:red, 255; green, 255; blue, 255 }  ,draw opacity=1 ][line width=3]    (172.67,105.16) -- (166.21,100.56) ;
\draw [color={rgb, 255:red, 0; green, 0; blue, 255 }  ,draw opacity=1 ][line width=1.5]    (186.57,86.19) .. controls (186.45,97.26) and (152.96,106.1) .. (154.07,127.71) ;
\draw [color={rgb, 255:red, 0; green, 0; blue, 255 }  ,draw opacity=1 ][line width=1.5]    (154.07,127.71) .. controls (154.19,138.77) and (180.31,154.52) .. (189.76,161.45) ;
\draw [color={rgb, 255:red, 255; green, 255; blue, 255 }  ,draw opacity=1 ][line width=3]    (173.67,149.16) -- (167.21,144.56) ;
\draw [color={rgb, 255:red, 0; green, 0; blue, 255 }  ,draw opacity=1 ][line width=1.5]    (187.57,127.19) .. controls (187.45,138.26) and (159.31,153.52) .. (152.57,161.19) ;
\draw [color={rgb, 255:red, 0; green, 0; blue, 255 }  ,draw opacity=1 ][line width=1.5]    (189.76,161.45) .. controls (195.31,166.52) and (229.83,178.21) .. (229.75,201.02) ;
\draw [color={rgb, 255:red, 255; green, 255; blue, 255 }  ,draw opacity=1 ][line width=3]    (213.67,171.16) -- (208.8,168.44) ;
\draw [color={rgb, 255:red, 255; green, 255; blue, 255 }  ,draw opacity=1 ][line width=3]    (136.31,176.52) -- (132.31,182.52) ;
\draw [color={rgb, 255:red, 0; green, 0; blue, 255 }  ,draw opacity=1 ][line width=1.5]    (154.96,201.45) .. controls (158.96,206.45) and (189.83,239.21) .. (206.34,237.27) ;
\draw [color={rgb, 255:red, 255; green, 255; blue, 255 }  ,draw opacity=1 ][line width=3]    (171.67,218.16) -- (165.21,213.56) ;
\draw [color={rgb, 255:red, 0; green, 0; blue, 255 }  ,draw opacity=1 ][line width=1.5]    (188.57,195.19) .. controls (180.96,202.45) and (159.96,221.45) .. (153.57,229.19) ;
\draw [color={rgb, 255:red, 0; green, 0; blue, 255 }  ,draw opacity=1 ][line width=1.5]    (110.18,100.83) .. controls (106.25,161.9) and (150.32,195.01) .. (154.96,201.45) ;
\draw [color={rgb, 255:red, 0; green, 0; blue, 255 }  ,draw opacity=1 ][line width=1.5]    (121.26,69.16) .. controls (147.26,63.16) and (152.02,80.45) .. (153.07,86.71) ;
\draw [color={rgb, 255:red, 0; green, 0; blue, 255 }  ,draw opacity=1 ][line width=1.5]    (223.26,68.16) .. controls (189.23,60.41) and (187.25,74.9) .. (186.57,86.19) ;
\draw [color={rgb, 255:red, 0; green, 0; blue, 255 }  ,draw opacity=1 ][line width=1.5]    (120.32,195.01) -- (97.31,219.52) ;
\draw [color={rgb, 255:red, 0; green, 0; blue, 255 }  ,draw opacity=1 ][line width=1.5]    (153.57,229.19) -- (127.64,256.57) ;
\draw [color={rgb, 255:red, 255; green, 255; blue, 255 }  ,draw opacity=1 ][line width=3]    (209.67,177.16) -- (203.21,173.56) ;
\draw [color={rgb, 255:red, 0; green, 0; blue, 255 }  ,draw opacity=1 ][line width=1.5]    (444.57,161.19) .. controls (437.25,168.13) and (417.28,190.27) .. (412.32,195.01) ;
\draw [color={rgb, 255:red, 0; green, 0; blue, 255 }  ,draw opacity=1 ][line width=1.5]    (497.33,192.8) .. controls (482.53,196.31) and (474.58,201.97) .. (465.89,193.47) .. controls (457.2,184.97) and (459.6,178.14) .. (462.02,175.54) .. controls (464.44,172.95) and (472.56,167.24) .. (479.72,178.67) .. controls (486.88,190.09) and (482.88,195.09) .. (480.38,211.67) ;
\draw [color={rgb, 255:red, 255; green, 255; blue, 255 }  ,draw opacity=1 ][line width=3.75]    (486.88,196.09) -- (477.88,197.09) ;
\draw [color={rgb, 255:red, 0; green, 0; blue, 255 }  ,draw opacity=1 ][line width=1.5]    (481.88,203.09) -- (483.88,193.09) ;
\draw [color={rgb, 255:red, 0; green, 0; blue, 255 }  ,draw opacity=1 ][line width=1.5]    (446.07,127.71) .. controls (446.19,138.77) and (485.26,151.29) .. (491.77,155.62) ;
\draw [color={rgb, 255:red, 255; green, 255; blue, 255 }  ,draw opacity=1 ][line width=3]    (469.67,146.16) -- (463.21,141.56) ;
\draw [color={rgb, 255:red, 0; green, 0; blue, 255 }  ,draw opacity=1 ][line width=1.5]    (479.57,127.19) .. controls (479.45,138.26) and (450.25,155.13) .. (444.57,161.19) ;
\draw [color={rgb, 255:red, 0; green, 0; blue, 255 }  ,draw opacity=1 ][line width=1.5]    (491.77,155.62) .. controls (496.77,158.62) and (544.77,178.62) .. (525.77,218.62) ;
\draw [color={rgb, 255:red, 255; green, 255; blue, 255 }  ,draw opacity=1 ][line width=3]    (503.67,176.16) -- (497.21,171.56) ;
\draw [color={rgb, 255:red, 255; green, 255; blue, 255 }  ,draw opacity=1 ][line width=3]    (429.25,176.13) -- (423.25,183.13) ;
\draw [color={rgb, 255:red, 0; green, 0; blue, 255 }  ,draw opacity=1 ][line width=1.5]    (446.96,202.45) .. controls (450.96,207.45) and (492.77,266.4) .. (525.77,218.62) ;
\draw [color={rgb, 255:red, 255; green, 255; blue, 255 }  ,draw opacity=1 ][line width=3]    (471.77,228.4) -- (465.77,222.97) ;
\draw [color={rgb, 255:red, 0; green, 0; blue, 255 }  ,draw opacity=1 ][line width=1.5]    (480.38,211.67) .. controls (475.52,233.86) and (446.52,221.86) .. (436.57,233.19) ;
\draw [color={rgb, 255:red, 0; green, 0; blue, 255 }  ,draw opacity=1 ][line width=1.5]    (412.32,195.01) -- (383.64,218.57) ;
\draw [color={rgb, 255:red, 0; green, 0; blue, 255 }  ,draw opacity=1 ][line width=1.5]    (436.57,233.19) -- (410.64,260.57) ;
\draw [color={rgb, 255:red, 0; green, 0; blue, 255 }  ,draw opacity=1 ][line width=1.5]    (530.03,100.33) .. controls (529.25,154.9) and (511.16,192.21) .. (497.33,192.8) ;
\draw [color={rgb, 255:red, 0; green, 0; blue, 255 }  ,draw opacity=1 ][line width=1.5]    (515.26,69.16) .. controls (535,74.38) and (543.54,85.44) .. (561.54,72.95) .. controls (579.55,60.46) and (577.07,47.83) .. (573.39,42.63) .. controls (569.71,37.42) and (559.25,30.34) .. (540.56,43.25) .. controls (521.86,56.15) and (530.29,72.85) .. (530.03,100.33) ;
\draw [color={rgb, 255:red, 255; green, 255; blue, 255 }  ,draw opacity=1 ][line width=3]    (529.03,78.86) -- (528.78,68.04) ;
\draw [color={rgb, 255:red, 0; green, 0; blue, 255 }  ,draw opacity=1 ][line width=1.5]    (532.13,75.42) -- (525.16,72.21) ;
\draw [color={rgb, 255:red, 0; green, 0; blue, 255 }  ,draw opacity=1 ][line width=1.5]    (402.18,101.83) .. controls (400,73.38) and (407,55.38) .. (389.65,45.22) .. controls (372.3,35.07) and (362.13,39.36) .. (358.65,44.56) .. controls (355.17,49.77) and (352.21,61.93) .. (369.82,74.93) .. controls (387.43,87.92) and (399,74.38) .. (413.26,70.16) ;
\draw [color={rgb, 255:red, 255; green, 255; blue, 255 }  ,draw opacity=1 ][line width=3]    (405.46,73.48) -- (395.85,77.41) ;
\draw [color={rgb, 255:red, 0; green, 0; blue, 255 }  ,draw opacity=1 ][line width=1.5]    (402,72.38) -- (401.81,78.35) ;
\draw [color={rgb, 255:red, 0; green, 0; blue, 255 }  ,draw opacity=1 ][line width=1.5]    (445.07,87.71) .. controls (445.19,98.77) and (480.96,111.1) .. (479.57,128.19) ;
\draw [color={rgb, 255:red, 255; green, 255; blue, 255 }  ,draw opacity=1 ][line width=3]    (464.67,106.16) -- (458.21,101.56) ;
\draw [color={rgb, 255:red, 0; green, 0; blue, 255 }  ,draw opacity=1 ][line width=1.5]    (478.57,87.19) .. controls (478.45,98.26) and (444.96,107.1) .. (446.07,128.71) ;
\draw [color={rgb, 255:red, 0; green, 0; blue, 255 }  ,draw opacity=1 ][line width=1.5]    (402.18,101.83) .. controls (398.25,162.9) and (442.32,196.01) .. (446.96,202.45) ;
\draw [color={rgb, 255:red, 0; green, 0; blue, 255 }  ,draw opacity=1 ][line width=1.5]    (413.26,70.16) .. controls (439.26,64.16) and (444.02,81.45) .. (445.07,87.71) ;
\draw [color={rgb, 255:red, 0; green, 0; blue, 255 }  ,draw opacity=1 ][line width=1.5]    (515.26,69.16) .. controls (481.23,61.41) and (479.25,75.9) .. (478.57,87.19) ;
\draw [color={rgb, 255:red, 255; green, 255; blue, 255 }  ,draw opacity=1 ][line width=3]    (520.16,171.21) -- (513.16,164.21) ;
\draw [color={rgb, 255:red, 255; green, 255; blue, 255 }  ,draw opacity=1 ][line width=3]    (516.16,177.21) -- (511.77,172.97) ;
\draw [color={rgb, 255:red, 0; green, 0; blue, 255 }  ,draw opacity=1 ][line width=1.5]    (97.51,278.45) .. controls (78.51,274.45) and (58.94,258.92) .. (41.01,271.52) .. controls (23.08,284.12) and (25.63,296.72) .. (29.34,301.91) .. controls (33.05,307.09) and (43.55,314.12) .. (62.17,301.1) .. controls (80.78,288.09) and (72.26,271.43) .. (72.37,243.96) ;
\draw [color={rgb, 255:red, 255; green, 255; blue, 255 }  ,draw opacity=1 ][line width=3]    (73.49,265.42) -- (73.8,276.23) ;
\draw [color={rgb, 255:red, 0; green, 0; blue, 255 }  ,draw opacity=1 ][line width=1.5]    (70.41,268.88) -- (77.39,272.05) ;
\draw [color={rgb, 255:red, 0; green, 0; blue, 255 }  ,draw opacity=1 ][line width=1.5]    (388.51,280.45) .. controls (369.51,276.45) and (349.94,260.92) .. (332.01,273.52) .. controls (314.08,286.12) and (316.63,298.72) .. (320.34,303.91) .. controls (324.05,309.09) and (334.55,316.12) .. (353.17,303.1) .. controls (371.78,290.09) and (363.26,273.43) .. (363.37,245.96) ;
\draw [color={rgb, 255:red, 255; green, 255; blue, 255 }  ,draw opacity=1 ][line width=3]    (364.49,267.42) -- (364.8,278.23) ;
\draw [color={rgb, 255:red, 0; green, 0; blue, 255 }  ,draw opacity=1 ][line width=1.5]    (361.41,270.88) -- (368.39,274.05) ;

\draw (154.54,325.9) node [anchor=north west][inner sep=0.75pt]    {$( a)$};
\draw (440.43,325.9) node [anchor=north west][inner sep=0.75pt]    {$( b)$};
\draw (481.57,130.59) node [anchor=north west][inner sep=0.75pt]    {$a_{6}$};
\draw (400.52,167.31) node [anchor=north west][inner sep=0.75pt]    {$a_{8}$};
\draw (486.01,196.91) node [anchor=north west][inner sep=0.75pt]    {$a_{4}$};
\draw (526.68,161.51) node [anchor=north west][inner sep=0.75pt]    {$a_{9}$};
\draw (379.16,250.32) node [anchor=north west][inner sep=0.75pt]  [color={rgb, 255:red, 0; green, 0; blue, 255 }  ,opacity=1 ,rotate=-319.97]  {$\dotsc $};
\draw (90,80.9) node [anchor=north west][inner sep=0.75pt]    {$a_{1}$};
\draw (242.13,77.82) node [anchor=north west][inner sep=0.75pt]    {$a_{2}$};
\draw (180.52,94.31) node [anchor=north west][inner sep=0.75pt]    {$a_{5}$};
\draw (185.52,137.31) node [anchor=north west][inner sep=0.75pt]    {$a_{6}$};
\draw (108.52,167.31) node [anchor=north west][inner sep=0.75pt]    {$a_{8}$};
\draw (231.71,226.31) node [anchor=north west][inner sep=0.75pt]    {$a_{4}$};
\draw (221.68,159.51) node [anchor=north west][inner sep=0.75pt]    {$a_{9}$};
\draw (89.16,250.54) node [anchor=north west][inner sep=0.75pt]  [color={rgb, 255:red, 0; green, 0; blue, 255 }  ,opacity=1 ,rotate=-319.97]  {$\dotsc $};
\draw (382,81.9) node [anchor=north west][inner sep=0.75pt]    {$a_{1}$};
\draw (534.13,78.82) node [anchor=north west][inner sep=0.75pt]    {$a_{2}$};
\draw (472.52,95.31) node [anchor=north west][inner sep=0.75pt]    {$a_{5}$};
\draw (52.34,272.34) node [anchor=north west][inner sep=0.75pt]    {$a_{3}$};
\draw (343.34,274.34) node [anchor=north west][inner sep=0.75pt]    {$a_{3}$};

\end{tikzpicture}

%% file: images/hatdsquare_=0/decomposition_1.tex
\begin{tikzpicture}[x=0.65pt,y=0.65pt,yscale=-1,xscale=1]

\draw  [draw opacity=0][fill={rgb, 255:red, 216; green, 216; blue, 216 }  ,fill opacity=1 ] (254.63,127.77) .. controls (274.63,112.77) and (307.63,108.77) .. (326.97,109.54) .. controls (324.79,140.67) and (324.79,173.67) .. (326.79,198.67) .. controls (303.63,200.77) and (273.63,199.77) .. (242.02,198.6) .. controls (241.63,159.77) and (237.63,140.77) .. (254.63,127.77) -- cycle ;
\draw  [draw opacity=0][fill={rgb, 255:red, 155; green, 155; blue, 155 }  ,fill opacity=1 ] (326.79,93.26) .. controls (326.79,70.89) and (344.92,52.76) .. (367.29,52.76) -- (367.29,52.76) .. controls (389.65,52.76) and (407.79,70.89) .. (407.79,93.26) -- (407.79,198.67) .. controls (407.79,198.67) and (407.79,198.67) .. (407.79,198.67) -- (326.79,198.67) .. controls (326.79,198.67) and (326.79,198.67) .. (326.79,198.67) -- cycle ;
\draw [color={rgb, 255:red, 144; green, 19; blue, 254 }  ,draw opacity=1 ][line width=1.5]    (198.79,199.61) -- (449.65,199.61) ;
\draw [color={rgb, 255:red, 0; green, 0; blue, 255 }  ,draw opacity=1 ][line width=1.5]    (326.22,198.61) .. controls (326.47,148.54) and (325.07,128.25) .. (325.63,99.54) .. controls (326.2,70.83) and (346.79,52.67) .. (366.18,52.76) ;
\draw [color={rgb, 255:red, 0; green, 0; blue, 255 }  ,draw opacity=1 ][line width=1.5]    (408.73,198.61) .. controls (408.98,148.54) and (408.15,120.54) .. (408.15,99.54) .. controls (408.15,78.54) and (394.79,53.67) .. (366.18,52.76) ;
\draw [color={rgb, 255:red, 0; green, 0; blue, 255 }  ,draw opacity=1 ][line width=1.5]    (242.02,198.6) .. controls (242.02,177.6) and (242.02,157.6) .. (242.02,150.6) .. controls (242.86,120.6) and (288.82,111.41) .. (322.16,109.41) ;
\draw [color={rgb, 255:red, 0; green, 0; blue, 255 }  ,draw opacity=1 ][line width=1.5]    (326.97,109.54) .. controls (330.16,109.41) and (335.16,109.41) .. (340.99,108.41) ;

\draw (290,234.4) node [anchor=north west][inner sep=0.75pt]    {$w=w'*w''$};
\draw (330,139.4) node [anchor=north west][inner sep=0.75pt]    {$\gamma $};
\draw (277.72,186.4) node [anchor=north west][inner sep=0.75pt]  [font=\scriptsize]  {$+$};
\draw (361.07,187.4) node [anchor=north west][inner sep=0.75pt]  [font=\scriptsize]  {$+$};
\draw (329.24,113.81) node [anchor=north west][inner sep=0.75pt]  [font=\scriptsize]  {$+$};
\draw (308.41,96.81) node [anchor=north west][inner sep=0.75pt]  [font=\scriptsize]  {$+$};
\draw (308.41,113.81) node [anchor=north west][inner sep=0.75pt]  [font=\scriptsize]  {$-$};
\draw (305.56,82.4) node [anchor=north west][inner sep=0.75pt]    {$a$};
\draw (271.39,204.4) node [anchor=north west][inner sep=0.75pt]    {$w'$};
\draw (355.49,206.4) node [anchor=north west][inner sep=0.75pt]    {$w''$};
\draw (451.91,193.01) node [anchor=north west][inner sep=0.75pt]  [color={rgb, 255:red, 144; green, 19; blue, 254 }  ,opacity=1 ]  {$\Gamma $};
\draw (416.99,126.4) node [anchor=north west][inner sep=0.75pt]  [color={rgb, 255:red, 0; green, 0; blue, 255 }  ,opacity=1 ]  {$\Lambda $};
\draw (330.07,96.81) node [anchor=north west][inner sep=0.75pt]  [font=\scriptsize]  {$-$};

\end{tikzpicture}

%% file: images/hatdsquare_=0/decomposition_2.tex
\begin{tikzpicture}[x=0.65pt,y=0.65pt,yscale=-1,xscale=1]

\draw  [draw opacity=0][fill={rgb, 255:red, 155; green, 155; blue, 155 }  ,fill opacity=1 ] (317.28,107.44) .. controls (317.28,85.07) and (335.41,66.94) .. (357.78,66.94) -- (357.78,66.94) .. controls (380.15,66.94) and (398.28,85.07) .. (398.28,107.44) -- (398.28,212.85) .. controls (398.28,212.85) and (398.28,212.85) .. (398.28,212.85) -- (317.28,212.85) .. controls (317.28,212.85) and (317.28,212.85) .. (317.28,212.85) -- cycle ;
\draw  [draw opacity=0][fill={rgb, 255:red, 216; green, 216; blue, 216 }  ,fill opacity=1 ] (245.63,142.01) .. controls (265.63,127.01) and (298.63,123.01) .. (317.97,123.78) .. controls (335.51,122.5) and (375.51,127.5) .. (398.28,135.87) .. controls (399.51,153.5) and (399.8,163.64) .. (399.79,171.28) .. controls (399.77,178.91) and (399.51,192.5) .. (399.73,212.85) .. controls (370.51,214.5) and (329.94,213.44) .. (315.22,213.85) .. controls (292.51,213.5) and (264.63,214.01) .. (233.02,212.84) .. controls (232.63,174.01) and (228.63,155.01) .. (245.63,142.01) -- cycle ;
\draw [color={rgb, 255:red, 144; green, 19; blue, 254 }  ,draw opacity=1 ][line width=1.5]    (189.79,213.85) -- (440.65,213.85) ;
\draw [color={rgb, 255:red, 0; green, 0; blue, 255 }  ,draw opacity=1 ][line width=1.5]    (316.28,131.87) .. controls (316.28,124.87) and (316.07,142.49) .. (316.63,113.78) .. controls (317.2,85.07) and (337.21,65.89) .. (357.18,67) ;
\draw [color={rgb, 255:red, 0; green, 0; blue, 255 }  ,draw opacity=1 ][line width=1.5]    (399.73,212.85) .. controls (399.98,162.78) and (399.15,134.78) .. (399.15,113.78) .. controls (399.15,92.78) and (386.21,65.89) .. (357.18,67) ;
\draw [color={rgb, 255:red, 0; green, 0; blue, 255 }  ,draw opacity=1 ][line width=1.5]    (233.02,212.84) .. controls (233.02,191.84) and (233.02,171.84) .. (233.02,164.84) .. controls (233.86,134.84) and (279.82,125.65) .. (313.16,123.65) ;
\draw [color={rgb, 255:red, 0; green, 0; blue, 255 }  ,draw opacity=1 ][line width=1.5]    (317.97,123.78) .. controls (343.28,121.87) and (384.28,130.87) .. (395.28,135.87) ;
\draw [color={rgb, 255:red, 0; green, 0; blue, 255 }  ,draw opacity=1 ][line width=1.5]    (403.14,138.78) .. controls (408.28,140.48) and (411.28,142.48) .. (418.28,145.48) ;

\draw (313,219.64) node [anchor=north west][inner sep=0.75pt]    {$w$};
\draw (314.72,199.64) node [anchor=north west][inner sep=0.75pt]  [font=\scriptsize]  {$+$};
\draw (319.97,127.18) node [anchor=north west][inner sep=0.75pt]  [font=\scriptsize]  {$+$};
\draw (299.41,111.05) node [anchor=north west][inner sep=0.75pt]  [font=\scriptsize]  {$+$};
\draw (299.41,128.05) node [anchor=north west][inner sep=0.75pt]  [font=\scriptsize]  {$-$};
\draw (294.56,98.64) node [anchor=north west][inner sep=0.75pt]    {$a$};
\draw (442.91,207.25) node [anchor=north west][inner sep=0.75pt]  [color={rgb, 255:red, 144; green, 19; blue, 254 }  ,opacity=1 ]  {$\Gamma $};
\draw (403.99,169.64) node [anchor=north west][inner sep=0.75pt]  [color={rgb, 255:red, 0; green, 0; blue, 255 }  ,opacity=1 ]  {$\Lambda $};
\draw (321.07,111.05) node [anchor=north west][inner sep=0.75pt]  [font=\scriptsize]  {$-$};
\draw (402.97,142.18) node [anchor=north west][inner sep=0.75pt]  [font=\scriptsize]  {$+$};
\draw (387.41,122.05) node [anchor=north west][inner sep=0.75pt]  [font=\scriptsize]  {$+$};
\draw (387.41,135.05) node [anchor=north west][inner sep=0.75pt]  [font=\scriptsize]  {$-$};
\draw (403.07,126.05) node [anchor=north west][inner sep=0.75pt]  [font=\scriptsize]  {$-$};

\end{tikzpicture}

%% file: images/hatdsquare_=0/decomposition_3.tex
\begin{tikzpicture}[x=0.65pt,y=0.65pt,yscale=-1,xscale=1]

\draw  [draw opacity=0][fill={rgb, 255:red, 155; green, 155; blue, 155 }  ,fill opacity=1 ] (526.91,95.25) .. controls (526.91,78.87) and (540.19,65.59) .. (556.57,65.59) -- (556.57,65.59) .. controls (572.95,65.59) and (586.23,78.87) .. (586.23,95.25) -- (586.23,175.68) .. controls (586.23,175.68) and (586.23,175.68) .. (586.23,175.68) -- (526.91,175.68) .. controls (526.91,175.68) and (526.91,175.68) .. (526.91,175.68) -- cycle ;
\draw  [draw opacity=0][fill={rgb, 255:red, 216; green, 216; blue, 216 }  ,fill opacity=1 ] (477.52,122.24) .. controls (491.63,110.92) and (514.91,107.9) .. (528.55,108.48) .. controls (542.2,106.54) and (547.15,133.03) .. (547.14,148.04) .. controls (549.25,180.48) and (553.48,176.71) .. (528.42,175.73) .. controls (512.09,177.32) and (490.92,176.56) .. (468.63,175.68) .. controls (468.35,146.38) and (465.52,132.04) .. (477.52,122.24) -- cycle ;
\draw [color={rgb, 255:red, 144; green, 19; blue, 254 }  ,draw opacity=1 ][line width=1.5]    (438.12,176.44) -- (615.1,176.44) ;
\draw [color={rgb, 255:red, 0; green, 0; blue, 255 }  ,draw opacity=1 ][line width=1.5]    (528.17,120.82) .. controls (527.46,113.28) and (526.76,109.51) .. (526.91,97.32) .. controls (526.76,80.08) and (538.67,65.17) .. (556.22,65.63) ;
\draw [color={rgb, 255:red, 0; green, 0; blue, 255 }  ,draw opacity=1 ][line width=1.5]    (586.23,175.68) .. controls (586.41,137.91) and (585.82,116.78) .. (585.82,100.93) .. controls (585.82,85.09) and (576.4,66.32) .. (556.22,65.63) ;
\draw [color={rgb, 255:red, 0; green, 0; blue, 255 }  ,draw opacity=1 ][line width=1.5]    (468.63,175.68) .. controls (468.63,159.83) and (468.63,144.74) .. (468.63,139.46) .. controls (469.21,116.82) and (501.64,109.89) .. (525.16,108.38) ;
\draw [color={rgb, 255:red, 0; green, 0; blue, 255 }  ,draw opacity=1 ][line width=1.5]    (528.55,108.48) .. controls (549.1,103.69) and (549.1,164.05) .. (549.1,173.86) ;
\draw [color={rgb, 255:red, 144; green, 19; blue, 254 }  ,draw opacity=1 ][line width=1.5]    (227.72,176.42) -- (404.7,176.42) ;
\draw [color={rgb, 255:red, 0; green, 0; blue, 255 }  ,draw opacity=1 ][line width=1.5]    (317.62,173.66) .. controls (317.8,135.88) and (316.81,122.57) .. (317.21,100.91) .. controls (317.61,79.25) and (332.13,65.55) .. (345.82,65.61) ;
\draw [color={rgb, 255:red, 0; green, 0; blue, 255 }  ,draw opacity=1 ][line width=1.5]    (375.83,175.66) .. controls (376.01,137.88) and (375.42,116.76) .. (375.42,100.91) .. controls (375.42,85.07) and (365.99,66.3) .. (345.82,65.61) ;
\draw [color={rgb, 255:red, 0; green, 0; blue, 255 }  ,draw opacity=1 ][line width=1.5]    (258.22,175.65) .. controls (258.22,159.81) and (258.22,144.72) .. (258.22,139.44) .. controls (258.81,116.8) and (291.24,109.87) .. (314.75,108.36) ;
\draw [color={rgb, 255:red, 0; green, 0; blue, 255 }  ,draw opacity=1 ][line width=1.5]    (318.15,108.46) .. controls (338.7,103.67) and (338.7,164.03) .. (338.7,173.84) ;
\draw  [draw opacity=0][fill={rgb, 255:red, 216; green, 216; blue, 216 }  ,fill opacity=1 ] (56.71,122.21) .. controls (70.82,110.89) and (94.1,107.88) .. (107.75,108.46) .. controls (106.21,131.95) and (106.21,156.84) .. (107.62,175.71) .. controls (91.28,177.29) and (70.12,176.54) .. (47.82,175.65) .. controls (47.54,146.36) and (44.72,132.02) .. (56.71,122.21) -- cycle ;
\draw  [draw opacity=0][fill={rgb, 255:red, 155; green, 155; blue, 155 }  ,fill opacity=1 ] (107.62,94.18) .. controls (107.62,78.4) and (120.41,65.61) .. (136.19,65.61) -- (136.19,65.61) .. controls (151.97,65.61) and (164.76,78.4) .. (164.76,94.18) -- (164.76,175.71) .. controls (164.76,175.71) and (164.76,175.71) .. (164.76,175.71) -- (107.62,175.71) .. controls (107.62,175.71) and (107.62,175.71) .. (107.62,175.71) -- cycle ;
\draw [color={rgb, 255:red, 144; green, 19; blue, 254 }  ,draw opacity=1 ][line width=1.5]    (17.32,176.42) -- (194.3,176.42) ;
\draw [color={rgb, 255:red, 0; green, 0; blue, 255 }  ,draw opacity=1 ][line width=1.5]    (107.22,175.66) .. controls (107.39,137.88) and (106.4,122.57) .. (106.81,100.91) .. controls (107.21,79.25) and (121.73,65.55) .. (135.41,65.61) ;
\draw [color={rgb, 255:red, 0; green, 0; blue, 255 }  ,draw opacity=1 ][line width=1.5]    (165.43,175.66) .. controls (165.6,137.88) and (165.01,116.76) .. (165.01,100.91) .. controls (165.01,85.07) and (155.59,66.3) .. (135.41,65.61) ;
\draw [color={rgb, 255:red, 0; green, 0; blue, 255 }  ,draw opacity=1 ][line width=1.5]    (47.82,175.65) .. controls (47.82,159.81) and (47.82,144.72) .. (47.82,139.44) .. controls (48.41,116.8) and (80.83,109.87) .. (104.35,108.36) ;
\draw [color={rgb, 255:red, 144; green, 19; blue, 254 }  ,draw opacity=1 ][line width=1.5]    (227.72,361.32) -- (404.7,361.32) ;
\draw [color={rgb, 255:red, 0; green, 0; blue, 255 }  ,draw opacity=1 ][line width=1.5]    (373.75,329.19) .. controls (339.18,329.19) and (318.6,318.05) .. (317.8,301.92) .. controls (317.23,294.14) and (317.58,294.63) .. (317.66,281.03) .. controls (317.74,267.42) and (331.72,249.67) .. (345.82,250.51) ;
\draw [color={rgb, 255:red, 0; green, 0; blue, 255 }  ,draw opacity=1 ][line width=1.5]    (375.83,360.56) .. controls (376.01,322.78) and (375.42,301.66) .. (375.42,285.81) .. controls (375.42,269.97) and (366.29,249.67) .. (345.82,250.51) ;
\draw [color={rgb, 255:red, 0; green, 0; blue, 255 }  ,draw opacity=1 ][line width=1.5]    (258.22,360.55) .. controls (258.22,344.71) and (258.22,329.62) .. (258.22,324.33) .. controls (258.81,301.7) and (291.24,294.77) .. (314.75,293.26) ;
\draw [color={rgb, 255:red, 0; green, 0; blue, 255 }  ,draw opacity=1 ][line width=1.5]    (319.56,293.36) .. controls (337.42,291.91) and (364.93,298.7) .. (372.69,302.48) ;
\draw [color={rgb, 255:red, 0; green, 0; blue, 255 }  ,draw opacity=1 ][line width=1.5]    (378.24,304.67) .. controls (381.86,305.95) and (383.98,307.46) .. (388.92,309.73) ;
\draw [color={rgb, 255:red, 0; green, 0; blue, 255 }  ,draw opacity=1 ][line width=1.5]    (378.62,329.02) .. controls (383.56,329.02) and (386.38,328.27) .. (389.9,326) ;
\draw  [draw opacity=0][fill={rgb, 255:red, 155; green, 155; blue, 155 }  ,fill opacity=1 ] (528.07,279.05) .. controls (528.07,263.27) and (540.86,250.47) .. (556.64,250.47) -- (556.64,250.47) .. controls (572.42,250.47) and (585.21,263.27) .. (585.21,279.05) -- (585.21,360.57) .. controls (585.21,360.57) and (585.21,360.57) .. (585.21,360.57) -- (528.07,360.57) .. controls (528.07,360.57) and (528.07,360.57) .. (528.07,360.57) -- cycle ;
\draw  [draw opacity=0][fill={rgb, 255:red, 216; green, 216; blue, 216 }  ,fill opacity=1 ] (477.52,307.12) .. controls (491.63,295.8) and (514.91,292.78) .. (528.55,293.36) .. controls (540.93,292.4) and (569.15,296.17) .. (585.21,302.48) .. controls (586.08,315.79) and (586.28,323.44) .. (586.27,329.2) .. controls (586.26,334.96) and (586.08,345.21) .. (586.23,360.57) .. controls (565.62,361.81) and (537,361.01) .. (526.61,361.32) .. controls (510.59,361.06) and (490.92,361.45) .. (468.63,360.56) .. controls (468.35,331.27) and (465.52,316.93) .. (477.52,307.12) -- cycle ;
\draw [color={rgb, 255:red, 144; green, 19; blue, 254 }  ,draw opacity=1 ][line width=1.5]    (438.12,361.32) -- (615.1,361.32) ;
\draw [color={rgb, 255:red, 0; green, 0; blue, 255 }  ,draw opacity=1 ][line width=1.5]    (528.2,301.93) .. controls (527.63,294.15) and (527.98,294.64) .. (528.07,281.03) .. controls (528.15,267.43) and (542.12,249.68) .. (556.22,250.52) ;
\draw [color={rgb, 255:red, 0; green, 0; blue, 255 }  ,draw opacity=1 ][line width=1.5]    (586.23,360.57) .. controls (586.41,322.79) and (585.82,301.66) .. (585.82,285.82) .. controls (585.82,269.97) and (576.69,249.68) .. (556.22,250.52) ;
\draw [color={rgb, 255:red, 0; green, 0; blue, 255 }  ,draw opacity=1 ][line width=1.5]    (468.63,360.56) .. controls (468.63,344.72) and (468.63,329.62) .. (468.63,324.34) .. controls (469.21,301.71) and (501.64,294.77) .. (525.16,293.27) ;
\draw [color={rgb, 255:red, 0; green, 0; blue, 255 }  ,draw opacity=1 ][line width=1.5]    (529.96,293.36) .. controls (547.82,291.92) and (575.33,298.71) .. (583.09,302.48) ;
\draw [color={rgb, 255:red, 0; green, 0; blue, 255 }  ,draw opacity=1 ][line width=1.5]    (588.64,304.68) .. controls (592.27,305.96) and (594.38,307.47) .. (599.32,309.73) ;
\draw  [draw opacity=0][fill={rgb, 255:red, 155; green, 155; blue, 155 }  ,fill opacity=1 ] (107.26,279.17) .. controls (107.26,263.32) and (120.11,250.47) .. (135.96,250.47) -- (135.96,250.47) .. controls (151.8,250.47) and (164.65,263.32) .. (164.65,279.17) -- (164.65,360.57) .. controls (164.65,360.57) and (164.65,360.57) .. (164.65,360.57) -- (107.26,360.57) .. controls (107.26,360.57) and (107.26,360.57) .. (107.26,360.57) -- cycle ;
\draw  [draw opacity=0][fill={rgb, 255:red, 216; green, 216; blue, 216 }  ,fill opacity=1 ] (56.71,307.12) .. controls (70.82,295.8) and (94.1,292.78) .. (107.75,293.36) .. controls (105.39,308.85) and (113.86,320.93) .. (130.79,325.45) .. controls (147.72,329.23) and (157.6,329.98) .. (165.47,329.2) .. controls (165.46,334.96) and (165.27,345.21) .. (165.43,360.57) .. controls (144.81,361.81) and (116.2,361.01) .. (105.81,361.32) .. controls (89.79,361.06) and (70.12,361.45) .. (47.82,360.56) .. controls (47.54,331.27) and (44.72,316.93) .. (56.71,307.12) -- cycle ;
\draw [color={rgb, 255:red, 144; green, 19; blue, 254 }  ,draw opacity=1 ][line width=1.5]    (17.32,361.32) -- (194.3,361.32) ;
\draw [color={rgb, 255:red, 0; green, 0; blue, 255 }  ,draw opacity=1 ][line width=1.5]    (163.35,329.2) .. controls (128.78,329.2) and (108.2,318.06) .. (107.4,301.93) .. controls (106.83,294.15) and (107.18,294.64) .. (107.26,281.03) .. controls (107.34,267.43) and (121.32,249.68) .. (135.41,250.52) ;
\draw [color={rgb, 255:red, 0; green, 0; blue, 255 }  ,draw opacity=1 ][line width=1.5]    (165.43,360.57) .. controls (165.6,322.79) and (165.01,301.66) .. (165.01,285.82) .. controls (165.01,269.97) and (155.89,249.68) .. (135.41,250.52) ;
\draw [color={rgb, 255:red, 0; green, 0; blue, 255 }  ,draw opacity=1 ][line width=1.5]    (47.82,360.56) .. controls (47.82,344.72) and (47.82,329.62) .. (47.82,324.34) .. controls (48.41,301.71) and (80.83,294.77) .. (104.35,293.27) ;
\draw [color={rgb, 255:red, 0; green, 0; blue, 255 }  ,draw opacity=1 ][line width=1.5]    (168.21,329.03) .. controls (173.15,329.03) and (175.97,328.27) .. (179.5,326.01) ;
\draw  [draw opacity=0][fill={rgb, 255:red, 128; green, 128; blue, 128 }  ,fill opacity=1 ] (48.53,508.32) .. controls (49.22,490.34) and (68.95,487.15) .. (72.46,484.3) .. controls (75.98,481.45) and (89.43,479.02) .. (107.04,478.85) .. controls (109.89,497.88) and (100.82,496.58) .. (88.39,502.24) .. controls (74.38,507.73) and (74.99,510.54) .. (48.53,508.32) -- cycle ;
\draw  [draw opacity=0][fill={rgb, 255:red, 216; green, 216; blue, 216 }  ,fill opacity=1 ] (64.77,506.34) .. controls (97.56,500.73) and (107.44,502.24) .. (107.75,478.85) .. controls (102.09,451.1) and (124.8,434.4) .. (135.41,436.01) .. controls (169.82,439.03) and (166.99,480.53) .. (164.41,487.97) .. controls (165.27,501.27) and (165.47,508.93) .. (165.47,514.69) .. controls (165.46,520.45) and (165.27,530.7) .. (165.43,546.06) .. controls (144.81,547.3) and (116.2,546.5) .. (105.81,546.81) .. controls (89.79,546.55) and (70.12,546.93) .. (47.82,546.05) .. controls (47.14,501.06) and (42.9,508.6) .. (64.77,506.34) -- cycle ;
\draw [color={rgb, 255:red, 144; green, 19; blue, 254 }  ,draw opacity=1 ][line width=1.5]    (17.32,546.81) -- (194.3,546.81) ;
\draw [color={rgb, 255:red, 0; green, 0; blue, 255 }  ,draw opacity=1 ][line width=1.5]    (45.89,508.19) .. controls (65.2,506.64) and (108.2,503.55) .. (107.4,487.41) .. controls (106.83,479.64) and (107.18,480.13) .. (107.26,466.52) .. controls (107.34,452.92) and (121.32,435.17) .. (135.41,436.01) ;
\draw [color={rgb, 255:red, 0; green, 0; blue, 255 }  ,draw opacity=1 ][line width=1.5]    (165.43,546.06) .. controls (165.6,508.28) and (165.01,487.15) .. (165.01,471.31) .. controls (165.01,455.46) and (155.89,435.17) .. (135.41,436.01) ;
\draw [color={rgb, 255:red, 0; green, 0; blue, 255 }  ,draw opacity=1 ][line width=1.5]    (48.98,504.37) .. controls (50.84,486.85) and (80.83,480.26) .. (104.35,478.75) ;
\draw [color={rgb, 255:red, 0; green, 0; blue, 255 }  ,draw opacity=1 ][line width=1.5]    (34.97,507.27) .. controls (39.91,508.78) and (39.91,508.02) .. (45.89,508.19) ;
\draw [color={rgb, 255:red, 0; green, 0; blue, 255 }  ,draw opacity=1 ][line width=1.5]    (47.82,546.05) .. controls (48.27,528.52) and (46.15,523.99) .. (48.53,509.83) ;
\draw [color={rgb, 255:red, 144; green, 19; blue, 254 }  ,draw opacity=1 ][line width=1.5]    (227.72,546.77) -- (404.7,546.77) ;
\draw [color={rgb, 255:red, 0; green, 0; blue, 255 }  ,draw opacity=1 ][line width=1.5]    (375.83,546.02) .. controls (376.01,508.24) and (375.42,487.11) .. (375.42,471.27) .. controls (375.42,455.42) and (366.29,435.13) .. (345.82,435.97) ;
\draw [color={rgb, 255:red, 0; green, 0; blue, 255 }  ,draw opacity=1 ][line width=1.5]    (319.56,478.81) .. controls (337.42,477.37) and (364.93,484.16) .. (372.69,487.93) ;
\draw [color={rgb, 255:red, 0; green, 0; blue, 255 }  ,draw opacity=1 ][line width=1.5]    (378.24,490.13) .. controls (381.86,491.41) and (383.98,492.92) .. (388.92,495.18) ;
\draw [color={rgb, 255:red, 0; green, 0; blue, 255 }  ,draw opacity=1 ][line width=1.5]    (257,508.15) .. controls (276.31,506.6) and (319.31,503.51) .. (318.51,487.38) .. controls (317.94,479.6) and (318.29,480.09) .. (318.37,466.48) .. controls (318.45,452.88) and (332.43,435.13) .. (346.52,435.97) ;
\draw [color={rgb, 255:red, 0; green, 0; blue, 255 }  ,draw opacity=1 ][line width=1.5]    (260.09,504.34) .. controls (261.95,486.81) and (291.94,480.22) .. (315.46,478.71) ;
\draw [color={rgb, 255:red, 0; green, 0; blue, 255 }  ,draw opacity=1 ][line width=1.5]    (246.08,507.23) .. controls (251.02,508.74) and (251.02,507.98) .. (257,508.15) ;
\draw [color={rgb, 255:red, 0; green, 0; blue, 255 }  ,draw opacity=1 ][line width=1.5]    (258.93,546.01) .. controls (259.38,528.48) and (257.26,523.95) .. (259.64,509.79) ;
\draw  [draw opacity=0][fill={rgb, 255:red, 155; green, 155; blue, 155 }  ,fill opacity=1 ] (528.07,464.5) .. controls (528.07,448.72) and (540.86,435.93) .. (556.64,435.93) -- (556.64,435.93) .. controls (572.42,435.93) and (585.21,448.72) .. (585.21,464.5) -- (585.21,546.03) .. controls (585.21,546.03) and (585.21,546.03) .. (585.21,546.03) -- (528.07,546.03) .. controls (528.07,546.03) and (528.07,546.03) .. (528.07,546.03) -- cycle ;
\draw  [draw opacity=0][fill={rgb, 255:red, 216; green, 216; blue, 216 }  ,fill opacity=1 ] (477.52,492.58) .. controls (491.63,481.26) and (514.91,478.24) .. (528.55,478.82) .. controls (540.93,477.85) and (569.15,481.63) .. (585.21,487.94) .. controls (586.08,501.24) and (586.28,508.9) .. (586.27,514.66) .. controls (586.26,520.42) and (586.08,530.67) .. (586.23,546.03) .. controls (565.62,547.27) and (537,546.47) .. (526.61,546.78) .. controls (510.59,546.52) and (490.92,546.9) .. (468.63,546.02) .. controls (468.35,516.72) and (465.52,502.39) .. (477.52,492.58) -- cycle ;
\draw [color={rgb, 255:red, 144; green, 19; blue, 254 }  ,draw opacity=1 ][line width=1.5]    (438.12,546.78) -- (615.1,546.78) ;
\draw [color={rgb, 255:red, 0; green, 0; blue, 255 }  ,draw opacity=1 ][line width=1.5]    (528.2,487.38) .. controls (527.63,479.6) and (527.98,480.1) .. (528.07,466.49) .. controls (528.15,452.89) and (542.12,435.14) .. (556.22,435.98) ;
\draw [color={rgb, 255:red, 0; green, 0; blue, 255 }  ,draw opacity=1 ][line width=1.5]    (586.23,546.03) .. controls (586.41,508.25) and (585.82,487.12) .. (585.82,471.27) .. controls (585.82,455.43) and (576.69,435.14) .. (556.22,435.98) ;
\draw [color={rgb, 255:red, 0; green, 0; blue, 255 }  ,draw opacity=1 ][line width=1.5]    (468.63,546.02) .. controls (468.63,530.17) and (468.63,515.08) .. (468.63,509.8) .. controls (469.21,487.16) and (501.64,480.23) .. (525.16,478.72) ;
\draw [color={rgb, 255:red, 0; green, 0; blue, 255 }  ,draw opacity=1 ][line width=1.5]    (529.96,478.82) .. controls (547.82,477.38) and (575.33,484.17) .. (583.09,487.94) ;
\draw [color={rgb, 255:red, 0; green, 0; blue, 255 }  ,draw opacity=1 ][line width=1.5]    (588.64,490.14) .. controls (592.27,491.42) and (594.38,492.93) .. (599.32,495.19) ;

\draw (261,22.5) node [anchor=north west][inner sep=0.75pt]   [align=left] {Strand Configuration};
\draw (61,22.5) node [anchor=north west][inner sep=0.75pt]   [align=left] {Degeneration 1};
\draw (495,22.5) node [anchor=north west][inner sep=0.75pt]   [align=left] {Degeneration 2};
\draw (281.5,165.71) node [anchor=north west][inner sep=0.75pt]  [font=\scriptsize]  {$+$};
\draw (345.94,165.71) node [anchor=north west][inner sep=0.75pt]  [font=\scriptsize]  {$+$};
\draw (317.84,110.55) node [anchor=north west][inner sep=0.75pt]  [font=\scriptsize]  {$+$};
\draw (303.14,97.72) node [anchor=north west][inner sep=0.75pt]  [font=\scriptsize]  {$+$};
\draw (303.14,110.55) node [anchor=north west][inner sep=0.75pt]  [font=\scriptsize]  {$-$};
\draw (299.13,84.11) node [anchor=north west][inner sep=0.75pt]    {$a$};
\draw (404.37,169.57) node [anchor=north west][inner sep=0.75pt]  [color={rgb, 255:red, 144; green, 19; blue, 254 }  ,opacity=1 ]  {$\Gamma $};
\draw (318.43,97.72) node [anchor=north west][inner sep=0.75pt]  [font=\scriptsize]  {$-$};
\draw (321.25,165.71) node [anchor=north west][inner sep=0.75pt]  [font=\scriptsize]  {$+$};
\draw (71.09,165.71) node [anchor=north west][inner sep=0.75pt]  [font=\scriptsize]  {$+$};
\draw (129.89,165.71) node [anchor=north west][inner sep=0.75pt]  [font=\scriptsize]  {$+$};
\draw (107.43,110.55) node [anchor=north west][inner sep=0.75pt]  [font=\scriptsize]  {$+$};
\draw (92.74,97.72) node [anchor=north west][inner sep=0.75pt]  [font=\scriptsize]  {$+$};
\draw (92.74,110.55) node [anchor=north west][inner sep=0.75pt]  [font=\scriptsize]  {$-$};
\draw (88.73,84.11) node [anchor=north west][inner sep=0.75pt]    {$a$};
\draw (193.97,169.57) node [anchor=north west][inner sep=0.75pt]  [color={rgb, 255:red, 144; green, 19; blue, 254 }  ,opacity=1 ]  {$\Gamma $};
\draw (108.02,97.72) node [anchor=north west][inner sep=0.75pt]  [font=\scriptsize]  {$-$};
\draw (503.9,165.71) node [anchor=north west][inner sep=0.75pt]  [font=\scriptsize]  {$+$};
\draw (557.75,165.71) node [anchor=north west][inner sep=0.75pt]  [font=\scriptsize]  {$+$};
\draw (528.24,110.57) node [anchor=north west][inner sep=0.75pt]  [font=\scriptsize]  {$+$};
\draw (513.54,97.74) node [anchor=north west][inner sep=0.75pt]  [font=\scriptsize]  {$+$};
\draw (513.54,110.57) node [anchor=north west][inner sep=0.75pt]  [font=\scriptsize]  {$-$};
\draw (509.54,84.14) node [anchor=north west][inner sep=0.75pt]    {$a$};
\draw (614.78,169.59) node [anchor=north west][inner sep=0.75pt]  [color={rgb, 255:red, 144; green, 19; blue, 254 }  ,opacity=1 ]  {$\Gamma $};
\draw (528.83,97.74) node [anchor=north west][inner sep=0.75pt]  [font=\scriptsize]  {$-$};
\draw (101.39,349.98) node [anchor=north west][inner sep=0.75pt]  [font=\scriptsize]  {$+$};
\draw (107.24,294.8) node [anchor=north west][inner sep=0.75pt]  [font=\scriptsize]  {$+$};
\draw (92.74,282.63) node [anchor=north west][inner sep=0.75pt]  [font=\scriptsize]  {$+$};
\draw (92.74,295.46) node [anchor=north west][inner sep=0.75pt]  [font=\scriptsize]  {$-$};
\draw (87.32,270.53) node [anchor=north west][inner sep=0.75pt]    {$a$};
\draw (193.97,354.48) node [anchor=north west][inner sep=0.75pt]  [color={rgb, 255:red, 144; green, 19; blue, 254 }  ,opacity=1 ]  {$\Gamma $};
\draw (108.02,282.63) node [anchor=north west][inner sep=0.75pt]  [font=\scriptsize]  {$-$};
\draw (165.8,331.02) node [anchor=north west][inner sep=0.75pt]  [font=\scriptsize]  {$+$};
\draw (152.7,316.58) node [anchor=north west][inner sep=0.75pt]  [font=\scriptsize]  {$+$};
\draw (154.11,328.66) node [anchor=north west][inner sep=0.75pt]  [font=\scriptsize]  {$-$};
\draw (165.17,318.85) node [anchor=north west][inner sep=0.75pt]  [font=\scriptsize]  {$-$};
\draw (312.29,349.98) node [anchor=north west][inner sep=0.75pt]  [font=\scriptsize]  {$+$};
\draw (317.65,294.79) node [anchor=north west][inner sep=0.75pt]  [font=\scriptsize]  {$+$};
\draw (303.14,282.62) node [anchor=north west][inner sep=0.75pt]  [font=\scriptsize]  {$+$};
\draw (303.14,295.45) node [anchor=north west][inner sep=0.75pt]  [font=\scriptsize]  {$-$};
\draw (297.72,270.52) node [anchor=north west][inner sep=0.75pt]    {$a$};
\draw (404.37,354.47) node [anchor=north west][inner sep=0.75pt]  [color={rgb, 255:red, 144; green, 19; blue, 254 }  ,opacity=1 ]  {$\Gamma $};
\draw (318.43,282.62) node [anchor=north west][inner sep=0.75pt]  [font=\scriptsize]  {$-$};
\draw (376.2,306.11) node [anchor=north west][inner sep=0.75pt]  [font=\scriptsize]  {$+$};
\draw (365.22,290.92) node [anchor=north west][inner sep=0.75pt]  [font=\scriptsize]  {$+$};
\draw (365.22,300.73) node [anchor=north west][inner sep=0.75pt]  [font=\scriptsize]  {$-$};
\draw (376.27,293.94) node [anchor=north west][inner sep=0.75pt]  [font=\scriptsize]  {$-$};
\draw (376.2,330.25) node [anchor=north west][inner sep=0.75pt]  [font=\scriptsize]  {$+$};
\draw (365.22,317.33) node [anchor=north west][inner sep=0.75pt]  [font=\scriptsize]  {$+$};
\draw (365.22,327.14) node [anchor=north west][inner sep=0.75pt]  [font=\scriptsize]  {$-$};
\draw (376.27,320.35) node [anchor=north west][inner sep=0.75pt]  [font=\scriptsize]  {$-$};
\draw (523.7,349.98) node [anchor=north west][inner sep=0.75pt]  [font=\scriptsize]  {$+$};
\draw (528.05,294.8) node [anchor=north west][inner sep=0.75pt]  [font=\scriptsize]  {$+$};
\draw (513.54,282.63) node [anchor=north west][inner sep=0.75pt]  [font=\scriptsize]  {$+$};
\draw (513.54,295.46) node [anchor=north west][inner sep=0.75pt]  [font=\scriptsize]  {$-$};
\draw (508.13,270.53) node [anchor=north west][inner sep=0.75pt]    {$a$};
\draw (614.78,354.48) node [anchor=north west][inner sep=0.75pt]  [color={rgb, 255:red, 144; green, 19; blue, 254 }  ,opacity=1 ]  {$\Gamma $};
\draw (528.83,282.63) node [anchor=north west][inner sep=0.75pt]  [font=\scriptsize]  {$-$};
\draw (586.6,306.12) node [anchor=north west][inner sep=0.75pt]  [font=\scriptsize]  {$+$};
\draw (575.62,290.93) node [anchor=north west][inner sep=0.75pt]  [font=\scriptsize]  {$+$};
\draw (575.62,300.74) node [anchor=north west][inner sep=0.75pt]  [font=\scriptsize]  {$-$};
\draw (586.68,293.95) node [anchor=north west][inner sep=0.75pt]  [font=\scriptsize]  {$-$};
\draw (101.39,534.95) node [anchor=north west][inner sep=0.75pt]  [font=\scriptsize]  {$+$};
\draw (107.24,480.29) node [anchor=north west][inner sep=0.75pt]  [font=\scriptsize]  {$+$};
\draw (92.74,468.12) node [anchor=north west][inner sep=0.75pt]  [font=\scriptsize]  {$+$};
\draw (92.74,480.94) node [anchor=north west][inner sep=0.75pt]  [font=\scriptsize]  {$-$};
\draw (87.32,456.02) node [anchor=north west][inner sep=0.75pt]    {$a$};
\draw (193.97,539.97) node [anchor=north west][inner sep=0.75pt]  [color={rgb, 255:red, 144; green, 19; blue, 254 }  ,opacity=1 ]  {$\Gamma $};
\draw (108.02,468.12) node [anchor=north west][inner sep=0.75pt]  [font=\scriptsize]  {$-$};
\draw (50.14,496.93) node [anchor=north west][inner sep=0.75pt]  [font=\scriptsize]  {$+$};
\draw (37.71,508.54) node [anchor=north west][inner sep=0.75pt]  [font=\scriptsize]  {$+$};
\draw (38.42,497.23) node [anchor=north west][inner sep=0.75pt]  [font=\scriptsize]  {$-$};
\draw (48.02,509.76) node [anchor=north west][inner sep=0.75pt]  [font=\scriptsize]  {$-$};
\draw (312.29,534.95) node [anchor=north west][inner sep=0.75pt]  [font=\scriptsize]  {$+$};
\draw (317.65,480.25) node [anchor=north west][inner sep=0.75pt]  [font=\scriptsize]  {$+$};
\draw (303.14,468.08) node [anchor=north west][inner sep=0.75pt]  [font=\scriptsize]  {$+$};
\draw (303.14,480.9) node [anchor=north west][inner sep=0.75pt]  [font=\scriptsize]  {$-$};
\draw (297.72,455.98) node [anchor=north west][inner sep=0.75pt]    {$a$};
\draw (404.37,539.93) node [anchor=north west][inner sep=0.75pt]  [color={rgb, 255:red, 144; green, 19; blue, 254 }  ,opacity=1 ]  {$\Gamma $};
\draw (318.43,468.08) node [anchor=north west][inner sep=0.75pt]  [font=\scriptsize]  {$-$};
\draw (376.2,491.57) node [anchor=north west][inner sep=0.75pt]  [font=\scriptsize]  {$+$};
\draw (365.22,476.38) node [anchor=north west][inner sep=0.75pt]  [font=\scriptsize]  {$+$};
\draw (365.22,486.19) node [anchor=north west][inner sep=0.75pt]  [font=\scriptsize]  {$-$};
\draw (376.27,479.39) node [anchor=north west][inner sep=0.75pt]  [font=\scriptsize]  {$-$};
\draw (261.25,496.14) node [anchor=north west][inner sep=0.75pt]  [font=\scriptsize]  {$+$};
\draw (248.82,507.75) node [anchor=north west][inner sep=0.75pt]  [font=\scriptsize]  {$+$};
\draw (249.52,496.43) node [anchor=north west][inner sep=0.75pt]  [font=\scriptsize]  {$-$};
\draw (259.13,508.96) node [anchor=north west][inner sep=0.75pt]  [font=\scriptsize]  {$-$};
\draw (523.7,534.95) node [anchor=north west][inner sep=0.75pt]  [font=\scriptsize]  {$+$};
\draw (528.05,480.26) node [anchor=north west][inner sep=0.75pt]  [font=\scriptsize]  {$+$};
\draw (513.54,468.09) node [anchor=north west][inner sep=0.75pt]  [font=\scriptsize]  {$+$};
\draw (513.54,480.91) node [anchor=north west][inner sep=0.75pt]  [font=\scriptsize]  {$-$};
\draw (508.13,455.99) node [anchor=north west][inner sep=0.75pt]    {$a$};
\draw (614.78,539.94) node [anchor=north west][inner sep=0.75pt]  [color={rgb, 255:red, 144; green, 19; blue, 254 }  ,opacity=1 ]  {$\Gamma $};
\draw (528.83,468.09) node [anchor=north west][inner sep=0.75pt]  [font=\scriptsize]  {$-$};
\draw (586.6,491.57) node [anchor=north west][inner sep=0.75pt]  [font=\scriptsize]  {$+$};
\draw (575.62,476.39) node [anchor=north west][inner sep=0.75pt]  [font=\scriptsize]  {$+$};
\draw (575.62,486.19) node [anchor=north west][inner sep=0.75pt]  [font=\scriptsize]  {$-$};
\draw (586.68,479.4) node [anchor=north west][inner sep=0.75pt]  [font=\scriptsize]  {$-$};

\end{tikzpicture}

%% file: images/hypertight_1.tex
\begin{tikzpicture}[x=0.55pt,y=0.55pt,yscale=-1,xscale=1]

\draw  [draw opacity=0][fill={rgb, 255:red, 216; green, 216; blue, 216 }  ,fill opacity=1 ] (437.79,103.57) .. controls (414.04,113.85) and (362.79,112.57) .. (341.97,108.54) .. controls (341.76,139.16) and (343.69,173.61) .. (341.73,198.61) .. controls (364.4,200.71) and (410.85,199.74) .. (441.79,198.57) .. controls (442.17,159.74) and (442.5,124.92) .. (437.79,103.57) -- cycle ;
\draw  [draw opacity=0][fill={rgb, 255:red, 216; green, 216; blue, 216 }  ,fill opacity=1 ] (162.79,101.57) .. controls (191.79,111.57) and (240.63,108.77) .. (259.97,109.54) .. controls (257.79,140.67) and (257.79,173.67) .. (259.79,198.67) .. controls (236.63,200.77) and (188.39,201.74) .. (156.79,200.57) .. controls (156.39,161.74) and (157.79,120.57) .. (162.79,101.57) -- cycle ;
\draw  [draw opacity=0][fill={rgb, 255:red, 155; green, 155; blue, 155 }  ,fill opacity=1 ] (259.06,93.62) .. controls (259.06,71.05) and (277.36,52.76) .. (299.92,52.76) -- (299.92,52.76) .. controls (322.49,52.76) and (340.79,71.05) .. (340.79,93.62) -- (340.79,198.67) .. controls (340.79,198.67) and (340.79,198.67) .. (340.79,198.67) -- (259.06,198.67) .. controls (259.06,198.67) and (259.06,198.67) .. (259.06,198.67) -- cycle ;
\draw [color={rgb, 255:red, 0; green, 0; blue, 255 }  ,draw opacity=1 ][line width=1.5]    (259.22,198.61) .. controls (259.47,148.54) and (258.07,128.25) .. (258.63,99.54) .. controls (259.2,70.83) and (279.79,52.67) .. (299.18,52.76) ;
\draw [color={rgb, 255:red, 0; green, 0; blue, 255 }  ,draw opacity=1 ][line width=1.5]    (341.73,198.61) .. controls (341.98,148.54) and (341.15,120.54) .. (341.15,99.54) .. controls (341.15,78.54) and (327.79,53.67) .. (299.18,52.76) ;
\draw [color={rgb, 255:red, 0; green, 0; blue, 255 }  ,draw opacity=1 ][line width=1.5]    (162.79,101.57) .. controls (200.79,112.57) and (225.63,111.54) .. (258.97,109.54) ;
\draw [color={rgb, 255:red, 0; green, 0; blue, 255 }  ,draw opacity=1 ][line width=1.5]    (323.06,107.23) .. controls (326.06,106.23) and (331.06,107.23) .. (338.06,108.23) ;
\draw [color={rgb, 255:red, 0; green, 0; blue, 255 }  ,draw opacity=1 ][line width=1.5]    (437.79,103.57) .. controls (411.79,115.57) and (377.79,113.09) .. (343.97,108.54) ;
\draw [color={rgb, 255:red, 144; green, 19; blue, 254 }  ,draw opacity=1 ][line width=1.5]    (147.79,199.57) -- (449.65,199.61) ;

\draw (264,241.4) node [anchor=north west][inner sep=0.75pt]    {$w=w'*w''$};
\draw (343,144.4) node [anchor=north west][inner sep=0.75pt]    {$\gamma '$};
\draw (292.72,188.4) node [anchor=north west][inner sep=0.75pt]  [font=\scriptsize]  {$+$};
\draw (389.07,188.4) node [anchor=north west][inner sep=0.75pt]  [font=\scriptsize]  {$+$};
\draw (325.06,110.63) node [anchor=north west][inner sep=0.75pt]  [font=\scriptsize]  {$+$};
\draw (342.79,96.66) node [anchor=north west][inner sep=0.75pt]  [font=\scriptsize]  {$+$};
\draw (343.97,111.94) node [anchor=north west][inner sep=0.75pt]  [font=\scriptsize]  {$-$};
\draw (346.56,81.4) node [anchor=north west][inner sep=0.75pt]    {$a$};
\draw (290.39,204.4) node [anchor=north west][inner sep=0.75pt]    {$w'$};
\draw (388.49,202.4) node [anchor=north west][inner sep=0.75pt]    {$w''$};
\draw (451.91,193.01) node [anchor=north west][inner sep=0.75pt]  [color={rgb, 255:red, 144; green, 19; blue, 254 }  ,opacity=1 ]  {$\Gamma $};
\draw (446.99,91.4) node [anchor=north west][inner sep=0.75pt]  [color={rgb, 255:red, 0; green, 0; blue, 255 }  ,opacity=1 ]  {$\Lambda $};
\draw (328.07,95.81) node [anchor=north west][inner sep=0.75pt]  [font=\scriptsize]  {$-$};
\draw (199.49,202.4) node [anchor=north west][inner sep=0.75pt]    {$w''$};
\draw (247.07,110.81) node [anchor=north west][inner sep=0.75pt]  [font=\scriptsize]  {$-$};
\draw (240,144.4) node [anchor=north west][inner sep=0.75pt]    {$\gamma $};

\end{tikzpicture}

%% file: images/hypertight_3.tex
\begin{tikzpicture}[x=0.55pt,y=0.55pt,yscale=-1,xscale=1]

\draw  [draw opacity=0][fill={rgb, 255:red, 216; green, 216; blue, 216 }  ,fill opacity=1 ] (147.79,66.7) .. controls (147.79,48.36) and (162.66,33.48) .. (181,33.48) -- (415.57,33.48) .. controls (433.91,33.48) and (448.79,48.36) .. (448.79,66.7) -- (448.79,199.57) .. controls (448.79,199.57) and (448.79,199.57) .. (448.79,199.57) -- (147.79,199.57) .. controls (147.79,199.57) and (147.79,199.57) .. (147.79,199.57) -- cycle ;
\draw  [draw opacity=0][fill={rgb, 255:red, 155; green, 155; blue, 155 }  ,fill opacity=1 ] (259.06,93.62) .. controls (259.06,71.05) and (277.36,52.76) .. (299.92,52.76) -- (299.92,52.76) .. controls (322.49,52.76) and (340.79,71.05) .. (340.79,93.62) -- (340.79,198.67) .. controls (340.79,198.67) and (340.79,198.67) .. (340.79,198.67) -- (259.06,198.67) .. controls (259.06,198.67) and (259.06,198.67) .. (259.06,198.67) -- cycle ;
\draw [color={rgb, 255:red, 0; green, 0; blue, 255 }  ,draw opacity=1 ][line width=1.5]    (259.22,198.61) .. controls (259.47,148.54) and (258.07,128.25) .. (258.63,99.54) .. controls (259.2,70.83) and (279.79,52.67) .. (299.18,52.76) ;
\draw [color={rgb, 255:red, 0; green, 0; blue, 255 }  ,draw opacity=1 ][line width=1.5]    (341.73,198.61) .. controls (341.98,148.54) and (341.15,120.54) .. (341.15,99.54) .. controls (341.15,78.54) and (327.79,53.67) .. (299.18,52.76) ;
\draw [color={rgb, 255:red, 144; green, 19; blue, 254 }  ,draw opacity=1 ][line width=1.5]    (147.79,199.57) -- (449.65,199.61) ;

\draw (264,241.4) node [anchor=north west][inner sep=0.75pt]    {$w=w'*w''$};
\draw (291.72,186.4) node [anchor=north west][inner sep=0.75pt]  [font=\scriptsize]  {$+$};
\draw (389.07,186.4) node [anchor=north west][inner sep=0.75pt]  [font=\scriptsize]  {$+$};
\draw (290.39,204.4) node [anchor=north west][inner sep=0.75pt]    {$w'$};
\draw (388.49,202.4) node [anchor=north west][inner sep=0.75pt]    {$w''$};
\draw (451.91,193.01) node [anchor=north west][inner sep=0.75pt]  [color={rgb, 255:red, 144; green, 19; blue, 254 }  ,opacity=1 ]  {$\Gamma $};
\draw (323.99,40.4) node [anchor=north west][inner sep=0.75pt]  [color={rgb, 255:red, 0; green, 0; blue, 255 }  ,opacity=1 ]  {$\Lambda $};
\draw (199.49,202.4) node [anchor=north west][inner sep=0.75pt]    {$w''$};
\draw (346,136.4) node [anchor=north west][inner sep=0.75pt]    {$\gamma \ =\ \gamma '$};

\end{tikzpicture}

%% file: images/S2xI.tex
\begin{tikzpicture}[x=0.75pt,y=0.75pt,yscale=-1,xscale=1]

\draw  [fill={rgb, 255:red, 246; green, 246; blue, 246 }  ,fill opacity=1 ] [line width=0.75]  (194.59,51) -- (390.79,51) -- (390.79,247.19) -- (194.59,247.19) -- cycle ;
\draw [color={rgb, 255:red, 144; green, 19; blue, 254 }  ,draw opacity=1 ][line width=1.5]    (194.59,247.19) -- (390.79,51) ;
\draw  [color={rgb, 255:red, 0; green, 0; blue, 255 }  ,draw opacity=1 ][line width=1.5]  (231.97,149.1) .. controls (231.97,115.56) and (259.15,88.37) .. (292.69,88.37) .. controls (326.23,88.37) and (353.41,115.56) .. (353.41,149.1) .. controls (353.41,182.63) and (326.23,209.82) .. (292.69,209.82) .. controls (259.15,209.82) and (231.97,182.63) .. (231.97,149.1) -- cycle ;
\draw  [line width=0.75]  (297.49,252.49) .. controls (292.76,249.49) and (288.03,247.69) .. (283.3,247.1) .. controls (288.03,246.49) and (292.76,244.7) .. (297.49,241.7) ;
\draw  [line width=0.75]  (385,159.94) .. controls (388,155.21) and (389.8,150.48) .. (390.39,145.75) .. controls (391,150.48) and (392.79,155.21) .. (395.79,159.94) ;
\draw  [color={rgb, 255:red, 0; green, 0; blue, 0 }  ,draw opacity=1 ][line width=0.75]  (277.3,46.2) .. controls (282.03,49.2) and (286.76,51) .. (291.49,51.6) .. controls (286.76,52.2) and (282.03,53.99) .. (277.3,56.99) ;
\draw  [color={rgb, 255:red, 0; green, 0; blue, 0 }  ,draw opacity=1 ][line width=0.75]  (289.3,46.2) .. controls (294.03,49.2) and (298.76,51) .. (303.49,51.6) .. controls (298.76,52.2) and (294.03,53.99) .. (289.3,56.99) ;

\draw  [color={rgb, 255:red, 0; green, 0; blue, 0 }  ,draw opacity=1 ][line width=0.75]  (199.79,139.75) .. controls (196.79,144.48) and (194.99,149.21) .. (194.39,153.94) .. controls (193.79,149.21) and (192,144.48) .. (189,139.75) ;
\draw  [color={rgb, 255:red, 0; green, 0; blue, 0 }  ,draw opacity=1 ][line width=0.75]  (199.79,151.75) .. controls (196.79,156.48) and (194.99,161.21) .. (194.39,165.94) .. controls (193.79,161.21) and (192,156.48) .. (189,151.75) ;

\draw (287,156.4) node [anchor=north west][inner sep=0.75pt]  [color={rgb, 255:red, 144; green, 19; blue, 254 }  ,opacity=1 ]  {$a$};
\draw (353,88.4) node [anchor=north west][inner sep=0.75pt]  [color={rgb, 255:red, 144; green, 19; blue, 254 }  ,opacity=1 ]  {$b$};
\draw (222.91,220.01) node [anchor=north west][inner sep=0.75pt]  [color={rgb, 255:red, 144; green, 19; blue, 254 }  ,opacity=1 ]  {$\Gamma $};
\draw (248.99,76.4) node [anchor=north west][inner sep=0.75pt]  [color={rgb, 255:red, 0; green, 0; blue, 255 }  ,opacity=1 ]  {$\Lambda $};
\draw (198.59,56.4) node [anchor=north west][inner sep=0.75pt]    {$\Sigma _{+}$};
\draw (364.59,225.4) node [anchor=north west][inner sep=0.75pt]    {$\Sigma _{-}$};

\end{tikzpicture}

%% file: images/RIV_images.tex
\begin{tikzpicture}[x=0.75pt,y=0.75pt,yscale=-1,xscale=1]

\draw [color={rgb, 255:red, 144; green, 19; blue, 254 }  ,draw opacity=1 ][line width=1.5]    (364.45,141.03) -- (459.99,141.03) ;
\draw [color={rgb, 255:red, 0; green, 0; blue, 255 }  ,draw opacity=1 ][line width=1.5]    (392.42,96.94) .. controls (392.55,111.4) and (430.47,126.86) .. (430.47,141.01) .. controls (430.47,155.17) and (429.96,157.41) .. (444.85,167.52) ;
\draw [color={rgb, 255:red, 255; green, 255; blue, 255 }  ,draw opacity=1 ][line width=3]    (413.32,121.04) -- (406.44,115.03) ;
\draw [color={rgb, 255:red, 0; green, 0; blue, 255 }  ,draw opacity=1 ][line width=1.5]    (428.15,96.27) .. controls (428.02,110.72) and (390.9,126.18) .. (390.9,140.34) .. controls (390.9,154.5) and (394.38,157.41) .. (379.8,167.52) ;
\draw  [color={rgb, 255:red, 0; green, 0; blue, 0 }  ,draw opacity=1 ][dash pattern={on 0.84pt off 2.51pt}] (360.52,136.62) .. controls (360.52,109.41) and (383.19,87.35) .. (411.15,87.35) .. controls (439.11,87.35) and (461.78,109.41) .. (461.78,136.62) .. controls (461.78,163.83) and (439.11,185.89) .. (411.15,185.89) .. controls (383.19,185.89) and (360.52,163.83) .. (360.52,136.62) -- cycle ;

\draw [color={rgb, 255:red, 144; green, 19; blue, 254 }  ,draw opacity=1 ][line width=1.5]    (529.01,130.15) -- (622.65,130.15) ;
\draw [color={rgb, 255:red, 0; green, 0; blue, 255 }  ,draw opacity=1 ][line width=1.5]    (593.82,176.86) .. controls (593.69,162.41) and (555.77,146.95) .. (555.77,132.79) .. controls (555.77,118.64) and (556.28,116.39) .. (541.39,106.28) ;
\draw [color={rgb, 255:red, 255; green, 255; blue, 255 }  ,draw opacity=1 ][line width=3]    (579.7,159.11) -- (572.81,153.1) ;
\draw [color={rgb, 255:red, 0; green, 0; blue, 255 }  ,draw opacity=1 ][line width=1.5]    (558.09,177.54) .. controls (558.22,163.08) and (595.34,147.62) .. (595.34,133.46) .. controls (595.34,119.31) and (591.86,116.39) .. (606.44,106.28) ;
\draw  [color={rgb, 255:red, 0; green, 0; blue, 0 }  ,draw opacity=1 ][dash pattern={on 0.84pt off 2.51pt}] (524.97,136.62) .. controls (524.97,109.41) and (547.64,87.35) .. (575.6,87.35) .. controls (603.56,87.35) and (626.23,109.41) .. (626.23,136.62) .. controls (626.23,163.83) and (603.56,185.89) .. (575.6,185.89) .. controls (547.64,185.89) and (524.97,163.83) .. (524.97,136.62) -- cycle ;

\draw    (518.88,136.54) -- (480.64,136.77) ;
\draw [shift={(478.64,136.78)}, rotate = 359.66] [color={rgb, 255:red, 0; green, 0; blue, 0 }  ][line width=0.75]    (10.93,-3.29) .. controls (6.95,-1.4) and (3.31,-0.3) .. (0,0) .. controls (3.31,0.3) and (6.95,1.4) .. (10.93,3.29)   ;
\draw [shift={(520.88,136.53)}, rotate = 179.66] [color={rgb, 255:red, 0; green, 0; blue, 0 }  ][line width=0.75]    (10.93,-3.29) .. controls (6.95,-1.4) and (3.31,-0.3) .. (0,0) .. controls (3.31,0.3) and (6.95,1.4) .. (10.93,3.29)   ;
\draw [color={rgb, 255:red, 144; green, 19; blue, 254 }  ,draw opacity=1 ][line width=1.5]    (114.63,140.34) -- (18.72,140.34) ;
\draw [color={rgb, 255:red, 0; green, 0; blue, 255 }  ,draw opacity=1 ][line width=1.5]    (88.31,95.1) .. controls (88.16,109.93) and (46.36,125.8) .. (46.36,140.32) .. controls (46.36,154.85) and (46.94,157.15) .. (30.52,167.52) ;
\draw [color={rgb, 255:red, 255; green, 255; blue, 255 }  ,draw opacity=1 ][line width=3]    (65.27,119.83) -- (72.86,113.66) ;
\draw [color={rgb, 255:red, 0; green, 0; blue, 255 }  ,draw opacity=1 ][line width=1.5]    (48.93,94.41) .. controls (49.07,109.24) and (89.99,125.11) .. (89.99,139.63) .. controls (89.99,154.16) and (86.15,157.15) .. (102.22,167.52) ;
\draw  [color={rgb, 255:red, 0; green, 0; blue, 0 }  ,draw opacity=1 ][dash pattern={on 0.84pt off 2.51pt}] (17,136.62) .. controls (17,109.41) and (39.67,87.35) .. (67.63,87.35) .. controls (95.59,87.35) and (118.26,109.41) .. (118.26,136.62) .. controls (118.26,163.83) and (95.59,185.89) .. (67.63,185.89) .. controls (39.67,185.89) and (17,163.83) .. (17,136.62) -- cycle ;

\draw [color={rgb, 255:red, 144; green, 19; blue, 254 }  ,draw opacity=1 ][line width=1.5]    (291.05,129.84) -- (199.66,129.84) ;
\draw [color={rgb, 255:red, 0; green, 0; blue, 255 }  ,draw opacity=1 ][line width=1.5]    (225.18,177.78) .. controls (225.32,162.94) and (265.85,147.08) .. (265.85,132.55) .. controls (265.85,118.03) and (265.29,115.73) .. (281.2,105.35) ;
\draw [color={rgb, 255:red, 255; green, 255; blue, 255 }  ,draw opacity=1 ][line width=3]    (240.27,159.56) -- (247.62,153.4) ;
\draw [color={rgb, 255:red, 0; green, 0; blue, 255 }  ,draw opacity=1 ][line width=1.5]    (263.36,178.47) .. controls (263.22,163.63) and (223.55,147.77) .. (223.55,133.24) .. controls (223.55,118.72) and (227.27,115.73) .. (211.69,105.35) ;
\draw  [color={rgb, 255:red, 0; green, 0; blue, 0 }  ,draw opacity=1 ][dash pattern={on 0.84pt off 2.51pt}] (194.81,136.62) .. controls (194.81,109.41) and (217.47,87.35) .. (245.43,87.35) .. controls (273.4,87.35) and (296.06,109.41) .. (296.06,136.62) .. controls (296.06,163.83) and (273.4,185.89) .. (245.43,185.89) .. controls (217.47,185.89) and (194.81,163.83) .. (194.81,136.62) -- cycle ;

\draw    (184.39,136.53) -- (144.7,136.76) ;
\draw [shift={(142.7,136.78)}, rotate = 359.67] [color={rgb, 255:red, 0; green, 0; blue, 0 }  ][line width=0.75]    (10.93,-3.29) .. controls (6.95,-1.4) and (3.31,-0.3) .. (0,0) .. controls (3.31,0.3) and (6.95,1.4) .. (10.93,3.29)   ;
\draw [shift={(186.39,136.52)}, rotate = 179.67] [color={rgb, 255:red, 0; green, 0; blue, 0 }  ][line width=0.75]    (10.93,-3.29) .. controls (6.95,-1.4) and (3.31,-0.3) .. (0,0) .. controls (3.31,0.3) and (6.95,1.4) .. (10.93,3.29)   ;

\draw (370.55,144.02) node [anchor=north west][inner sep=0.75pt]  [font=\footnotesize]  {$d$};
\draw (405.84,147.02) node [anchor=north west][inner sep=0.75pt]  [font=\footnotesize]  {$a$};
\draw (406.8,100.39) node [anchor=north west][inner sep=0.75pt]  [font=\footnotesize]  {$b$};
\draw (440.18,147.02) node [anchor=north west][inner sep=0.75pt]  [font=\footnotesize]  {$c$};
\draw (463.41,139.23) node [anchor=north west][inner sep=0.75pt]  [color={rgb, 255:red, 144; green, 19; blue, 254 }  ,opacity=1 ]  {$\Gamma $};
\draw (572.07,162.32) node [anchor=north west][inner sep=0.75pt]  [font=\footnotesize]  {$a$};
\draw (536.53,113.9) node [anchor=north west][inner sep=0.75pt]  [font=\footnotesize]  {$d$};
\draw (571.83,113.9) node [anchor=north west][inner sep=0.75pt]  [font=\footnotesize]  {$b$};
\draw (606.17,116.9) node [anchor=north west][inner sep=0.75pt]  [font=\footnotesize]  {$c$};
\draw (628.09,129.97) node [anchor=north west][inner sep=0.75pt]  [color={rgb, 255:red, 144; green, 19; blue, 254 }  ,opacity=1 ]  {$\Gamma $};
\draw (27.98,143.09) node [anchor=north west][inner sep=0.75pt]  [font=\footnotesize]  {$d$};
\draw (63.27,146.09) node [anchor=north west][inner sep=0.75pt]  [font=\footnotesize]  {$a$};
\draw (64.23,99.46) node [anchor=north west][inner sep=0.75pt]  [font=\footnotesize]  {$b$};
\draw (97.62,146.09) node [anchor=north west][inner sep=0.75pt]  [font=\footnotesize]  {$c$};
\draw (119.89,139.23) node [anchor=north west][inner sep=0.75pt]  [color={rgb, 255:red, 144; green, 19; blue, 254 }  ,opacity=1 ]  {$\Gamma $};
\draw (241.9,162.32) node [anchor=north west][inner sep=0.75pt]  [font=\footnotesize]  {$b$};
\draw (206.36,113.9) node [anchor=north west][inner sep=0.75pt]  [font=\footnotesize]  {$d$};
\draw (241.66,116.9) node [anchor=north west][inner sep=0.75pt]  [font=\footnotesize]  {$a$};
\draw (276,116.9) node [anchor=north west][inner sep=0.75pt]  [font=\footnotesize]  {$c$};
\draw (298.16,130.62) node [anchor=north west][inner sep=0.75pt]  [color={rgb, 255:red, 144; green, 19; blue, 254 }  ,opacity=1 ]  {$\Gamma $};
\draw (152.85,116.26) node [anchor=north west][inner sep=0.75pt]  [font=\small]  {$IVa$};
\draw (487.78,116.26) node [anchor=north west][inner sep=0.75pt]  [font=\small]  {$IVb$};
\draw (154.95,206) node [anchor=north west][inner sep=0.75pt]   [align=left] {(a)};
\draw (491.05,206) node [anchor=north west][inner sep=0.75pt]   [align=left] {(b)};

\end{tikzpicture}

%% file: images/polygon_correspondence.tex
\begin{tikzpicture}[x=0.75pt,y=0.75pt,yscale=-1,xscale=1]

\draw  [draw opacity=0][fill={rgb, 255:red, 216; green, 216; blue, 216 }  ,fill opacity=1 ] (463.35,92.13) .. controls (468.8,92.08) and (475.2,110.79) .. (468.67,130.69) .. controls (452.17,131.37) and (470.42,129.21) .. (445.84,129.58) .. controls (446.4,104.95) and (458.4,96.76) .. (463.35,92.13) -- cycle ;
\draw  [draw opacity=0][fill={rgb, 255:red, 216; green, 216; blue, 216 }  ,fill opacity=1 ] (389.9,92.13) .. controls (382.4,93.25) and (375.68,110.79) .. (380.87,130.69) .. controls (389.6,130.68) and (391.2,129.51) .. (404,129.8) .. controls (404,101.44) and (398.2,99.26) .. (389.9,92.13) -- cycle ;
\draw  [draw opacity=0][fill={rgb, 255:red, 183; green, 183; blue, 183 }  ,fill opacity=1 ] (463.35,92.13) .. controls (450.4,100.27) and (445,114.3) .. (446.84,129.58) .. controls (439.4,130.68) and (410.4,131.85) .. (404,129.8) .. controls (406.4,107.28) and (396,93.25) .. (390.48,92.13) .. controls (399.43,92.08) and (440.62,86.31) .. (463.35,92.13) -- cycle ;
\draw [color={rgb, 255:red, 144; green, 19; blue, 254 }  ,draw opacity=1 ][line width=1.5]    (468.67,130.69) -- (380.87,130.69) ;
\draw [color={rgb, 255:red, 0; green, 0; blue, 255 }  ,draw opacity=1 ][line width=1.5]    (404.62,187.03) .. controls (404.77,167.59) and (447.25,146.8) .. (447.25,127.77) .. controls (447.25,108.74) and (446.67,105.72) .. (463.35,92.13) ;
\draw [color={rgb, 255:red, 255; green, 255; blue, 255 }  ,draw opacity=1 ][line width=3]    (420.44,163.16) -- (428.15,155.08) ;
\draw [color={rgb, 255:red, 0; green, 0; blue, 255 }  ,draw opacity=1 ][line width=1.5]    (444.64,187.93) .. controls (444.5,168.5) and (402.91,147.71) .. (402.91,128.67) .. controls (402.91,109.64) and (406.81,105.72) .. (390.48,92.13) ;
\draw  [draw opacity=0][fill={rgb, 255:red, 216; green, 216; blue, 216 }  ,fill opacity=1 ] (261.14,88.18) .. controls (252,118.61) and (249.6,111.6) .. (239.2,119.78) .. controls (227.2,139.66) and (218.4,139.66) .. (217.17,159.37) .. controls (207.2,154.86) and (204.8,160.71) .. (195.19,159.4) .. controls (190.4,92.89) and (195.2,78.86) .. (261.14,88.18) -- cycle ;
\draw  [draw opacity=0][fill={rgb, 255:red, 216; green, 216; blue, 216 }  ,fill opacity=1 ] (239.2,119.78) .. controls (228.8,106.92) and (220.8,102.24) .. (219.86,87.09) .. controls (280.8,88.21) and (286.4,71.84) .. (285.74,159.4) .. controls (269.24,160.08) and (287.49,157.92) .. (262.9,158.29) .. controls (258.4,137.32) and (250.4,134.98) .. (239.2,119.78) -- cycle ;
\draw  [draw opacity=0][fill={rgb, 255:red, 183; green, 183; blue, 183 }  ,fill opacity=1 ] (261.14,88.18) .. controls (255.2,116.27) and (246.4,110.43) .. (241.99,121.26) .. controls (228.8,113.93) and (220.8,97.56) .. (219.86,87.09) .. controls (228.8,87.04) and (238.4,82.36) .. (261.14,88.18) -- cycle ;
\draw [color={rgb, 255:red, 144; green, 19; blue, 254 }  ,draw opacity=1 ][line width=1.5]    (285.74,159.4) -- (195.19,159.4) ;
\draw [color={rgb, 255:red, 0; green, 0; blue, 255 }  ,draw opacity=1 ][line width=1.5]    (261.14,88.18) .. controls (260.98,111.53) and (217.17,136.51) .. (217.17,159.37) .. controls (217.17,182.24) and (217.77,175.87) .. (200.56,192.2) ;
\draw [color={rgb, 255:red, 255; green, 255; blue, 255 }  ,draw opacity=1 ][line width=3]    (235.99,128.57) -- (240.94,121.78) ;
\draw [color={rgb, 255:red, 0; green, 0; blue, 255 }  ,draw opacity=1 ][line width=1.5]    (219.86,87.09) .. controls (220,110.44) and (262.9,135.42) .. (262.9,158.29) .. controls (262.9,181.15) and (258.88,175.87) .. (275.72,192.2) ;
\draw [color={rgb, 255:red, 216; green, 216; blue, 216 }  ,draw opacity=1 ][line width=3]    (240.99,121.26) -- (246.94,115.48) ;
\draw    (361.95,138.64) -- (320.16,138.87) ;
\draw [shift={(318.16,138.88)}, rotate = 359.68] [color={rgb, 255:red, 0; green, 0; blue, 0 }  ][line width=0.75]    (10.93,-3.29) .. controls (6.95,-1.4) and (3.31,-0.3) .. (0,0) .. controls (3.31,0.3) and (6.95,1.4) .. (10.93,3.29)   ;
\draw [shift={(363.95,138.62)}, rotate = 179.68] [color={rgb, 255:red, 0; green, 0; blue, 0 }  ][line width=0.75]    (10.93,-3.29) .. controls (6.95,-1.4) and (3.31,-0.3) .. (0,0) .. controls (3.31,0.3) and (6.95,1.4) .. (10.93,3.29)   ;

\draw (198.19,166.95) node [anchor=north west][inner sep=0.75pt]  [font=\footnotesize]  {$c$};
\draw (235.19,166.95) node [anchor=north west][inner sep=0.75pt]  [font=\footnotesize]  {$a$};
\draw (236.19,98.25) node [anchor=north west][inner sep=0.75pt]  [font=\footnotesize]  {$b$};
\draw (271.19,166.95) node [anchor=north west][inner sep=0.75pt]  [font=\footnotesize]  {$c$};
\draw (287.74,167.87) node [anchor=north west][inner sep=0.75pt]  [color={rgb, 255:red, 144; green, 19; blue, 254 }  ,opacity=1 ]  {$\Gamma $};
\draw (422.44,168.51) node [anchor=north west][inner sep=0.75pt]  [font=\footnotesize]  {$b$};
\draw (385.19,113.53) node [anchor=north west][inner sep=0.75pt]  [font=\footnotesize]  {$c$};
\draw (422.19,113.53) node [anchor=north west][inner sep=0.75pt]  [font=\footnotesize]  {$a$};
\draw (458.19,113.53) node [anchor=north west][inner sep=0.75pt]  [font=\footnotesize]  {$c$};
\draw (470.41,137.75) node [anchor=north west][inner sep=0.75pt]  [color={rgb, 255:red, 144; green, 19; blue, 254 }  ,opacity=1 ]  {$\Gamma $};

\end{tikzpicture}

%% file: images/RI_invariance.tex
\begin{tikzpicture}[x=0.55pt,y=0.55pt,yscale=-1,xscale=1]

\draw [color={rgb, 255:red, 144; green, 19; blue, 254 }  ,draw opacity=1 ][line width=1.5]    (273.19,136.78) -- (53.32,136.78) ;
\draw [color={rgb, 255:red, 0; green, 0; blue, 255 }  ,draw opacity=1 ][line width=1.5]    (122.41,69.88) .. controls (165.51,95.07) and (197.68,100.88) .. (197.68,137.22) .. controls (197.68,173.55) and (175.7,182.27) .. (163.81,182.27) .. controls (151.92,182.27) and (129.93,174.88) .. (129.84,137.22) .. controls (129.74,99.56) and (165.51,95.07) .. (207.34,68.43) ;
\draw [color={rgb, 255:red, 255; green, 255; blue, 255 }  ,draw opacity=1 ][line width=3]    (172.16,87.53) -- (155.32,97.54) ;
\draw [color={rgb, 255:red, 74; green, 74; blue, 255 }  ,draw opacity=1 ][line width=1.5]    (168.82,94.54) -- (159.47,89.54) ;
\draw  [line width=0.75]  (57.92,183.94) .. controls (57.9,188.61) and (60.22,190.95) .. (64.89,190.97) -- (119.78,191.17) .. controls (126.45,191.2) and (129.77,193.54) .. (129.75,198.21) .. controls (129.77,193.54) and (133.11,191.22) .. (139.78,191.24)(136.78,191.23) -- (194.66,191.44) .. controls (199.33,191.46) and (201.67,189.14) .. (201.68,184.47) ;
\draw [color={rgb, 255:red, 144; green, 19; blue, 254 }  ,draw opacity=1 ][line width=1.5]    (576.19,136.78) -- (356.32,136.78) ;
\draw   (130.92,233.94) .. controls (130.9,238.61) and (133.22,240.95) .. (137.89,240.97) -- (192.78,241.17) .. controls (199.45,241.2) and (202.77,243.54) .. (202.75,248.21) .. controls (202.77,243.54) and (206.11,241.22) .. (212.78,241.24)(209.78,241.23) -- (267.66,241.44) .. controls (272.33,241.46) and (274.67,239.14) .. (274.68,234.47) ;
\draw [color={rgb, 255:red, 0; green, 0; blue, 255 }  ,draw opacity=1 ][line width=1.5]    (559.81,79.36) .. controls (523.81,69.36) and (515.81,103.36) .. (461.81,104.36) .. controls (407.81,105.36) and (396.87,71.47) .. (362.87,79.47) ;
\draw    (328.39,120.53) -- (288.7,120.76) ;
\draw [shift={(286.7,120.78)}, rotate = 359.67] [color={rgb, 255:red, 0; green, 0; blue, 0 }  ][line width=0.75]    (10.93,-3.29) .. controls (6.95,-1.4) and (3.31,-0.3) .. (0,0) .. controls (3.31,0.3) and (6.95,1.4) .. (10.93,3.29)   ;
\draw [shift={(330.39,120.52)}, rotate = 179.67] [color={rgb, 255:red, 0; green, 0; blue, 0 }  ][line width=0.75]    (10.93,-3.29) .. controls (6.95,-1.4) and (3.31,-0.3) .. (0,0) .. controls (3.31,0.3) and (6.95,1.4) .. (10.93,3.29)   ;

\draw (582.19,136.97) node [anchor=north west][inner sep=0.75pt]  [color={rgb, 255:red, 144; green, 19; blue, 254 }  ,opacity=1 ]  {$\Gamma $};
\draw (224.67,143.01) node [anchor=north west][inner sep=0.75pt]    {$d_{i} '$};
\draw (174.17,260.07) node [anchor=north west][inner sep=0.75pt]    {$d_{i} =d_{i} '*a$};
\draw (276.19,137.97) node [anchor=north west][inner sep=0.75pt]  [color={rgb, 255:red, 144; green, 19; blue, 254 }  ,opacity=1 ]  {$\Gamma $};
\draw (158.17,145.21) node [anchor=north west][inner sep=0.75pt]    {$a$};
\draw (186.92,84.13) node [anchor=north west][inner sep=0.75pt]    {$b$};
\draw (77.67,143.01) node [anchor=north west][inner sep=0.75pt]    {$c_{i} '$};
\draw (101.17,210.07) node [anchor=north west][inner sep=0.75pt]    {$c_{i} =c_{i} '*a$};
\draw (303.85,100.26) node [anchor=north west][inner sep=0.75pt]  [font=\small]  {$I$};

\end{tikzpicture}

%% file: images/RI_invariance_2.tex
\begin{tikzpicture}[x=0.75pt,y=0.75pt,yscale=-1,xscale=1]

\draw  [draw opacity=0][fill={rgb, 255:red, 216; green, 216; blue, 216 }  ,fill opacity=1 ] (527.87,65.47) .. controls (549.14,64.87) and (554.14,32.87) .. (587.87,44.47) .. controls (591.31,48.83) and (590.1,62.14) .. (588.6,83.67) .. controls (572.1,84.14) and (492.73,83.42) .. (468.14,83.67) .. controls (469.11,72.64) and (465.28,48.01) .. (468.87,44.47) .. controls (500.14,36.87) and (507.14,63.87) .. (527.87,65.47) -- cycle ;
\draw  [draw opacity=0][fill={rgb, 255:red, 216; green, 216; blue, 216 }  ,fill opacity=1 ] (527.14,182.28) .. controls (548.42,181.68) and (552.14,144.87) .. (602.14,165.87) .. controls (605.58,170.23) and (604.14,181.87) .. (604.14,202.87) .. controls (587.65,203.33) and (492.73,203.23) .. (468.14,203.48) .. controls (469.11,192.46) and (464.55,164.83) .. (468.14,161.28) .. controls (499.42,153.68) and (506.42,180.68) .. (527.14,182.28) -- cycle ;
\draw  [draw opacity=0][fill={rgb, 255:red, 216; green, 216; blue, 216 }  ,fill opacity=1 ] (280.15,174.68) .. controls (302.84,159.52) and (327.83,151.76) .. (341.39,160.4) .. controls (342.43,164.87) and (341.83,182.51) .. (340.33,204.05) .. controls (323.83,204.51) and (285.83,203.51) .. (261.24,203.77) .. controls (263.17,181.71) and (272.98,182.43) .. (280.15,174.68) -- cycle ;
\draw  [draw opacity=0][fill={rgb, 255:red, 216; green, 216; blue, 216 }  ,fill opacity=1 ] (279.6,55.11) .. controls (302.29,39.95) and (315.5,37.95) .. (329.06,46.59) .. controls (332.5,50.95) and (341.28,62.95) .. (339.78,84.48) .. controls (323.28,84.95) and (285.28,83.95) .. (260.69,84.2) .. controls (262.62,62.15) and (272.43,62.87) .. (279.6,55.11) -- cycle ;
\draw  [draw opacity=0][fill={rgb, 255:red, 216; green, 216; blue, 216 }  ,fill opacity=1 ] (53.91,164.23) .. controls (75.91,153.23) and (83.91,163.23) .. (97.47,171.87) .. controls (100.91,176.23) and (119.91,181.23) .. (118.41,202.77) .. controls (101.91,203.23) and (63.91,202.23) .. (39.32,202.48) .. controls (39.91,174.23) and (42.91,170.23) .. (53.91,164.23) -- cycle ;
\draw  [draw opacity=0][fill={rgb, 255:red, 216; green, 216; blue, 216 }  ,fill opacity=1 ] (54.91,46.23) .. controls (76.91,35.23) and (84.91,45.23) .. (98.47,53.87) .. controls (101.91,58.23) and (120.91,63.23) .. (119.41,84.77) .. controls (102.91,85.23) and (64.91,84.23) .. (40.32,84.48) .. controls (40.91,56.23) and (43.91,52.23) .. (54.91,46.23) -- cycle ;
\draw [color={rgb, 255:red, 144; green, 19; blue, 254 }  ,draw opacity=1 ][line width=1.5]    (160.78,84.48) -- (40.32,84.48) ;
\draw [color={rgb, 255:red, 0; green, 0; blue, 255 }  ,draw opacity=1 ][line width=1.5]    (78.17,41.14) .. controls (101.78,57.46) and (119.41,61.22) .. (119.41,84.77) .. controls (119.41,108.31) and (107.37,113.96) .. (100.85,113.96) .. controls (94.34,113.96) and (82.29,109.17) .. (82.24,84.77) .. controls (82.19,60.36) and (101.78,57.46) .. (124.7,40.2) ;
\draw [color={rgb, 255:red, 216; green, 216; blue, 216 }  ,draw opacity=1 ][line width=3]    (101.43,56.57) -- (95.2,60.06) ;
\draw [color={rgb, 255:red, 0; green, 0; blue, 255 }  ,draw opacity=1 ][line width=1.5]    (103.6,57.11) -- (98.47,53.87) ;
\draw [color={rgb, 255:red, 144; green, 19; blue, 254 }  ,draw opacity=1 ][line width=1.5]    (339.78,84.48) -- (219.32,84.48) ;
\draw [color={rgb, 255:red, 0; green, 0; blue, 255 }  ,draw opacity=1 ][line width=1.5]    (257.17,41.14) .. controls (280.78,57.46) and (298.41,61.22) .. (298.41,84.77) .. controls (298.41,108.31) and (286.37,113.96) .. (279.85,113.96) .. controls (273.34,113.96) and (261.3,109.17) .. (261.24,84.77) .. controls (261.19,60.36) and (280.78,57.46) .. (303.7,40.2) ;
\draw [color={rgb, 255:red, 255; green, 255; blue, 255 }  ,draw opacity=1 ][line width=3]    (284.43,52.57) -- (275.2,59.06) ;
\draw [color={rgb, 255:red, 0; green, 0; blue, 255 }  ,draw opacity=1 ][line width=1.5]    (282.6,57.11) -- (277.47,53.87) ;
\draw   (295.38,234.04) .. controls (295.36,238.71) and (297.68,241.05) .. (302.35,241.07) -- (324.73,241.17) .. controls (331.4,241.2) and (334.72,243.54) .. (334.71,248.21) .. controls (334.72,243.54) and (338.06,241.22) .. (344.73,241.25)(341.73,241.24) -- (367.12,241.34) .. controls (371.79,241.36) and (374.13,239.04) .. (374.15,234.37) ;
\draw [color={rgb, 255:red, 144; green, 19; blue, 254 }  ,draw opacity=1 ][line width=1.5]    (369.78,203.48) -- (219.32,203.48) ;
\draw [color={rgb, 255:red, 0; green, 0; blue, 255 }  ,draw opacity=1 ][line width=1.5]    (257.17,160.14) .. controls (280.78,176.46) and (298.41,180.22) .. (298.41,203.77) .. controls (298.41,227.31) and (286.37,232.96) .. (279.85,232.96) .. controls (273.34,232.96) and (261.3,228.17) .. (261.24,203.77) .. controls (261.19,179.36) and (280.78,176.46) .. (303.7,159.2) ;
\draw [color={rgb, 255:red, 255; green, 255; blue, 255 }  ,draw opacity=1 ][line width=3]    (284.43,171.57) -- (275.2,178.06) ;
\draw [color={rgb, 255:red, 74; green, 74; blue, 255 }  ,draw opacity=1 ][line width=1.5]    (282.6,176.11) -- (277.47,172.87) ;
\draw [color={rgb, 255:red, 0; green, 0; blue, 255 }  ,draw opacity=1 ][line width=1.5]    (341.39,160.4) -- (341.39,221.18) ;

\draw [color={rgb, 255:red, 144; green, 19; blue, 254 }  ,draw opacity=1 ][line width=1.5]    (160.78,203.48) -- (40.32,203.48) ;
\draw [color={rgb, 255:red, 0; green, 0; blue, 255 }  ,draw opacity=1 ][line width=1.5]    (78.17,160.14) .. controls (101.78,176.46) and (119.41,180.22) .. (119.41,203.77) .. controls (119.41,227.31) and (107.37,232.96) .. (100.85,232.96) .. controls (94.34,232.96) and (82.29,228.17) .. (82.24,203.77) .. controls (82.19,179.36) and (101.78,176.46) .. (124.7,159.2) ;
\draw [color={rgb, 255:red, 0; green, 0; blue, 255 }  ,draw opacity=1 ][line width=1.5]    (103.6,176.11) -- (98.47,172.87) ;
\draw [color={rgb, 255:red, 144; green, 19; blue, 254 }  ,draw opacity=1 ][line width=1.5]    (588.6,83.67) -- (468.14,83.67) ;
\draw [color={rgb, 255:red, 0; green, 0; blue, 255 }  ,draw opacity=1 ][line width=1.5]    (587.87,44.47) .. controls (551.87,34.47) and (547.87,65.47) .. (527.87,65.47) .. controls (507.87,65.47) and (502.87,36.47) .. (468.87,44.47) ;

\draw [color={rgb, 255:red, 0; green, 0; blue, 255 }  ,draw opacity=1 ][line width=1.5]    (587.14,161.28) .. controls (551.14,151.28) and (547.14,182.28) .. (527.14,182.28) .. controls (507.14,182.28) and (502.14,153.28) .. (468.14,161.28) ;
\draw [color={rgb, 255:red, 0; green, 0; blue, 255 }  ,draw opacity=1 ][line width=1.5]    (603.57,156.4) -- (603.57,217.18) ;
\draw [color={rgb, 255:red, 144; green, 19; blue, 254 }  ,draw opacity=1 ][line width=1.5]    (618.6,203.48) -- (468.14,203.48) ;

\draw    (445.22,84.69) -- (399.99,84.75) ;
\draw [shift={(447.22,84.69)}, rotate = 179.93] [color={rgb, 255:red, 0; green, 0; blue, 0 }  ][line width=0.75]    (10.93,-3.29) .. controls (6.95,-1.4) and (3.31,-0.3) .. (0,0) .. controls (3.31,0.3) and (6.95,1.4) .. (10.93,3.29)   ;
\draw    (445.22,203.69) -- (399.99,203.75) ;
\draw [shift={(447.22,203.69)}, rotate = 179.93] [color={rgb, 255:red, 0; green, 0; blue, 0 }  ][line width=0.75]    (10.93,-3.29) .. controls (6.95,-1.4) and (3.31,-0.3) .. (0,0) .. controls (3.31,0.3) and (6.95,1.4) .. (10.93,3.29)   ;
\draw [color={rgb, 255:red, 255; green, 255; blue, 255 }  ,draw opacity=1 ][line width=3.75]    (104.43,51.57) -- (102.2,55.06) ;
\draw [color={rgb, 255:red, 216; green, 216; blue, 216 }  ,draw opacity=1 ][line width=3]    (94.2,180.06) ;
\draw [color={rgb, 255:red, 216; green, 216; blue, 216 }  ,draw opacity=1 ][line width=3]    (95.2,174.06) -- (102.51,179.71) ;
\draw [color={rgb, 255:red, 255; green, 255; blue, 255 }  ,draw opacity=1 ][line width=3.75]    (109.01,176.04) -- (99.47,168.87) ;

\draw (299.48,252.94) node [anchor=north west][inner sep=0.75pt]    {$d_{i}^{'} =d_{j} '*e_{j}$};
\draw (377.35,199.93) node [anchor=north west][inner sep=0.75pt]  [color={rgb, 255:red, 144; green, 19; blue, 254 }  ,opacity=1 ]  {$\Gamma $};
\draw (274.05,206.83) node [anchor=north west][inner sep=0.75pt]    {$a$};
\draw (289.8,166.69) node [anchor=north west][inner sep=0.75pt]    {$b$};
\draw (307.03,205.83) node [anchor=north west][inner sep=0.75pt]    {$d_{j} '$};
\draw (233.59,205.83) node [anchor=north west][inner sep=0.75pt]    {$c_{i} '$};
\draw (353.03,206.83) node [anchor=north west][inner sep=0.75pt]    {$e_{j}$};
\draw (193,75.4) node [anchor=north west][inner sep=0.75pt]    {$+$};
\draw (193,194.4) node [anchor=north west][inner sep=0.75pt]    {$+$};
\draw (159.71,82.58) node [anchor=north west][inner sep=0.75pt]  [color={rgb, 255:red, 144; green, 19; blue, 254 }  ,opacity=1 ]  {$\Gamma $};
\draw (95.05,87.83) node [anchor=north west][inner sep=0.75pt]    {$a$};
\draw (110.8,47.69) node [anchor=north west][inner sep=0.75pt]    {$b$};
\draw (49.59,86.83) node [anchor=north west][inner sep=0.75pt]    {$c_{i} '$};
\draw (126.41,86.83) node [anchor=north west][inner sep=0.75pt]    {$d_{i} '$};
\draw (159.71,201.58) node [anchor=north west][inner sep=0.75pt]  [color={rgb, 255:red, 144; green, 19; blue, 254 }  ,opacity=1 ]  {$\Gamma $};
\draw (95.05,206.83) node [anchor=north west][inner sep=0.75pt]    {$a$};
\draw (110.8,166.69) node [anchor=north west][inner sep=0.75pt]    {$b$};
\draw (49.59,205.83) node [anchor=north west][inner sep=0.75pt]    {$c_{i} '$};
\draw (126.41,205.83) node [anchor=north west][inner sep=0.75pt]    {$d_{i} '$};
\draw (340.35,81.93) node [anchor=north west][inner sep=0.75pt]  [color={rgb, 255:red, 144; green, 19; blue, 254 }  ,opacity=1 ]  {$\Gamma $};
\draw (274.05,87.83) node [anchor=north west][inner sep=0.75pt]    {$a$};
\draw (289.8,47.69) node [anchor=north west][inner sep=0.75pt]    {$b$};
\draw (307.03,86.83) node [anchor=north west][inner sep=0.75pt]    {$d_{i} '$};
\draw (233.59,86.83) node [anchor=north west][inner sep=0.75pt]    {$c_{i} '$};

\end{tikzpicture}

%% file: images/RI_invariance_3.tex
\begin{tikzpicture}[x=0.75pt,y=0.75pt,yscale=-1,xscale=1]

\draw  [draw opacity=0][fill={rgb, 255:red, 216; green, 216; blue, 216 }  ,fill opacity=1 ] (587.14,157.28) .. controls (560.39,149.38) and (542.39,175.38) .. (527.14,178.28) .. controls (514,175.64) and (493,158.64) .. (468.14,157.28) .. controls (454,157.64) and (425.39,160.38) .. (412.39,167.4) .. controls (412.39,180.38) and (412.39,193.13) .. (412.39,198.38) .. controls (395.89,198.84) and (396.73,199.23) .. (372.14,199.48) .. controls (371.39,189.38) and (371.39,172.38) .. (371.39,153.4) .. controls (405.39,107.38) and (511.39,95.38) .. (587.14,157.28) -- cycle ;
\draw  [draw opacity=0][fill={rgb, 255:red, 216; green, 216; blue, 216 }  ,fill opacity=1 ] (59.39,160.4) .. controls (81,121.38) and (160,93.38) .. (224.7,155.2) .. controls (215,162.38) and (210,164.38) .. (203.6,172.11) .. controls (184.8,159.75) and (141,130.38) .. (100.39,160.4) .. controls (100,179.38) and (100,173.38) .. (100,200.38) .. controls (89,195.38) and (66.39,200.38) .. (61,200.38) .. controls (59.4,194.91) and (59,168.64) .. (59.39,160.4) -- cycle ;
\draw [color={rgb, 255:red, 144; green, 19; blue, 254 }  ,draw opacity=1 ][line width=1.5]    (260.78,199.48) -- (40.32,199.48) ;
\draw [color={rgb, 255:red, 0; green, 0; blue, 255 }  ,draw opacity=1 ][line width=1.5]    (178.17,156.14) .. controls (201.78,172.46) and (219.41,176.22) .. (219.41,199.77) .. controls (219.41,223.31) and (207.37,228.96) .. (200.85,228.96) .. controls (194.34,228.96) and (182.29,224.17) .. (182.24,199.77) .. controls (182.19,175.36) and (199,175.38) .. (203.6,172.11) ;
\draw [color={rgb, 255:red, 0; green, 0; blue, 255 }  ,draw opacity=1 ][line width=1.5]    (203.6,172.11) -- (198.47,168.87) ;
\draw [color={rgb, 255:red, 0; green, 0; blue, 255 }  ,draw opacity=1 ][line width=1.5]    (587.14,157.28) .. controls (551.14,147.28) and (547.14,178.28) .. (527.14,178.28) .. controls (507.14,178.28) and (502,153.64) .. (468.14,157.28) ;
\draw [color={rgb, 255:red, 144; green, 19; blue, 254 }  ,draw opacity=1 ][line width=1.5]    (598.6,199.48) -- (348.14,199.48) ;
\draw [color={rgb, 255:red, 216; green, 216; blue, 216 }  ,draw opacity=1 ][line width=3]    (194.2,176.06) ;
\draw [color={rgb, 255:red, 0; green, 0; blue, 255 }  ,draw opacity=1 ][line width=1.5]    (59.39,160.4) -- (59.39,221.18) ;
\draw [color={rgb, 255:red, 0; green, 0; blue, 255 }  ,draw opacity=1 ][line width=1.5]    (100.39,160.4) -- (100.39,221.18) ;
\draw [color={rgb, 255:red, 0; green, 0; blue, 255 }  ,draw opacity=1 ][line width=1.5]  [dash pattern={on 5.63pt off 4.5pt}]  (59.39,160.4) .. controls (102.5,102.21) and (175.5,108.21) .. (224.7,155.2) ;
\draw [color={rgb, 255:red, 0; green, 0; blue, 255 }  ,draw opacity=1 ][line width=1.5]  [dash pattern={on 5.63pt off 4.5pt}]  (100.39,160.4) .. controls (125.5,144.21) and (145.5,142.21) .. (178.17,156.14) ;
\draw [color={rgb, 255:red, 255; green, 255; blue, 255 }  ,draw opacity=1 ][line width=3.75]    (203.6,176.11) -- (194.07,169.95) ;
\draw [color={rgb, 255:red, 0; green, 0; blue, 255 }  ,draw opacity=1 ][line width=1.5]    (207,169.38) .. controls (214,164.38) and (217,160.38) .. (224.7,155.2) ;
\draw [color={rgb, 255:red, 0; green, 0; blue, 255 }  ,draw opacity=1 ][line width=1.5]    (371.39,153.4) -- (371.39,214.18) ;
\draw [color={rgb, 255:red, 0; green, 0; blue, 255 }  ,draw opacity=1 ][line width=1.5]    (412.39,167.4) -- (412.39,214.18) ;
\draw [color={rgb, 255:red, 0; green, 0; blue, 255 }  ,draw opacity=1 ][line width=1.5]  [dash pattern={on 5.63pt off 4.5pt}]  (371.39,153.4) .. controls (414.5,95.21) and (537.95,110.29) .. (587.14,157.28) ;
\draw [color={rgb, 255:red, 0; green, 0; blue, 255 }  ,draw opacity=1 ][line width=1.5]  [dash pattern={on 5.63pt off 4.5pt}]  (412.39,167.4) .. controls (422,159.64) and (449,158.64) .. (468.14,157.28) ;
\draw    (334.22,177.69) -- (288.99,177.75) ;
\draw [shift={(336.22,177.69)}, rotate = 179.93] [color={rgb, 255:red, 0; green, 0; blue, 0 }  ][line width=0.75]    (10.93,-3.29) .. controls (6.95,-1.4) and (3.31,-0.3) .. (0,0) .. controls (3.31,0.3) and (6.95,1.4) .. (10.93,3.29)   ;

\draw (262.78,202.88) node [anchor=north west][inner sep=0.75pt]  [color={rgb, 255:red, 144; green, 19; blue, 254 }  ,opacity=1 ]  {$\Gamma $};
\draw (195.05,204.83) node [anchor=north west][inner sep=0.75pt]    {$a$};
\draw (216.8,163.69) node [anchor=north west][inner sep=0.75pt]    {$b$};
\draw (149.59,202.83) node [anchor=north west][inner sep=0.75pt]    {$c_{i} '$};
\draw (226.41,202.83) node [anchor=north west][inner sep=0.75pt]    {$d_{i} '$};
\draw (76,204.83) node [anchor=north west][inner sep=0.75pt]    {$y$};
\draw (386,202.4) node [anchor=north west][inner sep=0.75pt]    {$y$};
\draw (496.59,202.83) node [anchor=north west][inner sep=0.75pt]    {$c_{i} '*a*d_{i} '$};
\draw (600.6,202.88) node [anchor=north west][inner sep=0.75pt]  [color={rgb, 255:red, 144; green, 19; blue, 254 }  ,opacity=1 ]  {$\Gamma $};

\end{tikzpicture}